\documentclass{amsart}
\usepackage{amssymb}
\usepackage{amsthm}
\usepackage[all,cmtip]{xy}
\input diagxy
\usepackage[linktocpage=true,backref]{hyperref}
\usepackage[margin=1in]{geometry}
\usepackage{todonotes}
\usepackage{color}
\usepackage{tikz-cd}
\usepackage{comment}
\usetikzlibrary{matrix,intersections}
\pgfdeclarelayer{fg}
\pgfdeclarelayer{crossing over}
\pgfsetlayers{main,crossing over,fg}
\newtheorem{theorem}{Theorem}[section]
\newtheorem{prop}[theorem]{Proposition}
\newtheorem{corollary}[theorem]{Corollary}
\newtheorem{lemma}[theorem]{Lemma}
\newtheorem{notation}[theorem]{Notation}
\newtheorem{conjecture}[theorem]{Conjecture}
\newtheorem{assu}[theorem]{Assumption}

\newtheorem{definition}[theorem]{Definition}
\newtheorem{example}[theorem]{Example}

\numberwithin{equation}{section}

\theoremstyle{remark}
\newtheorem{remark}[theorem]{Remark}

\newcommand{\op}{\oplus}

\newcommand{\Lc}{\mathscr{L}}

\newcommand{\Cb}{\mathbb{C}}

\newcommand{\Rb}{\mathbb{R}}

\newcommand{\Det}{\T{Det}}

\newcommand{\Ker}{\T{Ker}}
\newcommand{\Coker}{\T{Coker}}

\theoremstyle{plain}

\theoremstyle{plain}

\usepackage{amscd}

\numberwithin{equation}{section}

\newcommand{\alpheqn}[1][\relax]{
     \refstepcounter{equation}
     \if#1\relax \relax
       \else \label{#1}
     \fi  
     \setcounter{saveeqn}{\value{equation}}%
    \setcounter{equation}{0}%
    \renewcommand{\theequation}{\thealphequation}}
\newcommand{\reseteqn}{\setcounter{equation}{\value{saveeqn}}%
     \renewcommand{\theequation}{\thearabicequation}}

\providecommand{\mathscr}{\mathcal} 
\IfFileExists{mathrsfs.sty}{
\usepackage{mathrsfs}}         
   {
     \IfFileExists{eucal.sty}{
        \usepackage[mathscr]{eucal}} 
   {
   }
}

\newcommand{\cd}{\cdot}

\newcommand{\clc}{\cdot\ldots\cdot}
\newcommand{\ot}{\otimes}

\newcommand{\hot}{\widehat \otimes}

\newcommand{\bop}{\bigoplus}

\newcommand{\olo}{\otimes\ldots\otimes}
\newcommand{\plp}{+ \ldots +}

\newcommand{\we}{\wedge}
\newcommand{\wlw}{\wedge\ldots\wedge}

\newcommand{\ci}{\circ}
\newcommand{\cilci}{\circ \ldots \circ}
\newcommand{\ti}{\times}
\newcommand{\nn}{\mathbb{N}}
\newcommand{\zz}{\mathbb{Z}}

\newcommand{\rr}{\mathbb{R}}
\newcommand{\cc}{\mathbb{C}}

\newcommand{\al}{\alpha}
\newcommand{\be}{\beta}
\newcommand{\ga}{\gamma}
\newcommand{\Ga}{\Gamma}
\newcommand{\de}{\delta}
\newcommand{\De}{\Delta}
\newcommand{\ep}{\varepsilon}

\newcommand{\io}{\iota}
\newcommand{\ka}{\kappa}
\newcommand{\la}{\lambda}
\newcommand{\La}{\Lambda}

\newcommand{\om}{\omega}
\newcommand{\Om}{\Omega}
\newcommand{\si}{\sigma}
\newcommand{\Si}{\Sigma}
\newcommand{\te}{\theta}
\newcommand{\Te}{\Theta}
\newcommand{\ze}{\zeta}
\newcommand{\pa}{\partial}
\newcommand{\da}{\dagger}

\newcommand{\C}[1]{\mathcal{#1}}
\newcommand{\G}[1]{\mathfrak{#1}}
\newcommand{\T}[1]{\textup{#1}}

\newcommand{\E}[1]{\emph{#1}}
\newcommand{\B}[1]{\mathbb{#1}}

\newcommand{\fork}[2]{\left\{ \begin{array}{#1} #2 \end{array} \right.} 

\newcommand{\ma}[2]{\left(\begin{array}{#1} #2 \end{array} \right)}

\newcommand{\su}{\subseteq}

\newcommand{\q}{\qquad}
\newcommand{\qq}{\qquad \qquad}
\newcommand{\qqq}{\qquad \qquad \qquad}
\newcommand{\qqqq}{\qquad \qquad \qquad \qquad}
\newcommand{\wit}{\widetilde}
\newcommand{\wih}{\widehat}

\newcommand{\inn}[1]{\langle #1 \rangle}
\newcommand{\binn}[1]{\big\langle #1 \big\rangle}

\newcommand{\sem}{\setminus}

\newcommand{\blu}{\color{blue}}

\newcommand\sL{\mathscr{L}}
\newcommand\sM{\mathscr{M}}

\makeatletter
\@namedef{subjclassname@2020}{%
  \textup{2020} Mathematics Subject Classification}
\makeatother

\title{A Higher Kac-Moody Extension for Two-Dimensional Gauge Groups}
\author{Jens Kaad, Ryszard Nest and Jesse Wolfson}

\address{Department of Mathematics and Computer Science, University of Southern Denmark,
	Campusvej 55, DK-5230 Odense M, Denmark}
\email{kaad@imada.sdu.dk}
\address{Department of Mathematics, University of Copenhagen, Universitetsparken
	5, 2100 Copenhagen, Denmark}
\email{rnest@math.ku.dk}
\address{ Department of Mathematics, University of California-Irvine, Rowland Hall 340H, Irvine, CA
92697, USA}
\email{wolfson@uci.edu}



\thanks{The first author gratefully acknowledge the financial support from the Independent Research Fund Denmark through grant no. 9040-00107B and 7014-00145B. The second author would like to acknowledge the support of the Danish National Research Foundation through the Center for symmetry and Deformations (SYM). The third author is partially supported by NSF grants DMS-1400349, DMS-1811846 and DMS-1944862.}

\begin{document}

\begin{abstract}
Let $\Gamma$ be a finite dimensional Lie group and consider the smooth double loop group, i.e. the Fr\'echet Lie group of smooth maps from the 2-torus to $\Gamma$. For a finite dimensional Hilbert space $V$, let $H$ denote the Hilbert space of vector valued $L^2$-functions on the 2-torus. The purpose of this paper is to construct a higher central extension of the smooth double loop group from the representation of the smooth double loop group on $H$ induced by a smooth action of $\Gamma$ on $V$. This higher central extension comes from an action of the smooth double loop group on a 2-category and yields a group cohomology class of degree 3 on the smooth double loop group. We show by a concrete computation that this group cohomology class is non-trivial in general. We relate our higher central extension to the Kac-Moody extension of the smooth single loop group as a higher dimensional analogue of the latter. More generally, given a group $G$ acting on a bipolarised Hilbert space, we apply higher category theory to construct a group cohomology class of degree 3 on $G$. As a second motivating example, we use these ideas to introduce a higher central extension of the group of invertible smooth functions on the noncommutative 2-torus.
\end{abstract}


\maketitle
\tableofcontents
\section{Introduction}

\subsection{Higher central extensions}
Let $\Gamma$ be a finite dimensional Lie group. The {\em smooth double loop group of $\Gamma$} is the Fr\'echet Lie group $C^\infty(\B T^2,\Gamma)$ of smooth maps from the $2$-torus to $\Gamma$. The purpose of this paper is to construct higher central extensions of the smooth double loop group by the group $\cc^*$ of invertible complex numbers. These higher central extensions yield classes in the third group cohomology of $C^\infty(\B T^2,\Gamma)$ (viewed as a discrete group) and we show that these classes are non-trivial in general.

The motivation for the constructions in this paper comes from the central extension of a group $G$ acting on a polarised Hilbert space $\C H$. For a representation $\rho \colon G \to GL(\C H)$ and an orthogonal projection $P \colon \C H \to \C H$ we define the bounded idempotent 
\[
P_g := \rho(g) P \rho(g^{-1}) \q \T{for all } g \in G \, .
\]
We recall that $P$ is a {\em polarisation} of $\rho \colon G \to GL(\C H)$ when the difference $P_g - P \colon \C H \to \C H$ is a Hilbert-Schmidt operator for all $g\in G$. The data of a polarisation of a representation determines a central extension of the group in question by the group $\cc^*$ of invertible complex numbers:
\[
1\to \cc^*\to E \to G\to 1 \, .
\]
Suppose for example that the finite dimensional Lie group $\Gamma$ acts smoothly on a finite dimensional Hilbert space $V$ and let $\C H := L^2(\B T)\otimes V$ denote the Hilbert space of vector valued $L^2$-functions on the $1$-torus. The smooth single loop group $G := C^\infty(\B T,\Gamma)$ is then represented on $\C H$ and the orthogonal projection $P$ onto the vector valued Hardy space $H^2(\B T)\otimes V \subseteq L^2(\B T)\otimes V$ gives a polarisation of this representation. The resulting central extension is the Kac-Moody extension of $C^\infty(\B T,\Gamma)$, and plays a central role in conformal field theory and quantum field theory. 

The two-dimensional analogue of a polarisation is given by the following notion of a bipolarisation:

\begin{definition}
Let $\rho\colon G\to GL(\C H)$ be a representation of a group on a separable Hilbert space. A {\em bipolarisation} of $\rho$ consists of a pair of orthogonal projections $P,Q\in \sL(\C H)$ such that the two operators 
\[
[P_g, Q_u] \, \, \, \mbox{and} \, \, \, \,
(P_g - P_h)(Q_u - Q_v) \colon \C H \to \C H
\]
are of trace class for all $u,v,g,h\in G$.
\end{definition}

The main application of our constructions is the following result (see Theorem \ref{t:grpcocyc}):

\begin{theorem}
Let $\rho\colon G\to GL(\C H)$ be a representation of a group on a separable Hilbert space. A bipolarisation of $\rho$ determines a higher central extension of $G$ by the group $\cc^*$, i.e. a crossed sequence
\[
1\to \cc^* \to N \to E \to G\to 1 \, .
\]
This extension yields a group cohomology class in $H^3(G,\cc^*)$ which is non-trivial in general.
\end{theorem}

Bipolarisations arise frequently in nature:

\begin{example}\label{ex:douloogrp}
Let $\Gamma$ be a Lie group acting smoothly on a finite dimensional Hilbert space $V$. The smooth double loop group $G := C^\infty(\B T^2,\Ga)$ can then be represented on the Hilbert space $\C H := L^2(\B T^2) \ot V$ of vector valued $L^2$-functions on the $2$-torus. Using the unitary isomorphism $L^2(\B T^2) \cong L^2(\B T) \hot L^2(\B T)$ we obtain two closed subspaces $( H^2(\B T) \hot L^2(\B T) ) \ot V \su \C H$ and $( L^2(\B T) \hot H^2(\B T) ) \ot V \su \C H$ and the corresponding orthogonal projections $P$ and $Q$ yield a bipolarisation of our representation of the smooth double loop group. In particular, we obtain an associated class in group cohomology $[c] \in H^3\big( C^\infty(\B T^2,\Ga),\cc^*\big)$ corresponding to a higher central extension of the smooth double loop group. We show by an explicit computation that this class is non-trivial in general (see Section \ref{s:nontriv}).
\end{example}

One of the advantages of our framework is that it easily adapts to a noncommutative setting:

\begin{example}
Let $\theta\in\Rb$ be an irrational number, and let $G := C^\infty(\B T_\theta^2)^*$ denote the group of smooth units of the noncommutative $2$-torus. Using the crossed product description of the noncommutative $2$-torus we may represent the group $G$ on the Hilbert space $\C H := L^2(\B T) \hot \ell^2(\zz)$. The Hilbert space $\C H$ contains the closed subspaces $H^2(\B T) \hot \ell^2(\zz)$ and $L^2(\B T) \hot \ell^2(\nn_0)$ and the corresponding orthogonal projections $P$ and $Q$ yield a bipolarisation of our representation of the smooth units of the noncommutative $2$-torus. In particular, we obtain an associated class in group cohomology $[c] \in H^3\big( C^\infty(\B T_\theta^2)^*,\cc^*\big)$ corresponding to a higher central extension of the smooth units of the noncommutative $2$-torus.
\end{example}

The group cohomology classes constructed in the present text can detect information about the singular homology (with integer coefficients) of the classifying space $BG^\delta$ where the superscript ``$\delta$'' signifies that the group $G$ is equipped with the discrete topology. Indeed, one may identify the singular homology of $BG^\delta$ with the group homology of the group $G$. In the above examples, the group $G$ can as well be considered as a topological group in a natural way (e.g. using that $C^\infty(\B T^2)$ is a Fr\'echet space) and we therefore also have the classifying space $BG^{\T{top}}$, where we are taking this topology into account. Explicit computations show that the origin of our group cohomology classes is ``not topological'' in the sense that they do not in general arise from singular cohomology classes in $H^\ast(BG^{\T{top}},\cc^*)$ via pullback along the comparison map $BG^{\delta} \to BG^{\T{top}}$.

In the case where the group $G = GL(R)$ agrees with the general linear group over a unital ring $R$ we obtain a class in group cohomology $[c] \in H^3( GL(R),\cc^*)$ from a bipolarisation of a representation of $GL(R)$ on a separable Hilbert space. We may use this class to obtain numerical information about the third algebraic $K$-group of $R$ by means of the Hurewicz homomorphism $K_3^{\T{alg}}(R) \to H_3(GL(R),\zz)$ and the pairing between group cohomology and group homology $H^3(GL(R),\cc^*) \ti H_3(GL(R),\zz) \to \cc^*$. The above examples fit in this context with the unital ring $R$ being either the smooth functions on the $2$-torus $C^\infty(\B T^2)$ or the smooth functions on the noncommutative $2$-torus $C^\infty(\B T_\theta^2)$. In particular, we obtain numerical invariants $K_3^{\T{alg}}(C^\infty(\B T^2)) \to \cc^*$ and $K_3^{\T{alg}}(C^\infty(\B T_\theta^2)) \to \cc^*$ and in Section~\ref{s:nontriv} we present concrete computations showing that the first of these invariants is in fact non-trivial.
\medskip

From the algebraic point of view, one may consider the formal double loops into $\cc$, namely the field $\cc((t))((s))$ (which is also a $2$-Tate space). The automorphism group of $\cc((t))((s))$ is an interesting object and in the papers \cite{ArKr:GTS,FrZh:GRD}, the construction of a group $3$-cocycle on this automorphism group is carried out. This group $3$-cocycle arises from the explicit description of $2$-category theoretic data and the corresponding character on the algebraic $K$-theory of the field $\cc((t))((s))$ is a particular case of the characters described in \cite{BrGrWo:GCR,OsZh:TCS,GoOs:HCL}. 
%

The difference between the present work and the above constructions of group $3$-cocycles can (to some extent) be clarified by considering the difference between the \emph{formal} double loops into $\cc$, $\cc((t))((s))$, and the \emph{smooth} double loops into $\cc$, $C^\infty(\B T^2)$. This passage from the formal setting to the smooth setting is the cause of a whole range of analytic problems which we take up in this paper. At the level of linear operators we are no longer working with infinite matrices subject to vanishing conditions on the entries but with bounded operators on Hilbert spaces. Similarly, we are systematically replacing finite rank operators with trace class operators on Hilbert spaces and their corresponding analytically defined Fredholm determinants. As a consequence, there seems to be no correct notion of lattices in our context: finite dimensionality conditions on (quotient) subspaces are not adequate for dealing with the presence of infinitely many eigenvalues subject to decay conditions. Lattices in $2$-Tate spaces form a core ingredient in the approach developed in \cite{ArKr:GTS,FrZh:GRD}. 

We introduce here an entirely new machinery which allows us to build group actions on $2$-categories from representations of groups on Hilbert spaces. Our ideas are related to the program of A. Connes in noncommutative geometry and in particular to the notion of finitely summable Fredholm modules and their link to analytic $K$-homology, Kasparov's $KK$-theory and index pairings, \cite{Con:NDG,Con:NCG,HiRo:AKH,Kas:OFE}. The two idempotents appearing in our definition of a bipolarisation are thought of as yielding two classes in $KK$-theory such that their Kasparov product yield a $3$-summable Fredholm module. The associated index pairings with the $K$-theory classes represented by elements in the group $G$ then yield actual Fredholm operators and their determinants provide us with the correct definition of the $2$-morphisms in our $2$-category. 

One may thus view the present paper as an attempt to align certain index theoretic constructions pertaining to noncommutative geometry with the framework of higher category theory. 
\medskip

We now explain the basic category theoretic ideas involved in the present approach to higher central extensions.

\subsection{Groups acting on higher categories and group cocycles}
Our construction of higher group cocycles is based on obstruction theory and mathematical physics. Namely, to produce an $(n+1)$-cocycle on a group, one constructs an $n$-category $\G C$ with a group $G$ acting on $\mathfrak{C}$. Provided that the $n$-category is sufficiently connected, this data yields an $(n+1)$-cocycle on the group $G$ with values in the automorphism group of an invertible $(n-1)$-morphism. 
%

The case of a category and an associated $2$-cocycle is fairly standard and will be described in detail in Section \ref{s:2coc}. For readers who are less familiar with notions of higher categories, let us explain the case of a group $G$ acting on a (sufficiently connected) $2$-category $\G C$. 

We denote the composition of $1$-morphisms and the horizontal composition of $2$-morphisms by $\ci_1$ and we denote the vertical composition of $2$-morphisms by $\ci_2$. The recipe is now as follows:
\begin{enumerate}
\item Choose an object $x$ in $\G C$.
\item Choose a $1$-isomorphism $\al_g \colon x \to g(x)$ for each element $g \in G$.
\item Choose a $2$-isomorphism $\be_{g,h} \colon g(\al_h) \ci_1 \al_g \to \al_{gh}$ for each pair of elements $g,h \in G$.
\end{enumerate}
For each triple of elements $g,h,k \in G$ we then have two different $2$-isomorphisms relating the $1$-morphisms $(gh)(\al_k) \ci_1 g(\al_h)\ci_1 \al_g$ and $\al_{ghk} \colon x \to (ghk)(x)$. These two $2$-isomorphisms are given by
\[
\be_{gh,k} \ci_2 \big( \T{id}_{(gh)(\al_k)} \ci_1 \be_{g,h}\big) \, \, \T{and} \, \, \, 
\be_{g,hk} \ci_2 \big(g(\be_{h,k}) \ci_1 \T{id}_{\al_g}\big) \colon (gh)(\al_k) \ci_1 g(\al_h)\ci_1 \al_g \to \al_{ghk} \, .
\]
The quotient of these $2$-isomorphisms yields an automorphism of the $1$-isomorphism $\al_{ghk} \colon x \to (ghk)(x)$ and this is exactly the value of our $3$-cocycle on the triple $(g,h,k)$ of group elements.

The structure which arises naturally in the context of bipolarisations of group representations is slightly different from a weak $2$-category. From our point of view there is no canonical way of defining the horizontal composition of $2$-morphisms so we are instead replacing this operation with a family of ``coproduct'' functors. These coproduct functors allow us to decompose any $2$-morphism $\be \colon \la \to \mu$ along a specified object $z$. In our context, these coproduct functors are in fact equivalences of $1$-categories and a choice of inverses (up to natural isomorphisms) provides us with a ``twisted'' weak $2$-category. The twisting arises from our systematic use of \emph{graded} tensor products of categories which, in turn, is dictated by the concrete analytic examples which we are investigating.  While we are unable to formulate a more precise statement at present, note that this coproduct structure is reminiscent of some of the challenges of Segal's approach to conformal field theory \cite{Seg:CFT}, where it is straightforward to cut a Riemann surface along a separating curve, but much more delicate to try to glue two Riemann surfaces with boundary together.

In this paper we provide a detailed description of group $3$-cocycles in the context of groups acting on coproduct categories. In particular, we explicitly verify the cocycle property and that the corresponding class in group cohomology is independent of various choices, see Section \ref{s:polar}. We record that the twisted nature of the associated weak $2$-category implies that our group $3$-cocycle more naturally takes values in $\cc^*/\{\pm 1\}$, but we may of course square it and obtain a group $3$-cocycle with values in $\cc^*$ as indicated earlier on in the introduction.

\subsubsection{Polarisations, categories and extensions}
The main work done in this paper concerns the construction of a twisted weak $2$-category with a $G$-action, given the data of a bipolarisation of a representation $\rho\colon G\to GL(\C H)$. As motivation, we now describe the $1$-category with a $G$-action associated to a polarisation of $\rho$. It is instructive to keep in mind the analogy between the present analytic constructions and the algebraic constructions carried out in \cite{ArCoKa:IRR,KaPe:SWI}. In fact, one may think of all the various idempotents $gPg^{-1}$, $g \in G$, as having ``commensurable'' images. Strictly speaking this is of course only the case when all the differences $P - gPg^{-1}$ are of finite rank and not just Hilbert-Schmidt. For pairs of group elements $g,h \in G$, the symbol $\big( \T{Im}(hPh^{-1}), \T{Im}(gPg^{-1})\big)$ (which is in general not well-defined in our setting) is then replaced by the (graded) determinant of the Fredholm operator $hPh^{-1} \cd gPg^{-1}$. The analogy with determinantal theories for lattices in $1$-Tate spaces as introduced in \cite{Kap:SSP} should now be apparent as well, see also \cite{FrZh:GRD} for further comments on this relationship.

We start out by investigating the universal setting for the ``one-dimensional story'':

Let $P \colon \C H \to \C H$ be an orthogonal projection with infinite dimensional kernel and cokernel. Consider the universal polarised representation, i.e. the {\em restricted general linear group}
\[
GL_{\T{res}}(\C H) := \big\{g\in GL(\C H)\mid P - gPg^{-1} \T{ is Hilbert-Schmidt} \big\} .
\]
Clearly, if $P$ determines a polarisation of $\rho \colon G \to GL(\C H)$ then the representation $\rho$ factors through a group homomorphism $\rho \colon G \to GL_{\T{res}}(\C H)$ and this links the universal setting to the particular setting. 

For any $g\in GL_{\T{res}}(\C H)$, we define the bounded idempotent $P_g :=gPg^{-1}$. For any two elements $g$ and $h$ of $GL_{\T{res}}(\C H)$, it then follows from Atkinson's theorem that the composition
\[
P_hP_g \colon P_g \C H \to P_h \C H
\]
is a Fredholm operator. Determinants of Fredholm operators play a key role in this paper, and in Section \ref{s:torsfred} and Section \ref{s:pert} we elaborate on the main constructions needed to describe the composition in the category here below. Notably, the condition that $P - P_g$ be Hilbert-Schmidt for each $g \in GL_{\T{res}}(\C H)$ and not just compact turns out to be important, since the composition uses the Carey-Pincus perturbation isomorphism as introduced in \cite{CaPi:PV}.

\begin{definition}\label{cat:polarisation}
Let $\G C_{\T{res}}$ be the following category with an action of $GL_{\T{res}}(\C H)$:
\begin{enumerate}
\item The objects in $\G C_{\T{res}}$ are given by the set of elements in the restricted general linear group $GL_{\T{res}}(\C H)$;
\item Given two objects $g,h \in GL_{\T{res}}(\C H)$, the morphisms from $g$ to $h$ are given by the graded determinant line 
\[
\Det (P_hP_g) = \Big( \Lambda^{\T{top}}\Ker(P_h P_g) \ot \Lambda^{\T{top}}\Coker(P_hP_g)^* , \T{Index}(P_h P_g) \Big) \, ;
\]
\item The restricted general linear group $GL_{\T{res}}(\C H)$ acts on $\G C_{\T{res}}$ by left multiplication on objects and by conjugation on morphisms.
\end{enumerate}
The composition of morphisms in $\G C_{\T{res}}$ is given by a combination of torsion isomorphisms and perturbation isomorphisms in the context of determinants of Fredholm operators (see Section~\ref{subsection:2-cocycle} for details). 
\end{definition}

The group $2$-cocycle on $GL_{\T{res}}(\C H)$ with values in $\cc^*$ constructed using this structure (see Section \ref{s:2coc}) appears in many places under different names. Its infinitesimal version is often called the ``Japanese cocycle'', while its global version plays a prominent role in conformal field theory. One of the reasons for its importance is that it is responsible for the central extensions of loop groups. First of all, the group $2$-cocycle on $GL_{\T{res}}(\C H)$ yields a central extension
\[
1 \to \cc^* \to E \to GL_{\T{res}}(\C H) \to 1  \, .
\]
Now let $\Gamma$ be a Lie group acting smoothly on a finite dimensional Hilbert space $V$ (not necessarily by unitary operators). As described earlier on, we then have a polarisation of the action of the smooth single loop group $G := C^\infty(\B T,\Gamma)$ on the Hilbert space $\C H := L^2(\mathbb{T},V)$. In particular, this polarisation yields a group homomorphism $\rho \colon C^\infty(\B T, \Gamma) \to GL_{\T{res}}(\C H)$. The central extension of the loop group $C^\infty(\B T,\Gamma)$ appearing in conformal field theory is the pull back along $\rho$ of the above central extension of $GL_{\T{res}}(\C H)$, \cite[Chapter 6]{PrSe:LG}. 

Another avatar of the group $2$-cocycle on $GL_{\T{res}}(\C H)$ turns out to be the low-dimensional version of the multiplicative character of Connes and Karoubi, \cite{CoKa:CMF}. To explain the link to Connes and Karoubi, remark that the orthogonal projection $P \colon \C H \to \C H$ yields a $2$-summable Fredholm module $\C F_{\T{uni}} = (\C M^1, \C H, 2 P - 1)$ where the unital $*$-algebra in question is defined by 
\[
\C M^1 := \big\{ T \in \sL(\C H) \mid [P,T] \T{ is Hilbert-Schmidt} \big\} \, .
\]
The restricted general linear group $GL_{\T{res}}(\C H)$ is exactly the group of invertible elements in $\C M^1$. In the one-dimensional case, the universal multiplicative character is a numerical invariant of the second algebraic $K$-group denoted by $M(\C F_{\T{uni}}) \colon K^{\T{alg}}_2(\C M^1) \to \cc^*$. The following holds:

\begin{theorem}\label{intro:two cocycle}
Let $[c_{\T{res}}] \in H^2(GL(\C M^1),\cc^*)$ denote the group $2$-cocycle associated to the group action of $GL(\C M^1)$ on a stabilised version of $\G C_{\T{res}}$. Then the diagram
\[
\xymatrix{
K^{\T{alg}}_2(\C M^1)\ar[rr]_h \ar[d]_{M(\C F_{\T{uni}})} && H_2(GL(\C M^1),\zz)\ar[dll]^{\, \, \, \, \inn{ [c_{\T{res}}], \cd}} \\
\cc^* && 
}
\]
commutes, where $h$ is the Hurewicz homomorphism and $\inn{ [c_{\T{res}}], \cd}$ comes from the pairing between group cohomology and group homology.
\end{theorem}

Let us briefly return to the particular case where the Hilbert space $\C H$ agrees with $L^2(\mathbb{T})$ and the orthogonal projection $P$ is the projection onto Hardy space. The smooth functions on the $1$-torus acts on $L^2(\mathbb{T})$ as multiplication operators and this action factors through $\C M^1$. Applying the functoriality of algebraic $K$-theory and composing with the universal multiplicative character, we obtain a numerical invariant $M(\C F_{\B T}) \colon K_2^{\T{alg}}(C^\infty(\B T)) \to \cc^*$. This numerical invariant extends the Tate tame symbol of pairs of meromorphic functions on a Riemann surface, see \cite{Del:SM,CaPi:JTH,Kaa:CSA}. In fact, for any pair of invertible smooth functions $u,v \colon \B T \to \cc$ we may form the Steinberg symbol $\{u,v\} \in K_2^{\T{alg}}(C^\infty(\B T))$ and the value $M(\C F_{\B T})( \{u,v\}) \in \cc^*$ is given explicitly by the formula
\[
\exp\big( \frac{1}{2\pi i} \int_{[0,1]} p^*( \frac{\log(u)}{v} d v)  \big) \cd v(1)^{-w(u)} ,
\]   
where $w(u) \in \zz$ denotes the winding number and $p \colon [0,1] \to \B T$ is defined by $p(t) = e^{2 \pi i t}$.

\subsubsection{$2$-categories associated to representations of a ring}\label{s:2category}
To extend the above to a higher dimensional setting, we proceed as follows. Let $R$ be a unital ring and let $I \su R$ be an ideal.

Suppose that $\{ \pi_\la\}_{\la \in \La}$ is a non-empty family of (not necessarily unital) representations of the ring $R$ as bounded operators on a separable Hilbert space $\C H$ satisfying the following conditions:

\begin{assu}\label{a:assurep}\mbox{}
\begin{enumerate}
\item for all $x\in R$ and pairs $\la,\mu \in \La$ the difference
\[
\pi_\la(x)\pi_\mu(1)-\pi_\la(1)\pi_\mu(x)
\]
is a trace class operator;
\item for all $i\in I$ and triples $\la,\mu,\nu \in \La$ the difference
\[
\pi_\la(i)\pi_\mu(1)\pi_\nu(1)-\pi_\la(i)\pi_\nu(1)
\]
is a trace class operator.
\end{enumerate}
\end{assu}

Let us for a moment fix two indices $\la,\mu \in \La$ and explain the link between the above assumptions and relative analytic $K$-homology. The fact that our conditions really involve triples of indices has to do with the existence of the composition in the categories which we are going to explain in a little while. For the moment we define the $\zz/2\zz$-graded separable Hilbert space
\[
\C H(\la,\mu) := \pi_\la(1) \C H \op \pi_\mu(1) \C H \, , 
\]
where the first component is even and the second component is odd. We equip this Hilbert space with the even representation
\[
\pi(\la,\mu) := \pi_\la \op \pi_\mu \colon R \to \C L\big( \pi_\la(1) \C H \op \pi_\mu(1) \C H \big) \, ,
\]
where we consider the unital ring $R$ to be trivially graded. We then define the odd bounded operator 
\[
F(\la,\mu) := \ma{cc}{0 & \pi_\la(1) \pi_\mu(1) \\ \pi_\mu(1) \pi_\la(1) & 0} 
\colon \pi_\la(1) \C H \op \pi_\mu(1) \C H \to \pi_\la(1) \C H \op \pi_\mu(1) \C H \, .
\]
With these definitions we record that our conditions $(1)$ and $(2)$ mean that the operators
\[
[ F(\la,\mu) , \pi(\la,\mu)(x) ] \q \T{and} \q \pi(\la,\mu)(i)\big( F(\la,\mu)^2 - \T{id}_{\C H(\la,\mu)} \big)
\]
are of trace class for all $x \in R$ and all $i \in I$. Comparing with \cite[Definition 8.5.1]{HiRo:AKH} we are then justified in viewing the triple
\[
\big(\pi(\la,\mu), \C H(\la,\mu), F(\la,\mu) \big)
\]
as an even relative Fredholm module with respect to the ideal $I$ inside the ring $R$. This even relative Fredholm module is then subject to the extra summability/dimensionality constraint imposed by replacing the compact operators on the Hilbert space $\C H(\la,\mu)$ by the trace class operators, see \cite{Con:NDG,Con:NCG} for more details on these matters. Remark also that $F(\la,\mu)$ becomes self-adjoint if we impose the extra condition that $\pi_\la(1)$ and $\pi_\mu(1)$ be orthogonal projections. This condition is however too restrictive for our present investigations (e.g. we are working with the whole group of invertible elements in $C^\infty(\B T^2)$ and not just the smooth functions with values in the circle).

Let us now also fix two idempotents $p,q \in R$ and assume that their difference $p-q$ belongs to the ideal $I$. This data then yields a class in the even relative $K$-theory with respect to the ideal $I$ inside the ring $R$. 

We recall that there is an index pairing between even relative $K$-theory and even relative $K$-homology with values in the group of integers. Thus, from our current data we know how to produce an integer
\[
\inn{ (p,q) , (\la,\mu) } \in \zz \, .
\]
This integer does in fact arise as the index of an explicit Fredholm operator acting between Hilbert spaces and we denote this Fredholm operator by
\[
F(\la,\mu)(p,q) \colon \pi_\la(p) \C H \op \pi_\mu(q) \C H \to \pi_\la(q) \C H \op \pi_\mu(p) \C H \, . 
\]
In our construction of a (twisted and weak) $2$-category we are not merely interested in this integer but instead in the graded determinant line associated to the Fredholm operator $F(\la,\mu)(p,q)$ so that the degree agrees with the corresponding index. Indeed, this graded determinant line agrees with the $2$-morphisms between the $1$-morphisms $\la, \mu$ (the source being $\la$ and the target $\mu$) acting between the objects $p$ and $q$ (here with source $p$ and target $q$). 

Still fixing the pair of idempotents $p,q \in R$ with difference in the ideal $I$, we obtain a $1$-category $\G L(p,q)$ defined as follows: 

\begin{enumerate}
\item The objects in $\G L(p,q)$ are elements of the index set $\La$;
\item Given two objects $\la,\mu \in \La$, the morphisms from $\la$ to $\mu$ are given by the graded determinant line 
\[
\Det \big(F(\la,\mu)(p,q)\big) \, .
\]
\end{enumerate}
As it happens for the category $\G C_{\T{res}}$ the composition in $\G L(p,q)$ is again given by a combination of torsion isomorphisms and perturbation isomorphisms (see Section~\ref{s:category} for details). The definition of this composition and the proof of associativity is however considerably more involved due to the complicated nature of the Fredholm operators $F(\la,\mu)(p,q)$ (see Section~\ref{s:proofcat}).

In order to incorporate group actions we suppose that our discrete group $G$ acts on the Hilbert space $\C H$, on the ring $R$ (with the ideal $I$ as an invariant subset) and on the index set $\La$. Imposing the equivariance condition
\[
g( \pi_\la(x) ) = \pi_{g(\la)}(g(x)) \q g \in G \, , \, \, \la \in \La \, , \, \, x \in R
\]
we then obtain isomorphisms of categories
\[
g \colon \G L(p,q) \to \G L(g(p), g(q)) \q g \in G 
\]
which are compatible with the group laws.

It turns out to be necessary to replace the collection of categories $\G L$ with a slightly different collection of categories which we label $\G H$. The definition of $\G H$ depends on the choice of a basepoint meaning that we choose an idempotent $p_0 \in R$. This choice of basepoint has to be compatible with the group action in the sense that $g(p_0) - p_0 \in I$. The categories appearing in the collection $\G H$ are then going to be the categories of morphisms between the objects in the associated twisted weak $2$-category. The objects are defined as follows
\[
\T{Obj}(\G H) := \big\{ p \in R \mid p \T{ idempotent and } p - p_0 \in I \big\}
\]
and for each pair of objects $p,q \in \T{Obj}(\G H)$ we define the category
\[
\G H(p,q) := \G L(p,p_0) \ot \G L(q,p_0)^\da \, ,
\]
where we apply a graded tensor product of categories and where the superscript ``$\da$'' means dual category.

The next natural step to take for constructing a twisted weak $2$-category with an action of $G$ out of our data would be to specify composition functors
\[
\G H(p,r) \ot \G H(r,q) \to \G H(p,q) \, ,
\]
where the twisting signifies that we are applying a graded tensor product of our categories and not just the cartesian product (which would correspond to the standard definition of a weak $2$-category). In our framework it turns out to be more convenient to construct coproduct functors instead of composition functors meaning that we instead construct coassociative functors
\[
\De_r \colon \G H(p,q) \to \G H(p,r) \ot \G H(r,q)
\]
for every object $r \in \T{Obj}(\G H)$. Contrary to the composition functors, which only satisfy twisted pentagon identities up to natural isomorphisms, our coproduct functors are in fact strictly coassociative. We refer to the collection of categories $\G H$ endowed with the collection of coproduct functors $\De$ as a \emph{coproduct category}.

The choice of a basepoint $p_0$ implies that the coproduct category $(\G H,\De)$ does not at first sight admit an action of the group $G$ (unless $p_0$ happens to be a fixed point for the group action on $R$). We prove in the present text that the coproduct category $(\G H,\De)$ is independent of the choice of basepoint up to isomorphism of coproduct categories. More specifically we show that for any alternative idempotent $p_0' \in R$ with $p_0 - p_0' \in I$ we have an isomorphism of coproduct categories
\[
\G B(p_0,p_0') \colon (\G H,\De) \to (\G H', \De')
\]
which we refer to as the \emph{change-of-base-point isomorphism} (here $(\G H',\De')$ denotes the coproduct category constructed using the idempotent $p_0'$ as the base point).

The main result of the present paper can now be stated as follows:

\begin{theorem}\label{t:intromain} The pair $(\G H,\De)$ can be given the structure of a coproduct category with an action of the group $G$ and this data satisfies the conditions for the construction of a group $3$-cocycle on $G$ with values in $\cc^*$. Moreover, an explicit computation for the case where $G$ equals the smooth double loop group $C^\infty(\B T^2)^*$ shows that this group $3$-cocycle is non-trivial in general (see Section~\ref{s:nontriv}).
\end{theorem}

Our principal example of a setting where the above theorem applies is as follows. Suppose that we are given a bipolarisation $(P,Q)$ of a representation $\rho \colon G \to GL(\C H)$. Let $R :=\inn{q_u ; u\in G}$ be the unital ring which is freely generated by idempotents $\{q_u\}_{u\in G}$ and let $I$ be the augmentation ideal defined as the kernel of the unital homomorphism $R \to R$ sending all $q_u$ to $q_e$ (with $e \in G$ being the neutral element). The fact that we are considering a bipolarisation allows us to show that every bounded operator of the form $P_g Q_u P_g$ is a trace class perturbation of an idempotent on the Hilbert space $P_g \C H$. In this fashion we obtain families of representations of $R$ parametrised by $g \in G$ and these families satisfy our two main conditions $(1)$ and $(2)$ detailed out under Assumption \ref{a:assurep}. With a little bit of extra care, we may also equip all our data with an action of $G$ and Theorem \ref{t:intromain} thereby applies and produces a group $3$-cocycle on the group $G$.


\begin{remark}
In distinction to the one-dimensional case, it is not clear if there exists a universal bipolarised representation, or equivalently a unique maximal group of bounded invertible operators on which the group $3$-cocycle is defined. 
\end{remark}


\subsection{Relation to the Tate tame symbol and the multiplicative character}
As remarked above, the first application of the results in this paper is to the algebra of smooth functions on the two-torus, $C^\infty (\mathbb{T}^2)$. The Hilbert space is $L^2(\mathbb{T}^2)$ and, given $\xi=\sum_{(m,n)\in \mathbb{Z}^2}\mu_{(m,n)} z_1^mz_2^n\in L^2(\mathbb{T}^2)$, we define the orthogonal projections $P,Q \in \sL\big( L^2(\mathbb{T}^2) \big)$ by
\begin{equation}\label{eq:bipolarisation of two-torus}
P\xi=\sum_{m\geq 0}\sum_{n\in \mathbb{Z}} \mu_{(m,n)} z_1^mz_2^n \, \, \mbox{ and } \,\,\,
Q\xi=\sum_{m\in \mathbb{Z}}\sum_{n\geq 0}\mu_{(m,n)} z_1^mz_2^n \, .
\end{equation}
The group $G := C^\infty(\B T^2)^*$ can be taken to be the group of smooth invertible functions on the $2$-torus acting by multiplication on $L^2(\B T^2)$. The corresponding $3$-cocycle $[c] \in H^3(C^\infty(\B T^2)^*,\cc^*)$ is non-trivial. In fact, for a constant function $\lambda\in \mathbb{C}^*$ we shall see that
\[
\inn{ [c], \{ z_1,z_2,\lambda \} } = \lambda^2 \, ,
\]
where $\{ z_1, z_2, \lambda \} \in H_3( C^\infty(\B T^2)^*, \zz)$ is the class of the group $3$-cycle
\[
(z_1, z_2, \la) - (z_1, \la, z_2) + (\la, z_1, z_2) - (\la, z_2, z_1) + (z_2,\la, z_1) - (z_2, z_1, \la)
\in \zz[ G^3 ] \, .
\]
This explicit computation of the pairing between group cohomology and group homology, the analogy with the one-dimensional setting as well as the situation for the formal double loop group $\cc((t))((s))$ lead us to the conjecture here below. The integral formula appearing should be compared with \cite[Theorem 2.7]{BrMc:MRL}, where this kind of integral formula is related to the product structure in Deligne cohomology, \cite{Bei:HRV,EsVi:DBC}:
	
\begin{conjecture}\label{intro:3-Beilinson}
 Suppose that $f$, $g$ and $h$ are smooth invertible functions on $\mathbb{T}^2$. Then
\[
\begin{split}
\inn{ [c], \{f,g,h\}} & =\exp \left( \frac{2}{(2\pi i)^2} \cd \int_{[0,1]^2} p^*\big(  \frac{\log(f)}{g \cd h} dg \we dh \big) \right) \\
& \qq \cd \exp\left( -\frac{w_1(f)}{\pi i} \cd \int_{ \{1\} \ti [0,1] } p^*\big( \frac{\log(g)}{h} dh \big)
+ \frac{w_2(f)}{\pi i} \cd \int_{ [0,1] \ti \{1\}} p^*\big( \frac{\log(g)}{h} dh  \big) \right) \\
& \qqq \cd h(1,1)^{ 2 w_1(f) w_2(g) - 2 w_1(g) w_2(f)} \, ,
\end{split}
\]
where $p \colon  [0,1]^2 \to \B T^2$ is the smooth map $p(t,s) := (\exp(2 \pi i \cd t), \exp(2 \pi i \cd s))$ and where $w_1(f) \, , \, \, w_2(f) \in \zz$ are the winding numbers of the smooth invertible functions on the circle $z \mapsto f(z,1)$ and $z \mapsto f(1,z)$, respectively.
\end{conjecture}

While this is not the language we adopt in this paper, much of our construction of the twisted weak $2$-category $\G H$ can be understood as an attempt to ``categorify'' constructions relating to the $K$-theory of operator algebras. Let $n \in \nn$ be fixed. Suppose that $\C A$ is an algebra over $\cc$, $\pi\colon \C A\to \sL(\C H)$ a representation of $\C A$ by bounded operators on a separable infinite dimensional Hilbert space $\C H$ and $F\in \sL(\C H)$ is a self-adjoint unitary satisfying that
\[
[F,\pi (a)]\in \mathscr{L}^n(\C H)  \q \forall a \in \C A \, ,
\]
where $\mathscr{L}^n(\C H)$ denotes the $n^{\T{th}}$-Schatten ideal. If $n$ is odd, we assume that $\C H$ is $\mathbb{Z}/2\mathbb{Z}$ graded, $F$ is odd and $\pi(a)$ is even for all $a \in \C A$. In this case, we denote the grading operator by $\ga \colon \C H \to \C H$. We will moreover assume that the spectral subspaces for $F$ (resp. $\ga$) are infinite dimensional when $n$ is even (resp. odd).
	
The Connes-Chern character of the $n$-summable Fredholm module $\C F := (\C A,\C H,F)$ is the class in cyclic cohomology $\T{Ch}(\C F) \in HC^{n-1}(\C A)$ given by the cyclic cocycle
\begin{equation}\label{eq:chern character}
\T{Ch}(\C F)(a_0,\ldots,a_{n-1}) = \frac{1}{2}\T{Tr}(\ga^n F[F,\pi(a_0)][F,\pi(a_1)]\ldots[F,\pi(a_{n-1})]) \q
a_0,\ldots,a_{n-1} \in \C A \, ,
\end{equation}
where $\ga^n := \T{id}_{\C H}$ for $n$ even and where $\T{Tr} \colon \mathscr{L}^1(\C H) \to \cc$ is the operator trace, \cite{Con:NDG,Con:NCG}.
	
The pattern for secondary invariants of algebraic $K$-theory is as follows. Let $\C F_{\T{uni}} := \big( \mathcal{M}^{n-1}, \C H, F_{\T{uni}} \big)$ be the universal $n$-summable Fredholm module. Then the topological $K$-theory is given by $K_n^{\T{top}}(\C M^{n-1})=0$, $K^{\T{top}}_{n+1}(\C M^{n-1})=\mathbb{Z}$ and we have a commuting diagram with exact rows:
\[
\xymatrix{
& \mathbb{Z}\ar[r]\ar@{=}[dd] & K^{\T{rel}}_{n}({\mathcal{M}^{n-1}})\ar[d]^{\T{Ch}^{\T{rel}}}\ar[r]&K_n^{\T{alg}}({\mathcal{M}^{n-1}})\ar[r]
\ar[dd]^{\T{M}(\C F_{\T{uni}})}\ar[r] &0  \\
& & HC_{n-1}({\mathcal{M}}^{n-1}) \ar[d]^{\T{Ch}(\C F_{\T{uni}})} & & \\
&\mathbb{Z}\ar[r]_{(2 \pi i)^{\lceil \frac{n}{2} \rceil}}&\mathbb{C}\ar[r]&\mathbb{C}/(2\pi i)^{\lceil \frac{n}{2}\rceil}\mathbb{Z}
}
\]
The right-most arrow is the universal Connes-Karoubi \emph{multiplicative character} and the Chern character $\T{Ch}^{\T{rel}}$ is the \emph{relative Chern character} which is defined on the relative $K$-theory of the Banach algebra $\C M^{n-1}$, \cite{Kar:HCK}. The cyclic homology group appearing is constructed using the projective tensor product of Banach algebras.

Any $n$-summable Fredholm module $\C F := (\C A,\C H,F)$ over $\C A$ yields an algebra homomorphism
\[
\pi\colon \C A\to \mathcal{M}^{n-1} \, \, \T{with} \q \pi^*( \C F_{\T{uni}}) = \C F
\]
(at least when the relevant spectral subspaces are infinite dimensional) and hence, using the functoriality of algebraic $K$-theory, produces the group homomorphism called the Connes-Karoubi \emph{multiplicative character}, \cite{CoKa:CMF}:
\[
\T{M}( \C F) \colon K^{\T{alg}}_n(\C A) \to \cc/ (2 \pi i)^{\lceil \frac{n}{2} \rceil} \zz \, .
\]
On the other hand, by the definition of the algebraic $K$-group $K_*^{\T{alg}}(\C A)$, we have the Hurewicz homomorphism
\[
K^{\T{alg}}_n(\C A)\to H_n(GL(\C A),\mathbb{Z})
\]
and the natural question is whether the multiplicative character of an $n$-summable Fredholm module factorises via group homology, i.e. whether the dotted arrow below exists making the diagram commute:
\[
\xymatrix{
K^{\T{alg}}_n(\C A)\ar[rr]\ar[d]_{\T{M}(\C F)}&& H_n(GL(\C A),\mathbb{Z})\ar@{-->}[dll]\\
\mathbb{C}/(2 \pi i)^{\lceil \frac{n}{2} \rceil} \zz && \, .
}
\]
	
In the case of the $2$-torus we obtain a $3$-summable Fredholm 
\[
\C F_{\B T^2} = \big( C^\infty(\B T^2 ), L^2(\B T^2) \op L^2(\B T^2), F_{\B T^2} \big)
\]
where the self-adjoint unitary operator $F_{\B T^2}$ is obtained by taking the phase of the Dirac operator
\[
D := \ma{cc}{0 & i \pa/\pa \te_1 - \pa/ \pa \te_2 \\ i \pa / \pa \te_1 + \pa/\pa \te_2 & 0 } \colon C^\infty(\B T^2,\cc^2) \to C^\infty(\B T^2,\cc^2) \, .
\]
Comparing with Conjecture \ref{intro:3-Beilinson} and considering the explicit computations of the multiplicative character carried out in \cite{Kaa:CMC} lead us to the following:
	
\begin{conjecture}\label{intro:3-Connes-Karoubi}
Consider the $3$-summable Fredholm module $\C F_{\B T^2}$ and let $[c] \in H^3\big( GL(C^\infty(\mathbb{T}^2)),\cc^* \big)$ denote the class of the group $3$-cocycle on the group of invertible matrices over $C^\infty(\mathbb{T}^2)$ constructed using the bipolarisation from \eqref{eq:bipolarisation of two-torus}. Then the diagram
\[
\xymatrix{
K^{\T{alg}}_3\big(C^\infty(\mathbb{T}^2)\big)\ar[r]\ar[d]_{\frac{1}{\pi i} \cd M(\C F_{\B T^2})} & H_3\big(GL(C^\infty(\mathbb{T}^2)),\mathbb{Z}\big) \ar[d]^{\inn{[c],\cdot}}\\
\cc/ (2 \pi i) \zz \ar[r]_{\exp} & \mathbb{C}^*
}
\]
commutes.
\end{conjecture}

We emphasise that the spectral triple $(C^\infty(\B T^2), L^2(\B T^2) \op L^2(\B T^2), D)$ coming from the Dirac operator on the $2$-torus (and hence also the associated Fredholm module $\C F_{\B T^2}$) is deeply related to the bipolarisation defined in \eqref{eq:bipolarisation of two-torus}. Indeed, the two orthogonal projections $P$ and $Q \colon L^2(\B T^2) \to L^2(\B T^2)$ are exactly the spectral projections onto the positive part of the spectrum of the two partial differential operators $i \pa/\pa \te_1$ and $i \pa/\pa \te_2$ (both acting as unbounded self-adjoint operators on the Hilbert space $L^2(\B T^2)$). In fact, one may regard the two partial differential operators $i \pa/\pa \te_1$ and $i \pa/\pa \te_2$ as providing a factorisation of the spectral triple $(C^\infty(\B T^2), L^2(\B T^2) \op L^2(\B T^2), D)$ and this point of view is compatible with the unbounded Kasparov product, see \cite{BaJu:TKO,Mes:UKC,KaLe:SFU}.

%

\subsection{The structure of the paper}
Since the determinants of Fredholm operators play a major role in our constructions, we collect all the necessary definitions and results in Section \ref{s:tors}-\ref{s:stab}. The definitions and main properties of determinant functors and their associated torsion isomorphisms are collected in Section \ref{s:tors} and Section \ref{s:torsfred}. In Section \ref{s:pert}, we recall the construction of the Carey-Pincus perturbation isomorphism in the context of trace class perturbations of Fredholm operators. In Section \ref{s:pert}, we also prove a fundamental relationship between torsion isomorphisms and perturbation isomorphisms. In Section \ref{s:stab}, we consider a particular kind of quasi-isomorphism, namely those obtained by stabilising Fredholm operators with an invertible operator.

In Section \ref{s:2coc}, we review the construction of group $2$-cocycles associated to group actions on categories. In particular, we apply this framework to obtain a group $2$-cocycle on the restricted general linear group associated to a polarisation of a Hilbert space $\C H$. In Section \ref{s:mulK2}, we then show that our group $2$-cocycle on the restricted linear group recovers the low-dimensional Connes-Karoubi multiplicative character on the second algebraic $K$-group, and hence that our group $2$-cocycle corresponds to the usual central extension of the restricted general linear group.

In Section \ref{s:polar}, we start our investigation of the two-dimensional setting, and provide an account of the category theoretic framework which we apply to exhibit group $3$-cocycles. In particular, we provide the main definitions relating to coproduct categories, and explain how group actions give rise to group $3$-cocycles in this particular context.

In Section \ref{s:category}-\ref{s:group}, we carry out the main constructions of the present paper. We thus explain how to apply representation theoretic data to obtain a coproduct category with an action of a group, and hence a group $3$-cocycle on the group in question. In these sections, we only present the main constructions and state the main theorems; the actual proofs are provided later on in Section \ref{s:proofcat} and Section \ref{s:propchange}. In this respect, we would like to highlight the computation of the Fredholm determinant in Lemma \ref{l:crux}, since this result can be viewed as the main reason for the functoriality of the change-of-base-point isomorphism. 

In Section \ref{s:bipolar} and \ref{s:nontriv}, we apply our main theorem to the setting of bipolarisations of group representations and we show by an explicit computation that our group $3$-cocycle is non-trivial in the case of the smooth double loop group $C^\infty(\B T^2)^*$, thus in the context of the bipolarisation coming from the spectral projections $P$ and $Q \colon L^2(\B T^2) \to L^2(\B T^2)$.

\subsection{Acknowledgements} 
The first author would like to thank Ulrich Bunke for the many very nice conversations on secondary invariants, and to thank Andreas Thom and the Technische Universit\"at Dresden for hospitality in the autumn of 2021 during the final stages of the writing of this paper.

The second author would like to thank Alexander Gorokhovsky and Boris Tsygan for many helpful discussions about various aspects of this paper. 

The third author thanks Oliver Braunling, Ezra Getzler, Michael Groechenig, Mikhail Kapranov and Boris Tsygan for helpful conversations, and thanks the organizers of the 2016 IBS-CGP workshop on Homotopical Methods in Quantum Field Theory for the opportunity to present an early version of this work.

\section{Determinants of $\zz/2\zz$-graded vector spaces}\label{s:tors}
We start out by reviewing some fundamental constructions relating to determinants of finite dimensional vector spaces. These constructions can be expanded and generalised in various directions leading to notions of determinant functors in different category theoretic settings, see e.g. \cite{KnMu:PMS,Knu:DEC,Del:DC,Bre:DFT,MuToWi:DFK}. Notice that we work in a $\zz/2\zz$-graded context instead of the more common $\zz$-graded context, where the objects arise as the cohomology groups of a bounded cochain complex, see also \cite{Kaa:JSC,KaNe:CHS} for more details.

\begin{notation} Let $\mathscr{L}$ be a complex line, i.e. a one-dimensional complex vector space. Given a non-zero vector $w\in \mathscr{L}$, we use $w^*$ to denote the vector in the dual line $\mathscr{L}^*= \T{Hom}(\mathscr{L},\Cb)$ given by $w^*(w)=1$.
\end{notation}

\begin{notation}\label{n:picard}
We let $\G{Pic}$ denote the Picard category of $\zz$-graded complex lines. The objects are thus pairs $(\sL,n)$ where $\sL$ is a complex line and $n$ is an integer. Moreover, the set of morphisms from $(\sL,n)$ to $(\sM,m)$ is the empty set for $n \neq m$ and for $n = m$, they are exactly the vector space isomorphisms from $\sL$ to $\sM$. We recall that the commutativity constraint in our Picard category is given by
\[
\epsilon \colon (\sL,n) \ot (\sM,m) \to (\sM,m) \ot (\sL,n) \q \epsilon(s \ot t) := (-1)^{n \cd m} \cd t \ot s \, .
\]
\end{notation}

\begin{notation}
The exterior algebra of a finite dimensional complex vector space $V$ is denoted by $\La(V)$. For a homogeneous element $v \in \La(V)$, we let
\[
\ep(v) \in \nn \cup \{0\}
\]
denote its degree. For a linear map $T \colon V \to W$ between finite dimensional vector spaces, we use the same notation
\[
T \colon \La(V) \to \La(W) \q T(v_1 \we v_2 \wlw v_k) := T(v_1) \we T(v_2) \wlw T(v_k)
\]
for the induced algebra homomorphism between the exterior algebras.
\end{notation}

\begin{definition}\label{d:torz2}
Given a finite dimensional complex vector space $V,$ its \emph{determinant} is the $\zz$-graded complex line
\[
\Det(V) := \big( \Lambda^{\T{top}}(V) , \dim(V) \big) \, .
\]
For a finite dimensional $\zz/2\zz$-graded vector space $V = V_+ \op V_-$, its \emph{determinant} is the $\zz$-graded complex line
\[
\Det(V) := \Det(V_+) \ot \Det(V_-)^* = \big( \Lambda^{\T{top}}(V_+) \ot \Lambda^{\T{top}}(V_-)^*, \dim(V_+) - \dim(V_-) \big) \, .
\]
We specify that 
\[
\Det(\{0\}) = ( \cc, 0) \, .
\] 
An even isomorphism $\Si = \ma{cc}{\Si_+ & 0 \\ 0 & \Si_-} \colon V \to W$ of $\zz/2\zz$-graded vector spaces induces an isomorphism of the corresponding determinant lines $\Det(\Si) \colon \Det(V) \to \Det(W)$ given by
\[
 v_+ \ot v_-^*\mapsto \Si_+(v_+) \ot (\Si_-(v_-) )^* \, .
\]
We will occasionally use the notation $|V|$ (resp. $|\Sigma|$) for $\Det(V)$ (resp. $\Det (\Sigma)$).
\end{definition}


\begin{definition}
An \E{exact triangle} of finite dimensional $\zz/2\zz$-graded vector spaces is an exact sequence $\De$ of the form
\[
\xymatrix{
            U \ar[rr]^{i} & & V  \ar[dl]^{q} \\
            & W \ar[ul]^{\pa} &         }
\]
where $i$ and $q$ are {\it even} and $\pa$ is an \emph{odd} linear map. Splitting the $\zz/2\zz$-graded vector spaces explicitly into the even and odd direct summands, such an exact triangle is given by the following six term exact sequence of finite dimensional vector spaces
\[
\xymatrix{ U_+ \ar[r]^{i_+} & V_+ \ar[r]^{q_+} & W_+ \ar[d]^{\pa_+} \\
W_-\ar[u]^{\pa_-} & V_- \ar[l]^{q_-} & U_- \ar[l]^{i_-} \, . }
\]
\end{definition}

For each exact triangle $\De$ of finite dimensional $\zz/2\zz$-graded vector spaces, we have the \emph{torsion isomorphism} of determinant lines
\[
\Det(\De) \colon \Det(V) \to \Det(U) \ot \Det(W) \, .
\]
Let us formulate its construction as a definition.

\begin{definition}\label{def:torsion} Let
\[
\De\colon\xymatrix{U\ar[r]^i&V\ar[r]^q&W\ar[r]^\pa& U
}
\]
be an exact triangle of finite dimensional $\zz/2\zz$-graded vector spaces. 

Choose homogeneous elements
\begin{equation}\label{eq:homogen}
\begin{split}
& u_+ \in \La^{\dim( \T{Im}(i_+))}(U_+) \q v_+ \in \La^{\dim(\T{Im}(q_+))}(V_+) \q w_+ \in \La^{\dim(\T{Im}(\pa_+))}(W_+) \\
& u_- \in \La^{\dim(\T{Im}(i_-))}(U_-) \q v_- \in \La^{\dim(\T{Im}(q_-))}(V_-) \q w_- \in \La^{\dim(\T{Im}(\pa_-))}(W_-)
\end{split}
\end{equation}
such that the following wedge products yield non-trivial elements:
\[
\begin{split}
& i_+(u_+) \we v_+ \in \Det(V_+) \q i_-(u_-) \we v_- \in \Det(V_-) \\
& \pa_-(w_-) \we u_+ \in \Det(U_+) \q \pa_+(w_+) \we u_- \in \Det(U_-) \\
& q_+(v_+) \we w_+ \in \Det(W_+) \q q_-(v_-) \we w_- \in \Det(W_-) \, .
\end{split}
\]

The \emph{torsion isomorphism}
\[
\Det(\De) \colon \Det(V) \to \Det(U) \ot \Det(W)
\]
is then defined by
\[
\begin{split}
& i_+(u_+) \we v_+ \ot ( i_-(u_-) \we v_- )^* \\
& \qq \mapsto (-1)^{\ep(\De)} \cd \pa_-(w_-) \we u_+ \ot (\pa_+(w_+) \we u_-)^* \ot q_+(v_+) \we w_+ \ot ( q_-(v_-) \we w_- )^* \, ,
\end{split}
\]
where the sign is determined by the exponent
\[
\ep(\De) := \dim( \T{Im}(q_+)) \cd \dim(U_+) + \dim( \T{Im}(i_-)) \cd \dim(W_+) + \dim( \T{Im}(\pa_-)) \cd \dim(V_-) \in \nn \cup \{0\} \, .
%
\]
\end{definition}

We record the following result (cf. \cite[Lemma 2.1.3]{Kaa:JSC}):

\begin{prop}
The torsion isomorphism $\Det(\De)$ is independent of the choices of elements in the exterior algebras used in Definition \ref{def:torsion}. 
\end{prop}

The next theorem summarises the main properties of our constructions. The first two (naturality and commutativity) follow directly from the definitions.  The proof of associativity is a (much longer) exercise in finite dimensional linear algebra. We refer the interested reader to \cite[Section 4]{Kaa:JSC}, and specifically to Remark 4.3.5 and Theorem 4.3.4 of {\em loc. cit.}. 

\begin{theorem}\label{t:torprop} The torsion isomorphisms satisfy the following properties:
\begin{enumerate}
\item {\bf Naturality.} Let
\[
\xymatrix{ U \ar[r] \ar[d]^a & V \ar[r] \ar[d]^b & W \ar[r] \ar[d]^c & U \ar[d]^a \\
U' \ar[r] & V' \ar[r] & W' \ar[r] & U' }
\]
be a commutative diagram of finite dimensional $\zz/2\zz$-graded vector spaces,
where the rows $\De$ and $\De'$ are exact triangles and the columns are even isomorphisms.

Then the diagram
\[
\xymatrix{ \Det(V) \ar[rr]^>>>>>>>>>>{\Det(\De)} \ar[d]_{\Det(b)} && \Det(U) \ot \Det(W) \ar[d]^{\Det(a) \ot \Det(c)} \\
\Det(V') \ar[rr]^>>>>>>>>>>{\Det(\De')} && \Det(U') \ot \Det(W') }
\]
is commutative.
\item {\bf Commutativity.} Associated to two finite dimensional $\zz/2\zz$-graded vector spaces $U = U_+ \op U_-$ and $V = V_+ \op V_-,$ we can construct two exact triangles
\[
\De_1 \colon \xymatrix{ U \ar[rr]^{\ma{c}{\T{id} \\ 0}} &&  U \op W \ar[rr]^{\ma{cc}{0 & \T{id}}} && W \ar[r]^{0} & U }
\]
 and
\[
\De_2 \colon \xymatrix{ W \ar[rr]^{\ma{c}{0 \\ \T{id}}} && U \op W \ar[rr]^{\ma{cc}{\T{id} & 0}} && U \ar[r]^{0} &  W } \, .
\]

Then the diagram
\[
\xymatrix{ \Det(U \op W) \ar[rr]^{\Det(\De_1)} \ar[rrd]_{\Det(\De_2) \, \, \, } && \Det(U) \ot \Det(W) \ar[d]^{\epsilon} \\
&& \Det(W) \ot \Det(U) }
\]
is commutative, where $\epsilon$ denotes the commutativity constraint from Notation \ref{n:picard}.
\item {\bf Associativity.} Let
\[
\xymatrix{
U \ar[r]^{i} \ar@{=}[d] & V \ar[r]^{q} \ar[d] & X  \ar[r] \ar[d] & U \ar@{=}[d] \\
U \ar[r] & W \ar[r] \ar[d] & Z \ar[r]^{d} \ar[d]^{\pi} & U \\
& Y \ar@{=}[r] \ar[d]_\pa & Y \ar[d] &  \\
& V \ar[r]^{q} & X & 
}
\]
be a commutative diagram of finite dimensional $\zz/2\zz$-graded vector spaces, where:
\begin{enumerate}
\item the two first rows, $\De_1$ and $\De_2$, and columns two and three, $\Ga_2$ and $\Ga_3$, are exact triangles;
\item the diagram
\[
\xymatrix{
Z \ar[r]^{d} \ar[d]_{\pi} & U \ar[d]^{i}\\
Y \ar[r]^{\pa} & V }
\]
is commutative.
\end{enumerate}
Then the diagram
\begin{equation}\label{eq:assotors}
\xymatrix{
\Det(W) \ar[rr]^{\Det(\De_2)} \ar[d]_{\Det(\Ga_2)} && \Det(U) \ot \Det(Z) \ar[d]^{\T{id} \ot \Det(\Ga_3)} \\
\Det(V) \ot \Det(Y) \ar[rr]^>>>>>>>>>>{\Det(\De_1) \ot \T{id}} && \Det(U) \ot \Det(X) \ot \Det(Y)
}
\end{equation}
is commutative.
\end{enumerate}
\end{theorem}

\section{The torsion isomorphism of Fredholm operators}\label{s:torsfred}
The determinants and torsion isomorphisms described in the previous section in the context of finite dimensional $\zz/2\zz$-graded vector spaces can be lifted to the analytic setting of Fredholm operators acting between separable Hilbert spaces.

\begin{definition}
Let $\C H$ and $\C K$ be separable Hilbert spaces. The \E{determinant line} of a Fredholm operator $T \colon \C H\to \C K$ is the graded line
\[
|T| :=\big( \Lambda^{\T{top}}\Ker(T)\otimes \Lambda^{\T{top}}\Coker(T)^*, \dim( \Ker (T)) - \dim( \Coker(T)) \big) \, .
\]
Thus, upon defining the finite dimensional $\zz/2\zz$-graded vector space $I(T) := \Ker(T) \oplus \Coker(T)$ we obtain that $|T| = \Det(I(T))$.
\end{definition}

\begin{definition}\label{d:torfre}
For two composable Fredholm operators $T$ and $S$, the associated \E{torsion isomorphism}
\[
\mathfrak{T} \colon |T|\otimes |S|\to |ST|
\]
is the torsion isomorphism $\Det( \De(S,T))^{-1}$ associated to the six term exact sequence $\De(S,T)$:
\[
\xymatrix{
\Ker(T) \ar[r]^{\iota} & \Ker(ST) \ar[r]^{T} & \Ker(S) \ar[d]^{\pa} \\ 
\Coker(S) \ar[u]^{0} & \Coker(ST) \ar[l]_{\pi} & \Coker(T) \ar[l]_{S} \, , }
\]
where $\iota$ is the inclusion, $\pi$ the quotient map and $\pa$ is given by the composition $\Ker(S)\to \C H\to \C H/\T{Im}(T)$.
\end{definition}

We specify that the six term exact sequence in the above definition comes from the exact triangle
\[
\De(S,T) \colon 
\xymatrix{ I(T) \ar[rr]^{\ma{cc}{\io & 0 \\ 0 & S}} && I(ST) \ar[rr]^{\ma{cc}{T & 0 \\ 0 & \pi}} && I(S) \ar[rr]^{\ma{cc}{0 & 0 \\ \pa & 0}} && I(T) } 
\]
of finite dimensional $\zz/2\zz$-graded vector spaces.


\begin{prop}\label{p:assotors}
Given three composable Fredholm operators $S$, $T$ and $R$, the following diagram commutes:
\[
\xymatrix{
|S|\otimes |T|\otimes |R| \ar[rr]^{\mathfrak{T} \ot \T{id}} \ar[d]_{\T{id} \ot \mathfrak{T}} && |TS|\otimes |R| \ar[d]^{\mathfrak{T}} \\
|S|\otimes |RT| \ar[rr]^{\G T} && |RTS| \, . } 
\]
\end{prop}
\begin{proof}
This is an easy consequence of the associativity property of the torsion isomorphism explained in Theorem \ref{t:torprop}.
\end{proof}

\begin{example}\label{ex:LR}
Suppose that $T \colon \C H_1 \to \C H_2$ is Fredholm and that $\Si_1 \colon \C G \to \C H_1$ and $\Si_2 \colon \C H_2 \to \C K$ are invertible. After identifying $|T| \ot |\Si_2|$ and $|\Si_1| \ot |T|$ with $|T|$, the torsion isomorphisms associated to the compositions get the following explicit description.
\begin{equation}\label{LR}
\begin{split}
& L(\Sigma_2) \colon |T| \to |\Sigma_2 T| \quad u_+ \otimes u_-^* \mapsto u_+ \otimes (\Sigma_2 u_-)^* \quad \mbox{and} \\
& R(\Sigma_1) \colon |T| \to | T \Sigma_1| \quad u_+ \otimes u_-^* \mapsto (\Sigma_1^{-1} u_+) \otimes u_-^* \, ,
\end{split}
\end{equation}
where $u_+ \in \Det( \T{Ker}(T))$ and $u_- \in \Det(\T{Coker}(T))$ are non-trivial vectors.
\end{example}

\begin{definition}
Let $T \colon \C H \to \C G$ and $T' \colon \C H' \to \C G'$ be Fredholm operators and let $\phi \colon \C H \to \C H'$ and $\psi \colon \C G \to \C G'$ be bounded operators. We say that the pair $(\phi,\psi)$ is a \E{quasi-isomorphism} from $T$ to $T'$ when $\psi \ci T = T' \ci \phi \colon \C H \to \C G'$ and when the induced linear maps
\[
\phi \colon \T{Ker}(T) \to \T{Ker}(T') \q \mbox{and} \q \psi \colon \T{Coker}(T) \to \T{Coker}(T')
\]
are isomorphisms. We let $|\phi,\psi| \colon |T| \to |T'|$ denote the isomorphism of graded lines induced by a quasi-isomorphism $(\phi,\psi)$.
\end{definition}

The next proposition explains the relationship between torsion isomorphisms and quasi-isomorphisms.

\begin{prop}\label{p:torquis}
Let $T \colon \C H \to \C G$, $S \colon \C G \to \C K$, $T' \colon \C H' \to \C G'$ and $S' \colon \C G' \to \C K'$ be Fredholm operators and suppose that $(\phi,\psi)$ and $(\psi,\tau)$ are quasi-isomorphisms from $T$ to $T'$ and from $S$ to $S'$, respectively. Then $(\phi,\tau)$ is a quasi-isomorphism from $ST$ to $S'T'$ and the following diagram commutes:
\[
\xymatrix{
|T|\otimes |S|\ar[r]^>>>>>>{\frak{T}}\ar[d]_{|\phi, \psi| \ot |\psi, \tau|} &|ST|\ar[d]^{|\phi,\tau|}\\
|T^{'}|\otimes |S^{'}|\ar[r]^>>>>>>{\frak{T}}& |S^{'}T^{'}| \, .
}
\]
\end{prop}
\begin{proof}
The claim follows from the five lemma, the definition of the torsion isomorphism and from its naturality, i.e. from Theorem \ref{t:torprop} $(1)$. 
\end{proof}

%

\section{The perturbation isomorphism}\label{s:pert}
Given a trace class operator $\delta \colon \C H \to \C H$ on a separable Hilbert space $\C H$, the following Fredholm determinant
\begin{equation}\label{GK}
\det(\T{id}_{\C H}+\delta)= 1 + \sum_{n=1}^\infty \T{Tr}(\Lambda^n \delta) \in \cc
\end{equation}
is well defined and is a multiplicative generalisation of the determinant of a finite matrix, see \cite[Chapter 3]{Sim:TIA}. For any pair of invertible bounded operators $T_1, T_2 \colon \C H \to \C G$ with trace class difference, one may use the Fredholm determinant to define the invertible complex number $\det(T_2 T_1^{-1})$. The perturbation isomorphism, invented by Carey and Pincus in \cite{CaPi:PV}, provides an extension of this assignment to the setting of Fredholm operators with trace class difference. For more information on these matters we also refer to the paper \cite{KaNe:CHS}.

\begin{theorem}\label{t:pervec}\cite[Theorem 11 and Theorem 12]{CaPi:PV}
Let $\C H$ and $\C G$ be separable Hilbert spaces. There is a universal construction which, to any pair of Fredholm operators $T_1,T_2 \colon \C H \to \C G$ with trace class difference associates an isomorphism
\[
\G P(T_1,T_2) \colon |T_1|\to |T_2|
\]
such that the following cocycle conditions hold:
\begin{enumerate}
\item $\G P(T_2,T_3) \ci \G P(T_1,T_2) = \G P(T_1,T_3)$;
\item $\G P(T_1,T_2)=\G P(T_2,T_1)^{-1}$.
\end{enumerate}
We refer to $\G P(T_1,T_2)$ as the \emph{perturbation isomorphism} and we sometimes apply the notation
\[
\mathfrak{P} := \G P(T_1,T_2) \colon |T_1|\to |T_2| \, .
\]
\end{theorem}

The following series of examples contain an explicit description of the perturbation isomorphism in various cases which will be useful later on.

\begin{example}\label{r:finrank}{\bf Finite rank case.} Suppose that $T_1$ and $T_2 \colon \C H \to \C G$ are two Fredholm operators such that $T_1 - T_2$ is of \emph{finite rank}. Let $V\subseteq \C H$ be a closed subspace of finite codimension such that $T_1|_V = T_2|_V = T \colon V \to \C G$. Let $\iota_V$ denote the inclusion $V\hookrightarrow \C H$. For $i = 1,2$, the torsion isomorphism from Definition \ref{d:torfre} then provides the isomorphism
\[
\xymatrix{
|T|= |T_i \ci \iota_V| \ar[rrr]^{\Det(\De(T_i,\io_V))} &&& |\io_V| \otimes |T_i| } \, .
\]
The perturbation isomorphism $\G P(T_1,T_2) \colon |T_1| \to |T_2|$ is then the unique isomorphism making the diagram 
\[
\xymatrix{ 
& & |T| \ar[drr]^{\, \, \, \, \, \, \, \,\Det(\De(T_2,\io_V))} \ar[dll]_{\Det(\De(T_1,\io_V)) \, \, \, \, \, \, \, \,} & & \\
|\io_V| \ot |T_1| \ar[rrrr]^{\T{id} \ot \G P(T_1,T_2)} &&&& |\io_V| \ot |T_2| }
\]
commute. 
\end{example}

\begin{example}\label{ex:invpert}{\bf Invertible case.} Suppose that $T_1$ and $T_2$ are both invertible, but that $T_1 - T_2$ is of trace class. Then
\[
|T_1|= |T_2| = ( \mathbb{C} , 0)
\]
and, as an automorphism $\mathbb{C}\to \mathbb{C}$, we have that
\begin{equation}\label{eq:perturbation 3}
\G P(T_1,T_2)(\lambda) = \det(T_2T_1^{-1}) \cd \lambda  \, .
\end{equation}
\end{example}

\begin{example}\label{ex:indzero}{\bf Index zero case.} Suppose that $T_1$ and $T_2 \colon \C H \to \C G$ are of index zero and that $T_1 - T_2$ is of trace class. Then the perturbation isomorphism $\G P(T_1,T_2)$ has the following description. For $i = 1,2$, choose a bounded finite rank operator $F_i \colon \C H \to \C G$ such that $T_i + F_i$ is invertible and such that $\Ker(F_i) \subseteq \C H$ is a vector space complement of $\Ker( T_i) \subseteq \C H$. Then each $F_i$ induces an isomorphism $F_i \colon \Ker(T_i) \to \Coker(T_i)$ and the invertible operator
\[
\Sigma := (T_2+ F_2)(T_1 + F_1)^{-1} \colon \C G \to \C G
\]
is of determinant class. Given basis vectors $t\in  \Lambda^{\T{top}}\Ker(T_1)$ and $s\in  \Lambda^{\T{top}}\Ker(T_2)$, the perturbation isomorphism is then given by
\begin{equation}
\label{eq:perturbation 2}
\G P(T_1,T_2) \colon t\otimes (F_1 t)^*\mapsto \det(\Sigma) \cd s \otimes (F_2 s)^* \, .
\end{equation}
\end{example}

\begin{example}\label{ex:general}{\bf The general case.} In full generality, the perturbation isomorphism can be described as follows. Choose $n,m \in \nn_0$ such that $m - n = \T{Index}(T_2) = \T{Index}(T_1)$ and let $0_{m,n} \colon \cc^n \to \cc^m$ be the trivial linear map. Then the Fredholm operators $T_1 \op 0_{m,n}$ and $T_2 \op 0_{m,n} \colon \C H \op \cc^n \to \C G \op \cc^m$ are of index zero and
\[
\G P(T_1,T_2) \ot \T{id} \colon |T_1| \ot |0_{m,n}| \to |T_2| \ot |0_{m,n}|
\]
agrees with the composition
\begin{equation}\label{eq:perturbation 1}
\xymatrix{ 
|T_1| \ot |0_{m,n}| \cong |T_1 \op 0_{m,n}| \ar[rrrr]^{\G P(T_1 \op 0_{m,n}, T_2 \op 0_{m,n})} &&&& | T_2 \op 0_{m,n}| \cong |T_2| \ot |0_{m,n}|  \, ,
} 
\end{equation}
where the first and the third isomorphism are given by the torsion isomorphisms associated to the six term exact sequences
\[
\xymatrix{
\T{Ker}(T_i) \ar[rr]^{\ma{c}{\T{id} \\ 0}} && \T{Ker}(T_i) \op \cc^m \ar[rr]^{\ma{cc}{0 & \T{id}}} && \cc^m \ar[d]^{0} \\ 
\cc^n \ar[u]^{0} && \T{Coker}(T_i) \op \cc^n \ar[ll]_{\ma{cc}{0 & \T{id}}} && \T{Coker}(T_i) \ar[ll]_{{\ma{c}{\T{id} \\ 0}}} }
\]
for $i = 1,2$.
\end{example}

\begin{remark}
Choose a parametrix $R \colon \C G \to \C H$ for $T_1$ such that $RT_1 - \T{id}_{\C H}$ and $T_1 R - \T{id}_{\C G}$ are of trace class. It follows from Theorem \ref{t:percom}, which we will prove in Subsection \ref{ss:percomtor}, that the perturbation $\G P(T_1,T_2) \colon |T_1| \to |T_2|$ agrees with the unique isomorphism making the following diagram commutative
\[
\xymatrix{
|T_1| \ot |R| \ar[r]^{\mathfrak{T}} \ar[d]_{\G P(T_1,T_2) \ot \T{id}} & |RT_1| \ar[d]^{\G P(RT_1,RT_2)} \\
|T_2| \ot |R| \ar[r]^{\mathfrak{T}} &  |RT_2| }
\]
commutative. Here both $RT_1$ and $RT_2 \colon \C H \to \C H$ have index $0$. In particular, $\G P(T_1,T_2)$ as described above is independent of the choice of parametrix $R \colon \C G \to \C H$.
\end{remark}

%

\subsection{Perturbation commutes with torsion}\label{ss:percomtor}
Our goal in this subsection is to prove a fundamental relationship between torsion isomorphisms and perturbation isomorphisms. This result can be found as Theorem \ref{t:percom} and will be applied throughout this text. A similar result was proved in \cite[Theorem 5.1]{KaNe:CHS} for finite rank perturbations of mapping cone Fredholm complexes. Since our sign conventions are slightly different from the conventions used in \cite{KaNe:CHS} and since we need to work with trace class perturbations, we provide a full proof for the case where the Fredholm complexes in question are the mapping cones of single Fredholm operators, i.e. of the form $0\to \C H \to^T \C G\to 0$. 

\begin{lemma}\label{l:finrank}
Let $\C H, \C G$ and $\C K$ be separable Hilbert spaces, let $T \colon \C H \to \C G$ and $S \colon \C G \to \C K$ be Fredholm operators and let $\de_T \colon \C H \to \C G$ be a finite rank operator. Then the following square commutes:
    \begin{equation}\label{eq:perttorfin}
        \xymatrix{
            |T|\otimes |S| \ar[r]^{\mathfrak{T}}\ar[d]_{\mathfrak{P} \otimes \T{id}} & |ST| \ar[d]^{\mathfrak{P}} \\
            |T + \de_T|\otimes |S| \ar[r]^>>>>>{\mathfrak{T}} & |S(T+ \de_T)| \, .
        }
    \end{equation}
\end{lemma}
\begin{proof}
Choose a closed subspace $V \subseteq \C H$ of finite codimension such that $V \subseteq \Ker(\de_T)$ and let $\io_V \colon V \to \C H$ denote the inclusion. The diagram in \eqref{eq:perttorfin} then commutes if and only if the diagram
\[
\xymatrix{
            |\io_V| \ot |T|\otimes |S| \ar[r]^{\T{id} \ot \mathfrak{T}} \ar[d]_{\T{id} \ot \mathfrak{P} \otimes \T{id}} &
|\io_V| \ot |ST| \ar[d]^{\T{id} \ot \mathfrak{P}} \\
            |\io_V| \ot |T + \de_T|\otimes |S| \ar[r]^>>>>>{\T{id} \ot \mathfrak{T}} & |\io_V| \ot |S(T+ \de_T)|
        }
\]
commutes. To see that this latter diagram commutes, it suffices to use the description of the perturbation isomorphism given in Example \ref{r:finrank} and the associativity of the torsion isomorphisms described in Proposition \ref{p:assotors}.
\end{proof}

\begin{lemma}\label{l:indzero}
Let $\C H, \C G$ and $\C K$ be separable Hilbert spaces, let $T \colon \C H \to \C G$ and $S \colon \C G \to \C K$ be Fredholm operators of index $0$ and let $\de_T \colon \C H \to \C G$ and $\de_S \colon \C G \to \C K$ be trace class operators. Then the following square commutes:
\begin{equation}\label{eq:indzero}
        \xymatrix{
            |T|\otimes |S| \ar[r]^{\mathfrak{T}} \ar[d]_{\mathfrak{P} \otimes \mathfrak{P}} & |ST| \ar[d]^{\mathfrak{P}} \\
            |T + \de_T|\otimes |S + \de_S| \ar[r]^>>>>>{\mathfrak{T}} & \big| (S + \de_S)(T+ \de_T)\big| \, .
        }
\end{equation}
\end{lemma}
\begin{proof}
Since both $T$ and $T + \de_T \colon \C H \to \C G$ are of index zero, we may find finite rank operators $F_T$ and $F_{T + \de_T} \colon \C H \to \C G$ such that $T + F_T$ and $T + \de_T + F_{T+ \de_T} \colon \C H \to \C G$ are invertible. Using Lemma \ref{l:finrank} and the transitivity of the perturbation isomorphism from Theorem \ref{t:pervec}, we may thus assume that both $T$ and $T + \de_T \colon \C H \to \C G$ are invertible.

Since $S$ and $S + \de_S \colon \C G \to \C K$ are of index zero, we may choose finite rank operators $F_S$ and $F_{S + \de_S} \colon \C G \to \C K$ such that $S + F_S$ and $S + \de_S + F_{S + \de_S} \colon \C G \to \C K$ are invertible and such that $\Ker(F_S)$ and $\Ker(F_{S + \de_S})$ are vector space complements of $\Ker(S)$ and $\Ker(S + \de_S)$, respectively. Since $T$ and $T + \de_T \colon \C H \to \C G$ are invertible, it then follows that $ST + F_S T \colon \C H \to \C K$ and $(S + \de_S)(T + \de_T) + F_{S + \de_S} (T + \de_T) \colon \C H \to \C K$ are invertible. Moreover, $\Ker(F_S T)$ is a vector space complement to $\Ker(ST)$, and $\Ker( F_{S + \de_S}(T + \de_T))$ is a vector space complement to $\Ker((S + \de_S)(T + \de_T))$. We now use the description of the perturbation isomorphism given in Example \ref{ex:invpert} and in Example \ref{ex:indzero}, together with the description of the torsion isomorphism given in Example \ref{ex:LR}. The commutativity of the diagram in \eqref{eq:indzero} then amounts to showing that the two isomorphisms $\T{LHS}$ and $\T{RHS} \colon |S| \to | (S + \de_S)(T+ \de_T) |$ given by
\[
\begin{split}
& \T{LHS} \colon s \ot (F_S s)^* \mapsto
\det\big( (T + \de_T) T^{-1}\big) \cd \det\big( (S + \de_S + F_{S + \de_S})(S + F_S)^{-1}\big) \\
& \qqq \qqq \cd \big( (T + \de_T)^{-1}(t) \ot (F_{S + \de_S} t)^* \big) \q \T{and} \\
& \T{RHS}\colon s \ot (F_S s)^* \mapsto \det\big( (S + \de_S + F_{S + \de_S})(T + \de_T)  ((S + F_S )T)^{-1} \big) \\
& \qqq \qqq \cd \big( r \ot \big( F_{S + \de_S}(T + \de_T) r \big)^* \big)
\end{split}
\]
agree, where $s$, $t$ and $r$ are basis vectors for $\Det(\Ker(S))$, $\Det( \Ker(S + \de_S))$ and $\Det(\Ker( (S + \de_S) (T + \de_T)) )$, respectively. Choosing $r := (T + \de_T)^{-1}(t)$, we only need to verify that the two Fredholm determinants
\[
\begin{split}
& \det\big( (T + \de_T) T^{-1}\big) \cd \det\big( (S + \de_S + F_{S + \de_S})(S + F_S)^{-1}\big) \q \T{and} \\
& \det\big( (S + \de_S + F_{S + \de_S})(T + \de_T)  ((S + F_S )T)^{-1} \big)
\end{split}
\]
coincide. But this follows immediately from basic properties of the Fredholm determinant: indeed, it suffices to use that the Fredholm determinant is multiplicative and that it is invariant under conjugation by invertible bounded operators (see e.g. \cite[Chapter 3]{Sim:TIA}).
\end{proof}

\begin{theorem}[{\bf Perturbation commutes with torsion}]\label{t:percom}
Let $\C H, \C G$ and $\C K$ be separable Hilbert spaces, let $T \colon \C H \to \C G$ and $S \colon \C G \to \C K$ be Fredholm operators and let $\de_T \colon \C H \to \C G$ and $\de_S \colon \C G \to \C K$ be trace class operators. Then the following square commutes:
    \begin{equation}\label{perttor}
        \xymatrix{
            |T|\otimes |S| \ar[r]^{\mathfrak{T}} \ar[d]_{\mathfrak{P} \otimes \mathfrak{P}} & |ST| \ar[d]^{\mathfrak{P}} \\
            |T + \de_T|\otimes |S + \de_S| \ar[r]^>>>>>{\mathfrak{T}} & | (S + \de_S)(T+ \de_T)| \, .
        }
    \end{equation}
\end{theorem}
\begin{proof}
Choose integers $n,m,k \geq 0$ such that $\T{Index}(T) = m - n$ and $\T{Index}(S) = k - m$. Recall that $0_{m,n} \colon \cc^n \to \cc^m$ and $0_{k,m} \colon \cc^m \to \cc^k$ denote the trivial maps. It can then be verified that the following diagram commutes:
\[
\xymatrix{
|T| \ot |S| \ot |0_{m,n}| \ot |0_{k,m}| \ar[rr]^{\G T \ot \G T} \ar[d] && |ST| \ot | 0_{k,n}| \ar[d] \\
| T \op 0_{m,n}| \ot |S \op 0_{k,m}| \ar[rr]^{\G T} && |ST \op 0_{k,n}| \, .
} 
\]
Notice here that the horizontal maps are torsion isomorphisms of Fredholm operators, the right vertical is the torsion isomorphism coming from the six term exact sequence in Example \ref{ex:general}, and the left vertical map is the composition of the commutativity constraint from Notation \ref{n:picard} and the torsion isomorphisms coming from the six term exact sequences in Example \ref{ex:general}. There is of course a similar commutative diagram when $T$ is replaced by $T + \de_T$ and $S$ is replaced by $S + \de_S$. Since $T \op 0_{m,n}$ and $S \op 0_{k,m}$ both have index zero, Lemma \ref{l:indzero} implies that the following diagram commutes:
\[
\xymatrix{
            | T \op 0_{m,n}| \ot | S \op 0_{k,m}| \ar[r]^{\mathfrak{T}} \ar[d]_{\mathfrak{P} \otimes \mathfrak{P}} & |ST \op 0_{k,n}| \ar[d]^{\mathfrak{P}} \\
            |(T + \de_T) \op 0_{m,n}| \ot |(S + \de_S) \op 0_{k,m}| \ar[r]^>>>>>{\mathfrak{T}} & | (S + \de_S)(T+ \de_T) \op 0_{k,n} | \, .
        }
\]
A combination of these observations, together with the description of the perturbation isomorphism from Example \ref{ex:general}, yields the commutativity of the diagram in \eqref{perttor}.
\end{proof}

\section{The stabilisation isomorphism}\label{s:stab}
In the more advanced part of this paper (when we investigate the two-dimensional setting), we shall often need an operation which allows us to change the domain and codomain of a Fredholm operator by adding an invertible operator. We call this operation {\em stabilisation}, and introduce and investigate it in this section. 

\begin{definition}\label{def:stabilisation} Let $\C H$ and $\C G$ be separable Hilbert spaces, and let $p \colon \C H \to \C H$ and $q \colon \C G \to \C G$ be idempotent operators. Suppose that $T \colon p \C H \to q \C G$ is a Fredholm operator, and that $\Ga \colon (1 - p) \C H \to (1 - q) \C G$ is an invertible operator. Then $T + \Ga \colon \C H \to \C G$ is a Fredholm operator, and the \emph{stabilisation isomorphism} is the isomorphism of graded lines
\[
\mathfrak{S} \colon |T|\to |T+ \Ga|
\]
induced by the quasi-isomorphism $(\io_p,\io_q)$ given by the inclusions $p\C H \su \C H$ and $q \C G \su \C G$.
\end{definition}


The next proposition describes the relationship between the torsion isomorphism and the stabilisation isomorphism. This result is a consequence of Proposition \ref{p:torquis}.

\begin{prop}[{\bf Torsion commutes with stabilisation}]\label{p:torsta}
Let $\C H$, $\C G$ and $\C K$ be separable Hilbert spaces, and let $e \colon \C H \to \C H$, $p \colon \C G \to \C G$ and $q \colon \C K \to \C K$ be idempotent operators. Suppose that $T \colon e\C H \to p \C G$ and $S \colon p \C G \to q \C K$ are Fredholm operators, and that $\Ga \colon (1 - e)\C H \to (1 - f) \C G$ and $\Te \colon (1-f)\C G \to (1 - q)\C K$ are invertible operators. Then the following diagram is commutative:
\[
\xymatrix{
|T|\otimes |S| \ar[r]^>>>>>>>>>>>>>{\frak{T}}\ar[d]_{\G S \ot \G S} &|S T|\ar[d]^{\G S}\\
|T + \Ga |\otimes |S + \Te| \ar[r]^>>>>>>{\frak{T}}& |S T + \Te \Ga| \, .
}
\]
\end{prop}

We proceed by explaining how the perturbation isomorphism and the stabilisation isomorphism relate to one another.

\begin{prop}[{\bf Perturbation commutes with stabilisation}]\label{p:persta} Let $\C H$ and $\C G$ be two separable Hilbert spaces, and let $p \colon \C H \to \C H$ and $q \colon \C G \to \C G$ be idempotent operators. Suppose that $T$ and $S \colon p \C H \to q \C G$ are Fredholm operator such that $T - S$ is of trace class and that $\Ga \colon (1 - p ) \C H \to (1 - q) \C H$ is an invertible operator. Then the following diagram is commutative:
\[
\xymatrix{
|T| \ar[r]^>>>>>>>>>{\frak{P}}\ar[d]_{\G S} &|S|\ar[d]^{\G S}\\
|T + \Ga| \ar[r]_>>>>>>{\frak{P}}& |S + \Ga| \, .
}
\]
\end{prop}
\begin{proof}
Upon applying Example \ref{ex:indzero}, the claimed identity is first verified in the case where the index of $T$ (and hence also the index of $S$) is equal to $0$. Then, by applying Example \ref{ex:general}, the general claim is verified by choosing $n,m \in \nn \cup \{0\}$ with $m - n = \T{Index}(T) = \T{Index}(S)$ and reducing to the index zero case by means of the trivial operator $0_{m,n} \colon \cc^n \to \cc^m$.
\end{proof}

We end this section with a result which elaborates on the commutativity property of the torsion isomorphism (Theorem \ref{t:torprop} $(2)$) in the context of Fredholm operators.

\begin{prop}\label{p:torsign}
Let $\C H$ be a separable Hilbert space and let $e,f,p,q \colon \C H \to \C H$ be idempotent operators such that $ep=pe = eq = q e = 0$ and $fp = p f = fq = q f = 0$. Suppose that $T \colon e \C H \to f \C H$ and $S \colon p \C H \to q \C H$ are Fredholm operators. Then the following diagram commutes:
\begin{equation}\label{eq:torsign}
\xymatrix{
|T| \ot |S| \ar[d]^{\G S \ot \G S} \ar[rr]_{\epsilon} && |S| \ot |T| \ar[d]_{\G S \ot \G S} \\
| T + p | \ot | S + f| \ar[dr]^{\G T} && |S + e| \ot |T + q | \ar[dl]_{\G T} \\
& \big| T + S \big| \, . & 
}
\end{equation}
\end{prop}
\begin{proof}
We remark that $T + p \colon (e + p) \C H \to (f + p) \C H$, $T + q \colon (e + q) \C H \to (f + q) \C H$, $S + f \colon (p + f) \C H \to (q + f) \C H$ and $S + e \colon (p + e) \C H \to (q + e) \C H$ are Fredholm operators satisfying that $(S + f)(T + p) = (T + q)(S + e) = T + S \colon (e + p) \C H \to (f + q)\C H$ so that the isomorphisms appearing in \eqref{eq:torsign} make sense. The result then follows by definition of the torsion isomorphisms and stabilisation isomorphisms appearing together with the properties $(1)$ and $(2)$ from Theorem \ref{t:torprop}. Indeed, one may identify the exact triangles $\De_{S + f,T + p}$ and $\De_{T + q,S + e}$ (yielding the torsion isomorphisms in question) with the exact triangles $\De_1$ and $\De_2$ relating to the direct sum decomposition $I(T) \op I(S)$ in Theorem \ref{t:torprop} $(2)$.
\end{proof}

\section{The group $2$-cocycle associated to a group action on a category}\label{s:2coc}
We shall now see how the data of a group acting \emph{strictly} on a category (subject to a few extra conditions) gives rise to a group $2$-cocycle. The constructions of this section should be compared with the constructions by Brylinski given in \cite{Bry:CRL} (see also \cite{Kap:ALT} and \cite{Str:AOS}). The reader might at this point think that having a strict action is too restrictive a condition, but we shall see in Section \ref{s:mulK2} that this condition suffices to give a description of the Connes-Karoubi multiplicative character on the second algebraic $K$-group. 

\begin{assu}\label{a:onecat} Assume that we are given a commutative unital ring $R$, a group $G$ and a category $\G C$ with at least one object satisfying the following conditions:
\begin{enumerate}
\item For any pair of objects $x,y$ in $\G C$, the morphisms $\T{Mor}(x,y)$ form a central bimodule over $R$, and the composition
\[
\T{Mor}(x,y) \ti \T{Mor}(y,z) \to \T{Mor}(x,z)
\]
induces a bimodule homomorphism
\[
\T{Mor}(x,y) \ot_R \T{Mor}(y,z) \to \T{Mor}(x,z) \, ;
\]
\item The set of isomorphisms between any two objects is non-empty;
\item For any object $x$ in $\G C$, the map $r \mapsto r \cd \T{id}_x$ is an isomorphism of bimodules over $R$:
\[
\io_x \colon R \to \T{Mor}(x,x) \, ;
\]
\item The group $G$ acts strictly on $\G{C}$, and the associated functors $g \colon \G C \to \G C$, $g \in G$, induce isomorphisms of bimodules $g \colon \T{Mor}(x,y) \to \T{Mor}(g(x),g(y))$ whenever $x,y$ are objects in $\G C$.
\end{enumerate}
\end{assu}

We let $R^*$ denote the abelian group of invertible elements in $R$.

\begin{definition}\label{def:one cocycle} Choose an object $x$ in $\G C$ and choose an isomorphism
\[
\alpha_g \colon x \to g(x)
\]
for every $g \in G$. Define the group $2$-cochain $c \colon \zz[ G^2 ] \to R^*$ by
\[
c(g_1,g_2) := \iota_x^{-1}\big( \alpha_{g_1}^{-1} \ci g_1(\alpha_{g_2}^{-1})\circ \alpha_{g_1g_2} \big) \q g_1, g_2 \in G \, .
\]
\end{definition}

\noindent One can visualise $c(g_1,g_2)$ as the difference between the two isomorphisms from $x$ to $(g_1g_2)(x)$ in the simplex
\[
\xymatrix{&g_1(x)\ar[dr]^{g_1(\alpha_{g_2})}&\\
x\ar[ru]^{\alpha_{g_1}}\ar[rr]_{\alpha_{g_1g_2}}&&(g_1g_2)(x)
}
\]

\begin{lemma}
The group $2$-cochain $c \colon \zz[ G^2 ] \to R^*$ is a group $2$-cocycle, and its class in group cohomology $[c] \in H^2(G,R^*)$ is independent of the choices made in Definition \ref{def:one cocycle}.
\end{lemma}
\begin{proof}
The fact that $c \colon \zz[ G^2 ] \to R^*$ is a group $2$-cocycle can be verified by a straightforward computation. In order to show that the corresponding class in group cohomology is independent of the choices made, we let $y$ be an alternative object in $\G C$ and $\be_g \colon y \to g(y)$ be an alternative choice of isomorphism for each $g \in G$. Choosing an isomorphism $\phi \colon x \to y,$ we then obtain that the quotient of the two $2$-cochains involved can be expressed as the coboundary of the $1$-cochain given by $g \mapsto \io_x^{-1}\big( \al_g^{-1} \ci g(\phi^{-1}) \ci \be_g \ci \phi\big)$.
\end{proof}

\subsection{The group $2$-cocycle on the restricted general linear group}\label{subsection:2-cocycle}
Let $P \colon \C H \to \C H$ be an orthogonal projection acting on a separable Hilbert space $\C H$. Let $\sL^2(\C H) \subseteq \sL(\C H)$ denote the ideal of Hilbert-Schmidt operators inside the bounded operators $\sL(\C H)$. The restricted general linear group, $GL_{\T{res}}(\C H)$, is defined as the group of bounded invertible linear transformations $u$ of $\C H$ satisfying
\[
P - u P u^{-1}\in \sL^2(\C H)  \, .
\]

We are now going to construct a category $\G C_{\T{res}}$ satisfying the conditions of Assumption \ref{a:onecat}, for the unital commutative ring $R := \cc$ and the  restricted general linear group $G := GL_{\T{res}}(\C H)$. In particular, we obtain a class
\[
[c_{\T{res}}] \in H^2\big( GL_{\T{res}}(\C H), \cc^*\big)
\]
in the second group cohomology of the restricted general linear group. 

For an element $u \in GL_{\T{res}}(\C H)$, we apply the notation
\[
P_u := u P u^{-1}
\]
for the idempotent operator on $\C H$ obtained by conjugating the orthogonal projection $P$ by $u$.

We begin by recording a few standard lemmas:

\begin{lemma}\label{l:tripledif}
For every triple of elements $u,v,w \in GL_{\T{res}}(\C H)$, the difference
\[
P_w P_v P_u - P_w P_u \colon \T{Im}(P_u) \to \T{Im}(P_w)
\]
is of trace class.
\end{lemma}
\begin{proof}
This follows from the identities
\[
P_w P_v P_u - P_w P_u = P_w (P_v - P_u) P_u
= P_w v [P, v^{-1} u] u^{-1} P_u
= w [P,w^{-1} v] [ P, v^{-1} u] P u^{-1}
\]
and the fact that the product of two Hilbert-Schmidt operators is a trace class operator.
\end{proof}

\begin{corollary}\label{l:fredone}
For every pair of elements $u,v \in GL_{\T{res}}(\C H)$, the bounded operator
\[
P_v P_u \colon \T{Im}(P_u) \to \T{Im}(P_v)
\]
is a Fredholm operator.
\end{corollary}
\begin{proof}
By Lemma \ref{l:tripledif} we have that $P_u P_v \colon \T{Im}(P_v) \to \T{Im}(P_u)$ is a parametrix.
\end{proof}

\begin{definition}
We define a category $\G C_{\T{res}}$ by the following criteria:
\begin{enumerate}
\item The objects of $\G C_{\T{res}}$ are the elements in the restricted general linear group;
\item For any pair of objects $u,v \in GL_{\T{res}}(\C H)$, the set of morphisms from $u$ to $v$ is given by the graded line
\[
\T{Mor}(u,v) := |P_v P_u| \, ,
\]
where $P_v P_u \colon \T{Im}(P_u) \to \T{Im}(P_v)$ is the Fredholm operator from Lemma \ref{l:fredone}.
\item For any triple of objects $u,v,w \in GL_{\T{res}}(\C H)$, the composition of morphisms is given by the isomorphism
\[
\xymatrix{
|P_v P_u| \ot |P_w P_v| \ar[r]^{\G T} & |P_w P_v P_u| \ar[r]^{\G P} & |P_w P_u| \, ,
}
\]
obtained from composing the torsion isomorphism of Fredholm operators and the perturbation isomorphism from Section \ref{s:torsfred} and Section \ref{s:pert}.
\end{enumerate}
\end{definition}

We continue by defining the action of $GL_{\T{res}}(\C H)$:

\begin{definition}
For each $g \in GL_{\T{res}}(\C H)$, we define a functor $g \colon \G C_{\T{res}} \to \G C_{\T{res}}$ by the following:
\begin{enumerate}
\item On the objects in $\G C$, the functor $g$ is given by left multiplication with the group element $g \in GL_{\T{res}}(\C H)$;
\item For any pair of objects $u,v \in GL_{\T{res}}(\C H)$, the isomorphism
\[
g \colon \T{Mor}(u,v) \to \T{Mor}(g \cd u, g \cd v)
\]
is given by the torsion isomorphism
\[
L(g) R(g^{-1}) \colon |P_v P_u| \to | g P_v P_u g^{-1}| = | P_{gv} P_{gu}|
\]
from Example \ref{ex:LR}.
\end{enumerate}
\end{definition}

Remark that the above isomorphism $L(g) R(g^{-1}) \colon |P_v P_u| \to | P_{gv} P_{gu}|$ can equally well be described as the quasi-isomorphism $|g,g| \colon |P_v P_u| \to | P_{gv} P_{gu}|$ induced by the bounded operators $g \colon P_u \C H \to P_{gu} \C H$ and $g \colon P_v \C H \to P_{gv} \C H$.

The rest of this section is devoted to proving that the above definition indeed provides us with a category equipped with a strict action of the restricted general linear group satisfying the conditions of Assumption \ref{a:onecat}.

\begin{prop}\label{prop:assoc}
$\G C_{\T{res}}$ is a category satisfying Assumption \ref{a:onecat} $(1)$, $(2)$ and $(3)$.
\end{prop}
\begin{proof}
The only non-trivial claim is that the composition is associative. For every four objects $u,v,w,x\in GL_{\T{res}}(\C H)$, it suffices to show that every square in the following diagram commutes:
    \begin{equation*}
        \xymatrix{
            |P_vP_u|\otimes |P_wP_v| \otimes |P_xP_w| \ar[r]^>>>>>>{\T{id} \ot \G T} \ar[d]_{\G T \ot \T{id}} & |P_vP_u|\otimes |P_xP_wP_v|\ar[r]^{\T{id} \ot \G P} \ar[d]_{\G T} & |P_vP_u| \otimes |P_xP_v| \ar[d]_{\G T} \\
            |P_wP_vP_u| \otimes  |P_xP_w| \ar[r]^{\G T} \ar[d]_{\G P \ot \T{id}} & |P_xP_wP_vP_u| \ar[r]^{\G P} \ar[d]_{\G P} & |P_xP_vP_u| \ar[d]_{\G P} \\
            |P_wP_v| \otimes |P_xP_w| \ar[r]_{\G T} & |P_xP_wP_v| \ar[r]_{\G P} & |P_xP_u| \, .
        }
    \end{equation*}
The upper left square commutes by the associativity of the torsion isomorphism from Proposition \ref{p:assotors}. The lower right square commutes by the cocycle property of the perturbation isomorphism from Theorem \ref{t:pervec}. The lower left and upper right squares commute since perturbation commutes with torsion, see Theorem \ref{t:percom}.
\end{proof}

\begin{prop}\label{prop:1action}
The group $GL_{\T{res}}(\C H)$ acts strictly on the category $\G C_{\T{res}}$.
\end{prop}
\begin{proof}
The only non-trivial assertion is that the group action on morphisms respects the composition of morphisms. Thus, let $g \in GL_{\T{res}}(\C H)$. By the definition of the composition and the group action, it suffices to check that we have a commuting diagram
    \begin{equation*}
      \xymatrix{
        |P_vP_u|\otimes |P_wP_v| \ar[r]^>>>>>>{\G T} \ar[d]_{L(g)R(g^{-1})\otimes L(g)R(g^{-1})} & |P_wP_vP_u| \ar[d]_{L(g) R(g^{-1})} \ar[r]^{\G P} & |P_wP_u|\ar[d]^{L(g) R(g^{-1})}\\
        |P_{gv}P_{gu}|\otimes |P_{gw}P_{gv}| \ar[r]_>>>>{\G T} & |P_{gw}P_{gv}P_{gu}| \ar[r]_{\G P} & |P_{gw}P_{gu}| \, .
      }
    \end{equation*}
But the commutativity of the first square follows from Proposition \ref{p:torquis}, and the commutativity of the second square follows since perturbation commutes with torsion, see Theorem \ref{t:percom}.
\end{proof}

\begin{corollary}\label{def:canonical two-cocycle}
The formula from Definition \ref{def:one cocycle} defines a class in the second group cohomology of the restricted general linear group, $[c_{\T{res}}] \in H^2(GL_{\T{res}}(\C H),\cc^*)$.
\end{corollary}

\section{Comparison with the multiplicative character}\label{s:mulK2}
Let us fix an orthogonal projection $P \colon \C H \to \C H$ on the separable Hilbert space $\C H$ with infinite dimensional kernel and image. 

Define the unital subalgebra $\C M^1 \su \sL(\C H)$ by
\[
\C M^1 := \big\{ x \in \sL(\C H) \mid [P,x] \T{ is Hilbert-Schmidt} \big\} \, .
\]
Notice that $\C M^1$ becomes a unital Banach algebra when equipped with the norm $\| x \| := \| x \|_\infty + \| [P,x] \|_2$, where $\| \cd \|_\infty \colon \sL(\C H) \to [0,\infty)$ is the operator norm and $\| \cd \|_2 \colon \sL^2(\C H) \to [0,\infty)$ is the Hilbert-Schmidt norm. We remark that the group of invertible elements $(\C M^1)^*$ is exactly the restricted general linear group
\[
(\C M^1)^* = GL_{\T{res}}(\C H) \, .
\]
The universal $2$-summable Fredholm module is given by the triple $\C F_{\T{uni}} := (\C M^1, \C H, 2P - 1)$. 

In order to describe the Connes-Karoubi multiplicative character
\[
M(\C F_{\T{uni}}) \colon K_2^{\T{alg}}(\C M^1) \to \cc^*
\]
using our category theoretic approach, we need to construct a polarisation for the general linear group $GL(\C M^1)$ consisting of invertible matrices with entries in $\C M^1$. To this end, we represent $GL(\C M^1)$ on the infinite Hilbert space direct sum $\C H^\infty := \op_{j = 1}^\infty \C H$ by defining the representation
\[
GL_n(\C M^1) \to GL(\C H^\infty) \q g( \sum_{j = 1}^\infty \xi_j  \cd e_j) := \sum_{i,j = 1}^n g_{ij}(\xi_j)\cd e_i  + \sum_{i = n + 1}^\infty \xi_i \cd e_i .
\]
for each $n \in \nn$ and verifying that these representations are compatible with the direct limit structure of $GL(\C M^1) = \displaystyle\lim_{n \to \infty} GL_n(\C M^1)$. The corresponding representation $GL(\C M^1) \to GL(\C H^\infty)$ is then polarised by the infinite direct sum of orthogonal projections
\[
P^\infty \colon \C H^\infty \to \C H^\infty \q P^\infty( \sum_{j = 1}^\infty \xi_j  \cd e_j ) := \sum_{j = 1}^\infty P(\xi_j) \cd e_j .
\]
By a slight abuse of notation, we let $[c_{\T{res}}] \in H^2(GL(\C M^1),\cc^*)$ denote the class in group cohomology coming from the action of $GL(\C M^1)$ on the category $\G C_{\T{res}}$ associated with the orthogonal projection $P^\infty \colon \C H^\infty \to \C H^\infty$.

By a result of Connes and Karoubi, \cite[Theorem 5.6]{CoKa:CMF}, we may compute the multiplicative character using a central extension of the connected component of the identity, $GL^0(\C M^1) \subseteq GL(\C M^1)$:
\begin{equation}\label{eq:central}
1 \to \cc^* \to  \Ga \to GL^0(\C M^1) \to 1 \, .
\end{equation}
This central extension is described in detail by Pressley and Segal in \cite[Chapter 6]{PrSe:LG} and arises from a combination of the Fredholm determinant and the exact sequence of Banach algebras
\[
0 \to \sL^1(P\C H) \to \C T^1 \to \C M^1 \to  0 \, .
\]
Recall here that the unital Banach algebra $\C T^1 \subseteq \C M^1 \ti \sL( P\C H)$ is defined by
\[
\C T^1 := \big\{ (x,y) \mid x \in \C M^1 \, , \, \, y \in \sL(P\C H) \, , \, \, PxP - y \in \sL^1(P\C H) \big\}
\]
and equipped with the norm $\| (x,y)\| := \| x \| + \| P x P - y \|_1$, where $\| \cd \|_1 \colon \sL^1(P \C H) \to [0,\infty)$ is the trace norm defined on all trace class operators.

\begin{theorem}\label{t:two cocycle and mcc}
The following diagram is commutative:
\[
\xymatrix{
K^{\T{alg}}_2(\C M^1)\ar[rr]^h \ar[d]_{M(\C F_{\T{uni}})} && H_2(GL(\C M^1),\zz)\ar[dll]^{\, \, \, \, \inn{ [c_{\T{res}}], \cd}} \\
\cc^* && 
}
\]
where $h$ is the Hurewicz homomorphism and $\inn{ [c_{\T{res}}], \cd}$ comes from the pairing between group cohomology and group homology.
\end{theorem}
\begin{proof}
By the results of Connes and Karoubi mentioned just before the statement of this theorem and by the fact that $K_2^{\T{alg}}(\C M^1)$ is isomorphic to the second group homology of the elementary matrices $E(\C M^1)$ (see for example \cite[Theorem 5.2.7]{Ros:AKA}), it suffices to show that the restriction of $[c_{\T{res}}] \in H^2(GL(\C M^1),\cc^*)$ to $H^2( GL^0(\C M^1),\cc^*)$ agrees with the group cohomology class provided by the central extension in \eqref{eq:central}. Moreover, since the passage from $\C M^1$ to the $(n \ti n)$-matrices $M_n(\C M^1)$ can be described by passing from the polarised Hilbert space $P\C H \subseteq \C H$ to the polarised Hilbert space $P^{\op n} \C H^{\op n} \subseteq \C H^{\op n}$ we may restrict our attention to the group of invertible elements in $\C M^1$. That is, to the restricted general linear group $GL_{\T{res}}(\C H)$.

An element $g \in GL_{\T{res}}(\C H)$ lies in the connected component of the identity, $GL_{\T{res}}^0(\C H)$, precisely when the Fredholm operator $P g P \colon P\C H \to P\C H$ has index zero. For each such $g$ we may thus choose a bounded finite rank operator $F_g \colon P\C H \to P\C H$ such that $P g P + F_g \colon \C H \to \C H$ is invertible. The group $2$-cocycle coming from the central extension in \eqref{eq:central} is then given by
\[
m_{\T{uni}}(g,h) := \det\big(  (P gh P + F_{gh})( P h P + F_h)^{-1} (P g P + F_g)^{-1} \big) \q g,h \in GL_{\T{res}}^0(\C H) \, .
\]

Let us now turn to the description of the restriction of $[c_{\T{res}}]$ to $H^2( GL_{\T{res}}^0(\C H), \cc^*)$. We start by choosing the unit $1 \in GL_{\T{res}}^0(\C H)$ as our object. For each $g \in GL_{\T{res}}^0(\C H)$ we have the perturbation isomorphism
\[
\G P( P_g \cd P, g (P g P + F_g)^{-1}) \colon | P_g \cd P| \to | g ( P g P + F_g)^{-1} | = (\cc,0)
\]
and we may thus choose the isomorphism
\[
\al_g := \G P\big( g (P g P + F_g)^{-1}, P_g \cd P\big)(1_{\cc}) \in \T{Mor}( 1, g) = | P_g \cd P | \, .
\]
Using that perturbation commutes with torsion, Theorem \ref{t:percom}, it can then be verified that the inverse is given by
\[
\al_g^{-1} = \G P\big( (P g P + F_g) g^{-1}, P \cd P_g\big)(1_\cc) \in | P \cd P_g| = \T{Mor}( g, 1) \, .
\]
For a pair of elements $g,h \in GL_{\T{res}}^0(\C H)$, the evaluation $c_{\T{res}}(g,h) \in \cc^*$ is therefore given by the image of
\[
\al_{gh} \ot g( \al_h^{-1}) \ot \al_g^{-1} \in | P_{gh}\cd P| \ot | g \cd P P_h \cd g^{-1}| \ot | P \cd P_g|
\]
under the composition
\[
\xymatrix{
| P_{gh} \cd  P| \ot | g \cd P P_h \cd g^{-1}| \ot | P \cd P_g| \ar[r]^>>>>>>{\G T} & 
| P P_g P_{gh} P| \ar[r]^>>>>{\G P} & |P| = (\cc,0) } \, .
\]
We record that $g(\al_h^{-1}) \in |g \cd P P_h \cd g^{-1}| = | P_g P_{gh}|$ agrees with the value
\[
\G P\big( g (P h P + F_h) h^{-1} g^{-1}, g \cd P P_h \cd g^{-1} \big)(1_\cc) 
\]
as can be verified by using the argument in the proof of Proposition \ref{prop:1action}. However, using that perturbation commutes with torsion, Theorem \ref{t:percom}, together with the transitivity property of the perturbation isomorphism, Theorem \ref{t:pervec}, we obtain that the diagram
\[
\xymatrix{
\big| gh (P gh P + F_{gh})^{-1}\big| \ot \big| g (P h P + F_h) h^{-1} g^{-1} \big| \ot \big| (PgP + F_g) g^{-1} \big|
\ar[d]^{\G T} \ar[r]^>>>>{\G P} & |P_{gh} \cd P| \ot |P_g \cd P_{gh}| \ot | P \cd P_g| \ar[d]^{\G T} \\
\big| (P g P + F_g) ( P h P + F_h) (P gh P + F_{gh})^{-1} \big| \ar[r]^>>>>>>>>>>>>>>>>{\G P} \ar[d]^{\G P}  & \big| P \cd P_g \cd P_{gh} \cd P \big| \ar[dl]^{\G P} \\
|P| & 
}
\]
is commutative. This entails that $c_{\T{res}}(g,h) \in \cc^*$ is given by the image of the unit $1_{\cc} \in \cc$ under the perturbation isomorphism
\[
\G P \colon \big| (P g P + F_g) ( P h P + F_h) (P gh P + F_{gh})^{-1} \big| \to |P| = (\cc,0) \, .
\]
But by Example \ref{ex:invpert}, this is exactly the Fredholm determinant
\[
\det\big( (P gh P + F_{gh}) (P h P + F_h)^{-1} ( P g P + F_g)^{-1}\big) \, .
\]
This shows that $c_{\T{res}}(g,h) = m_{\T{uni}}(g,h)$ for all $g,h \in GL_{\T{res}}^0(\C H)$ and the theorem is proved.
\end{proof}

We record the following corollary:

\begin{corollary}\label{c:two cocycle and mcc}
Let $\C F = (\C A, \C H, F)$ be a unital $2$-summable Fredholm module such that the orthogonal projection $P := (F +1)/2$ has infinite dimensional image and kernel. Let $\pi \colon \C A \to \C M^1$ denote an algebra homomorphism such that $\C F$ agrees with the pull back of the universal $2$-summable Fredholm module along $\pi$. Then the diagram
\[
\xymatrix{
K_2^{\T{alg}}(\C A) \ar[rr]^{h} \ar[d]^{M(\C F)} && H_2( GL(\C A),\zz ) \ar[dll]^{\inn{\pi^* [c_{\T{res}}], \cd}} \\
\cc^* & & }
\]
is commutative, where $\pi^*[c_{\T{res}}]$ denotes the pull back of the group cohomology class $[c_{\T{res}}]$ along the induced group homomorphism $\pi \colon GL(\C A) \to GL(\C M^1)$.
\end{corollary}

\section{The group $3$-cocycle associated to a group action on a coproduct category}\label{s:polar}
In this section we start our discussion of the two-dimensional situation and we shall explain how group actions on categories endowed with some extra structure give rise to group $3$-cocycles. The idea of constructing group cocycles from groups acting on higher categories appears in many places, see for example \cite{Str:AOS, JoSt:BTC, BrSp:GCF, BaLa:HDA, FrZh:GRD}. We are here taking a slightly different route and develop the notion of a group acting on a ``coproduct category'', since this is the structure which appears naturally in our analytic applications. Coproduct categories are related to weak $2$-categories, but our systematic use of graded tensor products makes this relationship less straightforward. The extra sign appearing in the composition of graded tensor products of morphisms also has consequences for our construction of group $3$-cocycles: we only obtain well-defined group $3$-cohomology classes with values in the quotient group $\cc^*/\{\pm 1\}$. 
%

\subsection{Coproduct categories}\label{ss:hopf}
%

We are going to consider categories where the morphisms between two objects are vectors in a $\zz$-graded complex line.

\begin{definition}\label{d:gralines}
A \emph{category of $\zz$-graded complex lines} is a category $\G C$ with at least one object such that
\begin{enumerate}
\item The set of morphisms $\T{Mor}(a,b)$ is a $\zz$-graded complex line whenever $a,b$ are objects in $\G C$;
\item The composition of morphisms $\ci \colon \T{Mor}(b,c) \ti \T{Mor}(a,b) \to \T{Mor}(a,c)$ induces an isomorphism of $\zz$-graded complex lines:
\[
M \colon \T{Mor}(a,b) \ot \T{Mor}(b,c) \to \T{Mor}(a,c) \, ;
\]
\item For each object $a$ in $\G H(x,y)$, the $\zz$-graded complex line $(\cc,0)$ is isomorphic to $\T{Mor}(a,a)$ via the map $\la \mapsto \la \cd \T{id}_a$. In particular, $\T{Mor}(a,a)$ has degree $0 \in \zz$.
\end{enumerate}
For a morphism $\al \in \T{Mor}(a,b) = (\sL,m)$ we write $\ep(\al) := m \in \zz$ for the \emph{degree}. 
\end{definition}
%

\begin{lemma}\label{l:connected}
If $\G C$ is a category of $\zz$-graded complex lines, then any two objects $a,b \in \G C$ are isomorphic.
\end{lemma}
\begin{proof}
Let $a,b \in \G C$. Since $\ci \colon \T{Mor}(a,b) \ot \T{Mor}(b,a) \to \T{Mor}(a,a)$ is an isomorphism of $\zz$-graded complex lines we may choose $\al \in \T{Mor}(a,b)$ and $\be \in \T{Mor}(b,a)$ such that $\be \ci \al = \T{id}_a$. Since the map $\la \mapsto \la \cd \T{id}_b$ is an isomorphism of $\zz$-graded complex lines (and the composition is associative and linear) it follows easily that $\al \ci \be = \T{id}_b$ as well.
\end{proof}

Let $\G C$ and $\G D$ be categories of $\zz$-graded complex lines. The \emph{graded tensor product}
\[
\G C \ot \G D
\]
is also a category of $\zz$-graded complex lines. The objects in this category are given by pairs of objects $(a,b)$, where $a \in \G C$ and $b \in \G D$. It is convenient already at this point to introduce the notation
\[
a \ot b := (a,b) \in \T{Obj}( \G C \ot \G D) \, .
\]
The morphisms $\T{Mor}( a \ot b , a' \ot b' \big)$ are given by the $\zz$-graded complex line:
\[
\T{Mor}\big( a \ot b , a' \ot b' \big) := \T{Mor}(a,a') \ot \T{Mor}(b,b') 
\]
A morphism from $a \ot b$ to $a' \ot b'$ is thus given by a tensor product
\[
\al \ot \be \colon a \ot b \to a' \ot b' \, ,
\]
where the morphism $\al \colon a \to a'$ has degree $\ep(\al) \in \zz$ and $\be \colon b \to b'$ has degree $\ep(\be) \in \zz$. 
The unit $\T{id}_{a \ot b} \colon a \ot b \to a \ot b$ is the tensor product of the units $\T{id}_a \in \T{Mor}(a,a)$ and $\T{id}_b \in \T{Mor}(b,b)$. The composition of morphisms is given by the formula
\[
M\big( (\al \ot \be) \ot (\al' \ot \be') \big) := (-1)^{\ep(\be) \cd \ep(\al')} \cd M(\al \ot \al') \ot M(\be \ot \be') .
\]

For an extra category $\G E$ of $\zz$-graded complex lines we may identify the graded tensor products
\[
(\G C \ot \G D) \ot \G E \q \T{and} \q \G C \ot (\G D \ot \G E)
\]
via an obvious isomorphism of categories, which we shall from now on tacitly suppress.

Let $\G C$ and $\G D$ be categories of $\zz$-graded complex lines. A \emph{linear functor} $F$ from $\G C$ to $\G D$ is a functor $F \colon \G C \to \G D$ such that the induced map
\[
F \colon \T{Mor}(a,b) \to \T{Mor}(F(a),F(b)) \q a,b \in \G C
\]
is a \emph{linear isomorphism} of $\zz$-graded complex lines.

\begin{definition}
Let $X$ be a non-empty set, and let $\G H(x,y)$ be a category of $\zz$-graded complex lines for every $x,y \in X$. We say that the collection $\big\{ \G H(x,y) \big\}_{x,y \in X}$ is a \emph{coproduct category} when it is equipped with linear functors
\[
\De_z \colon \mathfrak{H}(x,y) \to \mathfrak{H}(x,z) \ot \mathfrak{H}(z,y) \q x,y,z \in X
\]
referred to as \emph{coproducts} such that the coassociativity relation
\[
(\T{id} \ot \De_z) \De_w = (\De_w \ot \T{id}) \De_z 
\colon \mathfrak{H}(x,y) \to \mathfrak{H}(x,w) \ot \mathfrak{H}(w,z) \ot \mathfrak{H}(z,y) \, .
\]
holds for all elements $x,y,z,w \in X$.
\end{definition}


A {\em morphism} of coproduct categories $\Phi\colon\mathfrak{H}\to\mathfrak{H'}$ consists of a map $\varphi \colon X \to X'$ of the underlying sets and for each $x,y \in X$, a linear functor $\varphi \colon \G H(x, y) \to \G H'(\varphi(x),\varphi(y))$ such that the diagram
\[
\xymatrix{
\G H(x,y) \ar[r]^{\De_z} \ar[d]^{\varphi} & \G H(x,z) \ot \G H(z,y) \ar[d]^{\varphi \ot \varphi} \\
\G H'(\varphi(x),\varphi(y)) \ar[r]^<<<<{\De'_{\varphi(z)}} & \G H'(\varphi(x),\varphi(z)) \ot \G H'(\varphi(z),\varphi(y)) 
}
\]
commutes for all $x,y,z \in X$. The composition of morphisms of coproduct categories is carried out in the obvious way (composing the underlying maps between sets and the linear functors). 

An action of a group $G$ on a coproduct category $\G H$ is a group homomorphism
\[
G \to \T{Aut}( \G H) \, .
\]
\medskip

We now discuss the link between coproduct categories and ``twisted'' weak $2$-categories. The twisting comes from our use of \emph{graded} tensor products. We refer to \cite{Ben:ITB} for generalities regarding weak $2$-categories (bicategories). 

Suppose that $(\G H,\De,X)$ is a coproduct category. We start out by ``inverting'' the coproducts 
\[
\De_z \colon \G H(x,y) \to \G H(x,z) \ot \G H(z,y) \, .
\]

\begin{lemma}
For each triple of elements $x,y,z \in X$, it holds that the coproduct 
\[
\De_z \colon \G H(x,y) \to \G H(x,z) \ot \G H(z,y)
\]
is an equivalence of categories.
\end{lemma}
\begin{proof}
Let $x,y,z \in X$ be given. By definition, the functor $\De_z \colon \G H(x,y) \to \G H(x,z) \ot \G H(z,y)$ is fully faithful so we only need to argue that $\De_z$ is essentially surjective. This does in fact hold since any two objects $a \ot b$ and $a'\ot b'$ are isomorphic in $\G H(x,z) \ot \G H(z,y)$ as can be seen by applying Lemma \ref{l:connected}.
\end{proof}

Using the above lemma, for each $x,y,z \in X$, we may choose a functor $\ast \colon \G H(x,z) \ot \G H(z,y) \to \G H(x,y)$ together with natural isomorphisms $\xi \colon \ast \De_z \to \T{id}_{\G H(x,y)}$ and $\eta \colon \De_z \ast \to \T{id}_{\G H(x,z) \ot \G H(z,y)}$. By construction we have that $\ast$ is linear in the sense that the induced map 
\[
\ast \colon \T{Mor}\big( a \ot b, a' \ot b' \big) \to \T{Mor}\big(\ast(a \ot b), \ast(a' \ot b') \big)
\]
is linear at the level of the underlying complex lines. It need however not be the case that $\ast$ preserves the degree.

We now define our ``twisted'' weak $2$-category $\G H(2)$:

\begin{enumerate}
\item The objects in $\G H(2)$ are the elements in the set $X$; 
\item The $1$-morphisms from an object $x$ to an object $y$ are the objects in the category $\G H(x,y)$ and the composition of $1$-morphisms $a \in \G H(x,z)$ and $b \in \G H(z,y)$ is defined by $b\ci_1 a := \ast(a \ot b)$;
\item The $2$-morphisms from a $1$-morphism $a \in \G H(x,y)$ to a $1$-morphism $a' \in \G H(x,y)$ are the vectors in the $\zz$-graded complex line $\T{Mor}(a,a')$. The vertical composition of $2$-morphisms $\al \colon a \to a'$ and $\be \colon a' \to a''$ with $a,a',a'' \in \G H(x,y)$ is given by the composition in the category $\G H(x,y)$. The horizontal composition of $2$-morphisms $\al \colon a \to a'$ and $\be \colon b \to b'$ with $a,a' \in \G H(x,z)$ and $b,b' \in \G H(z,y)$ is defined by $\be \ci_1 \al := \ast(\al \ot \be) \colon b \ci_1 a \to b' \ci_1 a'$.
\end{enumerate}

The extra adjective ``twisted'' primarily signifies that the assignment
\[
\ci_1 \colon \G H(z,y) \ti \G H(x,z) \to \G H(x,y)
\]
defined on objects by $(b,a) \mapsto b \ci_1 a$ and on morphisms by $(\be,\al) \mapsto \be \ci_1 \al$ is \emph{not} a functor because of a sign defect. Indeed, for $2$-morphisms $\be \colon b \to b'$ and $\be' \colon b' \to b''$ in $\G H(z,y)$ and $2$-morphisms $\al \colon a \to a'$ and $\al' \colon a' \to a''$ in $\G H(x,y)$ we instead have that
\[
\begin{split}
(\be' \ci_2 \be) \ci_1 (\al' \ci_2 \al) 
& = \ast( (\al' \ci_2 \al) \ot (\be' \ci_2 \be) )
= \ast( (\al' \ot \be') \ci_2 (\al \ot \be) ) \cd (-1)^{\ep(\be) \cd \ep(\al')} \\
& = (\be' \ci_1 \al') \ci_2 (\be \ci_1 \al) \cd (-1)^{\ep(\be) \cd \ep(\al')} .
\end{split}
\]
One may define unitors and associators for our $2$-category theoretic data but these operations are only satisfying twisted versions of the usual naturality, pentagon and triangle identities, meaning that these identities are only satisfied up to specified signs. For objects $a \in \G H(x,z)$, $b \in \G H(z,w)$ and $c \in \G H(w,y)$ one may for example define the associator
\[
A_{a,b,c} \colon c \ci_1 (b \ci_1 a) \to (c \ci_1 b) \ci_1 a 
\]
as the unique isomorphism making the diagram here below commute:
\[
\xymatrix{
(\De_z \ot \T{id} ) \De_w( c \ci_1 (b \ci_1 a) ) \ar[dd]^{(\De_z \ot \T{id})\De_w(A_{a,b,c})} \ar[rrr]^>>>>>>>>>>>>>>>{(\De_z \ot \T{id})(\eta_{ (b \ci_1 a) \ot c})} 
&&& \De_z(b \ci_1 a) \ot c \ar[rd]^{\, \, \, \eta_{a \ot b} \ot \T{id}_c} & \\
&&&& a \ot b \ot c \\
(\T{id} \ot \De_w) \De_z( (c \ci_1 b) \ci_1 a ) \ar[rrr]^>>>>>>>>>>>>>>>{(\T{id} \ot \De_w)(\eta_{ a \ot (c \ci_1 b)})} &&& a \ot \De_w(c \ci_1 b)
\ar[ru]_{\, \, \, \T{id}_a \ot \eta_{b \ot c}} 
}
\]
Notice in this respect how the coassociativity of our coproduct functors is applied in order to make sense of the above diagram. 

We are not currently aware of a standardised treatment of weak $2$-categories which are twisted in the above sense and believe that this is a subject for further investigations. In the present text we work in the context of coproduct categories since these appear naturally in our example driven approach. Moreover, our coproduct categories arise without making any auxiliary choices and as a consequence all our identities are strict instead of the usual ``up to natural isomorphisms'' familiar from a $2$-category theoretic framework.

\subsection{$3$-cocycles}
Suppose that $(\G H,\De,X)$ is a coproduct category and that $G \to \T{Aut}(\G H)$ is a group action. We shall now see how to construct a group $3$-cocycle from this data. Because of our systematic use of graded tensor products, this group $3$-cocycle takes values in the quotient group $\cc^*/\{ \pm 1\}$. 

\begin{definition}\label{d:3cocyc}
Choose an element $x \in X$.

For each $g \in G$, choose an object $a_g \in \mathfrak{H}(x,gx)$ and for each pair of elements $g,h \in G$, choose an isomorphism
\[
\be_{g,h} \colon \De_{gx}( a_{gh} ) \to a_g \ot g(a_h) \, .
\]

For each triple of elements $g,h,k \in G$, define the automorphism
\[
\ga(g,h,k) \colon a_g \ot g(a_h) \ot (gh)(a_k) \to a_g \ot g(a_h) \ot (gh)(a_k) \, ,
\]
in the category $\G H(x,gx) \ot \G H(gx, gh x) \ot \G H(ghx, ghk x)$, such that the diagram here below commutes:
\[
\xymatrix{
& \De_{gx}( a_{gh}) \ot (gh)(a_k) \ar[rrr]^{\be_{g,h} \ot \T{id}} &&& a_g \ot g(a_h) \ot (gh)(a_k) \ar[dd]^{\ga(g,h,k)} \\
(\De_{gx} \ot \T{id}) \De_{ghx}(a_{ghk}) \ar[ru]^{(\De_{gx} \ot \T{id})(\be_{gh,k}) \q } \ar[rd]_{(\T{id} \ot \De_{ghx})(\be_{g,hk}) \q } &&&& \\ 
& a_g \ot \De_{ghx}( g (a_{hk}) ) \ar[rrr]^{\T{id} \ot (g \ot g)(\be_{h,k})} &&& a_g \ot g(a_h) \ot (gh)(a_k) 
}
\]
We define the \emph{group $3$-cochain} $c \colon \zz[G^3] \to \cc^*$ by
\[
c(g,h,k) \cd \T{id}_{a_g \ot g(a_h) \ot (gh)(a_k)} := \ga(g,h,k) \q g,h,k \in G \, .
\]
\end{definition}

\begin{prop}\label{p:cocycle3}
The group $3$-cochain $c \colon \zz[G^3] \to \cc^*$ satisfies the following twisted group $3$-cocycle relation:
\begin{equation}\label{eq:cocycle}
(-1)^{\ep(\be_{g,h}) \ep(\be_{k,l})} \cd c(gh,k,l) \cd c(g,h,kl) = c(h,k,l) \cd c(g,hk,l) \cd c(g,h,k) 
\end{equation}
for all $g,h,k,l \in G$. 
\end{prop}
\begin{proof}
Let $g,h,k,l \in G$ be given.

We start by moving each of the automorphisms $\ga(gh,k,l)$, $\ga(g,h,kl)$, $\ga(h,k,l)$, $\ga(g,hk,l)$ and $\ga(g,h,k)$ so that they all become automorphisms of the object $a_g \ot g(a_h) \ot (gh)(a_k) \ot (ghk)(a_l)$ in the category 
\[
\G H(x,gx) \ot \G H(gx,ghx) \ot \G H(ghx, ghkx) \ot \G H(ghkx,ghklx) \, .
\]
We notice that this operation does not change the associated elements $c(gh,k,l)$, $c(g,h,kl)$, $c(h,k,l)$, $c(g,hk,l)$ and $c(g,h,k)$ in $\cc^*$. Indeed, for the left hand side of \eqref{eq:cocycle}, we consider the following two automorphisms:
\[
\begin{split}
& (\T{id} \ot \T{id} \ot (gh \ot gh)(\be_{k,l})) \ci (\T{id} \ot \T{id} \ot \De_{ghk x})(\ga(g,h,kl)) \ci (\T{id} \ot \T{id} \ot (gh \ot gh)(\be_{k,l}^{-1})) \\
& \q = (\T{id} \ot \T{id} \ot (gh \ot gh)(\be_{k,l})) \ci
(\T{id} \ot \T{id} \ot \De_{ghkx})\big( (\T{id} \ot (g \ot g)(\be_{h,kl})) \ci (\T{id} \ot \De_{ghx})(\be_{g,hkl})\big) \\
& \qq \ci (\De_{gx} \ot \De_{ghkx})(\be_{gh,kl}^{-1}) \ci (\be_{g,h}^{-1} \ot \T{id} \ot \T{id}) \ci (\T{id} \ot \T{id} \ot (gh \ot gh)(\be_{k,l}^{-1})) \\
& \qqq \colon a_g \ot g(a_h) \ot (gh)(a_k) \ot (ghk)(a_l) \to a_g \ot g(a_h) \ot (gh)(a_k) \ot (ghk)(a_l)
\end{split}
\]
and
\[
\begin{split}
& ( \be_{g,h} \ot \T{id} \ot \T{id} ) \ci (\De_{gx} \ot \T{id} \ot \T{id})( \ga(gh,k,l)) \ci (\be_{g,h}^{-1} \ot \T{id} \ot \T{id}) \\
& \q = ( \be_{g,h} \ot \T{id} \ot \T{id} )
\ci (\T{id} \ot \T{id} \ot (gh \ot gh)(\be_{k,l})) \ci (\De_{gx} \ot \De_{ghk x})(\be_{gh,kl}) \\
& \qq \ci (\De_{gx} \ot \T{id} \ot \T{id})\big( (\De_{ghx} \ot \T{id})(\be_{ghk,l}^{-1}) \ci (\be_{gh,k}^{-1} \ot \T{id})\big)
\ci ( \be_{g,h}^{-1} \ot \T{id} \ot \T{id} ) \\
& \qqq \colon a_g \ot g(a_h) \ot (gh)(a_k) \ot (ghk)(a_l) \to a_g \ot g(a_h) \ot (gh)(a_k) \ot (ghk)(a_l) \, .
\end{split}
\]
We record that their composition is given by the automorphism
\begin{equation}\label{eq:leftcocyc}
\begin{split}
& (\T{id} \ot \T{id} \ot (gh \ot gh)(\be_{k,l})) \ci
(\T{id} \ot \T{id} \ot \De_{ghkx})\big( (\T{id} \ot (g \ot g)(\be_{h,kl})) \ci (\T{id} \ot \De_{ghx})(\be_{g,hkl})\big) \\
& \q \ci (\De_{gx} \ot \T{id} \ot \T{id})\big( (\De_{ghx} \ot \T{id})(\be_{ghk,l}^{-1}) \ci (\be_{gh,k}^{-1} \ot \T{id})\big)
\ci ( \be_{g,h}^{-1} \ot \T{id} \ot \T{id} ) \cd (-1)^{\ep(\be_{g,h}) \cd \ep(\be_{k,l})} \, ,
\end{split}
\end{equation}
where the extra sign $(-1)^{\ep(\be_{g,h}) \cd \ep(\be_{k,l})}$ comes from the definition of the composition in the graded tensor product of categories
\[
\G H(x,gx) \ot \G H(gx,ghx) \ot \G H(ghx,ghk x) \ot \G H(ghkx,ghklx) \, .
\]

For the right hand side of \eqref{eq:cocycle}, we consider the following three automorphisms:
\[
\begin{split}
& ( \T{id} \ot g \ot g \ot g)(\T{id} \ot \ga(h,k,l)) \\
& \q = (\T{id} \ot \T{id} \ot (gh \ot gh)(\be_{k,l})) \ci
(\T{id} \ot \T{id} \ot \De_{ghkx})(\T{id} \ot (g \ot g)(\be_{h,kl})) \\
& \qq \ci ( \T{id} \ot \De_{ghx} \ot \T{id})( \T{id} \ot (g \ot g)(\be_{hk,l}^{-1}))
\ci ( \T{id} \ot (g \ot g)(\be_{h,k}^{-1}) \ot \T{id}) \\
& \qqq \colon a_g \ot g(a_h) \ot (gh)(a_k) \ot (ghk)(a_l) \to a_g \ot g(a_h) \ot (gh)(a_k) \ot (ghk)(a_l)
\end{split}
\]
and
\[
\begin{split}
& ( \T{id} \ot (g \ot g)(\be_{h,k}) \ot \T{id}) \ci (\T{id} \ot \De_{ghx} \ot \T{id})(\ga(g,hk,l)) \ci
( \T{id} \ot (g \ot g)(\be_{h,k}^{-1}) \ot \T{id}) \\
& \q = ( \T{id} \ot (g \ot g)(\be_{h,k}) \ot \T{id}) \ci
( \T{id} \ot \De_{ghx} \ot \T{id})( \T{id} \ot (g \ot g)(\be_{hk,l})) \\
& \qq \ci (\T{id} \ot \T{id} \ot \De_{ghkx})(\T{id} \ot \De_{ghx})(\be_{g,hkl}) \\
& \qqq \ci (\De_{gx} \ot \T{id} \ot \T{id}) (\De_{ghx} \ot \T{id})(\be_{ghk,l}^{-1}) \\
& \qqqq \ci ( \T{id} \ot \De_{ghx} \ot \T{id} )( \be_{g,hk}^{-1} \ot \T{id})
\ci ( \T{id} \ot (g \ot g)(\be_{h,k}^{-1}) \ot \T{id}) \\
& \qqq \colon a_g \ot g(a_h) \ot (gh)(a_k) \ot (ghk)(a_l) \to a_g \ot g(a_h) \ot (gh)(a_k) \ot (ghk)(a_l)
\end{split}
\]
and
\[
\begin{split}
& \ga(g,h,k) \ot \T{id} \\
& \q = ( \T{id} \ot (g \ot g)(\be_{h,k}) \ot \T{id}) \ci ( \T{id} \ot \De_{ghx} \ot \T{id} )( \be_{g,hk} \ot \T{id}) \\
& \qq \ci (\De_{gx} \ot \T{id} \ot \T{id})(\be_{gh,k}^{-1} \ot \T{id})
\ci (\be_{g,h}^{-1} \ot \T{id} \ot \T{id}) \\
& \qqq \colon a_g \ot g(a_h) \ot (gh)(a_k) \ot (ghk)(a_l) \to a_g \ot g(a_h) \ot (gh)(a_k) \ot (ghk)(a_l) \, .
\end{split}
\]
We record that their composition is given by
\begin{equation}\label{eq:rightcocyc}
\begin{split}
& (\T{id} \ot \T{id} \ot (gh \ot gh)(\be_{k,l})) \ci (\T{id} \ot \T{id} \ot \De_{ghkx})(\T{id} \ot (g \ot g)(\be_{h,kl})) \\
& \q \ci (\T{id} \ot \T{id} \ot \De_{ghkx})(\T{id} \ot \De_{ghx})(\be_{g,hkl}) \ci (\De_{gx} \ot \T{id} \ot \T{id}) (\De_{ghx} \ot \T{id})(\be_{ghk,l}^{-1}) \\
& \qq \ci (\De_{gx} \ot \T{id} \ot \T{id})(\be_{gh,k}^{-1} \ot \T{id})
\ci (\be_{g,h}^{-1} \ot \T{id} \ot \T{id}) \, .
\end{split}
\end{equation}
Since the two automorphisms of $a_g \ot g(a_h) \ot (gh)(a_k) \ot (ghk)(a_l)$ given in \eqref{eq:leftcocyc} and \eqref{eq:rightcocyc} agree up to the sign $(-1)^{\ep(\be_{g,h}) \cd \ep(\be_{k,l})}$, we have proved the proposition.
\end{proof}

As a consequence of the above proposition, we obtain a group $3$-cocycle after passing to the quotient of $\cc^*$ by the subgroup $\{\pm 1\} \su \cc^*$. We denote the associated class in group cohomology by
\[
[c] \in H^3\big(G, \cc^*/\{\pm 1\} \big) .
\]

For any finite number of elements $x_1,x_2,\ldots,x_n \in X$, we introduce the notation
\[
\begin{split}
\De_{x_1,x_2,\ldots,x_n} & := ( \De_{x_1} \ot \T{id}^{\ot (n-1)} ) \ci (\De_{x_2} \ot \T{id}^{\ot (n-2)}) \cilci \De_{x_n} \\
& \q \colon \G H(x,y) \to \G H(x,x_1) \ot \G H(x_1,x_2) \ot \ldots \ot \G H(x_n,y) \q x,y \in X \, .
\end{split}
\]

\begin{lemma}\label{l:classindep}
The class $[c]\in H^3\big(G,\cc^*/\{\pm 1\} \big)$ is independent of the choices made in Definition \ref{d:3cocyc}.
\end{lemma}
\begin{proof}
Let $x' \in X$ be an alternative element, let $a_g' \in \G H(x', gx')$, $g \in G$, be alternative objects and let $\be_{g,h}' \colon \De_{g x'}(a_{gh}') \to a_g' \ot g(a_h')$, $g,h \in G$ be alternative isomorphisms. We let $c' \colon \zz[G^3] \to \cc^*$ denote the alternative $3$-cochain and we will show that $c'$ agrees with $c \colon \zz[G^3] \to \cc^*$ up to signs and a coboundary.

We start by choosing a set of intertwiners in the following way: Choose an object $\ze \in \G H(x,x')$. For each $g \in G$, $\ze \ot \De_{gx}(a_g')$ and $\De_{x'}(a_g) \ot g(\ze)$ are then both objects in the category $\G H(x,x') \ot \G H(x',gx) \ot \G H(gx,gx')$ and we may thus choose an isomorphism
\[
\si_g \colon \ze \ot \De_{gx}(a_g') \to \De_{x'}(a_g) \ot g(\ze) \, .
\]
For any finite number of group elements $g_1,g_2,\ldots,g_n \in G$, we then have the intertwining isomorphism
\[
\begin{split}
\si_{g_1,g_2,\ldots,g_n}
& \colon \ze \ot \De_{g_1 x}(a_{g_1}') \ot g_1^{\ot 2} \De_{g_2 x} (a_{g_2}') \ot \ldots \ot (g_1 \clc g_{n-1} )^{\ot 2} \De_{g_n x} (a_{g_n}') \\
& \q \to \De_{x'}(a_{g_1}) \ot g_1^{\ot 2} \De_{x'}(a_{g_2}) \olo (g_1 \clc g_{n-1})^{\ot 2}\De_{x'}(a_{g_n}) \ot (g_1 g_2 \clc g_n)(\ze) \, ,
\end{split}
\]
defined as the composition
\[
\begin{split}
\si_{g_1,g_2,\ldots,g_n} & := \big( \T{id}^{\ot 2(n-1)} \ot (g_1 g_2 \clc g_{n-1})^{\ot 3} ( \si_{g_n}) \big) \\
& \qq \cilci \big(\T{id}^{\ot 2} \ot g_1^{\ot 3} (\si_{g_2}) \ot \T{id}^{\ot 2(n-2)}\big) \ci \big(\si_{g_1} \ot \T{id}^{\ot 2(n-1)}\big) \, .
\end{split}
\]

For each $g,h \in G$, we define the automorphism
\[
\de(g,h) \colon \De_{x',gx,gx'}(a_{gh}) \ot (gh)(\ze) \to \De_{x',gx,gx'}(a_{gh}) \ot (gh)(\ze)
\]
as the composition of isomorphisms:
\[
\begin{split}
\De_{x',gx,gx'}(a_{gh}) \ot (gh)(\ze) & \to^{(\De_{x'} \ot \De_{gx'})(\be_{g,h}) \ot \T{id}}
\De_{x'}(a_g) \ot \De_{gx'} g( a_h) \ot (gh)(\ze) \\
& \to^{\si_{g,h}^{-1}}
\ze \ot \De_{gx}(a_g') \ot g^{\ot 2}\De_{hx}(a_h') \\
& \to^{\T{id} \ot (\De_{gx} \ot \De_{ghx})(\be_{g,h}')^{-1}}
\ze \ot \De_{gx,gx',ghx}(a_{gh}') \\
& \to^{(\T{id} \ot \De_{gx,gx'} \ot \T{id})(\si_{gh})}
\De_{x',gx,gx'}(a_{gh}) \ot (gh)(\ze) \, .
\end{split}
\]
We identify $\de(g,h)$ with the number $d(g,h) \in \cc^*$ and thereby obtain a $2$-cochain
\[
d \colon \zz[G^2] \to \cc^* \, .
\]
We claim that the coboundary of $d$ agrees with the quotient of $c$ and $c'$, thus that
\[
d(h,k) \cd d(gh,k)^{-1} \cd d(g,hk) \cd d(g,h)^{-1} = c(g,h,k) \cd c'(g,h,k)^{-1} \, ,
\]
for all $g,h,k \in G$.

Thus, let $g,h,k \in G$ be given. We are going to compare automorphisms of the object
\[
\De_{x'}(a_g) \ot \De_{gx'} g(a_h) \ot \De_{gh x'} (gh)(a_k) \ot (ghk)(\ze) \, .
\]
We thus start by representing the numbers $c(g,h,k)$ and $c'(g,h,k)^{-1}$ as automorphisms of this object. The number $c(g,h,k) \in \cc^*$ can be represented by the automorphism
\[
\begin{split}
& (\De_{x'} \ot \De_{gx'} \ot \De_{ghx'})\big(\ga(g,h,k) \big) \ot \T{id} \\
& \q = \big( \T{id}^{\ot 2} \ot g^{\ot 4} (\De_{x'} \ot \De_{hx'})(\be_{h,k}) \ot \T{id} \big)
\ci \big( (\De_{x'} \ot \De_{gx',ghx,ghx'})(\be_{g,hk}) \ot \T{id} \big) \\
& \qq \ci \big( (\De_{x',gx,gx'} \ot \De_{ghx'} )(\be_{gh,k}^{-1}) \ot \T{id}\big)
\ci \big( (\De_{x'} \ot \De_{gx'})(\be_{g,h}^{-1}) \ot \T{id}^{\ot 3} \big) \, ,
\end{split}
\]
whereas the number $c'(g,h,k)^{-1} \in \cc^*$ can be represented by the automorphism
\[
\begin{split}
& \si_{g,h,k} \ci \big( \T{id} \ot (\De_{gx} \ot \De_{ghx} \ot \De_{ghk x})(\ga'(g,h,k)^{-1}) \big) \ci \si_{g,h,k}^{-1} \\
& \q = \si_{g,h,k} \ci
\big( \T{id} \ot (\De_{gx} \ot \De_{ghx})(\be_{g,h}') \ot \T{id}^{\ot 2} \big)
\ci \big( \T{id} \ot (\De_{gx,gx',ghx} \ot \De_{ghkx} )(\be_{gh,k}') \big) \\
& \qq \ci \big( \T{id} \ot (\De_{gx} \ot \De_{ghx,ghx',ghkx})(\be_{g,hk}')^{-1} \big)
\ci \big( \T{id}^{\ot 3} \ot g^{\ot 4} (\De_{hx} \ot \De_{hkx})(\be_{h,k}')^{-1} \big)
\ci \si_{g,h,k}^{-1} \, .
\end{split}
\]

In order to compare the quotient $\frac{ c(g,h,k) }{c'(g,h,k)}$ with the quotient $\frac{d(h,k) \cd d(g,hk)}{d(gh,k) \cd d(g,h)}$, we now represent the numbers $d(g,h)^{-1}$, $d(gh,k)^{-1}$ and $d(g,hk)$, $d(h,k)$ as automorphisms of convenient objects. We represent $d(g,h)^{-1} \in \cc^*$ as the automorphism of the object
\[
( \De_{x', gx, gx'} \ot \De_{ghx'} )\big(a_{gh} \ot (gh)(a_k) \big) \ot (ghk)(\ze) \, ,
\]
given by
\[
\begin{split}
& \big( \T{id}^{\ot 4} \ot (gh)^{\ot 3} (\si_k) \big) \ci (\de(g,h)^{-1} \ot \T{id}^{\ot 2} ) \ci \big( \T{id}^{\ot 4} \ot (gh)^{\ot 3} (\si_k^{-1}) \big) \\
& \q = (-1)^{\ep(\si_k) \cd \ep(\be_{g,h})} \cd \big( (\De_{x'} \ot \De_{g x'})(\be_{g,h}^{-1}) \ot \T{id}^{\ot 3}\big) 
\ci \big( \T{id}^{\ot 4} \ot (gh)^{\ot 3} (\si_k) \big) \ci ( \si_{g,h} \ot \T{id}^{\ot 2} ) \\
& \qq \ci \big( \T{id} \ot (\De_{gx} \ot \De_{ghx})(\be'_{g,h}) \ot \T{id}^{\ot 2} \big)
\ci \big( ( \T{id} \ot \De_{gx,gx'} \ot \T{id})(\si_{gh}^{-1}) \ot \T{id}^{\ot 2} \big) \\
& \qqq \ci \big( \T{id}^{\ot 4} \ot (gh)^{\ot 3} (\si_k^{-1}) \big) \\
& \q = (-1)^{\ep(\si_k) \cd \ep(\be_{g,h})} \cd \big( (\De_{x'} \ot \De_{g x'})(\be_{g,h}^{-1}) \ot \T{id}^{\ot 3}\big) 
\ci \si_{g,h,k} \\
& \qq \ci \big( \T{id} \ot (\De_{gx} \ot \De_{ghx})(\be'_{g,h}) \ot \T{id}^{\ot 2} \big)
\ci ( \T{id} \ot \De_{gx,gx'} \ot \T{id}^{\ot 3} )(\si_{gh,k}^{-1}) \, .
\end{split}
\]
We represent $d(gh,k)^{-1} \in \cc^*$ as the automorphism of the object
\[
\De_{x',gx,gx',ghx,ghx'}(a_{ghk}) \ot (ghk)(\ze) \, ,
\]
given by
\[
\begin{split}
& ( \T{id} \ot \De_{gx,gx'} \ot \T{id}^{\ot 3})( \de(gh,k)^{-1}) \\
& \q = \big( ( \De_{x', gx, gx'} \ot \De_{gh x'})(\be_{gh,k}^{-1}) \ot \T{id} \big)
\ci ( \T{id} \ot \De_{gx,gx'} \ot \T{id}^{\ot 3} )(\si_{gh,k}) \\
& \qq \ci \big( \T{id} \ot (\De_{gx,gx',ghx} \ot \De_{ghk x})( \be_{gh,k}') \big)
\ci ( \T{id} \ot \De_{gx,gx',ghx, ghx'} \ot \T{id}) (\si_{ghk}^{-1}) \, .
\end{split}
\]
Similarly, we represent the number $d(g,hk) \in \cc^*$ by the automorphism of the object
\[
\De_{x',gx,gx',ghx,ghx'}(a_{ghk}) \ot (ghk)(\ze) \, ,
\]
given by
\[
\begin{split}
& (\T{id}^{\ot 3} \ot \De_{ghx,ghx'} \ot \T{id})( \de(g,hk) ) \\
& \q = ( \T{id} \ot \De_{gx,gx',ghx, ghx'} \ot \T{id}) (\si_{ghk})
\ci \big( \T{id} \ot ( \De_{gx} \ot \De_{ghx, ghx',ghkx})(\be'_{g,hk})^{-1} \big) \\
& \qq \ci (\T{id}^{\ot 3} \ot \De_{ghx,ghx'} \ot \T{id} )(\si_{g,hk}^{-1})
\ci \big( ( \De_{x'} \ot \De_{gx', ghx, ghx'})(\be_{g,hk}) \ot \T{id} \big) \, .
\end{split}
\]
We finally represent the number $d(h,k) \in \cc^*$ by the automorphism of the object
\[
\De_{x'}(a_g) \ot \De_{gx', ghx, ghx'} g(a_{hk}) \ot (ghk)(\ze) \, ,
\]
determined by
\[
\begin{split}
& \T{id}^{\ot 2} \ot g^{\ot 5} \de(h,k)  \\
& \q = \big( \T{id}^{\ot 2} \ot ( \T{id} \ot \De_{ghx, gh x'} \ot \T{id}) g^{\ot 3} \si_{hk} \big)
\ci \big( \T{id}^{\ot 3} \ot (\De_{ghx} \ot \De_{ghk x})g^{\ot 2} (\be'_{h,k})^{-1} \big) \\
& \qq \ci \big( \T{id}^{\ot 2} \ot g^{\ot 5} \si_{h,k}^{-1} \big)
\ci \big( \T{id}^{\ot 2} \ot g^{\ot 4} (\De_{x'} \ot \De_{hx'})(\be_{h,k}) \ot \T{id} \big) \, .
\end{split}
\]

With these identifications, we may represent the product $d(gh,k)^{-1} \cd d(g,h)^{-1} \in \cc^*$ by the automorphism
\[
\begin{split}
& (-1)^{\ep(\si_k) \cd \ep(\be_{g,h})} \cd \big( ( \De_{x', gx, gx'} \ot \De_{gh x'})(\be_{gh,k}^{-1}) \ot \T{id} \big) 
\ci \big( (\De_{x'} \ot \De_{g x'})(\be_{g,h}^{-1}) \ot \T{id}^{\ot 3}\big) 
\ci \si_{g,h,k} \\
& \q \ci \big( \T{id} \ot (\De_{gx} \ot \De_{ghx})(\be'_{g,h}) \ot \T{id}^{\ot 2} \big)
\ci \big( \T{id} \ot (\De_{gx,gx',ghx} \ot \De_{ghk x})( \be_{gh,k}') \big) \\
& \qq \ci ( \T{id} \ot \De_{gx,gx',ghx, ghx'} \ot \T{id}) (\si_{ghk}^{-1})
\end{split}
\]
of the object 
\[
\De_{x',gx,gx',ghx,ghx'}(a_{ghk}) \ot (ghk)(\ze) \, .
\]
Furthermore, we may represent the product $d(g,hk) \cd d(h,k) \in \cc^*$ by the automorphism
\[
\begin{split}
& \big( ( \De_{x'} \ot \De_{gx', ghx, ghx'})(\be_{g,hk}) \ot \T{id} \big) 
\ci ( \T{id} \ot \De_{gx,gx',ghx, ghx'} \ot \T{id}) (\si_{ghk}) \\
& \q \ci \big( \T{id} \ot ( \De_{gx} \ot \De_{ghx, ghx',ghkx})(\be'_{g,hk})^{-1} \big) 
\ci ( \si_g^{-1} \ot \T{id}^{\ot 4} )
\ci \big( \T{id}^{\ot 3} \ot (\De_{ghx} \ot \De_{ghk x})g^{\ot 2} (\be'_{h,k})^{-1} \big) \\
& \qq \ci \big( \T{id}^{\ot 2} \ot g^{\ot 5} \si_{h,k}^{-1} \big)
\ci \big( \T{id}^{\ot 2} \ot g^{\ot 4} (\De_{x'} \ot \De_{hx'})(\be_{h,k}) \ot \T{id} \big) \\
& \q = (-1)^{\ep(\si_g) \cd \ep(\be_{h,k}')} \cd \big( ( \De_{x'} \ot \De_{gx', ghx, ghx'})(\be_{g,hk}) \ot \T{id} \big) 
\ci ( \T{id} \ot \De_{gx,gx',ghx, ghx'} \ot \T{id}) (\si_{ghk}) \\
& \qq \ci \big( \T{id} \ot ( \De_{gx} \ot \De_{ghx, ghx',ghkx})(\be'_{g,hk})^{-1} \big)
\ci \big( \T{id}^{\ot 3} \ot (\De_{ghx} \ot \De_{ghk x})g^{\ot 2} (\be'_{h,k})^{-1} \big) \\
& \qqq \ci \si_{g,h,k}^{-1} 
\ci \big( \T{id}^{\ot 2} \ot g^{\ot 4} (\De_{x'} \ot \De_{hx'})(\be_{h,k}) \ot \T{id} \big) 
\end{split} 
\]
of the object
\[
\De_{x'}(a_g) \ot \De_{gx', ghx, ghx'} g(a_{hk}) \ot (ghk)(\ze) \, .
\]

Combining these observations, we may represent the quotient $\frac{c(g,h,k)}{c'(g,h,k)} \in \cc^*$ by the automorphism
\[
\begin{split}
& \big( \T{id}^{\ot 2} \ot g^{\ot 4} (\De_{x'} \ot \De_{hx'})(\be_{h,k}) \ot \T{id} \big)
\ci \big( (\De_{x'} \ot \De_{gx',ghx,ghx'})(\be_{g,hk}) \ot \T{id} \big) \\
& \qq \ci  (-1)^{\ep(\si_k) \cd \ep(\be_{g,h})} \cd d(gh,k)^{-1} d(g,h)^{-1} \cd (\T{id} \ot \De_{gx,gx',ghx,ghx'} \ot \T{id})(\si_{ghk}) \\
& \qq \ci \big( \T{id} \ot (\De_{gx} \ot \De_{ghx,ghx',ghkx})(\be_{g,hk}')^{-1} \big)
\ci \big( \T{id}^{\ot 3} \ot g^{\ot 4} (\De_{hx} \ot \De_{hkx})(\be_{h,k}')^{-1} \big)
\ci \si_{g,h,k}^{-1} \, ,
\end{split}
\]
acting on the object
\[
\De_{x'}(a_g) \ot \De_{gx'} g(a_h) \ot \De_{gh x'} (gh)(a_k) \ot (ghk)(\ze) \, .
\]
But this latter automorphism can now be seen to represent the quotient 
\[
(-1)^{\ep(\si_k) \cd \ep(\be_{g,h}) + \ep(\si_g) \cd \ep(\be_{h,k}')} \cd \frac{d(h,k) \cd d(g,hk)}{d(g,h) \cd d(gh,k)} \in \cc^* \, .
\]
This proves the lemma.
\end{proof}

\begin{lemma}\label{l:cocycconj}
Let $\Phi\colon \G H_1\to\G H_2$ be a morphism of coproduct categories. Let $G$ be a group and let 
\[
\rho_i\colon G\to \T{Aut}(\G H_i)
\]
be a group homomorphism for $i=1,2$. Assume that for all $g\in G$, $\rho_2(g)\Phi=\Phi\rho_1(g)$. Then
\[
[c_1]=[c_2]\in H^3\big(G,\cc^*/\{\pm 1\} \big) \, .
\]
\end{lemma}
\begin{proof}
This is a consequence of Lemma \ref{l:classindep}. Let $X_i$ be the set underlying $\G H_i$ for $i = 1,2$. Choose $x\in X_1$. For each $g\in G$, choose an object $a_g\in \G H_1(x,gx)$ and for each $g,h\in G$ choose an isomorphism
\[
\be_{g,h}\colon \Delta_{gx}(a_{gh})\to a_g\ot g(a_h) \, .
\]
These choices then yield the group $3$-cochain $c_1 \colon \zz[G^3] \to \cc^*$. 

Applying $\Phi$, we obtain $x':=\varphi(x)\in X_2$. Because $\Phi$ is $G$-equivariant, for each $g\in G$, $a'_g:=\varphi(a_g)\in\G H_2(x',gx')$.  Moreover, because $\Phi$ is a morphism of coproduct categories, for each $g,h\in G$, we obtain an isomorphism
\[
\be'_{g,h}:=\varphi(\be_{g,h})\colon\Delta_{gx'}(a'_{gh})\to a'_g\ot g(a'_h) \, .
\]
These choices then yield the group $3$-cochain $c_2 \colon \zz[G^3] \to \cc^*$.

For all $g,h,k\in G$, our choices determine automorphisms
\begin{align*}
\gamma(g,h,k)&\colon a_g\ot g(a_h)\ot gh(a_k)\to a_g\ot g(a_h)\ot gh(a_k)\\
\gamma'(g,h,k)&\colon a'_g\ot g(a'_h)\ot gh(a'_k)\to a'_g\ot g(a'_h)\ot gh(a'_k)
\end{align*}
as in Definition \ref{d:3cocyc}. In particular, we have that
\[
c_1(g,h,k) \cd \T{id}_{a_g\ot g(a_h)\ot gh(a_k)} = \ga(g,h,k) \, \, \T{ and } \, \, \,
c_2(g,h,k) \cd \T{id}_{a_g'\ot g(a_h')\ot gh(a_k')} = \ga'(g,h,k) \, .
\]
Further, because $\Phi$ is a $G$-equivariant morphism of coproduct categories,
\[
\varphi(\gamma(g,h,k))=\gamma'(g,h,k)\colon a'_g\ot g(a'_h)\ot gh(a'_k)\to a'_g\ot g(a'_h)\ot gh(a'_k) \, .
\]
Finally, because $\Phi$ is a morphism of categories, it is unital and we obtain that $c_1(g,h,k) = c_2(g,h,k)$. This proves the lemma.
\end{proof}





\section{Categories associated to representations of a ring}\label{s:category}
Throughout this section we fix a separable Hilbert space $\C H$, a unital ring $R$ and a two-sided ideal $I$ in $R$. We recall that $\sL^1(\C H) \subseteq \sL(\C H)$ denotes the ideal of trace class operators inside the bounded operators on the Hilbert space $\C H$. The unit in $R$ is denoted by $1$.

We work under the following assumption:

\begin{assu}\label{a:rep}
Suppose that we are given a family
\[
\mathfrak{Rep} = \{ \pi_\la \}_{\la \in \La},
\]
indexed by a \emph{non-empty} set $\Lambda$, of (not necessarily unital) representations of the unital ring $R$ as bounded operators on $\C H$ satisfying the following conditions:
\begin{enumerate}
\item for all $x\in R$ and pairs $\la,\mu \in \La$,
\[
\pi_\la(x)\pi_\mu(1)-\pi_\la(1)\pi_\mu(x)\in \sL^1(\C H) \, ;
\]
\item for all $i\in I$ and triples $\la,\mu,\nu \in \La$,
\[
\pi_\la(i)\pi_\mu(1)\pi_\nu(1)-\pi_\la(i)\pi_\nu(1) \in \sL^1(\C H) \, .
\]
\end{enumerate}
\end{assu}

Under Assumption \ref{a:rep} and given a fixed idempotent $p_0 \in R$ serving as a ``base point'', we shall in this section see how to build two categories $\G L_p$ and $\G L_p^\da$, whenever $p \in R$ is an idempotent with $p - p_0 \in I$. The category $\G L_p^\da$ will play the role of a dual to the category $\G L_p$.
\medskip

We start by introducing some notation:

\begin{notation}\label{n:invfred}
For $\la,\mu \in \La$ we define the following bounded operators:
\begin{enumerate}
\item For an idempotent $p \in R$, define
\[
\begin{split}
\Om(\la,\mu)(p) & := \ma{cc}{ \pi_\la(p) - \pi_\la(p) \pi_\mu(p) \pi_\la(p) & \pi_\la(p) \pi_\mu(p) \\
2 \pi_\mu(p)\pi_\la(p) - ( \pi_\mu(p)\pi_\la(p) )^2 & \pi_\mu(p) \pi_\la(p) \pi_\mu(p) - \pi_\mu(p)
} \\
& \q \colon \pi_\la(p)\C H \op \pi_\mu(p)\C H \to \pi_\la(p)\C H \op \pi_\mu(p)\C H \, .
\end{split}
\]
We sometimes consider $\Om(\la,\mu)(p)$ as a bounded operator on $\C H \op \C H$ by putting it equal to zero on the range of the idempotent $\T{id} - \pi_\la(p) \op \pi_\mu(p)$.
\item For idempotents $p,q \in R$ with $p - q \in I$, define
\[
\begin{split}
F(\la,\mu)(p,q) & := \big( \pi_\la(q) \op \pi_\mu(p) \big) \Om(\la,\mu)(1) \big( \pi_\la(p) \op \pi_\mu(q) \big) \\
& \q \colon \pi_\la(p) \C H \op \pi_\mu(q) \C H \to \pi_\la(q) \C H \op \pi_\mu(p) \C H \, .
\end{split}
\]
\end{enumerate}
\end{notation}

\begin{notation}\label{n:bigfred}
For $n \geq 3$, a tuple of indices $\la = (\la_1,\la_2,\ldots,\la_n)$ in $\La$ and $1 \leq i < j \leq n$ we apply the following notation:
\begin{enumerate}
\item For every tuple of idempotents $p = (p_1,p_2,\ldots,p_n)$ in $R$ with mutual differences in $I$ and with $p_i = p_j := q,$ we define 
\[
\Om^{ij}(\la)(p) := \Om(\la_i,\la_j)(q) + \sum_{k \neq i,j} \pi_{\la_k}(p_k) \colon \bop_{k = 1}^n \pi_{\la_k}( p_k ) \C H 
\to \bop_{k = 1}^n \pi_{\la_k}( p_k ) \C H \, .
\]
\item For every tuple of idempotents $p = (p_1,p_2,\ldots,p_n)$ in $R$ with mutual differences in $I,$ we define 
\[
F^{ij}(\la)(p) := F(\la_i,\la_j)(p_i,p_j) + \sum_{k \neq i,j} \pi_{\la_k}(p_k)
\colon \bop_{k = 1}^n \pi_{\la_k}(p_k) \C H \to \bop_{k = 1}^n \pi_{\la_k}(p_{\tau(k)}) \C H \, , 
\]
where $\tau \colon \{1,2,\ldots,n\} \to \{1,2,\ldots,n\}$ is the transposition which interchanges $i$ and $j$. 
\end{enumerate}
\end{notation}


In the following lemmas we present a few consequences of our assumptions, in particular we shall see that the bounded operators $F(\la,\mu)(p,q) \colon \pi_\la(p)\C H \op \pi_\mu(q) \C H \to \pi_\la(q)\C H \op \pi_\mu(p) \C H$ are in fact Fredholm operators.

\begin{lemma}\label{l:commut}
Suppose that $\la,\mu \in \La,$ and that $x,y \in R$ are two elements with $x - y \in I$. Then the difference
\[
\big( \pi_\la(y) \op \pi_\mu(x) \big) \Om(\la,\mu)(1)
- \Om(\la,\mu)(1) \big( \pi_\la(x) \op \pi_\mu(y) \big)
\]
is a trace class operator on $\C H \op \C H$.
\end{lemma}
\begin{proof}
We compute that
\[
\begin{split}
& \big( \pi_\la(y) \op \pi_\mu(x) \big) \Om(\la,\mu)(1)
- \Om(\la,\mu)(1) \big( \pi_\la(x) \op \pi_\mu(y) \big) \\
& \q = \ma{cc}{\pi_\la(y) - \pi_\la(y) \pi_\mu(1) \pi_\la(1)  & \pi_\la(y) \pi_\mu(1) \\
2 \pi_\mu(x)\pi_\la(1) - \pi_\mu(x)\pi_\la(1) \pi_\mu(1) \pi_\la(1) & \pi_\mu(x) \pi_\la(1) \pi_\mu(1) - \pi_\mu(x) } \\
& \qq - \ma{cc}{\pi_\la(x) - \pi_\la(1) \pi_\mu(1) \pi_\la(x) & \pi_\la(1) \pi_\mu(y) \\
2 \pi_\mu(1) \pi_\la(x) - \pi_\mu(1) \pi_\la(1) \pi_\mu(1) \pi_\la(x) & \pi_\mu(1) \pi_\la(1) \pi_\mu(y) - \pi_\mu(y) } \, ,
\end{split}
\]
and notice that the off diagonal terms lie in $\sL^1(\C H)$ by Assumption \ref{a:rep} $(1)$ and the diagonal terms lie in $\sL^1(\C H)$ by Assumption \ref{a:rep} $(1)$ and $(2)$. 
\end{proof}

\begin{lemma}\label{l:reduction}
Suppose that $\la,\mu \in \La,$ and that $p \in R$ is an idempotent. Then
\[
F(\la,\mu)(p,p) - \Om(\la,\mu)(p) \in \sL^1(\pi_\la(p)\C H \op \pi_\mu(p) \C H)
\]
and
\[
\big( \Om(\la,\mu)(p) \big)^2 =  \pi_\la(p) \op \pi_\mu(p) \, ,
\]
hence $\Om(\la,\mu)(p)$ is invertible on the range of the idempotent $\pi_\la(p)\oplus \pi_\mu(p)$.
\end{lemma}
\begin{proof}
Using Assumption \ref{a:rep} and the fact that $p \in R$ is an idempotent, we see that the difference $\pi_\la(p) \pi_\mu(1) - \pi_\la(p) \pi_\mu(p)$ is of trace class (and similarly for $\la$ and $\mu$ reversed). The first statement of the lemma now follows by noticing that
\[
F(\la,\mu)(p,p) = \ma{cc}{ \pi_\la(p) - \pi_\la(p) \pi_\mu(1) \pi_\la(p) & \pi_\la(p) \pi_\mu(p) \\
2 \pi_\mu(p)\pi_\la(p) - \pi_\mu(p)\pi_\la(1) \pi_\mu(1) \pi_\la(p) & \pi_\mu(p) \pi_\la(1) \pi_\mu(p) - \pi_\mu(p)
} \, .
\]
The second statement can be proved by setting $T := \pi_\la(p) \pi_\mu(p)$ and $T^\da := \pi_\mu(p) \pi_\la(p)$ and verifying that
\[
\begin{split}
& \Om(\la,\mu)(p) = \ma{cc}{\pi_\la(p) - T T^\da & T \\ 2 T^\da - T^\da T T^\da & T^\da T - \pi_\mu(p)} \\
& \q = \ma{cc}{\pi_\la(p) &  0 \\ T^\da & - \pi_\mu(p)} \cd \ma{cc}{\pi_\la(p) & T \\ 0 & \pi_\mu(p)} \cd \ma{cc}{\pi_\la(p) & 0 \\ - T^\da & \pi_\mu(p)} \\
& \q = \ma{cc}{\pi_\la(p) &  0 \\ T^\da & \pi_\mu(p)} \cd \ma{cc}{\pi_\la(p) & -T \\ 0 & \pi_\mu(p)} \cd \ma{cc}{\pi_\la(p) & 0 \\  T^\da & -\pi_\mu(p)}
 \, . \qedhere
\end{split}
\]
\end{proof}

One may also consider the case where $\la = \mu \in \La,$ and $p,q \in R$ are two idempotents with $p - q \in I$. Then
\[
\Om( \la, \la)(1) = \ma{cc}{0 & \pi_\la(1) \\ \pi_\la(1) & 0}
\]
and the bounded operator
\[
F(\la,\la)(p,q) = \ma{cc}{ 0 & \pi_\la(q) \\ \pi_\la(p) & 0} \colon \pi_\la(p)\C H \op \pi_\la(q) \C H \to 
\pi_\la(q)\C H \op \pi_\la(p) \C H
\]
is therefore invertible.

\begin{lemma}\label{l:Fredholm}
Suppose that $\la,\mu \in \La,$ and that $p,q \in R$ are idempotents with $p - q \in I$. Then
\[
F(\la,\mu)(p,q) \colon \pi_\la(p)\C H \op \pi_\mu(q)\C H \to \pi_\la(q)\C H \op \pi_\mu(p)\C H
\]
is a Fredholm operator with parametrix $F(\la,\mu)(q,p) \colon \pi_\la(q)\C H \op \pi_\mu(p)\C H \to \pi_\la(p)\C H \op \pi_\mu(q)\C H$. In fact we have that
\begin{equation}\label{eq:paramet}
F(\la,\mu)(q,p) F(\la,\mu)(p,q) - \pi_\la(p) \op \pi_\mu(q) \in \sL^1( \pi_\la(p) \C H \op \pi_\mu(q) \C H) \, .
\end{equation}
\end{lemma}
\begin{proof}
Since we may interchange the roles of the idempotents $p$ and $q$, it suffices by Atkinson's theorem to prove the inclusion in \eqref{eq:paramet}. But this inclusion follows in a straightforward way from Lemma \ref{l:commut} and Lemma \ref{l:reduction}. 
\end{proof}

\subsection{The category $\G L_p$ and its dual $\G L_p^\da$}\label{ss:catdual} We continue working under Assumption \ref{a:rep}, and on top of that we now fix an idempotent $p_0 \in R$ which will play the role of a base point. We shall moreover fix an idempotent $p \in R$ with $p - p_0 \in I$. 
\medskip

We construct a category of $\zz$-graded complex lines $\G L_p$ as follows (see Definition \ref{d:gralines}):
\begin{enumerate}
\item the non-empty set of objects of $\mathfrak{L}_p$ is the index set $\La$ for the family of representations $\G{Rep} = \{ \pi_\la \}_{\la \in \La},$
\item the morphisms from $\la$ to $\mu$ are the elements in the graded determinant line of the Fredholm operator $F(\la,\mu)(p,p_0)$, thus
\[
\G{L}_p(\la,\mu) := \big| F(\la,\mu)(p,p_0) \big| \, ,
\]
\item for each $\la \in \La$, the unit $\T{id}_\la \colon \la \to \la$ is the element
\[
1 \in (\cc,0) = \left| \ma{cc}{0 & \pi_\la(p_0) \\ \pi_\la(p) & 0} \right| = \G{L}_p(\la,\la) \, .
\]
\end{enumerate}
The composition of morphisms will be given below in Definition \ref{def:comp}.
\medskip

We also construct a category of $\zz$-graded complex lines $\G L_p^\da$ having the same objects as $\G L_p$ and with morphisms from $\la$ to $\mu$ given by the graded determinant line
\[
\G{L}_p^\da(\la,\mu) := \big| F(\la,\mu)(p_0,p) \big| \, .
\]
The unit $\T{id}_\la \colon \la \to \la$ is the multiplicative unit in $\cc$,
\[
1 \in (\cc,0) = \left| \ma{cc}{0 & \pi_\la(p) \\ \pi_\la(p_0) & 0} \right| = \G L_p^\da(\la,\la) \, .
\]
The composition of morphisms in the category $\G L_p^\da$ is given below in Definition \ref{def:dualcomp}.
\medskip

For each pair of elements $(\la,\mu)$ in the index set $\La$, there are isomorphisms of $\zz$-graded lines
\[
(\Cb,0) \to^{\varphi} \G L_p^\da(\la,\mu) \otimes \G L_p(\la,\mu)\q \T{and} \q
\G L_p(\la,\mu)\otimes \G L_p^\da(\la,\mu) \to^{\psi} (\cc,0) \, ,
\]
defined as follows: Using Lemma \ref{l:Fredholm} we define $\varphi$ as the composition
\begin{equation}\label{eq:dualI}
\begin{split}
        \varphi \colon (\Cb,0) = |\pi_\la(p_0)\op\pi_\mu(p)|
& \to^{\mathfrak{P}} |F(\la,\mu)(p,p_0)\cd F(\la,\mu)(p_0,p)| \\
& \to^{\mathfrak{T}^{-1}} |F(\la,\mu)(p_0,p)|\ot |F(\la,\mu)(p,p_0)| \, ,
\end{split}
\end{equation}
and similarly, we define $\psi$ as the composition
    \begin{equation}\label{eq:dualII}
\begin{split}
        \psi\colon |F(\la,\mu)(p,p_0)| \ot |F(\la,\mu)(p_0,p)|
& \to^{\mathfrak{T}} | F(\la,\mu)(p_0,p) \cd F(\la,\mu)(p,p_0)| \\
& \to^{\mathfrak{P}} | \pi_\la(p) \op \pi_\mu(p_0)| = (\cc,0) \, .
\end{split}
    \end{equation}
We shall sometimes write $\varphi(\la,\mu) \colon (\cc,0) \to \G L_p^\da(\la,\mu) \ot \G L_p(\la,\mu)$ and $\psi(\la,\mu) \colon \G L_p(\la,\mu) \ot \G L_p^\da(\la,\mu) \to (\cc,0)$ instead of $\varphi$ and $\psi$ when we want to be more specific about the objects involved.


\begin{lemma}\label{l:inverse}
Let $(\la,\mu)$ be a pair of elements in $\La$. It holds that $\G L_p^\da(\la,\mu)$ is right dual to $\G L_p(\la,\mu)$ in the sense that the following two diagrams
\[
\xymatrix{
\G L_p(\la,\mu) \ar[r]^>>>>>{\T{id} \ot \varphi} \ar[dr]_{\T{id}} & \G L_p(\la,\mu) \ot \G L_p^\da(\la,\mu) \ot \G L_p(\la,\mu) \ar[d]^{\psi \ot \T{id}} \\
& \G L_p(\la,\mu)
} \q \mbox{and} \q
\xymatrix{
\G L_p^\da(\la,\mu) \ar[r]^>>>>>{\varphi \ot \T{id} } \ar[dr]_{\T{id}} & \G L_p^\da(\la,\mu) \ot \G L_p(\la,\mu) \ot \G L_p^\da(\la,\mu) \ar[d]^{\T{id} \ot \psi} \\
& \G L_p^\da(\la,\mu)
}
\]
commute.
\end{lemma}
\begin{proof}
We only verify that $(\psi \otimes \T{id})\circ (\T{id} \otimes \varphi)=\T{id}$; the proof in the remaining case is similar.

We suppress the pair $(\la,\mu)$ from the notation and consider the following diagram
    \begin{equation*}
        \xymatrix{
            |F(p,p_0)| \ar[rr]^>>>>>>>>>{\T{id}\otimes\mathfrak{P}} \ar[drr]_{\mathfrak{P}} && |F(p,p_0)|\ot |F(p,p_0)\cd F(p_0,p)| \ar[rr]^{\T{id}\otimes \mathfrak{T}^{-1}} \ar[d]^{\mathfrak{T}} &&|F(p,p_0)|\ot |F(p_0,p)|\ot |F(p,p_0)| \ar[d]^{\mathfrak{T}\ot \T{id}}\\
            && |F(p,p_0)\cd F(p_0,p)\cd F(p,p_0)| \ar[rr]^{\mathfrak{T}^{-1}} \ar[drr]_{\mathfrak{P}} && |F(p_0,p)\cd F(p,p_0)|\ot |F(p,p_0)| \ar[d]^{\mathfrak{P}\ot \T{id}} \\
            &&&& |F(p,p_0)|
        }
    \end{equation*}
    The upper and lower triangles commute because torsion commutes with perturbation, see Theorem \ref{t:percom}. The square commutes by the associativity of torsion, see Proposition \ref{p:assotors}. The composition of the diagonal maps is equal to the identity on $\G L_p(\la,\mu)$ by the associativity of perturbation, see Theorem \ref{t:pervec} $(1)$. The composition of the upper horizontal maps with the right vertical maps is $(\psi\otimes \T{id})\circ(\T{id} \otimes \varphi)$ by definition.
\end{proof}

\subsection{The composition of morphisms}\label{ss:comp}
We continue working under Assumption \ref{a:rep} and we will moreover fix two idempotents $p$ and $p_0 \in R$ with $p - p_0 \in I$. The idempotent $p_0$ will serve as a base point, whereas the idempotent $p$ labels the tentative categories of $\zz$-graded complex lines $\G L_p$ and $\G L_p^\da$. For the remainder of this section, we define the composition of morphisms in $\G L_p$ and $\G L_p^\da$, see Definition \ref{def:comp} and Definition \ref{def:dualcomp}, and state the required properties. The main proofs are given in Section \ref{s:proofcat}.

Throughout this subsection $(\la,\mu,\nu)$ is a fixed triple of elements in the index set $\La$ and this triple will therefore often be suppressed from the notation. We define the Fredholm operators 
\[
\begin{split}
F^{12} & := F(\la,\mu)(p,p_0) + \pi_\nu(p_0) \\ 
& \q \colon \pi_\la(p) \C H \op \pi_\mu(p_0) \C H \op \pi_\nu(p_0) \C H 
\to \pi_\la(p_0) \C H \op \pi_\mu(p) \C H \op \pi_\nu(p_0) \C H  \\
F_\da^{12} & := F(\la,\mu)(p_0,p) + \pi_\nu(p_0) \\
& \q \colon \pi_\la(p_0) \C H \op \pi_\mu(p) \C H \op \pi_\nu(p_0) \C H 
\to \pi_\la(p) \C H \op \pi_\mu(p_0) \C H \op \pi_\nu(p_0) \C H 
\end{split}
\]
and similarly for the remaining cases, $F^{23} := F(\mu,\nu)(p,p_0) + \pi_\la(p_0)$, $F^{23}_\da := F(\mu,\nu)(p_0,p) + \pi_\la(p_0)$ and $F^{13} := F(\la,\nu)(p,p_0) + \pi_\mu(p_0)$, $F^{13}_\da := F(\la,\nu)(p_0,p) + \pi_\mu(p_0)$. Comparing with Notation \ref{n:bigfred}, we thus have that $F^{12} = F^{12}(\la,\mu,\nu)(p,p_0,p_0)$, $F^{12}_\da = F^{12}(\la,\mu,\nu)(p_0,p,p_0)$ and so on.

\begin{lemma}\label{l:1}
For any $i\in I,$ it holds that
\[
( \pi_\la(i)\oplus 0\oplus 0)\big(\Omega^{13}(1)\Omega^{23}(1)\Omega^{12}(1) - \pi_\la(1) \op \pi_\mu(1) \op \pi_\nu(1)\big)
\in \sL^1\big( \pi_\la(1)\C H \op \pi_\mu(1)\C H \op \pi_\nu(1) \C H \big) \, .
\]
\end{lemma}
\begin{proof}
We compute modulo the trace ideal inside the bounded operators, using the notation $\sim_1$ for the corresponding equivalence relation. Recalling the definitions from Notation \ref{n:invfred} and Notation \ref{n:bigfred} and applying Assumption \ref{a:rep}, we obtain that
\[
\begin{split}
& ( \pi_\la(i) \op 0 \op 0 ) \cd \Om^{13}(1) \sim_1 \ma{ccc}{0 & 0 & \pi_\la(1) \pi_\nu(1) \\ 0 & 0 & 0 \\ 0 & 0 & 0}
\cd ( 0 \op 0 \op \pi_\nu(i) ) \\
& ( 0 \op 0 \op \pi_\nu(i) ) \cd \Om^{23}(1) \sim_1 \ma{ccc}{0 & 0 & 0 \\ 0 & 0 & 0 \\ 0 & \pi_\nu(1) \pi_\mu(1) & 0}
\cd (0 \op \pi_\mu(i) \op 0)
\q \mbox{and} \\
& (0 \op \pi_\mu(i) \op 0) \cd \Om^{12}(1) \sim_1 \ma{ccc}{0 & 0 & 0 \\ \pi_\mu(1) \pi_\la(1) & 0 & 0 \\ 0 & 0 & 0 } \cd (\pi_\la(i) \op 0 \op 0 ) \, .
\end{split}
\]
The result of the lemma then follows from the computation
\[
\begin{split}
& \ma{ccc}{0 & 0 & \pi_\la(1) \pi_\nu(1) \\ 0 & 0 & 0 \\ 0 & 0 & 0}
\cd \ma{ccc}{0 & 0 & 0 \\ 0 & 0 & 0 \\ 0 & \pi_\nu(1) \pi_\mu(1) & 0}
\cd \ma{ccc}{0 & 0 & 0 \\ \pi_\mu(1) \pi_\la(1) & 0 & 0 \\ 0 & 0 & 0 } \cd (\pi_\la(i) \op 0 \op 0 ) \\
& \q = \pi_\la(1) \pi_\nu(1) \pi_\mu(1) \pi_\la(i) \op 0 \op 0
\sim_1 \pi_\la(i) \op 0 \op 0 \, . \qedhere
\end{split}
\]
\end{proof}

\begin{lemma}\label{l:pertcomp}
The following two bounded operators on $\pi_\la(1)\C H \op \pi_\mu(1)\C H \op \pi_\nu(1)\C H$,
\[
\begin{split}
& F^{13}_\da F^{23} F^{12} + \pi_\la(1 - p)\oplus \pi_\mu(1 - p_0)\oplus \pi_\nu(1 -p_0) \q \mbox{and} \\
& \Om^{13}(p_0) \Om^{23}(p_0) \Om^{12}(p_0) + \pi_\la(1 - p_0)\oplus \pi_\mu(1 - p_0)\oplus \pi_\nu(1 - p_0) \, ,
\end{split}
\]
agree modulo the trace ideal $\sL^1(\pi_\la(1)\C H \op \pi_\mu(1)\C H \op \pi_\nu(1)\C H)$.
\end{lemma}
\begin{proof} 
According to Notation \ref{n:bigfred} (and the conventions explained in the beginning of this subsection), we clarify that the composition of Fredholm operators $F^{13}_\da F^{23} F^{12}$ makes sense and yields a bounded endomorphism of the Hilbert space $\pi_\la(p) \C H \op \pi_\mu(p_0) \C H \op \pi_\nu(p_0) \C H$. Similarly, we record that $\Om^{13}(p_0) \Om^{23}(p_0) \Om^{12}(p_0)$ yields an invertible bounded endomorphism of the Hilbert space $\pi_\la(p_0) \C H \op \pi_\mu(p_0) \C H \op \pi_\nu(p_0) \C H$.

Using Lemma \ref{l:commut}, Lemma \ref{l:reduction} and Lemma \ref{l:1}, we obtain the result from the following computation modulo $\sL^1\big( \pi_\la(1)\C H \op \pi_\mu(1)\C H \op \pi_\nu(1) \C H \big)$:
\[
\begin{split}
& F^{13}_\da F^{23} F^{12} -
\pi_\la(p)\oplus \pi_\mu(p_0)\oplus \pi_\nu(p_0) \\
& \q \sim_1
\big(\pi_\la(p)\oplus \pi_\mu(p_0)\oplus \pi_\nu(p_0)\big) \cd
\Big(\Omega^{13}(1) \Omega^{23}(1) \Omega^{12}(1)
- \pi_\la(1) \op \pi_\mu(1) \op \pi_\nu(1) \Big) \\
& \q \sim_1
\big(\pi_\la(p_0)\oplus \pi_\mu(p_0)\oplus \pi_\nu(p_0)\big) \cd
\Big(\Omega^{13}(1) \Omega^{23}(1) \Omega^{12}(1)
- \pi_\la(1) \op \pi_\mu(1) \op \pi_\nu(1) \Big) \\
& \q \sim_1
\Omega^{13}(p_0) \Omega^{23}(p_0) \Omega^{12}(p_0)
- \pi_\la(p_0)\oplus \pi_\mu(p_0)\oplus \pi_\nu(p_0) \, . \qedhere
\end{split}
\]
\end{proof}

Using Lemma \ref{l:pertcomp}, we may define an isomorphism of $\zz$-graded lines:
\begin{equation}\label{eq:comp}
\begin{split}
\mu_p & \colon \big|F(\la,\mu)(p,p_0) \big| \ot \big|F(\mu,\nu)(p,p_0) \big| \ot \big| F(\la,\nu)(p_0,p) \big| 
\to^{\G S} \big| F^{12} \big| \ot \big| F^{23} \big| \ot \big| F^{13}_\da \big| \\
& \q \to^{\G T} \big| F^{13}_\da F^{23} F^{12}\big| 
\to^{\G S} \big| F^{13}_\da  F^{23} F^{12} + \pi_\la(1 - p) \op \pi_\mu(1 - p_0) \op \pi_\nu(1 - p_0) \big| \\
& \q \to^{\G P}
\big|\Om^{13}(p_0) \cd \Om^{23}(p_0) \cd \Om^{12}(p_0) + \pi_\la(1 - p_0) \op \pi_\mu(1 - p_0) \op \pi_\nu(1 - p_0) \big|= (\Cb,0) \, .
\end{split}
\end{equation}

We use the above trivialisation to define the composition of morphisms in $\G L_p$:

\begin{definition}\label{def:comp} For $\la,\mu,\nu \in \La$ the isomorphism of $\zz$-graded lines
\[
\mathfrak{M}_p \colon \mathfrak{L}_p(\la,\mu) \otimes \mathfrak{L}_p(\mu,\nu)\to \mathfrak{L}_p(\la,\nu) 
\]
is defined as the composition
\[
\begin{split}
& \big|F(\la,\mu)(p,p_0) \big| \ot \big|F(\mu,\nu)(p,p_0) \big| \\
& \q \to^{\T{id} \ot \T{id} \ot \varphi} \big|F(\la,\mu)(p,p_0) \big| \ot \big|F(\mu,\nu)(p,p_0) \big|
\ot \big| F(\la,\nu)(p_0,p) \big| \ot \big| F(\la,\nu)(p,p_0) \big| \\
& \q \to^{\mu_p \ot \T{id}} \big| F(\la,\nu)(p,p_0) \big| \, .
\end{split}
\]
\end{definition}

Using Lemma \ref{l:pertcomp} together with Lemma \ref{l:reduction} and Lemma \ref{l:Fredholm}, we may also define the following isomorphism of $\zz$-graded lines:
\begin{equation}\label{eq:dualcomp}
\begin{split}
\mu_p^\da & \colon \big|F(\la,\nu)(p,p_0) \big| \ot \big|F(\mu,\nu)(p_0,p) \big| \ot \big| F(\la,\mu)(p_0,p) \big|
\to^{\G S} \big| F^{13} \big| \ot \big| F^{23}_\da \big| \ot \big| F^{12}_\da \big| \\
& \q \to^{\G T} \big| F^{12}_\da F^{23}_\da F^{13}\big| 
\to^{\G S} \big| F^{12}_\da  F^{23}_\da F^{13} + \pi_\la(1 - p) \op \pi_\mu(1 - p_0) \op \pi_\nu(1 - p_0) \big| \\
& \q \to^{\G P}
\big|\Om^{12}(p_0) \cd \Om^{23}(p_0) \cd \Om^{13}(p_0) + \pi_\la(1 - p_0) \op \pi_\mu(1 - p_0) \op \pi_\nu(1- p_0) \big|= (\Cb,0) \, .
\end{split}
\end{equation}

We now define the composition of morphisms in $\G L_p^\da$:

\begin{definition}\label{def:dualcomp} For $\la,\mu,\nu \in \La$ the isomorphism of $\zz$-graded complex lines
\[
\mathfrak{M}_p^\da \colon  \mathfrak{L}^\da_p(\mu,\nu) \ot \mathfrak{L}^\da_p(\la,\mu) \to \mathfrak{L}^\da_p(\la,\nu)
\]
is defined as the composition
\[
\begin{split}
& \big|F(\mu,\nu)(p_0,p) \big| \ot \big|F(\la,\mu)(p_0,p) \big| \\
& \q \to^{\varphi \ot \T{id} \ot \T{id}} \big|F(\la,\nu)(p_0,p) \big| \ot \big|F(\la,\nu)(p,p_0) \big|
\ot \big| F(\mu,\nu)(p_0,p) \big| \ot \big| F(\la,\mu)(p_0,p) \big| \\
& \q \to^{\T{id} \ot \mu_p^\da} \big| F(\la,\nu)(p_0,p) \big| \, .
\end{split}
\]
\end{definition}

The next result explains the relationship between the compositions in the $\zz$-graded categories $\G L_p$ and $\G L_p^\da$ and we present a full proof later on in Section \ref{s:proofcat}.

\begin{prop}\label{p:dualcomp}
For each $\la,\mu,\nu \in \La$, the following diagram of isomorphisms of $\zz$-graded lines is commutative:
\[
            \xymatrix{
                \G L_p(\la,\mu) \ot \G L_p(\mu,\nu) \ot \G L_p^\da(\mu,\nu) \ot \G L_p^\da(\la,\mu) \ar[rr]^>>>>>>>>{\T{id}\ot \psi \ot \T{id}} \ar[d]_{\G M_p \ot \G M_p^\da} 
&& \G L_p(\la,\mu) \ot \G L_p^\da(\la,\mu) \ar[d]^{\psi} \\
\G L_p(\la,\nu) \ot \G L_p^\da(\la,\nu) \ar[rr]_{\psi} && (\cc,0)
            }
\]
\end{prop}

The next two theorems, which we also prove in Section \ref{s:proofcat}, show that $\G L_p$ is indeed a category of $\zz$-graded complex lines (see Definition \ref{d:gralines}). Similar theorems can be stated and proved for $\G L_p^\da$, in fact these unitality and associativity theorems can be verified using the duality relation in Proposition \ref{p:dualcomp} together with the corresponding theorems for $\G L_p$.

\begin{theorem}\label{thm:unitality} Suppose that the conditions in Assumption \ref{a:rep} are satisfied and that $p_0 \in R$ is a fixed idempotent. For any idempotent $p \in R$ with $p - p_0 \in I$, the pair $(\mathfrak{L}_p,\mathfrak{M}_p)$ satisfies the unitality condition, i.e. the following diagrams are commutative:
        \[
            \xymatrix{
                \mathfrak{L}_p(\la,\mu) \ar[rr]^{\T{id}\otimes \T{id}_\mu} \ar[drr]_{\T{id}} & & \mathfrak{L}_p(\la,\mu)\otimes\mathfrak{L}_p(\mu,\mu)\ar[d]^{\mathfrak{M}_p} & &\mathfrak{L}_p(\la,\mu) \ar[rr]^{\T{id}_\la \otimes \T{id}} \ar[drr]_{\T{id}} && \mathfrak{L}_p(\la,\la)\otimes\mathfrak{L}_p(\la,\mu) \ar[d]^{\mathfrak{M}_p}\\
                && \mathfrak{L}_p(\la,\mu) & & && \mathfrak{L}_p(\la,\mu)
            }
        \]
\end{theorem}

\begin{theorem}\label{t:associativity}
Suppose that the conditions in Assumption \ref{a:rep} are satisfied and that $p_0 \in R$ is a fixed idempotent. For any idempotent $p \in R$ with $p - p_0 \in I$, the pair $(\mathfrak{L}_p,\mathfrak{M}_p)$ satisfies the associativity condition, i.e. the following diagram is commutative:
\[
\xymatrix{
&\mathfrak{L}_p(\la,\mu)\otimes \mathfrak{L}_p(\mu,\nu)\otimes \mathfrak{L}_p(\nu,\tau)\ar[dl]_{\mathfrak{M}_p\otimes \T{id}}\ar[dr]^{\T{id}\otimes \mathfrak{M}_p}  &\\
\mathfrak{L}_p(\la,\nu)\otimes \mathfrak{L}_p(\nu,\tau)\ar[dr]_{\mathfrak{M}_p}&&\mathfrak{L}_p(\la,\mu)\otimes \mathfrak{L}_p(\mu,\tau)\ar[dl]^{\mathfrak{M}_p}\\
&\mathfrak{L}_p(\la,\tau)&
}
\]
\end{theorem}

\section{{ The coproduct category associated to representations of a ring}}\label{s:hopf}
Throughout this section we fix a unital ring $R$ and an ideal $I$ in $R$. We assume that
\[
\G{Rep} = \{ \pi_\la \}_{\la \in \La}
\]
is a family of (not necessarily unital) algebra homomorphisms from $R$ to the bounded operators on a fixed separable Hilbert space $\C H$. This family of representations is assumed to satisfy Assumption \ref{a:rep} and we fix an idempotent $p_0$ in the unital ring $R$. To ease the notation, we let
\[
\T{Idem}(R) \su R
\]
denote the subset of idempotent elements in $R$. We refer back to Section \ref{s:category} for the definition of the categories of $\zz$-graded complex lines $\G L_p$ and $\G L_q^\da$ appearing here below. The generalities needed regarding graded tensor products and coproduct categories can be found in Subsection \ref{ss:hopf}.

\begin{definition}\label{d:hopf}
We define a coproduct category $(\G H, \Delta, X)$ as follows:
\begin{enumerate}
\item The underlying set $X$ is given by $X := \big\{ p \in \T{Idem}(R) \mid p - p_0 \in I \big\}$.
\item For each $p,q \in X$ we define the category of $\zz$-graded complex lines $\G H(p,q)$ to be the full subcategory of the graded tensor product of 
$\G L_p$ and $\G L_q^\da$ on the subset of objects 
\[
D_\La := \big\{ \la \ot \la \mid \la \in \La \big\} \su \T{Obj}( \G L_p \ot \G L_q^\da) \, .
\]
Thus, we define
\[
\G H(p,q) := \big( \G L_p \ot \G L_q^\da \big) \big|_{D_\La}\, , 
\]
where $(-)|_{D_\La}$ denotes full subcategory on $D_\La$. We identify the set of objects in $\G H(p,q)$ with $\La$ via the isomorphism $\La \to D_\La$ given by $\la \mapsto \la \ot \la$.
\item Given $p,e,q \in X$, the coproduct
\[
\De_e \colon \G H(p,q) \to \G H(p,e) \ot \G H(e,q)
\]
is given by $\la \mapsto \la \ot \la$ on objects and on morphisms it is determined by the isomorphism
\[
\G L_p(\la,\mu) \ot \G L_q^\da(\la,\mu) 
\to^{\T{id} \ot \varphi \ot \T{id}}
\G L_p(\la,\mu) \ot \G L_e^\da(\la,\mu)  \ot \G L_e(\la,\mu) \ot \G L_q^\da(\la,\mu) \, , 
\]
of $\zz$-graded complex lines, where we recall that the isomorphism $\varphi \colon (\cc,0) \to \G L_e^\da(\la,\mu)  \ot \G L_e(\la,\mu)$ was introduced in Equation \eqref{eq:dualI}.
\end{enumerate}
We refer to the fixed idempotent $p_0 \in R$ as the \emph{base point} of the coproduct category $\G H$.
\end{definition}

We remark that our tentative coproduct category $\G H$ depends on the choice of the idempotent $p_0 \in R$. The dependency on this choice of base point will be examined in detail in Section \ref{s:change}.
\medskip

We now prove that Definition \ref{d:hopf} yields a coproduct category in the sense of Subsection \ref{ss:hopf}. For later reference, we state this result as a theorem and note that the proof is a consequence of Lemma \ref{l:category}, Lemma \ref{l:coassoc} and Proposition \ref{p:coprodfun}.

\begin{theorem}\label{t:copcat}
The assignments in Definition \ref{d:hopf} form a coproduct category $\G H$.
\end{theorem}

We start by showing that $\G H(p,q)$ is indeed a category for every $p,q \in X$:

\begin{lemma}\label{l:category}
Let $p,q \in X$. Then $\G H(p,q)$ is a category of $\zz$-graded complex lines.
\end{lemma}
\begin{proof}
The fact that $\G H(p,q)$ is a category of $\zz$-graded complex lines follows from Theorem \ref{thm:unitality} and Theorem \ref{t:associativity} together with the corresponding theorems for the dual category $\G L_q^\da$. See also the generalities regarding graded tensor products from Subsection \ref{ss:hopf}.
\end{proof}

For $p,q \in X$ and $\la,\mu,\nu \in \La$ we let
\[
\G M_{p,q} \colon \G H(p,q)(\la,\mu) \ot \G H(p,q)(\mu,\nu) \to \G H(p,q)(\la,\nu)
\]
denote the isomorphism of $\zz$-graded complex lines coming from the composition in $\G H(p,q)$.

The next lemma has a straightforward proof, which we therefore omit.

\begin{lemma}\label{l:coassoc}
The coproduct operation $\Delta$ of Definition \ref{d:hopf} is coassociative in the sense that
\[
(\T{id} \ot \De_f)\De_e = (\De_f \ot \T{id}) \De_e \colon \G H(p,q) \to \G H(p,e) \ot \G H(e,f) \ot \G H(f,q) \, ,
\]
whenever $p,e,f,q \in X$.
\end{lemma}

The proof that the various coproducts are functorial requires more care. We state the result here, but the proof will occupy the remainder of this section.

\begin{prop}\label{p:coprodfun}
The coproduct $\Delta$ of Definition \ref{d:hopf} is functorial in the sense that
    \begin{equation*}
\De_e \colon \G{H}(p,q)\to \G{H}(p,e) \ot \G{H}(e,q)
    \end{equation*}
is a linear functor, whenever $p,e,q \in X$.
\end{prop}

To prove Proposition \ref{p:coprodfun}, we need the following lemmas.

\begin{lemma}\label{l:trivmult1}
Let $\la,\mu,\nu \in \La$ and $p \in X$ be given. Then the trivialisation
\begin{equation}\label{eq:trivmul}
\mu_p^\da \ot \mu_p \colon \G L_p(\la,\nu) \ot \G L_p^\da(\mu,\nu) \ot \G L_p^\da(\la,\mu)
\ot \G L_p(\la,\mu) \ot \G L_p(\mu,\nu) \ot \G L_p^\da(\la,\nu) \to (\cc,0)
\end{equation}
agrees with the trivialisation
\begin{equation}\label{eq:trivvar}
\begin{split}
& \psi \circ (\T{id} \ot \varphi^{-1} \ot \T{id}) \circ (\T{id}^{\ot 2} \ot \varphi^{-1} \ot \T{id}^{\ot 2}) \\
& \q \colon \G L_p(\la,\nu) \ot \G L_p^\da(\mu,\nu) \ot \G L_p^\da(\la,\mu)
\ot \G L_p(\la,\mu) \ot \G L_p(\mu,\nu) \ot \G L_p^\da(\la,\nu) \to (\cc,0) \, ,
\end{split}
\end{equation}
where the trivialisations $\mu_p$ and $\mu_p^\da$ are defined in Equation \eqref{eq:comp} and Equation \eqref{eq:dualcomp}, respectively.
\end{lemma}
\begin{proof}
To ease the notation in this proof we suppress the idempotents $p$ and $p_0$ and put
\[
F(\la,\mu) := F(\la,\mu)(p,p_0) \q \T{and} \q F_\da(\la,\mu) := F(\la,\mu)(p_0,p) \, .
\]
We are sometimes also suppressing the tuple of indices $(\la,\mu,\nu)$ and hence applying the notation from the beginning of Subsection \ref{ss:comp}.

By the definition of $\varphi$ and $\psi$ from Equation \eqref{eq:dualI} and \eqref{eq:dualII}, the isomorphism in Equation \eqref{eq:trivvar} is given by the composition
\[
\begin{split}
&   |F(\la,\nu)|\ot|F_\da(\mu,\nu)|\ot|F_\da(\la,\mu)|\ot|F(\la,\mu)|\ot|F(\mu,\nu)|\ot|F_\da(\la,\nu)| \\
&\q \to^{\T{id}^{\ot 2}\ot(\G P \ci \G T) \ot \T{id}^{\ot 2}}
|F(\la,\nu)|\ot|F_\da(\mu,\nu)| \ot |F(\mu,\nu)|\ot |F_\da(\la,\nu)| \\
& \q \to^{\T{id} \ot (\G P \ci \G T) \ot \T{id}}
|F(\la,\nu)| \ot |F_\da(\la,\nu)|
\to^{\G P \ci \G T} | \pi_\la(p) \op \pi_\nu(p_0)| = (\cc,0) \, .
\end{split}
\]
Without altering this isomorphism, we could have stabilised beforehand (see Proposition \ref{p:torsta} and Proposition \ref{p:persta}), thus achieving the following alternative description of the above isomorphism:
\[
\begin{split}
&   |F(\la,\nu)|\ot|F_\da(\mu,\nu)|\ot|F_\da(\la,\mu)|\ot|F(\la,\mu)|\ot|F(\mu,\nu)|\ot|F_\da(\la,\nu)| \\
& \q \to^{\G S}
|F^{13}|\ot|F_\da^{23}|\ot|F_\da^{12}|\ot|F^{12}|\ot|F^{23}|\ot|F_\da^{13}| \\
&\q \to^{\T{id}^{\ot 2}\ot(\G P \ci \G T) \ot \T{id}^{\ot 2}}
|F^{13}|\ot|F_\da^{23}| \ot |F^{23}|\ot |F_\da^{13}| \\
& \q \to^{\T{id} \ot (\G P \ci \G T) \ot \T{id}}
|F^{13}| \ot |F_\da^{13}|
\to^{\G P \ci \G T} | \pi_\la(p) \op \pi_\mu(p_0) \op \pi_\nu(p_0)| = (\cc,0) \, .
\end{split}
\]
Gathering all the torsion isomorphisms and all the perturbation isomorphisms (using Theorem \ref{t:percom}) and using the associativity of these operations (Theorem \ref{t:torprop} and Theorem \ref{t:pervec}), we then obtain the formula
\[
\begin{split}
&   |F(\la,\nu)|\ot|F_\da(\mu,\nu)|\ot|F_\da(\la,\mu)|\ot|F(\la,\mu)|\ot|F(\mu,\nu)|\ot|F_\da(\la,\nu)| \\
& \q \to^{(\G T \ci \G S) \ot (\G T \ci \G S)}
\big| F_\da^{12} \cd F_\da^{23} \cd F^{13} \big| 
\ot \big| F_\da^{13} \cd F^{23} \cd F^{12} \big| \\
& \q \to^{\G T} \big| F_\da^{13} \cd F^{23} \cd F^{12} \cd  F_\da^{12} \cd F_\da^{23} \cd F^{13} \big| \\
& \q \to^{\G P} | \pi_\la(p) \op \pi_\mu(p_0) \op \pi_\nu(p_0)| = (\cc,0) \, ,
\end{split}
\]
for the isomorphism in Equation \eqref{eq:trivvar}. However, using that stabilisation commutes with torsion and perturbation (Proposition \ref{p:torsta} and Proposition \ref{p:persta}), and that torsion commutes with perturbation (Theorem \ref{t:percom}) we now see that the isomorphism in Equation \eqref{eq:trivvar} agrees with the isomorphism
\[
\begin{split}
& |F(\la,\nu)|\ot|F_\da(\mu,\nu)|\ot|F_\da(\la,\mu)|\ot|F(\la,\mu)|\ot|F(\mu,\nu)|\ot|F_\da(\la,\nu)| \\
& \q \to^{(\G{T} \ci \G S)\ot (\G{T} \ci \G S)} |F_\da^{12} \cd F_\da^{23}\cd F^{13} |\ot
|F_\da^{13}\cd F^{23}\cd F^{12}| \\
& \q \to^{\G S \ot \G S}
|F_\da^{12} \cd F_\da^{23}\cd F^{13} + \pi_\la(1 - p) \op \pi_\mu(1 - p_0) \op \pi_\mu(1 - p_0) | \\
& \qqq \qqq \ot
|F_\da^{13}\cd F^{23}\cd F^{12} + \pi_\la(1 - p) \op \pi_\mu(1 - p_0) \op \pi_\mu(1 - p_0)| \\
& \q \to^{\G P \ot \G P}
\big|\Om^{12}(p_0)\cd\Om^{23}(p_0)\cd \Om^{13}(p_0) + (\pi_\la \op \pi_\mu \op \pi_\nu)(1 - p_0) \big| \\
& \qqq \qqq \ot \big|\Om^{13}(p_0)\cd\Om^{23}(p_0)\cd\Om^{12}(p_0) + (\pi_\la \op \pi_\mu \op \pi_\nu)(1 - p_0) \big| \\
& \q \to^{\G T} \big|  (\pi_\la \op \pi_\mu \op \pi_\nu)(1) \big| = (\cc,0) \, .
\end{split}
\]
But this is exactly the isomorphism in Equation \eqref{eq:trivmul} and the lemma is therefore proved.
\end{proof}

Before stating the next lemma, we recall from Notation \ref{n:picard} that $\epsilon \colon (\sL,n) \ot (\sM,m) \to (\sM,m) \ot (\sL,n)$ denotes the commutativity constraint in the context of $\zz$-graded complex lines.

\begin{lemma}\label{l:trivmult2}
Let $\la,\mu,\nu \in \La$ and let $p \in X$. Then the composition
\begin{equation}\label{eq:trivcomp}
\begin{split}
& \G L^\da_p(\la,\mu) \ot \G L_p(\la,\mu) \ot \G L_p^\da(\mu,\nu) \ot \G L_p(\mu,\nu)
\to^{\epsilon} \G L^\da_p(\mu,\nu) \ot \G L_p^\da(\la,\mu) \ot \G L_p(\la,\mu) \ot \G L_p(\mu,\nu) \\
& \qq \to^{\G{M}^\da_p \ot \G{M}_p} \G L_p^\da(\la,\nu)  \ot \G L_p(\la,\nu)
\end{split}
\end{equation}
agrees with the composition
\[
\xymatrix{
\G L^\da_p(\la,\mu) \ot \G L_p(\la,\mu) \ot \G L_p^\da(\mu,\nu) \ot \G L_p(\mu,\nu)
\ar[rr]^>>>>>>>>{\varphi^{-1} \ot \varphi^{-1}} && (\cc,0) \ar[r]^>>>>>{\varphi} & \G L_p^\da(\la,\nu)  \ot \G L_p(\la,\nu) \, . }
 \]
\end{lemma}
\begin{proof}
By Lemma \ref{l:trivmult1} and the definition of the isomorphisms $\G{M}^\da_p$ and $\G{M}_p$, the isomorphism in \eqref{eq:trivcomp} agrees with the isomorphism
\[
  \begin{split}
& \G L^\da_p(\la,\mu) \ot \G L_p(\la,\mu) \ot \G L_p^\da(\mu,\nu) \ot \G L_p(\mu,\nu)
\to^{\epsilon} \G L^\da_p(\mu,\nu) \ot \G L_p^\da(\la,\mu) \ot \G L_p(\la,\mu) \ot \G L_p(\mu,\nu) \\
& \q \to^{\varphi\ot \T{id}^{\ot 4} \ot \varphi}
\G L_p^\da(\la,\nu) \ot \G L_p(\la,\nu) \ot \G L^\da_p(\mu,\nu) \ot \G L_p^\da(\la,\mu)
\ot \G L_p(\la,\mu) \ot \G L_p(\mu,\nu) \ot \G L_p^\da(\la,\nu) \ot \G L_p(\la,\nu) \\
& \q \to^{\T{id} \ot \big(\psi \circ(\T{id} \ot\varphi^{-1}\ot \T{id})\circ(\T{id}^{\ot 2}\ot\varphi^{-1}\ot \T{id}^{\ot 2})\big)\ot \T{id}}
\G L_p^\da(\la,\nu)  \ot \G L_p(\la,\nu) \, .
\end{split}
\]
Since the sign related to the above commutativity constraint is equal to one, we obtain the result of the lemma by an application of the duality relation from Lemma \ref{l:inverse}.
\end{proof}

\begin{proof}[Proof of Proposition \ref{p:coprodfun}]
Let $p$, $q$ and $e$ be elements in $X$. We must show that $\Delta_e \colon \G H(p,q) \to \G H(p,e) \ot \G H(e,q)$ is unital and respects the composition of morphisms. For $\la \in \La$, one may verify that the unitality condition follows from the identity
\[
 \varphi(\la,\la)(1)=\T{id}_\la \ot \T{id}_\la \in 
\G L_e^\da(\la,\la) \ot \G L_e(\la,\la) = \left| \ma{cc}{ 0 & \pi_\la(e) \\ \pi_\la(p_0) & 0}\right|\ot \left| \ma{cc}{ 0 & \pi_\la(p_0)\\ \pi_\la(e) & 0}\right| \, ,
\]
which in turn is a straightforward consequence of the definition of $\varphi \colon (\cc,0) \to \G L_e^\da(\la,\la) \ot \G L_e(\la,\la)$ (see \eqref{eq:dualI}).

To show that $\Delta_e$ respects the composition of morphisms, we let $\la,\mu$ and $\nu$ be indices in $\La$. We need to prove that the following diagram commutes
    \begin{equation}\label{eq:cofunc}
\xymatrix{
\G H(p,q)(\la,\mu) \ot \G H(p,q)(\mu,\nu)
\ar[rr]^>>>>>>>>>>>{\De_e \ot \De_e} \ar[d]_{\G M_{p,q}}
& & \big( \G H(p,e)(\la,\mu) \ot \G H(e,q)(\la,\mu) \big) \ot \big( \G H(p,e)(\mu,\nu) \ot \G H(e,q)(\mu,\nu) \big)
\ar[d]^{\G M_{p,e,q}} \\ 
\G H(p,q)(\la,\nu) \ar[rr]_<<<<<<<<<<<<<<<<<<<<<{\De_e} & & \G H(p,e)(\la,\nu) \ot \G H(e,q)(\la,\nu) \, ,
}
\end{equation}
where the vertical arrow to the right is the isomorphism of $\zz$-graded complex lines induced by the composition in the graded tensor product $\G H(p,e) \ot \G H(e,q)$.
%

Let $s_p \in \G L_p(\la,\mu)$, $s_q \in \G L_q^\da(\la,\mu)$ and $t_p \in \G L_p(\mu,\nu)$, $t_q \in \G L_q^\da(\mu,\nu)$ be given and let $n_p, n_q, m_p$ and $m_q \in \zz$ denote their respective degrees. By definition, the evaluation of the left-hand side of the diagram in \eqref{eq:cofunc} on the morphism $s_p \ot s_q \ot t_p \ot t_q \in \G H(p,q)(\la,\mu) \ot \G H(p,q)(\mu,\nu)$ is given by
    \[
        \big(\De_e \ci \G{M}_{p,q})\big( (s_p \ot s_q) \ot (t_p \ot t_q) \big)
        = (-1)^{n_q \cd ( m_p + m_q)}
        \cd \G{M}_p(s_p \ot t_p) \ot \varphi(\la,\nu)(1) \ot \G{M}_q^\da(t_q \ot s_q) \, .
    \]
On the other hand, let us write
\[
\begin{split}
\varphi(\la,\mu)(1) & = \varphi_-(\la,\mu) \ot \varphi_+(\la,\mu) \in \G L_e^\da(\la,\mu) \ot \G L_e(\la,\mu) \q \T{and} \\
\varphi(\mu,\nu)(1) & = \varphi_-(\mu,\nu) \ot \varphi_+(\mu,\nu) \in \G L_e^\da(\la,\mu) \ot \G L_e(\la,\mu) \, .
\end{split}
\]
The evaluation of the right-hand side of \eqref{eq:cofunc} on the same morphism is then given by
    \[
        \begin{split}
        & \big( \G{M}_{p,e,q} \ci (\De_e \ot \De_e) \big)\big( (s_p \ot s_q) \ot (t_p \ot t_q) \big)\\
& \q =
(-1)^{n_q \cd (m_p + m_q) }
\cd \G{M}_p(s_p \ot t_p) \ot
\G{M}_e^\da(\varphi_-(\mu,\nu) \ot \varphi_-(\la,\mu))  \\
& \qqq \qqq \qqq  \ot \G M_e(\varphi_+(\la,\mu) \ot \varphi_+(\mu,\nu)) \ot \G{M}_q^\da( t_q \ot s_q) \, .
        \end{split}
    \]
But this proves the proposition, since Lemma \ref{l:trivmult2} shows that 
\[
\G{M}_e^\da(\varphi_-(\mu,\nu) \ot \varphi_-(\la,\mu))  \ot \G{M}_e(\varphi_+(\la,\mu) \ot \varphi_+(\mu,\nu)) = \varphi(\la,\nu)(1) \, .
\]
\end{proof}


\section{Change of base point}\label{s:change}
Throughout this section we fix a separable Hilbert space $\C H$, a unital ring $R$ and a two-sided ideal $I$ in $R$.

We assume that $\G{Rep} = \{ \pi_\la\}_{\la \in \La}$ is a family of (not necessarily unital) representations of $R$ as bounded operators on $\C H$, satisfying Assumption \ref{a:rep} with respect to the ideal $I \su R$. 

Consider two idempotents $p_0,p_0'\in R$ with $p_0-p_0'\in I$. In this section, we define a linear isomorphism of coproduct categories
\[
    \G B(p_0,p_0')\colon \G H\to \G H' \, ,
\]
where $\G H$ and $\G H'$ are the coproduct categories constructed using the base points $p_0$ and $p_0'$, respectively (see Definition \ref{d:hopf}). We refer to this isomorphism as the \emph{change-of-base-point isomorphism}.
%

The isomorphism $\G B(p_0,p_0')$ is the identity on the set 
\[
X = \{ p \in \T{Idem}(R) \mid p - p_0 \in I\} = \{ p \in \T{Idem}(R) \mid p - p_0' \in I\} = X' \, .
\]
Given $p,q\in X = X'$, we define the linear functor
\begin{equation*}
    \G B(p_0,p_0')\colon \G H(p,q)\to \G H'(p,q)
\end{equation*}
to be the identity map on objects (thus elements in the index set $\La$), and to be given on morphisms by an isomorphism
\begin{equation*}
    \G B(p_0,p_0')\colon \G H(p,q)(\la,\mu)\to \G H'(p,q)(\la,\mu) \q \la,\mu \in \La
\end{equation*}
of $\zz$-graded complex lines, see Definition \ref{def:change}.

From now on, we fix the two indices $\la,\mu \in \La$ and the two idempotents $p,q \in X = X'$.

We often suppress the tuples of indices $(\la,\mu)$ and $(\la,\mu,\mu)$ and refer the reader to Notation \ref{n:bigfred} for an explanation of the notation applied.

\begin{lemma}\label{l:changepert}
The difference of the two Fredholm operators
\[
\begin{split}
& F^{12}(p_0,q,p) F^{13}(p,q,p_0)  + \ma{ccc}{ 0 & 0 & 0 \\ 0 & 0 & \pi_\mu(1 - p_0) \\ 0 & 0 & 0 } \q \mbox{and} \\
& F^{12}(p_0',q,p) F^{13}(p,q,p_0') + \ma{ccc}{ 0 & 0 & 0 \\ 0 & 0 & \pi_\mu(1 - p_0') \\ 0 & 0 & 0 } \, ,
\end{split}
\]
both acting from the Hilbert space $\pi_\la(p) \C H \op \pi_\mu(q) \C H \op \pi_\mu(1)\C H$ to the Hilbert space $\pi_\la(q) \C H \op \pi_\mu(1)\C H \op \pi_\mu(p)\C H$, is of trace class.
\end{lemma}
\begin{proof}
Using Lemma \ref{l:commut}, we compute modulo trace class operators:
\[
\begin{split}
F^{12}(p_0,q,p) F^{13}(p,q,p_0) \sim_1
( \pi_\la(q) \op \pi_\mu(p_0) \op \pi_\mu(p) ) \cd \Om^{12}(1) \Om^{13}(1) \, .
\end{split}
\]
Since a similar computation holds with $p_0$ replaced by $p_0'$, it suffices to show that the difference
\[
( 0 \op \pi_\mu(p_0 - p_0') \op 0) \cd \Om^{12}(1) \Om^{13}(1)
- \ma{ccc}{ 0 & 0 & 0 \\ 0 & 0 & \pi_\mu(p_0 - p_0') \\ 0 & 0 & 0 }
\]
is of trace class. But this follows from Assumption \ref{a:rep} since $p_0 - p_0' \in I$. Indeed, as in Lemma \ref{l:1} we see that
\[
\begin{split}
& ( 0 \op \pi_\mu(p_0 - p_0') \op 0) \cd \Om^{12}(1) \Om^{13}(1) \\
& \q \sim_1 \ma{ccc}{0 & 0 & 0 \\ \pi_\mu(1) \pi_\la(1) & 0 & 0 \\ 0 & 0 & 0} \ma{ccc}{0 & 0 & \pi_\la(1) \pi_\mu(1) \\ 0 & 0 & 0 \\
0 & 0 & 0} (0 \op 0 \op \pi_\mu(p_0- p_0')) \\
& \q \sim_1 \ma{ccc}{ 0 & 0 & 0 \\ 0 & 0 & \pi_\mu(p_0 - p_0') \\ 0 & 0 & 0 } \, . \qedhere
\end{split}
\]
\end{proof}

It follows from the above Lemma \ref{l:changepert} that the perturbation isomorphism appearing in our definition of the change-of-base-point isomorphism makes sense: 

\begin{definition}\label{def:change} For $\la,\mu \in \La$ and idempotents $p,q \in X = X'$, the isomorphism of $\zz$-graded complex lines
\[
\G B(p_0,p_0') \colon \G H(p,q)(\la,\mu) \to \G H'(p,q)(\la,\mu)
\]
is defined as the composition
\[
\begin{split}
& |F(p,p_0)| \ot | F(p_0,q)|
\to^{\G T \ci \G S}
\big| F^{12}(p_0,q,p) \cd F^{13}(p,q,p_0) \big| \\
& \q \to^{\G S^{-1} \ci \G P \ci \G S}
\big| F^{12}(p_0', q, p) \cd F^{13}(p,q,p_0') \big|
\to^{(\G T \ci \G S)^{-1}}
|F(p,p_0')| \ot | F(p_0',q)| \, .
\end{split}
\]
\end{definition}

Applying the cocycle property of the perturbation isomorphism from Theorem \ref{t:pervec}, we obtain that the change-of-base-point isomorphism satisfies the following properties:

\begin{lemma}\label{l:equivalence}
The change-of-base-point isomorphism is reflexive, skew-symmetric and transitive in the sense that
\begin{enumerate}
\item $\G B(p_0,p_0) = \T{id} \colon \G H(p,q) \to \G H(p,q)$;
\item $\G B(p_0,p_0') = \G B(p_0',p_0)^{-1} \colon \G H(p,q) \to \G H'(p,q)$;
\item $\G B(p_0', p_0'')\ci \G B(p_0,p_0') = \G B(p_0, p_0'') \colon \G H(p,q) \to \G H''(p,q)$,
\end{enumerate}
whenever $p_0,p_0',p_0''$ are idempotents in $R$ which agree modulo the ideal $I \su R$.
\end{lemma}

It is moreover not hard to see that the change-of-base-point satisfies the unitality condition, thus that it sends units in the category $\G H(p,q)$ to the corresponding units in the category $\G H'(p,q)$ for all $p,q \in X$. Recall in this respect that the unit morphism in $\G H(p,q)(\la,\la)$ is defined by the unit $1 \in \cc$ under the isomorphism
\[
(\cc,0) \cong \Big| \ma{cc}{0 & \pi_\la(p_0) \\ \pi_\la(p) & 0}\Big| \ot \Big| \ma{cc}{0 & \pi_\la(q) \\ \pi_\la(p_0) & 0} \Big| \, .
\]
\medskip

The proofs of the following two results are more involved, and will be given in Subsection \ref{ss:cofunc} and Subsection \ref{ss:func}.

\begin{prop}\label{p:cofunc}
The change-of-base-point isomorphism commutes with the coproducts. Thus, for each $\la,\mu \in \La$ and each $p,e,q \in X = X'$, the following diagram is commutative:
\[
\xymatrix{
\G H(p,q)(\la,\mu) \ar[r]^>>>>>{\De_e} \ar[d]_{\G B(p_0,p_0')} & \G H(p,e)(\la,\mu) \ot \G H(e,q)(\la,\mu) \ar[d]^{\G B(p_0,p_0') \ot \G B(p_0,p_0')} \\
\G H'(p,q)(\la,\mu) \ar[r]_>>>>>{\De'_e} & \G H'(p,e)(\la,\mu) \ot \G H'(e,q)(\la,\mu) \, .
}
\]
\end{prop}

\begin{prop}\label{p:func}
The change-of-base-point isomorphism commutes with the products. Thus, for $\la,\mu,\nu \in \La$ and $p,q \in X = X'$, the following diagram is commutative:
\[
\xymatrix{
\G H(p,q)(\la,\mu) \ot \G H(p,q)(\mu,\nu) \ar[r]^<<<<<{\G M_{p,q}} \ar[d]_{\G B(p_0,p_0') \ot \G B(p_0,p_0')} &
\G H(p,q)(\la,\nu) \ar[d]^{\G B(p_0,p_0')} \\
\G H'(p,q)(\la,\mu) \ot \G H'(p,q)(\mu,\nu) \ar[r]_<<<<<{\G M'_{p,q}} & \G H'(p,q)(\la,\nu) \, .
}
\]
\end{prop}

We summarise the results of this section into the following statement:

\begin{theorem}\label{t:baseiso}
The change-of-base-point isomorphism $\G B(p_0,p_0') \colon \G H \to \G H'$ is an isomorphism of coproduct categories.
\end{theorem}


\section{Group actions and main result}\label{s:group}
We continue working under Assumption \ref{a:rep} and we fix an idempotent $p_0 \in R$. Associated to the separable Hilbert space $\C H$, we have the group $GL(\C H)$ of invertible bounded operators on $\C H,$ and associated to the unital ring $R,$ we have the group of automorphisms $\T{Aut}(R)$. Finally we have the group of permutations $\T{Perm}(\La)$ of the index set $\La$ (indexing the representations of $R$ on the separable Hilbert space $\C H$). On top of Assumption \ref{a:rep} we impose the following:

\begin{assu}\label{a:group}
Suppose that we are given a group $G$ together with three group homomorphisms
\[
\alpha \colon G\to GL(\C H) \q \be \colon G \to \T{Aut}(R) \q
\ga \colon G \to \T{Perm}(\La) \, ,
\]
which satisfy the following conditions:
\begin{enumerate}
\item $\be(g)(i) \in I$ for all $i \in I$ and $g \in G$;
\item $\be(g)(p_0) - p_0 \in I$ for all $g \in G$;
\item $\alpha (g)\pi_\la (r)\alpha(g)^{-1} = \pi_{ \ga(g)(\la)}\big( \be(g)(r)\big)$ for all $r \in R$, $\la \in \La$ and $g \in G$.
\end{enumerate}
\end{assu}

We emphasise that we do not impose any continuity conditions on the group homomorphisms $\al, \be$ and $\ga$.
%

We now define a strict action of $G$ on the coproduct category $\G H_{p_0}$ given in Definition \ref{d:hopf}, where the subscript $p_0$ indicates the base point.

For each $g \in G$, we first define an isomorphism of coproduct categories:
\[
t_g \colon \G H_{p_0} \to \G H_{\be(g)(p_0)} \, .
\]
On the underlying set $X_{p_0} = \{ p \in \T{Idem}(R) \mid p - p_0 \in I \}$ we have the bijection
\begin{equation}\label{eq:xgroup}
t_g \colon X_{p_0} \to X_{\be(g)(p_0)} = X_{p_0} \q p \mapsto \be(g)(p) \, .
\end{equation}
Note that our assumptions on $\be(g) \colon R \to R$ imply that
\[
X_{p_0} = X_{\be(g)(p_0)}
\]
and that this set is invariant under the automorphism $\be(g)$. For each $p,q \in X,$ we have the assignment
\[
t_g \colon \G H_{p_0}(p,q) \to \G H_{\be(g)(p_0)}\big( \be(g)(p),\be(g)(q)\big)
\]
defined as follows: on objects, i.e. on the index set $\La$, we put
\begin{equation}\label{eq:grpobj}
t_g := \ga(g) \colon \La \to \La
\end{equation}
and on morphisms we have the isomorphism of $\zz$-graded complex lines
\begin{equation}\label{eq:grpmor}
\begin{split}
t_g & := \T{Ad}( \al(g) \op \al(g)) \ot \T{Ad}( \al(g) \op \al(g)) \\
& \q \colon \G H_{p_0}(p,q)(\la,\mu) \to \G H_{\be(g)(p_0)}\big(\be(g)(p),\be(g)(q)\big)(\ga(g)(\la),\ga(g)(\mu)) \, ,
\end{split}
\end{equation}
where $\T{Ad}( \al(g) \op \al(g))$ is notation for the isomorphism of $\zz$-graded complex lines given by $\T{Ad}( \al(g) \op \al(g)) := L( \al(g) \op \al(g)) R(\al(g^{-1}) \op \al(g^{-1}))$, see Example \ref{ex:LR}. Note that the equivariance condition from Assumption \ref{a:group} implies that
\[
\ma{cc}{ \al(g) & 0 \\ 0 & \al(g)} F(\la,\mu)(p,p_0) \ma{cc}{\al(g)^{-1} & 0 \\ 0 & \al(g)^{-1}}
= F\big( \ga(g)(\la), \ga(g)(\mu)\big)\big( \be(g)(p), \be(g)(p_0)\big) \, ,
\]
and similarly with $(p,p_0)$ replaced by $(p_0,q)$.

\begin{lemma}\label{l:conjact}
For each $g \in G$, the above assignments $t_g$ defined in Equation \eqref{eq:xgroup}, \eqref{eq:grpobj} and \eqref{eq:grpmor} yield an isomorphism of coproduct categories $t_g \colon \G H_{p_0} \to \G H_{\be(g)(p_0)}$. Moreover, it holds that $t_e = \T{id} \colon \G H_{p_0} \to \G H_{p_0}$ and that the following diagram is commutative for all $g,h \in G$:
\begin{equation}\label{eq:conjact}
\xymatrix{ \G H_{p_0} \ar[r]^{t_g} \ar[dr]_{t_{hg}} & \G H_{\be(g)(p_0)} \ar[d]^{t_h} \\ & \G H_{\be(hg)(p_0)} \, .}
\end{equation}
\end{lemma}
\begin{proof}
We suppress the group homomorphisms $\al,\be$ and $\ga$ from the notation. The fact that $t_e = \T{id}$ and that the diagram in Equation \eqref{eq:conjact} commutes is a straightforward consequence of the definition of $t_g$ for $g \in G$. So for each $p,q,e \in X$ and each $g \in G$, we focus on proving that $t_g \colon \G H_{p_0}(p,q) \to \G H_{g(p_0)}(g(p), g(q))$ is a functor and that the following diagram is commutative:
\begin{equation}\label{eq:copconj}
\xymatrix{ \G H_{p_0}(p,q) \ar[rr]^{\De_e} \ar[d]^{t_g} & & \G H_{p_0}(p,e) \ot \G H_{p_0}(e,q) \ar[d]^{t_g \ot t_g} \\
\G H_{g(p_0)}\big(g(p), g(q)\big) \ar[rr]^>>>>>>>>>>{\De_{g(e)}} & & \G H_{g(p_0)}\big( g(p), g(e) \big) \ot \G H_{g(p_0)}\big(g(e),g(q)\big) \, . }
\end{equation}
Let us fix three objects $\la,\mu,\nu \in \La$.  Using the definition of the coproduct from Definition \ref{d:hopf}, we see that verifying the commutativity of the diagram in Equation \eqref{eq:copconj} amounts to showing that the diagram here below is commutative
\begin{equation}\label{eq:duaconj}
\xymatrix{
(\cc,0) \ar[r]^>>>>>>>>>>>>{\varphi} \ar[dr]^{\varphi} & \big| F(\la,\mu)(p_0,e) \big| \ot \big| F(\la,\mu)(e,p_0) \big|
\ar[d]^{\T{Ad}( g \op g) \ot \T{Ad}(g \op g)} \\ 
& \big| F( g(\la), g(\mu))(g(p_0),g(e)) \big| \ot \big| F( g(\la), g(\mu))(g(e),g(p_0)) \big| \, . } 
\end{equation}
This commutativity result in turn follows from the definition of the duality operation $\varphi$ (see Equation \eqref{eq:dualI}) together with the associativity of torsion and the fact that torsion commutes with perturbation (Proposition \ref{p:assotors} and Theorem \ref{t:percom}). We continue by investigating the functoriality of $t_g$. The unitality condition is straightforward to verify, so we focus on showing that $t_g$ is compatible with the compositions. We now argue that the following composition of isomorphisms of $\zz$-graded complex lines 
\begin{equation}\label{eq:mulconj}
\begin{split}
& \big| F(\la,\mu)(p,p_0) \big| \ot \big| F(\mu,\nu)(p,p_0) \big| \ot \big| F(\la,\nu)(p_0,p) \big| \\
& \q \to^{\T{Ad}(g \op g)^{\ot 3}} \big| F(g(\la),g(\mu))(g(p),g(p_0)) \big| \ot \big| F(g(\mu),g(\nu))(g(p),g(p_0)) \big| 
\ot \big| F(g(\la),g(\nu))(g(p_0),g(p)) \big| \\
& \q \to^{\mu_{g(p)}} (\cc,0)
\end{split}
\end{equation}
agrees with the trivialisation
\[
\mu_p \colon \big| F(\la,\mu)(p,p_0) \big| \ot \big| F(\mu,\nu)(p,p_0) \big| \ot \big| F(\la,\nu)(p_0,p) \big| \to (\cc,0) \, .
\]
Recalling the definition of the trivialisation $\mu_p$ (and $\mu_{g(p)}$) from Equation \eqref{eq:comp} we see that the identity between the above two isomorphisms follows since torsion is associative and commutes with both perturbation and stabilisation (Proposition \ref{p:assotors}, Theorem \ref{t:percom} and Proposition \ref{p:torsta}). A similar argument shows that the two trivialisations in the dual case
\[
\mu_{g(p)}^\da \ci \T{Ad}(g \op g)^{\ot 3} \, \T{ and } \, \, \mu_p^\da
\colon \big| F(\la,\nu)(p,p_0) \big| \ot \big| F(\mu,\nu)(p_0,p) \big| \ot \big| F(\la,\mu)(p_0,p) \big| \to (\cc,0)
\]
also agree, see Equation \eqref{eq:dualcomp}. The fact that $t_g$ is compatible with compositions is now a consequence of these observations together with the commutativity of the diagram in Equation \eqref{eq:duaconj}.
\end{proof}

The action of the group $G$ on the coproduct category $\G H_{p_0}$ is then defined by the composition:
\begin{equation}\label{eq:funcgroup}
    \rho(g) \colon \G{H}_{p_0} \to^{t_g}\G H_{\be(g)(p_0)} \to^{\G B(\be(g)(p_0),p_0)}\G H_{p_0} \q \T{for all } g \in G \, ,
\end{equation}
where we recall that $\G B(\be(g)(p_0),p_0)$ denotes the change-of-base-point isomorphism from Section \ref{s:change}. It follows from Lemma \ref{l:conjact} and Theorem \ref{t:baseiso} that $\rho(g)$ is an automorphism of the coproduct category $\G H_{p_0}$ for every $g \in G$.

\begin{lemma}\label{l:grpactcop}
The assignment $\rho \colon G \to \T{Aut}( \G H_{p_0})$ given in Equation \eqref{eq:funcgroup} is a group homomorphism.
\end{lemma}
\begin{proof}
Given the properties of the change-of-base-point isomorphism from Lemma \ref{l:equivalence} and the properties of the isomorphisms $t_g \colon \G H_{p_0} \to \G H_{\be(g)(p_0)}$, $g \in G$, established in Lemma \ref{l:conjact}, we only need to show that the following diagram of isomorphisms is commutative:
\begin{equation}\label{eq:baseconj}
\xymatrix{ \G H_{\be(g)(p_0)} \ar[r]^{t_h} \ar[d]_{\G B(\be(g)(p_0),p_0)} & \G H_{\be(hg)(p_0)} \ar[d]^{\G B\big(\be(hg)(p_0), \be(h)(p_0)\big)} \\ \G H_{p_0} \ar[r]_{t_h} & \G H_{\be(h)(p_0)} \, ,}
\end{equation}
for all $g,h \in G$. However, recalling the definition of the change-of-base-point isomorphism from Definition \ref{def:change}, we see that the commutativity of the diagram in Equation \eqref{eq:baseconj} again follows by an application of the associativity of torsion and the commutativity property of torsion with respect to stabilisation and perturbation (Proposition \ref{p:assotors}, Proposition \ref{p:torsta} and Theorem \ref{t:percom}).
\end{proof}

We summarise what we have achieved so far, thus stating the main result of this paper, combining Proposition \ref{p:cocycle3}, Theorem \ref{t:copcat} and Lemma \ref{l:grpactcop}:

\begin{theorem}\label{t:hopfgroup}
Let $R$ be a unital ring equipped with a fixed idempotent element $p_0 \in R$, let $I \su R$ be an ideal, let $G$ be a group and let $\C H$ be a separable Hilbert space. Suppose that we have a non-empty family
\[
\G{Rep} = \{ \pi_\la \}_{\la \in \La}
\]
consisting of (not necessarily unital) ring homomorphisms $\pi_\la \colon R \to \sL(\C H)$, $\la \in \La$, satisfying the conditions of Assumption \ref{a:rep}. Suppose moreover that we have three group homomorphisms
\[
\alpha\colon G\to GL(\C H) \q \be \colon G \to \T{Aut}(R) \q
\ga \colon G \to \T{Perm}(\La)
\]
satisfying the conditions of Assumption \ref{a:group}. Then the constructions presented in Definition \ref{d:hopf} yield a coproduct category $\G H_{p_0}$ and the assignment in Equation \eqref{eq:funcgroup} yields a strict action of the group $G$ on $\G H_{p_0}$. In particular, we have a group cohomology class $[c_{\G H_{p_0}}] \in H^3(G,\cc^*/\{\pm 1\})$ given by the constructions in Definition \ref{d:3cocyc}.
\end{theorem}

\section{Coproduct categories and group cocycles from bipolarised representations}\label{s:bipolar}
In this section, we provide a general operator theoretic framework, which gives rise to coproduct categories with group actions. In particular, this framework gives rise to group $3$-cocycles. We fix a separable Hilbert space $\C H$ and recall that we have the unital $C^*$-algebra $\sL(\C H)$ of bounded operators on $\C H$, the trace ideal $\sL^1(\C H) \su \sL(\C H)$ and the group of invertible bounded operators $GL(\C H) \su \sL(\C H)$.

\begin{definition}\label{d:bipol}
Let $G$ be a group equipped with a group homomorphism $\al \colon G \to GL(\C H)$.

We say that a pair $(P,Q)$ of idempotents in $\sL(\C H)$ is a \emph{bipolarisation} of the group homomorphism $\al \colon G \to GL(\C H)$, when the following holds for all $u,v,g,h \in G$:
\begin{enumerate}
\item $\big[\al(g)P \al(g^{-1}), \al(u)Q \al(u^{-1})\big] \in \sL^1(\C H)$;
\item $\big(\al(g)P \al(g^{-1})- \al(h)P \al(h^{-1})\big)\big( \al(u)Q \al(u^{-1}) - \al(v)Q \al(v^{-1}) \big) \in \sL^1(\C H)$.
\end{enumerate}
We apply the notation
\[
P_g=\al(g)P\al(g^{-1}) \mbox{ and } Q_u= \al(u)Q\al(u^{-1}) \mbox{ for all } g,u\in G \, .
\]
\end{definition}

We shall in this section see how to construct a coproduct category $\G H(P,Q,\al)$ for any group homomorphism $\al \colon G \to GL(\C H)$ equipped with a bipolarisation $(P,Q)$. Moreover, this coproduct category will carry an action of the group $G,$ and it therefore yields a group $3$-cocycle on $G$ with values in $\cc^*/\{\pm 1\}$. By construction, the cohomology class of this group $3$-cocycle is canonically associated to our data.

We start by giving a different description of a bipolarisation:

\begin{lemma}\label{l:equivbipol}
Let $\al \colon G \to GL(\C H)$ be a group homomorphism and let $P,Q \colon \C H \to \C H$ be bounded idempotents. The pair $(P,Q)$ is a bipolarisation of $\al \colon G \to GL(\C H)$ if and only if
\[
[P,\al(g)][Q,\al(u)] \in \sL^1(\C H) \q \mbox{and} \q [ \al(g)P \al(g^{-1}), Q] \in \sL^1(\C H)
\]
for all $g,u \in G$. 
\end{lemma}
\begin{proof}
Using that $\al \colon G \to GL(\C H)$ is a group homomorphism and that $\sL^1(\C H) \su \sL(\C H)$ is an ideal our claim becomes a consequence of the identities
\[
\begin{split}
\al(u^{-1})\big[\al(g)P \al(g^{-1}), \al(u)Q \al(u^{-1})\big] \al(u) & = [ \al(u^{-1}g) P \al(g^{-1} u), Q ] \qq \T{and} \\
( \al(g)P \al(g^{-1}) - P )( \al(u) Q \al(u^{-1}) - Q)  & = -\al(g) [P,\al(g^{-1})] [Q,\al(u)]\al(u^{-1}) \, ,
\end{split}
\]
which are valid for all $g,u \in G$.
\end{proof}

In order to proceed, we need the following operator theoretic result, where we recall that $\sL^p(\C H) \su \sL(\C H)$ denotes the $p^{\T{th}}$ Schatten ideal (for $p \in [1,\infty)$):

\begin{lemma}\label{l:idempo}
Let $p \in [1,\infty)$ and suppose that $T \colon \C H \to \C H$ is a bounded operator with $T^2 - T \in \sL^p(\C H)$. Then there exists a bounded idempotent $E \colon \C H \to \C H$ with $T - E \in \sL^p(\C H)$.
\end{lemma}
\begin{proof}
We apply the notation
\[
B_r(z) := \{ w \in \cc \mid | w - z| < r \}
\]
for the open ball around $z \in \cc$ with radius $r > 0$. Furthermore, we let $\T{Sp}(T) \su \cc$ denote the spectrum of $T$.

The operator $T^2 - T \in \sL^p(\C H)$ is in particular compact and hence, by the Riesz-Schauder theorem, it has discrete spectrum with $0 \in \cc$ as the only possible limit point. Moreover, by holomorphic functional calculus, it holds that
\[
\T{Sp}(T^2 - T) = \{ z^2 - z \mid z \in \T{Sp}(T) \} ,
\]
and we conclude that the set
\[
\T{Sp}(T) \sem \big( \T{Sp}(T) \cap ( B_{\ep}(0) \cup B_{\ep}(1) ) \big)
\]
is finite for all $\ep \in (0,1)$. 

Let us choose a $\de \in (0,1/2)$ such that
\[
\T{Sp}(T) \cap \{ \de \cd \exp(2 \pi i t) \mid t \in [0,1] \} = \emptyset =
\T{Sp}(T) \cap \{ 1 + \de \cd \exp(2 \pi i t) \mid t \in [0,1] \} \, .
\]
Using the holomorphic functional calculus we then obtain two idempotents
\[
E_0 := \chi_{B_\de(0)}(T) \, , \, \, E_1 := \chi_{B_\de(1)}(T) \in \sL(\C H) \, ,
\]
where $\chi_U \colon \cc \to [0,1]$ refers to the indicator function associated with a subset $U \su \cc$. We claim that
\[
T - E_1  = T E_0 - (1 - T) E_1  + T ( 1 - E_0 - E_1 ) \in \sL^p(\C H) \, .
\]
To see this, we first remark that $T ( 1 - E_0 - E_1 ) \colon \C H \to \C H$ is of finite rank. Indeed, this follows since the compact operator $T^2 - T$ restricts to an invertible operator on the image of the bounded idempotent $1 - E_0 - E_1 \colon \C H \to \C H$ and this image must therefore be finite dimensional. We may thus focus our attention on showing that
\[
X := T E_0 \in \sL^p(\C H) \q \T{and} \q Y := (1 - T) E_1 \in \sL^p(\C H) \, .
\]
To this end we consider the invertible operators $\exp(2 \pi i X)$ and $\exp(2 \pi i Y) \colon \C H \to \C H$. We have that
\[
\begin{split}
\exp(2 \pi i X) - 1
& = \sum_{n = 1}^\infty \frac{(2 \pi i)^n}{n!} E_0 T^n = \sum_{n = 1}^\infty \frac{(2 \pi i)^n}{n!} E_0 (T^n - T) \, .
\end{split}
\]
For every $n \geq 2$ we notice that $T^n - T = (T^2 - T)(1 + \ldots + T^{n-2})$. It follows that $T^n - T  \in \sL^p(\C H)$ and that we have the estimate
\[
\| T^n - T \|_p \leq \| T^2 - T\|_p \cd \big( 1 + \| T \| \plp \| T \|^{n-2} \big)
\]
for the $p$-norm $\| \cd \|_p \colon \sL^p(\C H) \to [0,\infty)$ (here $\| \cd \| \colon \sL(\C H) \to [0,\infty)$ is the operator norm). The sum $\sum_{n = 1}^\infty \frac{(2 \pi i)^n}{n!} E_0 (T^n - T)$ therefore converges absolutely in the $p$-norm and we may thus conclude that $\exp(2 \pi i X) - 1 \in \sL^p(\C H)$ (since $\sL^p(\C H)$ is a Banach space when equipped with the $p$-norm). A similar computation shows that $\exp(2 \pi i Y) - 1 \in \sL^p(\C H)$ as well.

Notice now that the holomorphic function $f \colon z \mapsto (\exp(2 \pi i z) - 1)/z$ is invertible on the open ball $B_{1/2}(0)$. Since both $\T{Sp}(X) \su B_{1/2}(0)$ and $\T{Sp}(Y) \su B_{1/2}(0)$ we must have that $f(X)$ and $f(Y)$ are invertible. But this implies that $X, Y \in \sL^p(\C H)$ since
\[
\exp(2 \pi i X) - 1 = X \cd f(X) \q \T{and} \q \exp(2 \pi i Y) - 1 = Y \cd f(Y) \, . \qedhere
\]
\end{proof}

As a corollary we have the following:

\begin{corollary}\label{c:traper}
Suppose that $E,F \colon \C H \to \C H$ are two bounded idempotents with $[E,F] \in \sL^2(\C H)$. Then there exists a bounded idempotent $E' \colon E \C H \to E \C H$ such that $EFE - E' \in \sL^1(E \C H)$.
\end{corollary}
\begin{proof}
This follows from Lemma \ref{l:idempo} since
\[
(EFE)^2 - EFE = EFEFE - EFE = EF[E,F]E = E[E,F][E,F]E \in \sL^1(E\C H) \qedhere
\]
\end{proof}


\begin{notation}\label{n:freeidem} For any group $G$, we denote by $R_G$ the free unital ring with one generator $q_u$ for each $u \in G$ subject to the relation $q_u^2 = q_u$. We let $I_G$ denote the two-sided ideal in $R_G$ generated by the subset $\{q_u-q_v \mid u,v\in G\} \su R_G$.
\end{notation}

We emphasise that the unit $1$ in $R_G$ is different from the idempotent $q_e \in R_G$, where $e \in G$ is the neutral element in $G$.

Clearly, the group $G$ acts on $R_G$ via the group homomorphism $\be \colon G \to \T{Aut}(R_G)$ defined by
\begin{equation}\label{eq:autoring}
\be(g) \colon q_u \mapsto q_{gu} \q  \T{for all } g,u \in G \, ,
\end{equation}
and this action is compatible with the ideal $I_G \su R_G$ in the sense that
\[
\be(g)( i ) \in I_G \q \T{and} \q \be(g)(q_u) - q_u \in I_G
\]
for all $i \in I_G$ and all $g,u \in G$.

\begin{definition}\label{d:admrep}
Let $(P,Q)$ be a bipolarisation of a group homomorphism $\al \colon G \to GL(\C H)$. We will say that a family
$\si = \{ \si_g \}_{g \in G}$ of (not necessarily unital) representations of the unital ring $R_G$ as bounded operators on $\C H$ is \emph{admissible} when the relations
\begin{enumerate}
\item $\si_g(1) = P_g$;
\item $\si_g(q_u)-P_gQ_uP_g\in \mathscr{L}^1(\C H)$;
\item $\si_{hg}(q_{hu})=\al(h)\si_g(q_u)\al(h^{-1})$
\end{enumerate}
hold for all $g,h,u \in G$.
\end{definition}

For any bipolarisation $(P,Q)$ of a group representation $\al \colon G \to GL(\C H)$ we define the set
\[
\T{Adm}(P,Q,\al) := \big\{ \si \mid \si = \{ \si_g \}_{g \in G} \T{ is an admissible family of representations of } R_G \T{ on } \C H \big\}
\]
and we will consider the index set defined as the cartesian product:
\begin{equation}\label{eq:index}
\La(P,Q,\al) := G \ti \T{Adm}(P,Q,\al) 
= \{ (h, \si) \mid h \in G \, , \, \, \si \in \T{Adm}(P,Q,\al) \} \, .
\end{equation}
We define an action of $G$ on the index set $\La = \La(P,Q,\al)$ by permutations:
\begin{equation}\label{eq:permindex}
\ga \colon G \to \T{Perm}(\La) \q \ga(k)( h,\si) := (kh, \si) \, .
\end{equation}
For each $\la = (h,\si) \in \La(P,Q,\al)$ we then have the representation $\pi_\la := \si_h \colon R_G \to \sL(\C H)$. This yields the family of representations of $R_G$ as bounded operators on $\C H$:
\[
\G{Rep}(P,Q,\al) := \{ \pi_\la \}_{\la \in \La} \, .
\]

We shall now see that our index set $\La := \La(P,Q,\al)$ is non-empty:

\begin{lemma}\label{l:repres}
Suppose that $(P,Q)$ is a bipolarisation of a group representation $\al \colon G \to GL(\C H)$. Then for each $g \in G$, we may choose a ring homomorphism 
\[
\si_g \colon R_G \to \sL(\C H)
\]
such that the associated family $\si = \{\si_g\}_{g \in G}$ is admissible. In particular, we have an element $(h, \si) \in \La(P,Q,\al)$ for all $h \in G$.
\end{lemma}
\begin{proof}
Using Corollary \ref{c:traper}, we may choose idempotents $E_u \in \sL( P \C H)$, $u\in G$, satisfying
\[
E_u-PQ_uP\in \mathscr{L}^1(P\C H).
\]
For each $g \in G$, we define the representation $\si_g \colon R_G \to \sL(\C H)$ by setting
\[
\si_{g}(q_{u}) := \al(g) \cd E_{g^{-1}u} \cd \al(g^{-1}) \colon \C H \to \C H \q \T{for all } u \in G
\]
and $\si_g(1) := P_g$. Remark in this respect that $\si_g(1) \cd \si_g(q_u) = \si_g(q_u) = \si_g(q_u) \cd \si_g(1)$ since $P E_{g^{-1} u} = E_{g^{-1} u} = E_{g^{-1} u } P$ for all $u \in G$. We leave it to the reader to verify that the relations in Definition \ref{d:admrep} are satisfied. 
%
\end{proof}



The next theorem is now mainly a consequence of Theorem \ref{t:hopfgroup}:

\begin{theorem}\label{t:grpcocyc}
Suppose that $(P,Q)$ is a bipolarisation of a group homomorphism $\al \colon G \to GL(\C H)$. Then we have an associated group $3$-cohomology class
\[
    [c(P,Q,\alpha)]\in H^3(G,\cc^*/\{\pm 1 \}) \, .
\]
More precisely, our family of representations 
\[
\G{Rep}(P,Q,\al) = \{ \pi_\la\}_{\la \in \La(P,Q,\al)}
\]
of $R_G$ satisfies Assumption \ref{a:rep} with respect to the ideal $I_G \su R_G$. Moreover, the three group homomorphisms
\[
\alpha\colon G \to GL(\C H) \q \be \colon G \to \T{Aut}(R_G) \q
\ga \colon G \to \T{Perm}(\La(P,Q,\al))
\]
satisfy the conditions in Assumption \ref{a:group} with respect to the fixed idempotent $q_e \in R_G$ (and the ideal $I_G \su R_G$). In particular, we obtain a coproduct category $\G H(P,Q,\al)$ equipped with a strict action of the group $G$ and this data yields our group $3$-cohomology class $[c(P,Q,\al)] \in H^3(G,\cc^*/\{\pm 1 \})$ for the group $G$ with values in $\cc^*/\{\pm 1 \}$. 
%
\end{theorem}
\begin{proof}
By Theorem \ref{t:hopfgroup} and Lemma \ref{l:repres} we only need to verify condition $(1)$ and $(2)$ in Assumption \ref{a:rep} and condition $(1)$, $(2)$ and $(3)$ in Assumption \ref{a:group}. We have already argued that $\be \colon G \to \T{Aut}(R_G)$ satisfies condition $(1)$ and $(2)$ in Assumption \ref{a:group}. To check the equivariance condition $(3)$, we let $(h,\si) \in \La(P,Q,\al)$ and $k,u \in G$ be given and notice that
\[
\al(k) \pi_{(h,\si)}(q_u) \al(k)^{-1} = \al(k) \si_h(q_u) \al(k)^{-1}
= \si_{kh}(q_{ku}) = \pi_{(kh,\si)}(q_{ku}) = \pi_{\ga(k)(h,\si)}(\be(k)(q_u)) \, ,
\] 
where the second equality sign follows by the admissibility of the family of representations $\si = \{\si_g\}_{g \in G}$. We are thus left with condition $(1)$ and $(2)$ in Assumption \ref{a:rep}. 

Let $(h,\si), (k,\rho)$ and $(l, \tau) \in \La$ be given. Since $\sL^1(\C H) \su \sL(\C H)$ is an ideal (and we are working with ring homomorphisms), it suffices to verify the commutator condition $(1)$ on the generators $q_u \in R_G$, $u \in G$. But for each $u \in G$ we have that
\[
\si_h(q_u) \rho_k(1) = \si_h(q_u) P_k \sim_1 P_h Q_u P_h P_k \sim_1 P_h Q_u P_k \sim_1 P_h P_k Q_u P_k \sim_1 P_h \rho_k(q_u) = \si_h(1) \rho_k(q_u) \, ,
\] 
where the first and fourth equivalence follow by admissibility and the second and third equivalence follow since $(P,Q)$ is a bipolarisation of $\al \colon G \to GL(\C H)$ (see Definition \ref{d:admrep} and Definition \ref{d:bipol}). 

Since we have already verified the commutator condition $(1)$, it suffices to check condition $(2)$ of Assumption \ref{a:rep} for the generators $q_u - q_v \in I_G$, $u,v \in G$.  But for each $u,v \in G$ we have that
\[
\begin{split}
\si_h(q_u - q_v) \rho_k(1) \tau_l(1) & \sim_1 P_h (Q_u - Q_v) P_h P_k P_l \sim_1 P_h (Q_u - Q_v) P_k P_l \\ 
& \sim_1 P_h (Q_u - Q_v) P_h P_l \sim_1 \si_h(q_u - q_v) \tau_l(1) ,
\end{split}
\]
where we are again using admissibility and the fact that $(P,Q)$ is a bipolarisation of $\al \colon G \to GL(\C H)$.
\end{proof}

\subsection{Group $3$-cocycles on double loop groups}
Let $\Ga$ be a Lie group and suppose that $\rho \colon \Ga \to GL(V)$ is a representation of $\Ga$ by \emph{invertible} operators on a finite dimensional complex Hilbert space $V$. We suppose that the representation is smooth in the sense that the map $\Ga \to \cc$, $\rho_{v,w} \colon g \mapsto \inn{v,\rho(g)w}$ is smooth for all $v,w \in V$. We are using the convention that inner products on complex Hilbert spaces are linear in the second variable and conjugate linear in the first variable.

We are interested in the double loop group $G := C^\infty(\B T^2, \Ga)$ consisting of smooth maps from the $2$-torus with values in $\Ga$ (where the group structure is given by point-wise application of the group law in $\Ga$) and we are going to construct a group $3$-cocycle on $G$ associated to the smooth representation $\rho \colon \Ga \to GL(V)$. 

We let $L^2(\B T^2)$ denote the separable Hilbert space of equivalence classes of square integrable functions on the $2$-torus and we put $\C H := L^2(\B T^2) \ot V$. The unital $*$-algebra of smooth functions on the $2$-torus $C^\infty(\B T^2)$ acts on the Hilbert space $L^2(\B T^2)$ as multiplication operators and we denote the corresponding unital $*$-homomorphism by $m \colon C^\infty(\B T^2) \to \sL( L^2(\B T^2))$. Let $\{ e_i\}_{i = 1}^n$ denote an orthonormal basis for $V$ and define the group homomorphism
\begin{equation}\label{eq:bilooprep}
\al_\rho \colon G \to GL(\C H) \q \al_\rho(f)(\xi \ot e_j) := \sum_{i = 1}^n m(\rho_{e_i,e_j} \ci f)(\xi) \ot e_i\, ,
\end{equation}
for all $\xi \in L^2(\B T^2)$, $f \in G$ and $j \in \{1,2,\ldots,n\}$. The group homomorphism $\al \colon G \to GL(\C H)$ is independent of the choice of orthonormal basis $\{e_i\}_{i = 1}^n$ for $V$.

The Hilbert space $L^2(\B T^2)$ has an orthonormal basis $\{ z_1^{n_1} z_2^{n_2} \}_{(n_1,n_2) \in \zz^2}$, where the functions $z_1$ and $z_2 \colon \B T^2 \to \cc$ denote the projections onto the two factors in the Cartesian product $\B T^2 = \B T \ti \B T$ followed by the inclusion $\B T \su \cc$. We may then define two orthogonal projections $P$ and $Q \colon L^2(\B T^2) \to L^2(\B T^2)$ by
\begin{equation}\label{eq:bipolbiloop}
\begin{split}
& P(z_1^{n_1} z_2^{n_2}) := \fork{ccc}{z_1^{n_1} z_2^{n_2}  & \T{for} & n_1 \geq 0 \\ 0 & \T{for} & n_1 < 0 } \q \T{and} \\
& Q(z_1^{n_1} z_2^{n_2}) := \fork{ccc}{z_1^{n_1} z_2^{n_2}  & \T{for} & n_2 \geq 0 \\ 0 & \T{for} & n_2 < 0 \, .}  
\end{split}
\end{equation}
It holds that $PQ = QP$ and the image of the orthogonal projection $PQ \colon L^2(\B T^2) \to L^2(\B T^2)$ is then exactly the Hardy-space for the bi-disc $H^2(\B D^2) \su L^2(\B T^2)$. From the orthogonal projections $P$ and $Q$ we obtain orthogonal projections $P \ot \T{id}_V$ and $Q \ot \T{id}_V$ on the Hilbert space $\C H = L^2(\B T^2) \ot V$.

The main application of our work can now be summarised in the following result:


\begin{theorem}\label{t:bipolbiloop}
The pair of orthogonal projections $(P \ot \T{id}_V,Q \ot \T{id}_V)$ defined in Equation \eqref{eq:bipolbiloop} is a bipolarisation of the group homomorphism $\al_\rho \colon C^\infty(\B T^2,\Ga) \to GL(\C H)$ defined in Equation \eqref{eq:bilooprep}. In particular, we obtain an associated group $3$-cohomology class $\big[ c(P \ot \T{id}_V, Q \ot \T{id}_V,\al_\rho) \big] \in H^3( C^\infty(\B T^2,\Ga), \cc^*/\{\pm 1\})$.
\end{theorem}
\begin{proof}
Using Lemma \ref{l:equivbipol} and the relation $PQ = QP$ we see that the first part of the theorem follows if we can verify that $[P,m(f)][Q,m(g)] \in \sL^1( L^2(\mathbb{T}^2))$ and $\big[ [P,m(f)] , Q\big] \in \sL^1( L^2(\mathbb{T}^2))$ for all $f,g \in C^\infty(\B T^2)$. Since the functions $f$ and $g$ are both smooth, we know that their Fourier coefficients, respectively $\{ a_{n_1,n_2} \}_{n_1,n_2 \in \zz}$ and $\{ b_{m_1,m_2} \}_{m_1,m_2 \in \zz}$, are rapidly decreasing in the sense that the series 
\[
\sum_{n_1,n_2 \in \zz} |a_{n_1,n_2}| |n_1|^k |n_2|^l \q \T{and} \q \sum_{m_1,m_2 \in \zz} |b_{m_1,m_2}| |m_1|^k |m_2|^l
\]
are convergent for all $k,l \in \nn \cup \{0\}$. We compute as follows: 
\[
\begin{split}
[P,m(f)][Q,m(g)] & = \sum_{n_1,n_2,m_1,m_2 \in \zz} a_{n_1,n_2} b_{m_1,m_2} \cd z_2^{n_2} [P,z_1^{n_1}][Q,z_2^{m_2}] z_1^{m_1} \q \T{and} \\
\big[ [P,m(f)] , Q\big] & = - \sum_{n_1,n_2 \in \zz} a_{n_1,n_2} [P,z_1^{n_1}][Q,z_2^{n_2}] 
\end{split}
\]
so that it now suffices to show that $\big\| [P,z_1^n][Q,z_2^m] \big\|_1 = |n| \cd |m|$ for all $n,m \in \zz$ (where we recall that $\| \cd \|_1 \colon \sL^1( L^2(\B T^2)) \to [0,\infty)$ denotes the trace norm). This identity however follows from the fact that $( P - z_1^{|n|} P z_1^{-|n|} ) (Q - z_2^{|m|}Q z_2^{-|m|})$ is an orthogonal projection of rank $|n| \cd |m|$ (and hence has trace norm equal to $|n| \cd |m|$).

The second part of the theorem is an immediate consequence of the first part and of Theorem \ref{t:grpcocyc}.
\end{proof}

We end this subsection by giving one more example of a bipolarised group homomorphism, but this time in the context of the smooth noncommutative $2$-torus, so that we obtain a group $3$-cocycle on the invertible elements in this unital $*$-algebra. Thus, we fix an irrational number $\te \in \rr \sem \B Q$ and consider the unitary operators $U$ and $V \colon L^2(\B T^2) \to L^2(\B T^2)$ defined by
\[
U( z_1^{n_1} z_2^{n_2} ) := z_1^{n_1 + 1} z_2^{n_2} \q \T{and} \q V(z_1^{n_1} z_2^{n_2}) := \exp(2\pi i n_1 \te) z_1^{n_1} z_2^{n_2 + 1}\, . 
\] 
We notice the relation $U V = \exp(-2\pi i \te) V U$. The smooth noncommutative $2$-torus $C^\infty(\B T^2_\te)$ is then defined as the unital $*$-subalgebra of $\sL(L^2(\B T^2))$ consisting of all bounded operators of the form 
\[
\sum_{n_1,n_2 \in \zz} a_{n_1,n_2} U^{n_1} V^{n_2}  \in \sL(L^2(\B T^2)) \, ,
\]
where the sequence of complex numbers $\{ a_{n_1,n_2} \}_{n_1,n_2 \in \zz}$ is rapidly decreasing, i.e.
\[
\sum_{n_1,n_2 \in \zz} |a_{n_1,n_2}| |n_1|^k |n_2|^l < \infty \, ,
\]
for all $k,l \in \nn \cup \{0\}$. The smooth noncommutative $2$-torus is in fact a unital Fr\'echet $*$-algebra where the corresponding countable family of semi-norms $\{ \| \cd \|_{k,l} \}_{k,l \in \nn \cup \{0\}}$ is given by
\[
\big\| \sum_{n_1,n_2 \in \zz} a_{n_1,n_2} U^{n_1} V^{n_2} \big\|_{k,l} 
:= \| \sum_{n_1,n_2 \in \zz} a_{n_1,n_2} n_1^k n_2^l U^{n_1} V^{n_2} \|_\infty \, .
\]
We let $C^\infty(\B T_\te^2)^*$ denote the group of invertible elements in $C^\infty(\B T_\te^2)$ and define the group homomorphism $\al_\te \colon C^\infty(\B T_\te^2)^* \to GL( L^2(\B T^2))$ induced by the unital inclusion $C^\infty(\B T^2_\te) \su \sL( L^2(\B T^2))$. The proof of the next proposition is then very similar to the proof of Theorem \ref{t:bipolbiloop} and will therefore not be repeated here.

\begin{prop}\label{p:ncbiloop}
The pair of orthogonal projections $(P ,Q )$ defined in Equation \eqref{eq:bipolbiloop} is a bipolarisation of the group homomorphism $\al_\te \colon C^\infty(\B T^2_\te)^* \to GL(L^2(\B T^2))$. In particular, we obtain an associated group $3$-cohomology class $\big[ c(P, Q,\al_\te) \big] \in H^3( C^\infty(\B T^2_\te)^*, \cc^*/\{\pm 1 \})$.
\end{prop}

\section{Non-triviality of group $3$-cocycles on double loop groups}\label{s:nontriv}
We now analyse in more detail the case of the double loop group $G = C^\infty(\mathbb{T}^2)^*$ of invertible smooth functions on the $2$-torus. Let $\al \colon G \to GL(L^2(\B T^2))$ denote the group homomorphism determined by the unital representation $m \colon C^\infty(\B T^2) \to \sL(L^2(\B T^2))$ of $C^\infty(\B T^2)$ as multiplication operators on $\C H := L^2(\B T^2)$. We recall from Theorem \ref{t:bipolbiloop} that the two orthogonal projections $P$ and $Q \colon L^2(\B T^2) \to L^2(\B T^2)$ defined in Equation \eqref{eq:bipolbiloop} yield a bipolarisation of $\al$. We thus have an associated group $3$-cohomology class $[c(P,Q,\al)] \in H^3(G,\cc^*/\{\pm 1 \})$.

For a fixed $\la \in \cc^*$ we consider the class in group $3$-homology $\{ z_1,z_2,\la\} \in H_3(G,\zz)$ coming from the 
group $3$-cycle
\begin{equation}\label{eq:classgrouphom}
(z_1,z_2,\la) - (z_1,\la,z_2) + (\la, z_1,z_2) - (\la,z_2,z_1) + (z_2,\la,z_1) - (z_2,z_1,\la) \in \zz[G^3]
\end{equation}
We are going to compute the pairing between our group $3$-cohomology class $[c(P,Q,\al)]$ and the group $3$-homology class $\{z_1,z_2,\la\}$. This yields the result $\inn{[c(P,Q,\al)], \{z_1,z_2,\la\} } = [\la]$, where $[\cd ] \colon \cc^* \to \cc^*/\{\pm 1 \}$ denotes the quotient map. In particular, we obtain that our group $3$-cohomology class $[c] = [c(P,Q,\al)]$ is non-trivial. The computation of the pairing $\inn{ [c(P,Q,\al)], \{z_1,z_2,\la\} }$ is carried out in several steps and the reader can in these steps see how the various category theoretic constructions appearing in this paper are working in practice. 

We recall from Notation \ref{n:freeidem} that $R_G$ denotes the free unital ring with one generator $q_u$ for every $u \in G$ and that these generators are subject to the relation $q_u^2 = q_u$ so that they become elements in $\T{Idem}(R_G)$. The ideal $I_G \su R_G$ is the smallest two-sided ideal containing the difference $q_u - q_v$ for every $u,v \in G$. According to Theorem \ref{t:hopfgroup}, Theorem \ref{t:grpcocyc} and Definition \ref{d:hopf}, the coproduct category $\G H := \G H(P,Q,\al)$ is based on the underlying set
\[
X := \big\{ p \in \T{Idem}(R_G) \mid p - q_e \in I_G \big\}
\]
together with the choice of basepoint $q_e \in X$. For every pair of elements $p,q \in X$ we have the category $\G H(p,q)$. The objects in $\G H(p,q)$ are elements in the index set
\[
\La := \La(P,Q,\al) := G \ti \T{Adm}(P,Q,\al) ,
\]
where $\T{Adm}(P,Q,\al)$ is the set of admissible families of representations of $R_G$, see Definition \ref{d:admrep}. For every pair of elements 
$(g,\si), (h,\tau) \in \La$ the morphisms from $(g,\si)$ to $(h,\tau)$ are given by the $\zz$-graded complex line
\[
\G H(p,q)\big( (g,\si), (h,\tau) \big) := \G L_p\big( (g,\si), (h,\tau) \big) \ot \G L_q^\da\big( (g,\si), (h,\tau)\big) .
\]
We remind the reader of the notation 
\[
\C H := L^2(\B T^2) \q P_g := \al(g) P \al(g^{-1}) \, \, \T{ and } \, \, \, Q_u := \al(u) Q \al(u^{-1}) \, .
\]
We then recall that the $\zz$-graded complex lines $\G L_p\big( (g,\si), (h,\tau) \big)$ and $\G L_q^\da\big( (g,\si), (h,\tau)\big)$ are respectively the determinants of the two Fredholm operators
\[
\begin{split}
& F\big( (g,\si), (h,\tau) \big)(p,q_e) := 
\big( \si_g(q_e) \op \tau_h(p) \big) \Om( (g,\si), (h,\tau))(1) \big( \si_g(p) \op \tau_h(q_e) \big)
\q \T{and} \\
& F\big( (g,\si), (h,\tau) \big)(q_e,q) :=  
\big( \si_g(q) \op \tau_h(q_e) \big) \Om( (g,\si), (h,\tau))(1) \big( \si_g(q_e) \op \tau_h(q) \big) ,
\end{split}
\]
where the matrix in the middle is given by
\[
\Om( (g,\si), (h,\tau))(1) := \ma{cc}{P_g - P_g P_h P_g & P_g P_h \\ 2 P_h P_g - (P_h P_g)^2 & P_h P_g P_h - P_h} ,
\]
see Notation \ref{n:invfred}. The composition in the category $\G H(p,q)$ is described in detail in Subsection \ref{ss:comp}. For an extra element $f \in X$ we have the coproduct functor
\[
\De_f \colon \G H(p,q) \to \G H(p,f) \ot \G H(f,q) ,
\]
introduced in Definition \ref{d:hopf}. The group action of the group $G$ on the coproduct category $\G H(P,Q,\al)$ is determined by the three group homomorphisms
\[
\al \colon G \to GL( L^2(\B T^2))  \, , \, \, \be \colon G \to \T{Aut}(R_G) \, , \, \, \ga \colon G \to \T{Perm}( \La(P,Q,\al)) ,
\]
where we recall that $\be(k)(q_u) = q_{k \cd u}$ and $\ga(k)( g,\si) = (k \cd g,\si)$. At the level of the underlying set $X$, the automorphism $\rho(k) \colon \G H(P,Q,\al) \to \G H(P,Q,\al)$ is induced by $\be(k) \colon R_G \to R_G$ and for every pair $(p,q) \in X \ti X$ the automorphism $\rho(k) \colon \G H(p,q) \to \G H(\be(k)(p), \be(k)(q))$ is given in Equation \eqref{eq:grpmor} and Equation \eqref{eq:funcgroup}. Notice that this automorphism involves the change-of-base-point isomorphism in an essential way.

In order to find a representative for the class $[c(P,Q,\al)] \in H^3(G,\cc^*/\{\pm 1 \})$ we follow the recipe from Definition \ref{d:3cocyc}. We start by choosing the element $q_e \in X$. For each $g \in G$ we should then choose an object $a_g \in \G H(q_e,q_g)$, so we need to choose $a_g \in \La(P,Q,\al) = G \ti \T{Adm}(P,Q,\al)$. We put $a_g := (e,\si)$, where $\si$ is a specific admissible family of representations, which we now construct following the procedure described in Lemma \ref{l:repres}. For every $u \in G$ we choose an idempotent $E_u \in \sL(P\C H)$ with $E_u - P Q_u P \in \sL^1(P \C H)$ and we arrange that
\[
E_u := P Q_{\xi \cd z_1^{t_1} \cd z_2^{t_2}} P = P Q_{z_2^{t_2}} \in \sL( P \C H) ,
\]
whenever $u = \xi \cd z_1^{t_1} \cd z_2^{t_2}$ is an invertible monomial. This provides us with the non-unital representation $\si_e \colon R_G \to \sL(\C H)$ given by
\[
\si_e \colon q_u \mapsto E_u \q \T{and} \q \si_e(1) := P \, .
\]
For each $k \in G$, we then have the non-unital representation $\si_k \colon R_G \to \sL(\C H)$ given by
\[
\si_k( q_u) := \al(k)( E_{k^{-1} u}) \al(k^{-1}) \q \T{and} \q \si_k(1) := P_k \, .
\]
In particular, it holds that
\begin{equation}\label{eq:sigmaspecial}
\si_k( q_u) = P_{z_1^{l_1}} Q_{z_2^{t_2}} \, ,
\end{equation}
whenever $k = \ka \cd z_1^{l_1} z_2^{l_2}$ and $u = \xi \cd z_1^{t_1} z_2^{t_2}$ are invertible monomials. According to Lemma \ref{l:repres} we then have our admissible family of representations $\si := \{ \si_k \}_{k \in G}$ and hence the objects $a_g := (e,\si) \in \G H(q_e,q_g)$. Remark that our choices entail that the group element $g \in G$ only influences which category the object $a_g = (e,\si)$ belongs to. 

According to Definition \ref{d:3cocyc}, the next step is to choose an isomorphism
\[
\be_{g,h} \colon \De_{q_g}(a_{gh}) \to a_g \ot g(a_h)
\]
for every pair of elements $(g,h) \in G \ti G$. The isomorphism $\be_{g,h}$ belongs to the category $\G H(q_e ,q_g) \ot \G H(q_g, q_{gh})$ and travels from the object $\De_{q_g}(a_g) = (e,\si) \ot (e,\si)$ to the object $a_g \ot g(a_h) = (e,\si) \ot (g,\si)$. We may thus choose $\be_{g,h}$ of the form
\[
\be_{g,h} := \T{id}_{ (e,\si)} \ot b_{g,h} \, ,
\]
where $b_{g,h} \colon (e,\si) \to (g,\si)$ is an isomorphism in the category $\G H(q_g,q_{gh})$ so that $b_{g,h}$ is a non-zero vector in the graded determinant line
\[
\G L_{q_g}\big( (e,\si), (g,\si)\big) \ot \G L_{q_{gh}}^\da\big( (e,\si), (g,\si)\big)
= \big| F\big( (e,\si), (g,\si) \big)(q_g,q_e) \big| \ot \big| F\big( (e,\si), (g,\si) \big)(q_e,q_{gh}) \big| \, .
\]

We now compute the kernel and cokernel of the involved Fredholm operators for specific choices of $g,h \in G$. It will in this respect be convenient to introduce the finite rank orthogonal projection
\[
\Ga_{n,m,t,s} := (P_{z_1^m} - P_{z_1^n})(Q_{z_2^s} - Q_{z_2^t}) \colon \C H \to \C H
\]
defined for $n, m,t,s \in \nn \cup \{0\}$ with $n \geq m$ and $t \geq s$.

\begin{lemma}\label{l:kercoker}
Let $g = \mu \cd z_1^{n_1}  z_2^{n_2}$, $h = \nu \cd z_1^{m_1} z_2^{m_2}$ and $u = \xi \cd z_1^{t_1} z_2^{t_2}, v = \eta \cd z_1^{s_1} z_2^{s_2}$ be invertible monomials with $n_1,m_1,t_2,s_2 \in \nn \cup \{0\}$ and $n_1 \geq m_1$, $t_2 \geq s_2$. It holds that
\[
\begin{split}
& \T{Ker}\Big( F( (h,\si), (g,\si))(q_v,q_u) \Big) = \T{Im}( \Ga_{n_1,m_1,t_2,s_2}) \op \{0\} \\
& \T{Coker}\Big( F((h,\si),(g,\si) )(q_v,q_u) \Big) = \{0\} \q \mbox{and} \\
& \T{Ker}\Big( F( (g,\si), (h,\si))(q_u,q_v) \Big) = \{0\} \op \T{Im}( \Ga_{n_1,m_1,t_2,s_2} ) \\
& \T{Coker}\Big( F((g,\si),(h,\si) )(q_u,q_v) \Big) = \{0\} \, .
\end{split}
\]
Moreover, it holds that 
\[
\begin{split}
& F( (h,\si),(g,\si) )(q_u,q_v) = \big( F((h,\si),(g,\si))(q_v,q_u) \big)^* \q \mbox{and} \\ 
& F((g,\si), (h,\si))(q_v,q_u) = \big( F( (g,\si), (h,\si))(q_u,q_v) \big)^*
\end{split}
\]
and hence that 
\[
\begin{split}
& \T{Ker}\Big( F( (h,\si),(g,\si) )(q_u,q_v) \Big) = \{ 0 \} \\ 
& \T{Coker}\Big( F( (h,\si),(g,\si) )(q_u,q_v) \Big) \cong \T{Im}( \Ga_{n_1,m_1,t_2,s_2} ) \op \{0\} \q \mbox{and} \\
& \T{Ker}\Big( F( (g,\si), (h,\si))(q_v,q_u) \Big) = \{0\} \\
& \T{Coker}\Big( F((g,\si),(h,\si) )(q_v,q_u) \Big) \cong \{0\} \op \T{Im}( \Ga_{n_1,m_1,t_2,s_2}) \, .
\end{split}
\]
\end{lemma}
\begin{proof}
We will only treat the case where the object $(h,\si)$ appears to the left of the object $(g,\si)$ since the remaining identities follow by similar considerations. 

Since $n_1 \geq m_1 \geq 0$ and thus $P_{z_1^{n_1}} P_{z_1^{m_1}} = P_{z_1^{n_1}} = P_{z_1^{m_1}} P_{z_1^{n_1}}$ we have that
\[
\Om( (h,\si),(g,\si) )(1) = \ma{cc}{P_{z_1^{m_1}} - P_{z_1^{n_1}} & P_{z_1^{n_1}} \\ P_{z_1^{n_1}} & 0 } \, .
\]
In particular $\Om( (h,\si),(g,\si) )(1)$ is self-adjoint. Recall next that $\si_g(q_u) := P_{z_1^{n_1}} Q_{z_2^{t_2}}$ and hence
\[
\begin{split}
& F( (h,\si), (g,\si))(q_v,q_u) = (P_{z_1^{m_1}} Q_{z_2^{t_2}} \oplus P_{z_1^{n_1}} Q_{z_2^{s_2}}) \Om( (h,\si),(g,\si) )(1) (P_{z_1^{m_1}} Q_{z_2^{s_2}} \oplus P_{z_1^{n_1}} Q_{z_2^{t_2}}) \q \T{and} \\
  & F( (h,\si), (g,\si))(q_u,q_v) = (P_{z_1^{m_1}} Q_{z_2^{s_2}} \oplus P_{z_1^{n_1}} Q_{z_2^{t_2}}) \Om( (h,\si),(g,\si) )(1) (P_{z_1^{m_1}} Q_{z_2^{t_2}} \oplus P_{z_1^{n_1}} Q_{z_2^{s_2}}) \, .
\end{split}
\]
Since the idempotents in the above formulae are all self-adjoint we see that 
\[
F( (h,\si), (g,\si))(q_u,q_v) = \big( F( (h,\si), (g,\si))(q_v,q_u)\big)^* \, .
\]
It thus suffices to compute the kernel and cokernel of $F( (h,\si), (g,\si))(q_v,q_u)$.

Since $t_2 \geq s_2 \geq 0$ as well, we obtain that
\[
\begin{split}
F( (h,\si), (g,\si))(q_v,q_u)
& = (P_{z_1^{m_1}} Q_{z_2^{t_2}} \oplus P_{z_1^{n_1}} Q_{z_2^{s_2}}) \ma{cc}{P_{z_1^{m_1}} - P_{z_1^{n_1}} & P_{z_1^{n_1}} \\ P_{z_1^{n_1}} & 0 } 
(P_{z_1^{m_1}} Q_{z_2^{s_2}} \oplus P_{z_1^{n_1}} Q_{z_2^{t_2}})  \\
& = \ma{cc}{ (P_{z_1^{m_1}} - P_{z_1^{n_1}}) Q_{z_2^{t_2}} & P_{z_1^{n_1}} Q_{z_2^{t_2}} \\ P_{z_1^{n_1}} Q_{z_2^{s_2}} & 0} \\
& \qq \colon (P_{z_1^{m_1}}Q_{z_2^{s_2}}) \C H \op (P_{z_1^{n_1}} Q_{z_2^{t_2}}) \C H  \to ( P_{z_1^{m_1}} Q_{z_2^{t_2}} )\C H \oplus ( P_{z_1^{n_1}} Q_{z_2^{s_2}} ) \C H \, .
\end{split}
\]
But this Fredholm operator is surjective and has kernel given by the image of the finite rank projection $(P_{z_1^{m_1}} - P_{z_1^{n_1}})(Q_{z_2^{s_2}} - Q_{z_2^{t_2}}) \op 0 \colon \C H \op \C H \to \C H \op \C H$. The lemma is proved.
\end{proof}

For every $n, m,t,s \in \nn \cup \{0\}$ with $n > m$ and $t > s$, we choose the non-zero vector
\begin{equation}\label{eq:omegavec}
\begin{split}
\om_{n,m,t,s} & := \big( [z_1^m z_2^s] \wlw [z_1^{n - 1} z_2^s] \big) \we \big( [z_1^m z_2^{s+1}] \wlw [z_1^{n - 1} z_2^{s+1}] \big) \\
& \q \wlw \big( [z_1^m z_2^{t - 1}] \wlw [z_1^{n - 1} z_2^{t - 1}] \big) \in \La^{\T{top}}\big( \T{Im}(\Ga_{n,m,t,s} ) \big)
\end{split}
\end{equation}
and for $n,m,t,s \in \nn \cup \{0\}$ with $n = m$ or $t = s$ we choose $\om_{n,m,t,s} = 1 \in \cc = \La^{\T{top}}( \{0\})$. In particular, for invertible monomials $g = \mu \cd z_1^{n_1} z_2^{n_2}$ and $h = \nu \cd z_1^{m_1} z_2^{m_2}$ with $n_1,n_2,m_1,m_2 \in \nn \cup \{0\}$ we choose the isomorphism
\begin{equation}\label{eq:bexpdef}
b_{g,h} := \om_{n_1,0,n_2,0}^* \ot \om_{n_1,0,n_2 + m_2,0} \in \big| F\big( (e,\si), (g,\si) \big)(q_g,q_e) \big| \ot \big| F\big( (e,\si), (g,\si) \big)(q_e,q_{gh}) \big| \, ,
\end{equation}
where the gradings of the two $\zz$-graded complex lines are given by the indices $-n_1 \cd n_2$ and $n_1 \cd (n_2 + m_2)$, respectively.

Whenever $g = \mu \cd z_1^{n_1} z_2^{n_2}, h = \nu \cd z_1^{m_1} z_2^{m_2}$ and $k = \ka \cd z_1^{l_1} z_2^{l_2}$ are invertible monomials with exponents $n_1,n_2,m_1,m_2,l_1,l_2$ in $\nn \cup \{0\}$, we are interested in computing the number $c(g,h,k) \in \cc^*$ coming from the automorphism
\begin{equation}\label{eq:automorph}
(g,\si) \ot (gh,\si) \to^{b_{g,h}^{-1} \ot b_{gh,k}^{-1}} (e,\si) \ot (e,\si) \to^{\De_{q_{gh}}(b_{g,hk})} (g,\si) \ot (g,\si) 
\to^{\T{id} \ot g(b_{h,k})} (g,\si) \ot (gh,\si)
\end{equation}
in the category $\G H(q_g,q_{gh}) \ot \G H(q_{gh},q_{ghk})$, see Definition \ref{d:3cocyc}. In order to carry out this computation we need to have a better understanding of the coproduct, the group action and the composition relating to the coproduct category $\G H(P,Q,\al)$. We are now going to investigate these operations. 

\subsection{The coproduct}
We consider invertible monomials $u = \xi \cd z_1^{t_1}z_2^{t_2}, v = \eta \cd z_1^{s_1}z_2^{s_2}$ and $w = \ze \cd z_1^{r_1} z_2^{r_2}$ as well as 
$g = \mu \cd z_1^{n_1} z_2^{n_2}$ and $h = \nu \cd z_1^{m_1} z_2^{m_2}$. Recall from Definition \ref{d:hopf} that the coproduct functor $\De_{q_w} \colon \G H( q_u,q_v) \to \G H(q_u,q_w) \ot \G H(q_w,q_v)$ is given by 
\[
\De_{q_w} \colon (g,\si) \mapsto (g,\si) \ot (g,\si) \q \T{and} \q
\De_{q_w} \colon \om_+ \ot \om_- \mapsto \om_+ \ot \varphi(1) \ot \om_-
\]
on objects and morphisms, respectively. We specify that 
\[
\om_+ \ot \om_- \in \big| F( (g,\si), (h,\si))(q_u,q_e) \big| \ot \big| F( (g,\si), (h,\si))(q_e,q_v) \big|
\]
and recall that 
\[
\varphi \colon (\cc,0) \to \big| F( (g,\si), (h,\si) )(q_e,q_w)\big| \ot \big|F( (g,\si),(h,\si))(q_w,q_e) \big|
\]
is the duality isomorphism from Equation \eqref{eq:dualI}. Computing the coproduct functor thus really amounts to computing the duality isomorphism:

\begin{lemma}\label{l:dualexam}
Let $g = \mu \cd z_1^{n_1} z_2^{n_2}$, $h = \nu \cd z_1^{m_1} z_2^{m_2}$ and $w = \ze \cd z_1^{r_1} z_2^{r_2}$ be invertible monomials with $n_1 \geq m_1 \geq 0$ and $r_2 \geq 0$. Then the duality isomorphism $\varphi \colon (\cc,0) \to \big| F( (h,\si), (g,\si) )(q_e,q_w)\big| \ot \big|F( (h,\si),(g,\si))(q_w,q_e) \big|$ is given explicitly by
\[
\varphi(1) = \om_{n_1,m_1,r_2,0} \ot \om_{n_1,m_1,r_2,0}^* \, .
\]
Similarly, the duality isomorphism $\varphi \colon (\cc,0) \to \big| F( (g,\si), (h,\si) )(q_e,q_w)\big| \ot \big|F( (g,\si),(h,\si))(q_w,q_e) \big|$ takes the form
\[
\varphi(1) = \om_{n_1,m_1,r_2,0}^* \ot \om_{n_1,m_1,r_2,0} \, .
\]
\end{lemma}
\begin{proof}
We restrict our attention to the first of the two duality isomorphisms since the computation of the second one is similar but easier. We recall from Equation \eqref{eq:dualI} that the duality isomorphism in question is defined as the composition
\begin{equation}\label{eq:dualexam}
\begin{split}
(\cc,0) = \big| \si_h(q_e) \op \si_g(q_w) \big| & \to^{\G P} \big| F( (h,\si),(g,\si))(q_w,q_e) \cd F( (h,\si), (g,\si) )(q_e,q_w) \big| \\
& \to^{\G T^{-1}} \big| F( (h,\si), (g,\si) )(q_e,q_w)\big| \ot \big|F( (h,\si),(g,\si))(q_w,q_e) \big| \, .
\end{split}
\end{equation}
The product of the Fredholm operators appearing is given by
\[
\begin{split}
& F( (h,\si),(g,\si))(q_w,q_e) \cd F( (h,\si), (g,\si) )(q_e,q_w) \\
& \q = \ma{cc}{ (P_{z_1^{m_1}} - P_{z_1^{n_1}}) Q_{z_2^{r_2}} & P_{z_1^{n_1}} Q \\ P_{z_1^{n_1}} Q_{z_2^{r_2}} & 0 } \cd \ma{cc}{ (P_{z_1^{m_1}} - P_{z_1^{n_1}}) Q_{z_2^{r_2}} & P_{z_1^{n_1}} Q_{z_2^{r_2}} \\ P_{z_1^{n_1}} Q & 0 } \\
& \q = \ma{cc}{P_{z_1^{m_1}}Q - (P_{z_1^{m_1}} - P_{z_1^{n_1}})(Q - Q_{z_2^{r_2}}) & 0 \\ 0 & P_{z_1^{n_1}} Q_{z_2^{r_2}}}
\end{split}
\]
acting on the Hilbert space $(P_{z_1^{m_1}} Q)\C H \op (P_{z_1^{n_1}} Q_{z_2^{r_2}}) \C H$. To ease the notation we put $\Ga := \Ga_{n_1,m_1,r_2,0} = (P_{z_1^{m_1}} - P_{z_1^{n_1}})(Q - Q_{z_2^{r_2}})$ and the perturbation isomorphism appearing in Equation \eqref{eq:dualexam} can then be computed using Example \ref{ex:indzero}:
\[
\G P \colon 1 \mapsto \om_{n_1,m_1,r_2,0} \ot \om_{n_1,m_1,r_2,0}^* \in \big( \La^{\T{top}}( \T{Im}(\Ga)), (n_1 - m_1)r_2 \big)
\ot \big( \La^{\T{top}}( \T{Im}(\Ga))^*, (m_1 - n_1)r_2 \big) \, .
\]
Moreover, it can be verified using Lemma \ref{l:kercoker} that the torsion isomorphism appearing in Equation \eqref{eq:dualexam} comes from the six term exact sequence
\[
\xymatrix{ \T{Im}(\Ga) \op \{0\} \ar@{=}[r] & \T{Im}(\Ga) \op \{0\} \ar[r] & \{0\} \ar@{=}[d] \\ 
\T{Im}(\Ga) \op \{0\} \ar[u]^{0} & \T{Im}(\Ga) \op \{0\} \ar@{=}[l] & \{0\} \ar[l]
} 
\]
so that in fact $\G T^{-1}(\om_{n_1,m_1,r_2,0} \ot \om_{n_1,m_1,r_2,0}^*) = \om_{n_1,m_1,r_2,0} \ot \om_{n_1,m_1,r_2,0}^*$, see Definition \ref{d:torfre} and Definition \ref{def:torsion}. This proves the lemma.
\end{proof}

We apply Lemma \ref{l:dualexam} to advance our computation of the number $c(g,h,k) \in \cc^*$, where $g = \mu \cd z_1^{n_1} z_2^{n_2}, h = \nu \cd z_1^{m_1} z_2^{m_2}$ and $k = \ka \cd z_1^{l_1} z_2^{l_2}$ are invertible monomials with $n_1,n_2,m_1,m_2,l_1,l_2 \in \nn_0$. Indeed, we obtain that
\[
\begin{split}
\De_{q_{gh}}( b_{g,hk}) 
& = \om^*_{n_1,0,n_2,0} \ot \varphi(1) \ot \om_{n_1,0,n_2 + m_2 + l_2,0} \\
& = \om^*_{n_1,0,n_2,0} \ot \om_{n_1,0,n_2 + m_2,0} \ot \om_{n_1,0,n_2 + m_2,0}^* \ot \om_{n_1,0,n_2 + m_2 + l_2,0} \\
& = b_{g,h} \ot \om_{n_1,0,n_2 + m_2,0}^* \ot \om_{n_1,0,n_2 + m_2 + l_2,0} \colon (e,\si) \ot (e,\si) \to (g,\si) \ot (g,\si)
\end{split}
\]
as an isomorphism in the category $\G H(q_g,q_{gh}) \ot \G H(q_{gh},q_{ghk})$. The automorphism in Equation \eqref{eq:automorph} yielding the number $c(g,h,k) \in \cc^*$ then reduces to the composition of isomorphisms
\begin{equation}\label{eq:automorphII}
(gh,\si) \to^{b_{gh,k}^{-1}} (e,\si) \to^{ \om_{n_1,0,n_2 + m_2,0}^* \ot \om_{n_1,0,n_2 + m_2 + l_2,0}} (g,\si) \to^{g(b_{h,k})} (gh,\si)
\end{equation}
inside the category $\G H(q_{gh},q_{ghk})$ at least up to the sign $(-1)^{n_1 \cd m_2 \cd l_2 \cd (m_1 + 1)}$. This sign comes from commuting $b_{g,h}$ past $b_{gh,k}^{-1}$ since these isomorphisms lie in $\zz$-graded complex lines with gradings $n_1 \cd m_2$ and $-(n_1 + m_1) \cd l_2$, respectively.

\subsection{The group action}
Our aim is now to compute the group action $\rho(k) \colon \G H(P,Q,\al) \to \G H(P,Q,\al)$ in the case where $k = \ka \cd z_1^{l_1} z_2^{l_2}$ is an invertible monomial with $l_1,l_2 \geq 0$. We are interested in the situation where the elements in the underlying set $X$ are idempotents of the form $q_u, q_v \in X$ for invertible monomials $u = \xi \cd z_1^{t_1} z_2^{t_2}$ and $v = \eta \cd z_1^{s_1} z_2^{s_2}$. This means that we are looking at the isomorphism of categories
\[
\rho(k) \colon \G H(q_u,q_v) \to \G H( q_{ku}, q_{kv}) \, .
\]
For two invertible monomials $g = \mu \cd z_1^{n_1} z_2^{n_2}$ and $h = \nu \cd z_1^{m_1} z_2^{m_2}$ we have the objects $(h,\si)$ and $(g,\si)$ in the category $\G H(q_u,q_v)$ with corresponding $\zz$-graded complex line of morphisms given by
\[
\G H(q_u,q_v)( (h,\si),(g,\si)) = \big| F( (h,\si), (g,\si))(q_u,q_e) \big| \ot \big| F( (h,\si), (g,\si))(q_e,q_v) \big| \, .
\]
Supposing moreover that $n_1 \geq m_1 \geq 0$ and that $s_2, t_2 \geq 0$ we have computed the above graded determinant line explicitly in Lemma \ref{l:kercoker} and we may choose the non-trivial vector
\[
\om_{n_1,m_1,t_2,0}^* \ot \om_{n_1,m_1,s_2,0} \in \G H(q_u,q_v)( (h,\si),(g,\si)) ,
\]
see Equation \eqref{eq:omegavec} for the notation. Recalling the definition of $\rho(k)$ from Equation \eqref{eq:grpmor} and Equation \eqref{eq:funcgroup} and the computation in Example \ref{ex:LR} we obtain that
\begin{equation}\label{eq:grpconcrete}
\begin{split}
\rho(k)\big( \om_{n_1,m_1,t_2,0}^* \ot \om_{n_1,m_1,s_2,0} \big)
& = \G B( q_k, q_e)\big( t_k( \om_{n_1,m_1,t_2,0}^* \ot \om_{n_1,m_1,s_2,0} ) \big) \\
& = \G B( q_k, q_e)\big( \ka^{(n_1 - m_1)(s_2 - t_2)} \cd \om_{n_1 + l_1, m_1 + l_1, t_2 + l_2,l_2}^* \ot \om_{n_1 + l_1,m_1 + l_1,s_2 + l_2 , l_2} \big) \, ,
\end{split}
\end{equation}
where the non-trivial vector 
\[
\begin{split}
& \big( \al(k)(\om_{n_1, m_1, t_2,0}) \big)^* \ot \al(k)(\om_{n_1,m_1,s_2, 0} ) \\
& \q = ( \ka^{(n_1 - m_1)t_2} \cd \om_{n_1 + l_1, m_1 + l_1, t_2 + l_2,l_2})^* \ot \ka^{ (n_1 - m_1) s_2} \cd \om_{n_1 + l_1,m_1 + l_1,s_2 + l_2 , l_2} \\
& \q = \ka^{(n_1 - m_1)(s_2 - t_2)} \cd \om_{n_1 + l_1, m_1 + l_1, t_2 + l_2,l_2}^* \ot \om_{n_1 + l_1,m_1 + l_1,s_2 + l_2 , l_2}
\end{split}
\]
belongs to the graded determinant line
\[
\big| F( (kh,\si), (kg,\si))(q_{ku},q_k) \big| \ot \big| F( (kh,\si), (kg,\si))(q_k,q_{kv}) \big| \, .
\]
To finish our computation of the group action we therefore need to investigate the change-of-basepoint isomorphism more carefully in the present context. 

\begin{lemma}\label{l:baseconcrete}
Let $g = \mu \cd z_1^{n_1} z_2^{n_2}$, $h = \nu \cd z_1^{m_1} z_2^{m_2}$ and $u = \xi \cd z_1^{t_1} z_2^{t_2}$, $v = \eta \cd z_1^{s_1} z_2^{s_2}$, $w = \ze \cd z_1^{r_1} z_2^{r_2}$ be invertible monomials with $n_1 \geq m_1 \geq 0$ and $t_2,s_2 \geq r_2 \geq 0$. The change-of-basepoint isomorphism
\[
\begin{split}
\G B(q_w,q_e) & \colon \big| F( (h,\si), (g,\si))(q_u,q_w) \big| \ot \big| F( (h,\si), (g,\si))(q_w,q_v) \big| \\ 
& \q \to \big| F( (h,\si), (g,\si))(q_u,q_e) \big| \ot \big| F( (h,\si), (g,\si))(q_e,q_v) \big|
\end{split}
\]
is given explicitly by
\[
\G B(q_w,q_e) \colon \om_{n_1,m_1,t_2,r_2}^* \ot \om_{n_1,m_1,s_2,r_2} \mapsto (-1)^{ (n_1 - m_1)(t_2 - s_2) r_2} \cd \om_{n_1,m_1,t_2,0}^* \ot \om_{n_1,m_1,s_2,0} \, .
\]
\end{lemma}
\begin{proof}
We focus on the case where $t_2 \geq s_2$ since the case where $s_2 \geq t_2$ follows from a similar computation. Applying a slight abuse of notation we write $\si_g$ instead of $(g,\si)$. We are also going to suppress the triple of representations $(\si_h,\si_g,\si_g)$ from our formulae. Recall from Section \ref{s:change} that the change-of-basepoint isomorphism fits in the commutative diagram
\begin{equation}\label{eq:compbase}
\xymatrix{
\big| F( \si_h, \si_g)(q_u,q_w) \big| \ot \big| F( \si_h, \si_g)(q_w,q_v) \big| \ar[r]^{\G B(q_w,q_e)} \ar[d]^{\G S \G T \G S} & \big| F( \si_h, \si_g)(q_u,q_e) \big| \ot \big| F( \si_h, \si_g)(q_e,q_v) \big| \ar[d]^{\G S \G T \G S} \\
\big| F^{12}(q_w, q_v,q_u) F^{13}(q_u,q_v,q_w) + e_{23}( \si_g(1 - q_w)) \big| \ar[r]^{\G P} & 
\big| F^{12}(q_e, q_v,q_u) F^{13}(q_u,q_v,q_e) + e_{23}( \si_g(1 - q_e)) \big| 
}
\end{equation}
where $e_{23}( \si_g(1-q_w))$ refers to the matrix of operators which zeroes everywhere except for the operator $\si_g(1 - q_w)$ in position $(2,3)$. The first Fredholm operator in the lower line is given by
\[
\begin{split}
& F^{12}(q_w, q_v,q_u) F^{13}(q_u,q_v,q_w) + e_{23}( \si_g(1 - q_w)) \\
& \q = \ma{ccc}{ (P_{z_1^{m_1}} - P_{z_1^{n_1}}) Q_{z_2^{s_2}} & P_{z_1^{n_1}} Q_{z_2^{s_2}} & 0 \\
P_{z_1^{n_1}} Q_{z_2^{r_2}} & 0 & 0 \\
0 & 0 & P_{z_1^{n_1}} Q_{z_2^{t_2}}}
\ma{ccc}{ (P_{z_1^{m_1}} - P_{z_1^{n_1}}) Q_{z_2^{t_2}} & 0 & P_{z_1^{n_1}} Q_{z_2^{r_2}} \\
0 & P_{z_1^{n_1}} Q_{z_2^{s_2}} & 0 \\
P_{z_1^{n_1}} Q_{z_2^{t_2}} & 0 & 0 
 } \\ 
& \qqq + e_{23}( \si_g(1 - q_w)) \\
& \q = \ma{ccc}{ (P_{z_1^{m_1}} - P_{z_1^{n_1}}) Q_{z_2^{t_2}} & P_{z_1^{n_1}} Q_{z_2^{s_2}} & 0 \\ 
0 & 0 & P_{z_1^{n_1}} \\
P_{z_1^{n_1}} Q_{z_2^{t_2}} & 0 & 0} \, .
\end{split}
\]
A similar computation shows that the same formula holds for the second Fredholm operator in the lower line and hence that
\[
F^{12}(q_w, q_v,q_u) F^{13}(q_u,q_v,q_w) + e_{23}( \si_g(1 - q_w)) = F^{12}(q_e, q_v,q_u) F^{13}(q_u,q_v,q_e) + e_{23}( \si_g(1 - q_e)) \, .
\]
The perturbation isomorphism appearing in the expression for the change-of-basepoint isomorphism is therefore equal to the identity map and we may focus on computing the vertical isomorphisms in Equation \eqref{eq:compbase}. The stabilisation isomorphisms appearing are invisible in the final expressions since they merely assure that the involved Fredholm operators act on the correct Hilbert spaces without interfering with kernels and cokernels in an essential way. We therefore restrict attention to the torsion isomorphisms. We record that the kernel and cokernel of the Fredholm operator
\[
\begin{split}
& G := \ma{ccc}{ (P_{z_1^{m_1}} - P_{z_1^{n_1}}) Q_{z_2^{t_2}} & P_{z_1^{n_1}} Q_{z_2^{s_2}} & 0 \\ 
0 & 0 & P_{z_1^{n_1}} \\
P_{z_1^{n_1}} Q_{z_2^{t_2}} & 0 & 0} \\
& \q \colon (P_{z_1^{m_1}} Q_{z_2^{t_2}}) \C H \op ( P_{z_1^{n_1}} Q_{z_2^{s_2}}) \C H \op P_{z_1^{n_1}} \C H
\to (P_{z_1^{m_1}} Q_{z_2^{s_2}}) \C H \op P_{z_1^{n_1}} \C H \op ( P_{z_1^{n_1}} Q_{z_2^{t_2}}) \C H 
\end{split}
\]
are given by 
\[
\T{Ker}(G) = \{0\} \q \T{and} \q \T{Coker}(G) \cong \T{Im}\big(  \Ga_{n_1,m_1,t_2,s_2} \big) \op \{0\} \op \{0\} \, .
\]
Combining this computation of the kernel and the cokernel with the expressions given in Lemma \ref{l:kercoker} we see from Definition \ref{d:torfre} that the relevant six term exact sequence for the torsion isomorphism in the left hand side of Equation \eqref{eq:compbase} is given by
\[
\xymatrix{ \{0\} \ar@{=}[r] & \{0\} \ar[r] & \T{Im}\big( \Ga_{n_1,m_1,s_2,r_2}  \big) \ar[d] \\
\{0\} \ar@{=}[u] & \T{Im}\big(  \Ga_{n_1,m_1,t_2,s_2} \big) \ar[l] & 
\T{Im}\big( \Ga_{n_1,m_1,t_2,r_2} \big) \ar[l]^{Q_{z_2^{s_2}}} }
\]
Using that $\om_{n_1,m_1,t_2,r_2} = \om_{n_1,m_1,s_2,r_2} \we \om_{n_1,m_1,t_2,s_2}$ (see Equation \eqref{eq:omegavec}) and comparing with Definition \ref{def:torsion} we obtain that the left hand side of Equation \eqref{eq:compbase} operates as follows:
\[
\G S \G T \G S \colon \om_{n_1,m_1,t_2,r_2}^* \ot \om_{n_1,m_1,s_2,r_2}  \mapsto (-1)^{ (n_1 - m_1)(t_2 - s_2)(s_2 - r_2)} \cd \om_{n_1,m_1,t_2,s_2}^* \, . 
\]
Similarly, we find that the relevant six term exact sequence with regards to the right hand side is given by
\[
\xymatrix{ \{0\} \ar@{=}[r] & \{0\} \ar[r] & \T{Im}\big( \Ga_{n_1,m_1,s_2,0} \big) \ar[d] \\
\{0\} \ar@{=}[u] & \T{Im}\big(  \Ga_{n_1,m_1,t_2,s_2}  \big) \ar[l] & 
\T{Im}\big( \Ga_{n_1,m_1,t_2,0} \big) \ar[l]^{Q_{z_2^{s_2}}} } 
\]
and hence that the right hand side of Equation \eqref{eq:compbase} takes the form
\[
\G S \G T \G S \colon \om_{n_1,m_1,t_2,0}^* \ot \om_{n_1,m_1,s_2,0} \mapsto (-1)^{ (n_1 - m_1)(t_2 - s_2)s_2} \cd \om_{n_1,m_1,t_2,s_2}^* \, .
\]
We conclude that the change-of-basepoint isomorphism is given explicitly by
\[
\G B(q_w,q_e) \colon \om_{n_1,m_1,t_2,r_2}^* \ot \om_{n_1,m_1,s_2,r_2} \mapsto (-1)^{ (n_1 - m_1)(t_2 - s_2) r_2} \cd \om_{n_1,m_1,t_2,0}^* \ot \om_{n_1,m_1,s_2,0} \, . \qedhere
\]
\end{proof}

We now combine the above Lemma \ref{l:baseconcrete} with the computation in Equation \eqref{eq:grpconcrete} and obtain the expression:
\[
\begin{split}
& \rho(k)( \om_{n_1,m_1,t_2,0}^* \ot \om_{n_1,m_1,s_2,0})
= \G B(q_k,q_e)( \ka^{(n_1 - m_1)(s_2 - t_2)} \cd \om_{n_1 + l_1,m_1 + l_1,t_2 + l_2,l_2}^* \ot \om_{n_1+ l_1,m_1 + l_1,s_2 + l_2,l_2}) \\
& \q = (-1)^{(n_1 - m_1)(t_2 - s_2)l_2} \cd \ka^{(n_1 - m_1)(s_2 - t_2)} \cd \om_{n_1 + l_1,m_1 + l_1,t_2 + l_2,0 }^* \ot \om_{n_1+ l_1,m_1 + l_1,s_2 + l_2,0} \, .
\end{split}
\]

To finish our discussion of the group action in this concrete setting we continue our computation of the number $c(g,h,k) \in \cc^*$ in the case where $g = \mu \cd z_1^{n_1} z_2^{n_2}$, $h = \nu \cd z_1^{m_1} z_2^{m_2}$ and $k = \ka \cd z_1^{l_1} z_2^{l_2}$ are invertible monomials with $n_1,n_2,m_1,m_2,l_1,l_2 \geq 0$. Recalling the definition of $b_{h,k}$ from Equation \eqref{eq:bexpdef} we record that
\[
g(b_{h,k}) = g( \om_{m_1,0,m_2,0}^* \ot \om_{m_1,0,m_2 + l_2,0})
= (-1)^{m_1 l_2 n_2} \cd \mu^{m_1 l_2} \cd \om_{m_1 + n_1,n_1,m_2 + n_2,0}^* \ot \om_{m_1 + n_1,n_1,m_2 + l_2 + n_2,0} \, .
\]
Hence, comparing with Equation \eqref{eq:automorphII}, we obtain that the automorphism yielding the number $c(g,h,k) \in \cc^*$ is given by
\begin{equation}\label{eq:automorphIII}
\begin{split}
(gh,\si) \to^{b_{gh,k}^{-1}} (e,\si) & \to^{ \om_{n_1,0,n_2 + m_2,0}^* \ot \om_{n_1,0,n_2 + m_2 + l_2,0}} (g,\si) \\
& \q \to^{ \mu^{m_1 l_2} \cd \om_{n_1 + m_1,n_1,n_2 + m_2,0}^* \ot \om_{n_1 + m_1,n_1,n_2 + m_2 + l_2,0} } (gh,\si)
\end{split}
\end{equation}
up to the sign $(-1)^{l_2 (n_2 m_1 + n_1 m_2 m_1 + n_1 m_2)}$.

\subsection{The composition}
We are now going to describe the composition inside the category $\G H(q_v,q_u)$ in the case where $u = \xi \cd z_1^{t_1} z_2^{t_2}$ and $v = \eta \cd z_1^{s_1} z_2^{s_2}$ are invertible monomials with $s_2,t_2 \geq 0$. We moreover restrict our attention to the composition of morphisms in $\G H(q_v,q_u)((k,\si),(h,\si))$ and in $\G H(q_v,q_u)((h,\si),(g,\si))$ where $g = \mu \cd z_1^{n_1} z_2^{n_2}$, $h = \nu \cd z_1^{m_1} z_2^{m_2}$ and $k = \ka \cd z_1^{l_1} z_2^{l_2}$ are invertible monomials with $n_1 \geq m_1 \geq l_1 \geq 0$. In this situation we know that the morphisms $\G H(q_v,q_u)((k,\si),(h,\si))$ agree with the $\zz$-graded complex line
\[
\G L_{q_v}\big( (k,\si), (h,\si) \big) \ot \G L_{q_u}^\da\big( (k,\si), (h,\si)\big)
= \Big| F\big( (k,\si), (h,\si)\big)(q_v,q_e)  \Big| \ot \Big| F\big( (k,\si), (h,\si)\big)(q_e,q_u) \Big|
\]
Moreover, from Lemma \ref{l:kercoker} we have that
\[
\begin{split}
& \Big| F\big( (k,\si), (h,\si)\big)(q_v,q_e)  \Big| 
\cong \Big( \La^{\T{top}}\big( \T{Im}(\Ga_{m_1,l_1,s_2,0}) \op \{0\} \big)^*, s_2 \cd (l_1 - m_1)  \Big) \\
& \Big| F\big( (k,\si), (h,\si)\big)(q_e,q_u) \Big| 
= \Big( \La^{\T{top}}\big( \T{Im}(\Ga_{m_1,l_1,t_2,0}) \op \{0\} \big), t_2 \cd (m_1 - l_1) \Big)
\end{split}
\]
and we may single out the explicit morphism:
\[
\om_{m_1,l_1,s_2,0}^* \ot \om_{m_1,l_1,t_2,0} 
\in \G L_{q_v}\big( (k,\si), (h,\si) \big) \ot \G L_{q_u}^\da\big( (k,\si), (h,\si)\big) ,
\]
see Equation \eqref{eq:omegavec}. A similar description applies to the morphisms from $(h,\si)$ to $(g,\si)$ and from $(k,\si)$ to $(g,\si)$. The definition of the composition isomorphism appearing in the lemma here below can be found in Definition \ref{d:hopf}.

\begin{lemma}\label{l:expcomp}
Let $u = \xi \cd z_1^{t_1} z_2^{t_2}$, $v = \eta \cd z_1^{s_1} z_2^{s_2}$ and $g = \mu \cd z_1^{n_1} z_2^{n_2}$, $h = \nu \cd z_1^{m_1} z_2^{m_2}$ and $k = \ka \cd z_1^{l_1} z_2^{l_2}$ be invertible monomials with $s_2,t_2 \geq 0$ and $n_1 \geq m_1 \geq l_1 \geq 0$. The composition isomorphism
\[
\G M_{q_v,q_u} \colon \G H(q_v,q_u)\big( (k,\si),(h,\si)\big) \ot \G H(q_v,q_u)\big( (h,\si),(g,\si)\big)
\to \G H(q_v,q_u)\big( (k,\si),(g,\si)\big) 
\]
is given explicitly by
\[
\begin{split}
\G M_{q_v,q_u} & \colon \om_{m_1,l_1,s_2,0}^* \ot \om_{m_1,l_1,t_2,0} \ot \om_{n_1,m_1,s_2,0}^* \ot \om_{n_1,m_1,t_2,0} \\
& \q \mapsto (-1)^{(n_1 -m_1)(m_1 - l_1) t_2 (t_2 - s_2)} \cd ( \om_{m_1,l_1,s_2,0} \we \om_{n_1,m_1,s_2,0} )^* \ot ( \om_{m_1,l_1,t_2,0} \we \om_{n_1,m_1,t_2,0}) .
\end{split}
\]
\end{lemma}
\begin{proof}
The sign $(-1)^{(n_1 -m_1)(m_1 - l_1) t_2 (t_2 - s_2)}$ comes from the symmetry constraint when passing from the $\zz$-graded complex line
\[
\G L_{q_v}\big( (k,\si), (h,\si) \big) \ot \G L_{q_u}^\da\big( (k,\si), (h,\si)\big) 
\ot \G L_{q_v}\big( (h,\si), (g,\si) \big) \ot \G L_{q_u}^\da\big( (h,\si), (g,\si)\big) 
\]
to the $\zz$-graded complex line
\[
\G L_{q_v}\big( (k,\si), (h,\si) \big) \ot \G L_{q_v}\big( (h,\si), (g,\si) \big) 
\ot \G L_{q_u}^\da\big( (h,\si), (g,\si)\big) \ot \G L_{q_u}^\da\big( (k,\si), (h,\si)\big) \, .
\]
We are thus claiming that 
\[
\begin{split}
& \G M_{q_v}( \om_{m_1,l_1,s_2,0}^* \ot  \om_{n_1,m_1,s_2,0}^* ) = ( \om_{m_1,l_1,s_2,0} \we \om_{n_1,m_1,s_2,0} )^* \q \T{and} \\
& \G M_{q_u}^\da(  \om_{n_1,m_1,t_2,0} \ot \om_{m_1,l_1,t_2,0}) = \om_{m_1,l_1,t_2,0} \we \om_{n_1,m_1,t_2,0} \, .
\end{split}
\]
We shall only establish this claim for the case of the isomorphism 
\[
\G M_{q_v} \colon \G L_{q_v}\big( (k,\si), (h,\si) \big) \ot \G L_{q_v}\big( (h,\si), (g,\si) \big)  \to \G L_{q_v}\big( (k,\si), (g,\si) \big)
\]
since the proof in the case of $\G M_{q_u}^\da$ follows a similar pattern. Alternatively, it is possible to derive the formula for $\G M_{q_u}^\da$ by applying the duality relation from Proposition \ref{p:dualcomp}. For more details on the isomorphism $\G M_{q_v}$ we refer to Definition \ref{def:comp}.

In view of the description of the coproduct from Lemma \ref{l:dualexam}, sending the unit $1 \in \B C$ to the element 
\[
\om_{m_1,l_1,s_2,0} \we \om_{n_1,m_1,s_2,0} \ot ( \om_{m_1,l_1,s_2,0} \we \om_{n_1,m_1,s_2,0} )^*
\in \G L_{q_v}^\da \big( (k,\si), (g,\si) \big) \ot \G L_{q_v} \big( (k,\si), (g,\si) \big)
\]
it suffices to show that
\[
\mu_{q_v}\big( \om_{m_1,l_1,s_2,0}^* \ot  \om_{n_1,m_1,s_2,0}^* \ot \om_{m_1,l_1,s_2,0} \we \om_{n_1,m_1,s_2,0}\big) = 1 \, .
\]

We introduce the Fredholm operators
\[
\begin{split}
F^{12} & := \ma{ccc}{ (P_{z_1^{l_1}} - P_{z_1^{m_1}}) Q_{z_2^{s_2}} & P_{z_1^{m_1}} Q & 0 \\
P_{z_1^{m_1}} Q_{z_2^{s_2}} & 0 & 0 \\
0 & 0 & P_{z_1^{n_1}} Q } \\
& \q \colon P_{z_1^{l_1}} Q_{z_2^{s_2}} \C H \op P_{z_1^{m_1}} Q \C H \op P_{z_1^{n_1}} Q \C H
\to P_{z_1^{l_1}} Q \C H \op P_{z_1^{m_1}} Q_{z_2^{s_2}} \C H \op P_{z_1^{n_1}} Q \C H \\
F^{23} & := \ma{ccc}{ P_{z_1^{l_1}} Q & 0 & 0 \\
0 & (P_{z_1^{m_1}} - P_{z_1^{n_1}}) Q_{z_2^{s_2}} & P_{z_1^{n_1}} Q \\
0 & P_{z_1^{n_1}} Q_{z_2^{s_2}} & 0 } \\
& \q \colon P_{z_1^{l_1}} Q \C H \op P_{z_1^{m_1}} Q_{z_2^{s_2}} \C H \op P_{z_1^{n_1}} Q \C H
\to P_{z_1^{l_1}} Q \C H \op P_{z_1^{m_1}} Q \C H \op P_{z_1^{n_1}} Q_{z_2^{s_2}} \C H \q \T{and} \\
F^{13}_\da & := 
\ma{ccc}{ (P_{z_1^{l_1}} - P_{z_1^{n_1}})Q_{z_2^{s_2}} & 0 & P_{z_1^{n_1}} Q_{z_2^{s_2}} \\
0 & P_{z_1^{m_1}} Q & 0 \\
P_{z_1^{n_1}} Q & 0  & 0 } \\
& \q \colon P_{z_1^{l_1}} Q \C H \op P_{z_1^{m_1}} Q \C H \op P_{z_1^{n_1}} Q_{z_2^{s_2}} \C H
\to P_{z_1^{l_1}} Q_{z_2^{s_2}} \C H \op P_{z_1^{m_1}} Q \C H \op P_{z_1^{n_1}} Q \C H \, .
\end{split}
\]
These Fredholm operators are stabilised versions of the Fredholm operators $F\big( (k,\si),(h,\si)\big)(q_v,q_e)$, \newline $F\big( (h,\si),(g,\si)\big)(q_v,q_e)$ and $F\big( (k,\si),(g,\si)\big)(q_e,q_v)$, respectively. Our first task is to compute the torsion isomorphism
\[
\G T = \G T \ci (\T{id} \ot \G T) 
\colon |F^{12}| \ot |F^{23}| \ot |F^{13}_\da| \to 
|F^{12}| \ot |F^{13}_\da F^{23}| \to 
|F^{13}_\da F^{23} F^{12}| \, .
\]
The various products of Fredholm operators appearing can be computed and are given by
\[
\begin{split}
F^{13}_\da F^{23} & = \ma{ccc}{
(P_{z_1^{l_1}} - P_{z_1^{n_1}})Q_{z_2^{s_2}} & P_{z_1^{n_1}} Q_{z_2^{s_2}} & 0 \\
0 & (P_{z_1^{m_1}} - P_{z_1^{n_1}}) Q_{z_2^{s_2}} & P_{z_1^{n_1}} Q \\
P_{z_1^{n_1}} Q & 0 & 0 
}  \\ 
&\q \colon P_{z_1^{l_1}} Q \C H \op P_{z_1^{m_1}} Q_{z_2^{s_2}} \C H \op P_{z_1^{n_1}} Q \C H
\to P_{z_1^{l_1}} Q_{z_2^{s_2}} \C H \op P_{z_1^{m_1}} Q \C H \op P_{z_1^{n_1}} Q \C H 
\q \T{and} \\
F^{13}_\da F^{23} F^{12} 
& = \ma{ccc}{ 
(P_{z_1^{l_1}} - P_{z_1^{m_1}})Q_{z_2^{s_2}} + P_{z_1^{n_1}} Q_{z_2^{s_2}} & (P_{z_1^{m_1}} - P_{z_1^{n_1}}) Q_{z_2^{s_2}} & 0 \\
(P_{z_1^{m_1}} - P_{z_1^{n_1}}) Q_{z_2^{s_2}} & 0 & P_{z_1^{n_1}} Q \\
0 & P_{z_1^{n_1}} Q & 0 
} \\
& \q \colon P_{z_1^{l_1}} Q_{z_2^{s_2}} \C H \op P_{z_1^{m_1}} Q \C H \op P_{z_1^{n_1}} Q \C H
\to P_{z_1^{l_1}} Q_{z_2^{s_2}} \C H \op P_{z_1^{m_1}} Q \C H \op P_{z_1^{n_1}} Q \C H \, .
\end{split}
\]
Moreover, the respective kernels and cokernels are given by
\[
\begin{split}
& \T{Ker}(F^{12}) = \{0\} = \T{Ker}(F^{23}) = \T{Coker}(F_{\da}^{13}) \\
& \T{Coker}(F^{12}) \cong \T{Im}(\Ga_{m_1,l_1,s_2,0}) \op \{0\} \op \{0\} \\
& \T{Coker}(F^{23}) \cong \{0\} \op \T{Im}(\Ga_{n_1,m_1,s_2,0}) \op \{0\} \cong \T{Coker}(F^{13}_\da F^{23}) \\
& \T{Ker}(F_\da^{13}) = \T{Im}(\Ga_{n_1,l_1,s_2,0}) \op \{0 \} \op \{0\} = \T{Ker}(F^{13}_\da F^{23}) \\
& \T{Ker}( F^{13}_\da F^{23} F^{12} ) = \{0\} \op \T{Im}(\Ga_{n_1,m_1,s_2,0}) \op \{0\} 
\cong \T{Coker}( F^{13}_\da F^{23} F^{12} )
\end{split}
\]
The torsion isomorphism $\G T \colon |F^{23}| \ot |F^{13}_\da| \to |F^{13}_\da F^{23}|$ therefore comes from the six term exact sequence
\[
\xymatrix{ \{0\} \ar[r] & \T{Im}(\Ga_{n_1,l_1,s_2,0}) \ar@{=}[r] & \T{Im}( \Ga_{n_1,l_1,s_2,0} ) \ar[d]^{0} \\
\{0\} \ar@{=}[u] & \T{Im}(  \Ga_{n_1,m_1,s_2,0} ) \ar[l] & 
\T{Im}( \Ga_{n_1,m_1,s_2,0} ) \ar@{=}[l] } 
\]
whereas the torsion isomorphism $\G T \colon |F^{12}| \ot |F^{13}_\da F^{23}| \to |F^{13}_\da F^{23} F^{12}|$ comes from the six term exact sequence
\[
\xymatrix{ \{0\} \ar[r] & \T{Im}(\Ga_{n_1,m_1,s_2,0}) \ar[r]^{\io} & \T{Im}( \Ga_{n_1,l_1,s_2,0} ) \ar[d]^{P_{z_1^{l_1}} - P_{z_1^{m_1}}} \\
\T{Im}(\Ga_{n_1,m_1,s_2,0}) \ar[u] & \T{Im}(  \Ga_{n_1,m_1,s_2,0} ) \ar@{=}[l] & 
\T{Im}( \Ga_{m_1,l_1,s_2,0} ) \ar[l]^{0} } 
\]
where the map $\io$ is the inclusion, see Definition \ref{d:torfre}. The isomorphism $\G T \colon |F^{23}| \ot |F^{13}_\da| \to |F^{13}_\da F^{23}|$ is therefore given explicitly by
\[
\G T \colon \om_{n_1,m_1,s_2,0}^* \ot \om_{n_1,l_1,s_2,0} \mapsto (-1)^{(n_1 -l_1) (n_1 - m_1) s_2} \cd \om_{n_1,l_1,s_2,0} \ot \om_{n_1,m_1,s_2,0}^*
\]
whereas the isomorphism $\G T \colon |F^{12}| \ot |F^{13}_\da F^{23}| \to |F^{13}_\da F^{23} F^{12}|$ is given explicitly by
\[
\G T \colon \om_{m_1,l_1,s_2,0}^* \ot \om_{n_1,m_1,s_2,0} \we \om_{m_1,l_1,s_2,0} \ot \om_{n_1,m_1,s_2,0}^* 
\mapsto \om_{n_1,m_1,s_2,0} \ot \om_{n_1,m_1,s_2,0}^*  \, .
\]
For details on these torsion isomorphisms we refer to Definition \ref{def:torsion}. Combining these computations we conclude that the torsion isomorphism $\G T \colon |F^{12}| \ot |F^{23}| \ot |F^{13}_\da| \to |F^{13}_\da F^{23} F^{12}|$ is described by
\[
\G T \colon \om_{m_1,l_1,s_2,0}^* \ot  \om_{n_1,m_1,s_2,0}^* \ot \om_{n_1,m_1,s_2,0} \we \om_{m_1,l_1,s_2,0}
\mapsto (-1)^{(n_1 -l_1) (n_1 - m_1) s_2} \cd \om_{n_1,m_1,s_2,0} \ot \om_{n_1,m_1,s_2,0}^* \, .
\]

Comparing with Equation \eqref{eq:comp} and Equation \eqref{eq:sigmaspecial} we now introduce the idempotent
\[
f := P_{z_1^{l_1}}(Q - Q_{z_2^{s_2}}) \op 0 \op 0 \, .
\]
We are going to disregard the extra subspace given by the idempotent $P_{z_1^{l_1}}(1 - Q) \op P_{z_1^{m_1}}(1 - Q) \op P_{z_1^{n_1}}(1 - Q)$ since it appears twice in the stabilisation procedure present in Equation \eqref{eq:comp} and therefore does not influence the final result. We apply $f$ to stabilise the Fredholm operator $F^{13}_\da F^{23} F^{12}$ and we thereby obtain the Fredholm operator
\[
F_\da^{13} F^{23} F^{12} + f = \ma{ccc}{ 
(P_{z_1^{n_1}} - P_{z_1^{m_1}})Q_{z_2^{s_2}} + P_{z_1^{l_1}} Q & (P_{z_1^{m_1}} - P_{z_1^{n_1}}) Q_{z_2^{s_2}} & 0 \\
(P_{z_1^{m_1}} - P_{z_1^{n_1}}) Q_{z_2^{s_2}} & 0 & P_{z_1^{n_1}} Q \\
0 & P_{z_1^{n_1}} Q & 0 
}
\]
which acts as an endomorphism of the Hilbert space $P_{z_1^{l_1}} Q \C H \op P_{z_1^{m_1}} Q \C H \op P_{z_1^{n_1}} Q \C H$. We are going to trivialise the $\zz$-graded complex line associated to the above Fredholm operator and to this end we introduce the invertible operators
\[
\begin{split}
\Om^{12} & := \ma{ccc}{ (P_{z_1^{l_1}} - P_{z_1^{m_1}}) Q & P_{z_1^{m_1}} Q & 0 \\
P_{z_1^{m_1}} Q & 0 & 0 \\
0 & 0 & P_{z_1^{n_1}} Q } \\
\Om^{23} & := \ma{ccc}{ P_{z_1^{l_1}} Q & 0 & 0 \\
0 & (P_{z_1^{m_1}} - P_{z_1^{n_1}}) Q & P_{z_1^{n_1}} Q \\
0 & P_{z_1^{n_1}} Q & 0 } \q \T{and} \\
\Om^{13} & := 
\ma{ccc}{ (P_{z_1^{l_1}} - P_{z_1^{n_1}})Q & 0 & P_{z_1^{n_1}} Q \\
0 & P_{z_1^{m_1}} Q & 0 \\
P_{z_1^{n_1}} Q & 0  & 0 } 
\end{split}
\]
all acting as automorphisms of the Hilbert space $P_{z_1^{l_1}} Q \C H \op P_{z_1^{m_1}} Q \C H \op P_{z_1^{n_1}} Q \C H$. The product of these invertible operators is then given by
\[
\Om^{13} \Om^{23} \Om^{12} = \ma{ccc}{ (P_{z_1^{l_1}} - P_{z_1^{m_1}} + P_{z_1^{n_1}}) Q & (P_{z_1^{m_1}} - P_{z_1^{n_1}}) Q & 0 \\
(P_{z_1^{m_1}} - P_{z_1^{n_1}})Q & 0 & P_{z_1^{n_1}} Q  \\
0 & P_{z_1^{n_1}} Q & 0} \, .
\]
We remark that the difference 
\[
\Om^{13} \Om^{23} \Om^{12} - (F_\da^{13} F^{23} F^{12} + f)
= \ma{ccc}{ (P_{z_1^{n_1}} - P_{z_1^{m_1}})(Q - Q_{z_2^{s_2}}) & (P_{z_1^{m_1}} - P_{z_1^{n_1}})(Q - Q_{z_2^{s_2}}) & 0  \\ 
(P_{z_1^{m_1}} - P_{z_1^{n_1}})(Q - Q_{z_2^{s_2}}) & 0 & 0 \\ 0 & 0 & 0
}
\]
makes sense (the operators in question act on the same Hilbert space) and is of trace class (in fact of finite rank). Our task is now to compute the perturbation isomorphism
\[
\G P \colon \big| F_\da^{13} F^{23} F^{12} + f \big| \to \big| \Om^{13} \Om^{23} \Om^{12} \big| = (\cc,0) \, .
\]
We consider the finite rank operator $N := 0 \op \Ga_{n_1,m_1,s_2,0} \op 0$ which induces an isomorphism 
\[
N \colon \T{Ker}( F_\da^{13} F^{23} F^{12} + f ) \to \T{Coker}(F_\da^{13} F^{23} F^{12} + f) \, ,
\]
makes $F_\da^{13} F^{23} F^{12} + f + N$ invertible, and has $\T{Ker}(N)$ as a vector space complement of $\T{Ker}(F_\da^{13} F^{23} F^{12} + f)$. It thus follows from Example \ref{ex:indzero} that
\[
\G P( \om_{n_1,m_1,s_2,0} \ot \om_{n_1,m_1,s_2,0}^*) = \T{det}(\Si)
\]
where $\Si := (\Om^{13} \Om^{23} \Om^{12})( F_\da^{13} F^{23} F^{12} + f + N)^{-1}$. We record that
\[
(\Om^{13} \Om^{23} \Om^{12})^{-1} = \ma{ccc}{ (P_{z_1^{l_1}} - P_{z_1^{m_1}} + P_{z_1^{n_1}}) Q & (P_{z_1^{m_1}} - P_{z_1^{n_1}}) Q & 0 \\
(P_{z_1^{m_1}} - P_{z_1^{n_1}})Q & 0 & P_{z_1^{n_1}} Q  \\
0 & P_{z_1^{n_1}} Q & 0 }
\]
and hence that
\[
\begin{split}
\Si^{-1} 
& = (F_\da^{13} F^{23} F^{12} + f + N)(\Om^{13} \Om^{23} \Om^{12})^{-1} \\
& = \ma{ccc}{ 
(P_{z_1^{n_1}} - P_{z_1^{m_1}})Q_{z_2^{s_2}} + P_{z_1^{l_1}} Q & (P_{z_1^{m_1}} - P_{z_1^{n_1}}) Q_{z_2^{s_2}} & 0 \\
(P_{z_1^{m_1}} - P_{z_1^{n_1}}) Q_{z_2^{s_2}} & \Ga_{n_1,m_1,s_2,0} & P_{z_1^{n_1}} Q \\
0 & P_{z_1^{n_1}} Q & 0 
} \\ 
& \qq \cd \ma{ccc}{ (P_{z_1^{l_1}} - P_{z_1^{m_1}} + P_{z_1^{n_1}}) Q & (P_{z_1^{m_1}} - P_{z_1^{n_1}}) Q & 0 \\
(P_{z_1^{m_1}} - P_{z_1^{n_1}})Q & 0 & P_{z_1^{n_1}} Q  \\
0 & P_{z_1^{n_1}} Q & 0 } \\
& = \ma{ccc}{ P_{z_1^{l_1}}Q - \Ga_{n_1,m_1,s_2,0} & \Ga_{n_1,m_1,s_2,0} & 0 \\
\Ga_{n_1,m_1,s_2,0} &  P_{z_1^{m_1}} Q - \Ga_{n_1,m_1,s_2,0} & 0 \\
0 & 0 & P_{z_1^{n_1}} Q} \, .
\end{split}
\]
The Fredholm determinant of this determinant class operator can be computed and is given by $\T{det}(\Si^{-1}) = (-1)^{(n_1 - m_1)s_2}$.
We thus obtain that $\G P( \om_{n_1,m_1,s_2,0} \ot \om_{n_1,m_1,s_2,0}^*) = (-1)^{(n_1 - m_1) s_2}$.

Combining the above computation of the torsion isomorphism $\G T \colon |F^{12}| \ot |F^{23}| \ot |F^{13}_\da| \to \big| F^{13}_\da F^{23} F^{12}\big|$ and the perturbation isomorphism $\G P \colon \big| F^{13}_\da F^{23} F^{12} + f\big| \to \big| \Om^{13} \Om^{23} \Om^{12}\big| = (\cc,0)$ we obtain that the trivialisation 
\[
\mu_{q_v} = (\G P \G S) \ci (\G T \G S) \colon \big| F\big( (k,\si),(h,\si)\big)(q_v,q_e) \big| \ot \big| F\big( (h,\si),(g,\si)\big)(q_v,q_e)\big| \ot \big| F\big( (k,\si),(g,\si)\big)(q_e,q_v) \big| \to (\cc,0)
\]
is described by the assignment
\[
\begin{split}
& \mu_{q_v}( \om_{m_1,l_1,s_2,0}^* \ot  \om_{n_1,m_1,s_2,0}^* \ot \om_{m_1,l_1,s_2,0} \we \om_{n_1,m_1,s_2,0} ) \\
& \q = (-1)^{(n_1 - l_1)(n_1 - m_1) s_2 + (m_1 - l_1)(n_1 - m_1) s_2} \cd \G P( \om_{n_1,m_1,s_2,0} \ot \om_{n_1,m_1,s_2,0}^*)  \\
& \q = (-1)^{(n_1 - l_1)(n_1 - m_1) s_2 + (m_1 - l_1)(n_1 - m_1) s_2 + (n_1 - m_1) s_2} = 1 \, .
\end{split}
\]
This ends the proof of the present lemma.
\end{proof}

\subsection{Non-triviality}
We now return to the setting where $g = \mu \cd z_1^{n_1} z_2^{n_2}$, $h = \nu \cd z_1^{m_1} z_2^{m_2}$ and $k = \ka \cd z_1^{l_1} z_2^{l_2}$ are invertible monomials with exponents $n_1,n_2,m_1,m_2,l_1,l_2$ in $\nn \cup \{0\}$. According to Equation \eqref{eq:automorphIII} and Lemma \ref{l:expcomp}, the automorphism yielding the number $c(g,h,k) \in \cc^*$ is described by the composition 
\begin{equation}\label{eq:automorphIV}
(gh,\si) \to^{b_{gh,k}^{-1}} (e,\si) \to^{ \mu^{m_1 l_2} \cd (\om_{n_1,0,n_2 + m_2,0} \we \om_{n_1 + m_1,n_1,n_2 + m_2,0})^* \ot ( \om_{n_1,0,n_2 + m_2 + l_2,0} 
\we \om_{n_1 + m_1,n_1,n_2 + m_2 + l_2,0} ) } (gh,\si)
\end{equation}
up to the sign 
\[
(-1)^{l_2 (n_2 m_1 + n_1 m_2 m_1 + n_1 m_2) + n_1 \cd m_1 \cd (n_2 + m_2 + l_2) \cd l_2} = (-1)^{(n_1m_2 + n_1 m_1 + n_2 m_1) l_2 + n_1 n_2 m_1 l_2} \, .
\]

The result of the next proposition relies on the explicit choices made in the beginning of Subsection \ref{s:nontriv} near the statement of Lemma \ref{l:kercoker}.

\begin{prop}\label{p:compmono}
Let $g = \mu \cd z_1^{n_1} z_2^{n_2}$, $h = \nu \cd z_1^{m_1} z_2^{m_2}$ and $k = \ka \cd z_1^{l_1} z_2^{l_2}$ be invertible monomials with exponents $n_1,n_2,m_1,m_2,l_1,l_2$ in $\nn \cup \{0\}$. It holds that 
\[
c(g,h,k) = c(P,Q,\al)(g,h,k) = \mu^{m_1 l_2} 
\cd (-1)^{\ep(g,h,k)} \, ,
\]
where the sign is given by
\[
\begin{split}
\ep(g,h,k)
& := (n_1 m_2 + n_1 m_1 + n_2 m_1) l_2  +  n_1 n_2 m_1 l_2  \\ 
& \q +  n_1 m_1 \cd ( n_2 + m_2 - 1)(n_2 + m_2)/2  +  n_1 m_1 \cd ( n_2 + m_2 + l_2 - 1)(n_2 + m_2 + l_2)/2 \, .
\end{split}
\]
\end{prop}
\begin{proof}
We recall from Equation \eqref{eq:bexpdef} that $b_{gh,k} = \om^*_{n_1 + m_1,0,n_2 + m_2,0} \ot \om_{n_1 + m_1,0,n_2 + m_2 + l_2,0}$, thus to compute the remaining composition in Equation \eqref{eq:automorphIV}, we consult Equation \eqref{eq:omegavec} and obtain that
\[
\begin{split}
& \om_{n_1,0,n_2 + m_2,0} \bigwedge \om_{n_1 + m_1,n_1,n_2 + m_2,0} \\
& \q = ( [1] \wlw [z_1^{n_1 - 1}] ) \we ( [z_2] \wlw [z_1^{n_1 - 1} z_2] ) 
\wlw (  [z_2^{n_2 + m_2 - 1}] \wlw [z_1^{n_1 - 1} z_2^{n_2 + m_2 - 1}]) \\
& \qq \bigwedge ( [z_1^{n_1}] \wlw [z_1^{n_1 + m_1 - 1}] ) \we ( [z_1^{n_1} z_2] \wlw [z_1^{n_1 + m_1 - 1} z_2] ) \\
& \qqq \wlw (  [z_1^{n_1} z_2^{n_2 + m_2 - 1}] \wlw [z_1^{n_1 + m_1 - 1} z_2^{n_2 + m_2 - 1}]) \\
& \q = (-1)^{n_1 \cd  m_1 \cd \sum_{j = 1}^{n_2 + m_2 - 1} j } \cd ( [1] \wlw [z_1^{n_1 + m_1 - 1}] ) \we ( [z_2] \wlw [z_1^{n_1 + m_1 - 1} z_2] ) \\ 
& \qqq \wlw (  [z_2^{n_2 + m_2 - 1}] \wlw [z_1^{n_1 + m_1 - 1} z_2^{n_2 + m_2 - 1}]) \\
& \q = (-1)^{n_1 \cd  m_1 \cd ( n_2 + m_2 - 1)(n_2 + m_2)/2 } \cd \om_{n_1 + m_1,0,n_2 + m_2,0} \, .
\end{split}
\]
Similarly, we have that
\[
\om_{n_1,0,n_2 + m_2 + l_2,0} \we \om_{n_1 + m_1,n_1,n_2 + m_2 + l_2,0} 
= (-1)^{n_1 \cd  m_1 \cd ( n_2 + m_2 + l_2 - 1)(n_2 + m_2 + l_2)/2 } \cd \om_{n_1 + m_1,0,n_2 + m_2 + l_2,0} \, .
\]
These computations prove the proposition.
\end{proof}

As a consequence of the above Proposition \ref{p:compmono} we may show that our group $3$-cocycle yields a non-trivial cohomology class $[c(P,Q,\al)] \in H^3\big(C^\infty(\B T^2)^*, \cc^*/\{\pm 1 \} \big)$. We recall the definition of the class in group $3$-homology $\{ z_1,z_2,\la\} \in H_3( C^\infty(\B T^2)^*,\zz)$ from Equation \eqref{eq:classgrouphom}. 

\begin{corollary}\label{c:nontriv}
We have the identity 
\[
\binn{[c(P,Q,\al)], \{z_1,z_2,\la\}} = [\la] \in \cc^* / \{\pm 1 \}
\]
for the pairing between the group $3$-cohomology class $[c(P,Q,\al)]$ and the group $3$-homology class $\{z_1,z_2,\la\}$.
\end{corollary}
\begin{proof}
The result of Proposition \ref{p:compmono} implies that $c(\la, z_1,z_2) = \la$ whereas
\[
c(\la,z_2,z_1) = c(z_1,\la,z_2) = c(z_1,z_2,\la) = c(z_2, \la,z_1) = c(z_2, z_1,\la) = 1 \, . \qedhere
\]
\end{proof}

We immediately obtain the following:

\begin{corollary}\label{c:gennontriv}
Let $V$ be a finite dimensional Hilbert space, and let $\Gamma$ be a Lie group with a smooth linear action on $V$ such that scalar multiplication $\cc^* \to \sL(V)$ factors through a homomorphism $\cc^* \to \Gamma$. Then the class $[c] \in H^3\big(C^\infty(\B T^2,\Gamma),\cc^*/\{\pm 1 \}\big)$ associated to the representation of $\Ga$ on $V$ as in Theorem \ref{t:bipolbiloop} is non-trivial. 
\end{corollary}


\section{Proofs of properties of the composition in $\G L_p$ and its dual $\G L_p^\da$}\label{s:proofcat}
We return to the setup described in Section \ref{s:category}. We will thus work under the conditions in Assumption \ref{a:rep} and $p_0 \in R$ will be a fixed idempotent. Moreover, we will fix an idempotent $p \in R$ such that $p-p_0\in I$.

We are going to prove that $\G L_p$ is indeed a category, thus that the composition is associative and satisfies both left and right unitality conditions. We shall moreover show that the composition $\G L_p^\da$ can be obtained from the composition in $\G L_p$ using the duality isomorphisms $\varphi \colon (\cc,0) \to \G L_p^\da(\la,\mu) \ot \G L_p(\la,\mu)$ and $\psi \colon \G L_p(\la,\mu) \ot \G L_p^\da(\la,\mu) \to (\cc,0)$, for $\la,\mu \in \La$. This result ensures that $\G L_p^\da$ is a category as well.  

Since both of the idempotents $p$ and $p_0$ are fixed in this section we apply the following notation, see Notation \ref{n:invfred} for further information:

\begin{notation}
For any pair of indices $\la,\mu$ in the non-empty index set $\La$ we define
\[
\begin{split}
& F(\la,\mu) := F(\la,\mu)(p,p_0) \colon \pi_\la(p) \C H \op \pi_\mu(p_0) \C H \to \pi_\la(p_0) \C H \op \pi_\mu(p) \C H \q \mbox{and} \\
& F_\da(\la,\mu) := F(\la,\mu)(p_0,p) \colon \pi_\la(p_0) \C H \op \pi_\mu(p) \C H \to \pi_\la(p) \C H \op \pi_\mu(p_0) \C H \, .
\end{split}
\]
\end{notation}

\subsection{Associativity}\label{ss:proofcat}
Throughout this subsection, we fix a quadruple $(\la,\mu,\nu,\tau)$ of indices in $\La$. This quadruple will therefore often be suppressed. We introduce a couple of abbreviations (which can be compared with the notation applied in Section \ref{s:category}):

For $1 \leq i < j \leq 4$ we use the notation $F^{ij}$, $F^{ij}_\da$, meaning that the corresponding Fredholm operator relates to the representations sitting in position $i$ and $j$ in the direct sum $\pi_\la \op \pi_\mu \op \pi_\nu \op \pi_\tau$ and that the base point $p_0 \in \T{Idem}(R)$ is used to stabilise. For example, we have that
\[
\begin{split}
F^{12}_\da & = F_\da(\la,\mu) \op \pi_\nu(p_0) \op \pi_\tau(p_0) \\
& \q \colon \pi_\la(p_0) \C H \op \pi_\mu(p) \C H \op \pi_\nu(p_0) \C H \op \pi_\tau(p_0) \C H \\
& \qq \to \pi_\la(p) \C H \op \pi_\mu(p_0) \C H \op \pi_\nu(p_0) \C H \op \pi_\tau(p_0) \C H \, .
\end{split}
\]
For an extra idempotent $q \in R$ we also apply the notation $\Om^{ij}(q) \in \sL\big( \pi_\la(q) \C H \op \pi_\mu(q) \C H \op \pi_\nu(q) \C H \op \pi_\tau(q) \C H  \big)$, which means that the $\Om(q)$-operator from Notation \ref{n:bigfred} relates to the representations in position $i$ and $j$ and that we have stabilised with the idempotent $q$. 

For $1 \leq i < j < k \leq 4$, we also define
\[
F^{ijk} := F_\da^{ik} F^{jk} F^{ij} \q \mbox{and} \q \Om^{ijk}(q) := \Om^{ik}(q) \Om^{jk}(q) \Om^{ij}(q) \, .
\]

We consider the following idempotent operators on $\C H^{\op 4}$:
\[
\begin{split}
& e_0 := \pi_\la(1 - p_0) \op \pi_\mu(1 - p_0) \op \pi_\nu(1 - p_0) \op \pi_\tau(1 - p_0) \\
& e := \pi_\la(1 - p) \op \pi_\mu(1 - p_0) \op \pi_\nu(1 - p_0) \op \pi_\tau(1 - p_0) \q \T{and} \\ 
& f := \pi_\la(1 - p_0) \op \pi_\mu(1 - p) \op \pi_\nu(1 - p_0) \op \pi_\tau(1 - p_0) \, .
\end{split}
\]
Using the analogue of Lemma \ref{l:pertcomp} for $(4\times 4)$-matrices we obtain the following isomorphism:
\begin{align*}
    \mu_p\colon & | F(\la,\mu) | \ot | F(\mu,\nu) | \ot |F(\nu,\tau)| \ot | F_\da(\la,\tau)| \to^{\G T \circ \G S}
\big| F_\da^{14} F^{34} F^{23} F^{12} \big| \\
& \q \to^{\G S} \big| F_\da^{14} F^{34} F^{23} F^{12} + e \big|
   \to^{\G P} \big|\Om^{14}(p_0) \Om^{34}(p_0) \Om^{23}(p_0) \Om^{12}(p_0) + e_0 \big| = (\Cb,0) \, .
\end{align*}
This isomorphism allows us to define a ternary version of the composition in $\G L_p$.

\begin{definition}\label{def:tern}
The isomorphism of $\zz$-graded complex lines
    \[
        \mathfrak{M}_p \colon \mathfrak{L}_p(\la,\mu) \otimes \mathfrak{L}_p(\mu,\nu)\ot \mathfrak{L}_p(\nu,\tau)\to \mathfrak{L}_p(\la,\tau)
    \]
    is defined as the composition
    \begin{align*}
| F(\la,\mu) | \ot | F(\mu,\nu) | \ot |F(\nu,\tau)|
 &\to^{\T{id}^{\ot 3}\ot\varphi}
| F(\la,\mu) | \ot | F(\mu,\nu) | \ot |F(\nu,\tau)| \ot |F_\da(\la,\tau)| \ot |F(\la,\tau)| \\
        &\to^{\mu_p \ot \T{id}} |F(\la,\tau)| \, .
    \end{align*}
\end{definition}

We are going to use the ternary multiplication operator to prove the associativity of the composition in $\G L_p$. More precisely, we are going to prove Theorem \ref{t:associativity} by showing that each of the two triangles in the following diagram commutes:
\begin{equation}\label{eq:fullasso}
\xymatrix{
&\mathfrak{L}_{p}(\la,\mu)\otimes \mathfrak{L}_{p}(\mu,\nu)\otimes \mathfrak{L}_{p}(\nu,\tau)\ar[dl]_{\mathfrak{M}_{p}\otimes \T{id}}\ar[dr]^{\T{id}\otimes \mathfrak{M}_{p}} \ar[dd]^{\mathfrak{M}_{p}} &\\
\mathfrak{L}_{p}(\la,\nu)\otimes \mathfrak{L}_{p}(\nu,\tau)\ar[dr]_{\mathfrak{M}_{p}}&&\mathfrak{L}_{p}(\la,\mu)\otimes \mathfrak{L}_{p}(\mu,\tau)\ar[dl]^{\mathfrak{M}_{p}}\\
&\mathfrak{L}_{p}(\la,\tau)&
}
\end{equation}
It turns out that the commutativity of the left hand side and the right hand side can not be established by similar methods even though one could believe this to be the case at a first glance. Establishing the commutativity of the diagram on the right is more involved so we start with the diagram on the left:

\begin{prop}\label{p:LHSasso}
The following diagram of isomorphisms of $\zz$-graded complex lines is commutative:
\[
\xymatrix{
\mathfrak{L}_{p}(\la,\mu)\otimes \mathfrak{L}_{p}(\mu,\nu)\otimes \mathfrak{L}_{p}(\nu,\tau)\ar[d]_{\mathfrak{M}_{p}\otimes \T{id}} \ar[dr]^{\mathfrak{M}_{p}} \\
\mathfrak{L}_{p}(\la,\nu)\otimes \mathfrak{L}_{p}(\nu,\tau)\ar[r]_{\mathfrak{M}_{p}}& \mathfrak{L}_{p}(\la,\tau)
}
\]
\end{prop}
\begin{proof}
From the definition of the binary and ternary multiplications, the commutativity of our diagram is equivalent to proving that the trivialisation
\begin{equation}\label{eq:assotrivI}
\begin{split}
& |F(\la,\mu)|\ot |F(\mu,\nu)|\ot |F_\da(\la,\nu)| \ot |F(\la,\nu)| \ot |F(\nu,\tau)| \ot | F_\da(\la,\tau) | \\
& \q \to^{\T{id}^{\ot 2} \ot \varphi^{-1} \ot \T{id}^{\ot 2}} 
|F(\la,\mu)|\ot |F(\mu,\nu)|\ot |F(\nu,\tau)| \ot | F_\da(\la,\tau) |
\to^{\G T \ci \G{S}}
        \big| F_\da^{14} F^{34} F^{23} F^{12}\big| \\
& \q \to^{\G S} \big| F_\da^{14} F^{34}  F^{23}  F^{12} + e \big|
        \to^{\mathfrak{P}}\big|\Om^{14}(p_0) \Om^{34}(p_0) \Om^{23}(p_0) \Om^{12}(p_0) + e_0 \big|= (\Cb,0)
\end{split}
\end{equation}
agrees with the trivialisation
\begin{equation}\label{assoproofII}
\begin{split}
& |F(\la,\mu)|\ot |F(\mu,\nu)|\ot |F_\da(\la,\nu)| \ot |F(\la,\nu)| \ot |F(\nu,\tau)| \ot | F_\da(\la,\tau) | \\
& \q \to^{(\G T \ci \G S ) \ot (\G T \ci \G S) }
|F^{13}_\da  F^{23} F^{12} \big|\ot
\big|F_\da^{14} F^{34} F^{13} \big| 
\to^{\G S \ot \G S} |F^{13}_\da F^{23} F^{12} + e \big|\ot
\big|F_\da^{14} F^{34}  F^{13} + e \big| \\
        & \q \to^{\G{P} \ot \G P }
\big|\Om^{13}(p_0) \Om^{23}(p_0) \Om^{12}(p_0) + e_0 \big|
\ot \big|\Om^{14}(p_0) \Om^{34}(p_0) \Om^{13}(p_0) + e_0 \big|
\cong (\cc ,0) \, .
\end{split}
\end{equation}

    Using that torsion commutes with perturbation and stabilisation (Theorem \ref{t:percom} and Proposition \ref{p:torsta}), the trivialisation in Equation \eqref{assoproofII} becomes equal to
\begin{equation}\label{eq:assoproofIII}
\begin{split}
& |F(\la,\mu)|\ot |F(\mu,\nu)|\ot |F_\da(\la,\nu)|
\ot |F(\la,\nu)| \ot |F(\nu,\tau)| \ot | F_\da(\la,\tau) | \\
& \q \to^{\G T \ci \G S}
\big|F^{14}_\da  F^{34}  F^{13}  F^{13}_\da F^{23}  F^{12}\big| \to^{\G S} 
\big|F^{14}_\da  F^{34}  F^{13}  F^{13}_\da  F^{23}  F^{12} + e \big| \\
& \q \to^{\G P}
\big|\Om^{14}(p_0) \Om^{34}(p_0) \Om^{23}(p_0) \Om^{12}(p_0) + e_0 \big|= (\Cb,0) \, .
\end{split}
\end{equation}

Using the associativity of perturbation (Theorem \ref{t:pervec}) and that perturbation commutes with stabilisation (Proposition \ref{p:persta}) the trivialisation in Equation \eqref{eq:assoproofIII} can be rewritten as
\begin{equation}\label{eq:assotrivII}
\begin{split}
& |F(\la,\mu)|\ot |F(\mu,\nu)|\ot |F_\da(\la,\nu)|
\ot |F(\la,\nu)| \ot |F(\nu,\tau)| \ot | F_\da(\la,\tau) | \\
& \q \to^{\G T \ci \G S}
\big|F^{14}_\da  F^{34}  F^{13} F^{13}_\da  F^{23}  F^{12} \big|
\to^{\G P} \big|F^{14}_\da  F^{34}  F^{23}  F^{12} \big| \\
& \q \to^{\G S}
\big|F^{14}_\da  F^{34}  F^{23}  F^{12} + e \big|
\to^{\G P}
\big|\Om^{14}(p_0) \Om^{34}(p_0) \Om^{23}(p_0) \Om^{12}(p_0) + e_0 \big|= (\Cb,0) \, .
\end{split}
\end{equation}

Since the last half of the trivialisation in Equation \eqref{eq:assotrivI} agrees with the last half of the trivialisation in Equation \eqref{eq:assotrivII} we focus on proving that the following two isomorphisms agree:
\begin{equation}\label{eq:assotwo}
\begin{split}
& |F(\la,\mu)|\ot |F(\mu,\nu)|\ot |F_\da(\la,\nu)|
\ot |F(\la,\nu)| \ot |F(\nu,\tau)| \ot | F_\da(\la,\tau) | \\
& \q \to^{\T{id}^{\ot 2} \ot \varphi^{-1} \ot \T{id}^{\ot 2} }
|F(\la,\mu)|\ot |F(\mu,\nu)|\ot |F(\nu,\tau)| \ot | F_\da(\la,\tau) |
\to^{\G T \ci \G S}
\big|F^{14}_\da  F^{34}  F^{23}  F^{12} \big| \q \T{and} \\
& \\
& |F(\la,\mu)|\ot |F(\mu,\nu)|\ot |F_\da(\la,\nu)|
\ot |F(\la,\nu)| \ot |F(\nu,\tau)| \ot | F_\da(\la,\tau) | \\
& \q \to^{\G T \ci \G S}
\big|F^{14}_\da  F^{34}  F^{13} F^{13}_\da  F^{23}  F^{12} \big|
\to^{\G P} \big|F^{14}_\da  F^{34}  F^{23}  F^{12} \big|  \, .
\end{split}
\end{equation}

Commuting perturbation past torsion (Theorem \ref{t:percom}) and applying the associativity of torsion (Proposition \ref{p:assotors}) yield that the second of the isomorphisms in Equation \eqref{eq:assotwo} is equal to the composition
\[
\begin{split}
& |F(\la,\mu)|\ot |F(\mu,\nu)|\ot |F_\da(\la,\nu)|
\ot |F(\la,\nu)| \ot |F(\nu,\tau)| \ot | F_\da(\la,\tau) | \\
& \q \to^{\G S \ot ( \G T \ci \G S ) \ot \G S }
|F^{12} |\ot |F^{23}|\ot |F^{13} F_\da^{13}| \ot |F^{34}| \ot | F_\da^{14} | \\
& \q \to^{\T{id}^{\ot 2} \ot \G P \ot \T{id}^{\ot 2}}
|F^{12} |\ot |F^{23}|\ot |F^{34}| \ot | F_\da^{14}|
\to^{\G T} \big|F^{14}_\da  F^{34}  F^{23}  F^{12} \big| \, .
\end{split}
\] 
But this isomorphism agrees with the first isomorphism in Equation \eqref{eq:assotwo} and we have proved the proposition.
\end{proof}


\subsubsection{Commutativity of the right hand side of Equation \eqref{eq:fullasso}}
We start with a small lemma whose proof is almost identical to the proof of Lemma \ref{l:pertcomp} (and it will therefore not be repeated here). The notation which we apply was introduced in the beginning of Subsection \ref{ss:proofcat}.

\begin{lemma}\label{l:pertcomp4}
The two bounded operators 
\[
F^{12}_\da  F^{234}  F^{12} + e \q \mbox{and} \q
\Om^{12}(p_0)  \Om^{234}(p_0)  \Om^{12}(p_0) + e_0 
\]
agree up to an element in the trace ideal $\sL^1\big(\pi_\la(1)\C H \op \pi_\mu(1) \C H \op \pi_\nu(1) \C H \op \pi_\tau(1) \C H \big)$.
\end{lemma}

The next proposition is in a way the analogue of Proposition \ref{p:LHSasso} and it shows why proving the commutativity of the right hand side of Equation \eqref{eq:fullasso} is more complicated than proving the commutativity of the left hand side. We notice that Lemma \ref{l:pertcomp} and Lemma \ref{l:pertcomp4} imply that the perturbation isomorphisms appearing in the statement make sense.

\begin{prop}\label{p:RHSassoI}
The right hand side of the diagram in Equation \eqref{eq:fullasso} commutes if and only if the following two trivialisations coincide:
\begin{equation}\label{eq:assoR}
\begin{split}
& \big| F^{12} \big| \otimes \big| F^{234} \big| \ot \big| F^{12}_\da \big| 
\to^{\G T} \big| F^{12}_\da  F^{234}  F^{12} \big| \to^{\G S} 
\big| F^{12}_\da  F^{234}  F^{12} + e \big| \\ 
& \q \to^{\G P} 
\big| \Om^{12}(p_0)  \Om^{234}(p_0)  \Om^{12}(p_0) + e_0 \big| = (\cc,0) \q \mbox{and} \\
& \big| F^{12} \big| \otimes \big| F^{234} \big| \ot \big| F^{12}_\da \big| 
\to^{\T{id} \ot \G S \ot \T{id}} 
\big| F^{12} \big| \otimes \big| F^{234} + f \big| \ot \big| F^{12}_\da \big|
\to^{\T{id} \ot \G P \ot \T{id}}
\big| F^{12} \big| \otimes \big| \Om^{234}(p_0) + e_0 \big| \ot \big| F^{12}_\da \big| \\
& \q \to^{=} \big| F^{12} \big| \ot \big| F^{12}_\da \big| 
\to^{\psi \ci \G S^{-1}} (\cc,0) \, .
\end{split}
\end{equation}
\end{prop}
\begin{proof}
First of all, an argument which is similar to the proof of Proposition \ref{p:LHSasso} shows that the ternary trivialisation
\begin{align*}
    \mu_p\colon & | F(\la,\mu) | \ot | F(\mu,\nu) | \ot |F(\nu,\tau)| \ot | F_\da(\la,\tau)| \to^{\G T \ci \G S}
\big| F_\da^{14}  F^{34} F^{23} F^{12} \big|\\
    & \q \to^{\G S} \big| F_\da^{14}  F^{34} F^{23} F^{12} + e \big| 
\to^{\G P} \big|\Om^{14}(p_0)\cd \Om^{34}(p_0)\cd \Om^{23}(p_0)\cd \Om^{12}(p_0) + e_0 \big| = (\Cb,0)
\end{align*}
agrees with the following trivialisation:
\begin{equation}\label{eq:middle}
\begin{split}
& | F(\la,\mu) | \ot | F(\mu,\nu) | \ot |F(\nu,\tau)| \ot | F_\da(\la,\tau)| \\ 
& \q \to^{(\T{id}^{\ot 4} \ot \varphi \ot \T{id})(\T{id}^{\ot 3} \ot \varphi \ot \T{id})} | F(\la,\mu) | \ot | F(\mu,\nu) | \ot |F(\nu,\tau)| \\
& \qqq \qqq \qq \ot |F_\da(\mu,\tau)| \ot 
|F_\da(\la,\mu)| \ot |F(\la,\mu)| \ot 
|F(\mu,\tau)| \ot | F_\da(\la,\tau)| \\
& \q \to^{(\G T \ci \G S) \ot \mu_p}
\big| F^{12}_\da \cd F^{234} \cd F^{12} \big| \to^{\G S}\big| F^{12}_\da \cd F^{234} \cd F^{12} + e \big| \\
& \q \to^{\G P}
\big| \Om^{12}(p_0) \cd \Om^{234}(p_0) \cd \Om^{12}(p_0) + e_0 \big| = (\Cb, 0) \, .
\end{split}
\end{equation}
On the other hand, the trivialisation corresponding to the composition
\[
\mathfrak{L}_{p}(\la,\mu)\otimes \mathfrak{L}_{p}(\mu,\nu)\otimes \mathfrak{L}_{p}(\nu,\tau) \to^{\T{id}\otimes \mathfrak{M}_{p}} 
\mathfrak{L}_{p}(\la,\mu)\otimes \mathfrak{L}_{p}(\mu,\tau)\to^{\mathfrak{M}_{p}} \mathfrak{L}_{p}(\la,\tau)
\]
appearing in Equation \eqref{eq:fullasso} is given by
\begin{equation}\label{eq:middleII}
\begin{split}
& | F(\la,\mu) | \ot | F(\mu,\nu) | \ot |F(\nu,\tau)| \ot | F_\da(\la,\tau)| \\
& \q \to^{\T{id}^{\ot 3} \ot \varphi \ot \T{id}} | F(\la,\mu) | \ot | F(\mu,\nu) | \ot |F(\nu,\tau)| \ot |F_\da(\mu,\tau)| \ot |F(\mu,\tau)| \ot | F_\da(\la,\tau)| \\
& \q \to^{\T{id} \ot \mu_p \ot \T{id}^{\ot 2}} |F(\la,\mu)| \ot |F(\mu,\tau)| \ot | F_\da(\la,\tau)|
\to^{\mu_p} (\cc,0) \, .
\end{split}
\end{equation}
Now, using Lemma \ref{l:inverse} the trivialisation in Equation \eqref{eq:middleII} agrees with the trivialisation
\begin{align*}
& | F(\la,\mu) | \ot | F(\mu,\nu) | \ot |F(\nu,\tau)| \ot | F_\da(\la,\tau)| \\
& \q \to^{(\T{id}^{\ot 4} \ot \varphi \ot \T{id})(\T{id}^{\ot 3} \ot \varphi \ot \T{id})} | F(\la,\mu) | \ot | F(\mu,\nu) | \ot |F(\nu,\tau)| \\
& \qqq \qqq \qq \ot |F_\da(\mu,\tau)| \ot 
|F_\da(\la,\mu)| \ot |F(\la,\mu)| \ot 
|F(\mu,\tau)| \ot | F_\da(\la,\tau)| \\
& \q \to^{\T{id} \ot \mu_p \ot \T{id} \ot \mu_p} |F(\la,\mu)| \ot |F_\da(\la,\mu)| \to^{\psi} (\cc,0) \, .
\end{align*}
Comparing with Equation \eqref{eq:middle} we see that the right hand side of Equation \eqref{eq:fullasso} commutes if and only if the trivialisation
\begin{align*}
& | F(\la,\mu) | \ot | F(\mu,\nu) | \ot |F(\nu,\tau)| \ot |F_\da(\mu,\tau)| \ot |F_\da(\la,\mu)|
\to^{\G T \ci \G S}
\big| F^{12}_\da \cd F^{234} \cd F^{12} \big| \\ 
& \q \to^{\G S} \big| F^{12}_\da \cd F^{234} \cd F^{12} + e \big| \to^{\G P}
\big| \Om^{12}(p_0) \cd \Om^{234}(p_0) \cd \Om^{12}(p_0) + e_0 \big| = (\Cb, 0)
\end{align*}
agrees with the trivialisation
\[
| F(\la,\mu) | \ot | F(\mu,\nu) | \ot |F(\nu,\tau)| \ot |F_\da(\mu,\tau)| \ot |F_\da(\la,\mu)|
\to^{\T{id} \ot \mu_p \ot \T{id}} |F(\la,\mu)| \ot |F_\da(\la,\mu)| \to^{\psi} (\cc,0) \, .
\]
The proposition is now proved since it follows by an application of the associativity of torsion (Proposition \ref{p:assotors}) and the definition of the binary trivialisation $\mu_p$ that this latter identity of trivialisations is equivalent to the identity of the two trivialisations in Equation \eqref{eq:assoR}.
\end{proof}

We are now going to prove that the two trivialisations appearing in Proposition \ref{p:RHSassoI} do in fact agree and this will be accomplished in a few steps. We start by establishing some extra notation together with a crucial lemma.

Define the Fredholm operator
\[
\begin{split}
G(\la,\mu) & := \big( \pi_\la(1 - p_0) \op \pi_\mu(1 - p) \big) \Om(\la,\mu)(1) \big( \pi_\la(1 - p) \op \pi_\mu(1 - p_0) \big) \\
& \q \colon \pi_\la(1 - p) \C H \op \pi_\mu(1 - p_0) \C H \to \pi_\la(1 - p_0) \C H \op \pi_\mu(1 - p) \C H \, .
\end{split}
\]

\begin{lemma}\label{l:addcompl}
The Fredholm operator
\[
F(\la,\mu) + G(\la,\mu) \colon \pi_\la(1) \C H \op \pi_\mu(1) \C H \to \pi_\la(1) \C H \op \pi_\mu(1) \C H
\]
is a trace class perturbation of the invertible operator $\Om(\la,\mu)(1) \in \sL( \pi_\la(1) \C H \op \pi_\mu(1) \C H)$.
\end{lemma}
\begin{proof}
This is a consequence of Lemma \ref{l:commut}. Indeed,
\[
\begin{split}
F(\la,\mu) + G(\la,\mu) 
& = (\pi_\la(p_0) \op \pi_\mu(p) ) \Om(\la,\mu)(1) (\pi_\la(p) \op \pi_\mu(p_0)) \\
& \q + (\pi_\la(1 - p_0) \op \pi_\mu(1 - p) ) \Om(\la,\mu)(1) (\pi_\la(1 - p) \op \pi_\mu(1- p_0)) \\
& \sim_1 \Om(\la,\mu)(1) (\pi_\la(p) \op \pi_\mu(p_0) + \pi_\la(1 - p) \op \pi_\mu(1- p_0)) 
= \Om(\la,\mu)(1) \, . \qedhere
\end{split}
\]
\end{proof}

Let us choose $n,m \in \nn \cup \{0\}$ such that
\[
\T{Index}( F(\la,\mu) ) = n - m  \, .
\]
It follows from the above lemma that
\[
\T{Index}( G(\la,\mu) ) = - \T{Index}( F(\la,\mu) ) 
\]
and we may thus find a bounded finite rank operator
\[
L \colon \pi_\la(1 - p) \C H \op \pi_\mu(1 - p_0) \C H \op \B C^n \to \pi_\la(1 - p_0) \C H \op \pi_\mu(1 - p) \C H \op \B C^m
\]
such that the bounded operator
\[
\begin{split}
\Ga := G(\la,\mu) + L & \colon \pi_\la(1 - p) \C H \op \pi_\mu(1 - p_0) \C H \op \B C^n  \\
& \q \to \pi_\la(1 - p_0) \C H \op \pi_\mu(1 - p) \C H \op \B C^m
\end{split}
\]
is invertible. We introduce the Hilbert spaces
\[
\begin{split}
& \C H_1 := \pi_\la(1) \C H \op \pi_\mu(1) \C H \op \B C^n \op \pi_\nu(1) \C H \op \pi_\tau(1) \C H \q \T{and} \\
& \C H_2 := \pi_\la(1) \C H \op \pi_\mu(1) \C H \op \B C^m \op \pi_\nu(1) \C H \op \pi_\tau(1) \C H 
\end{split}
\]
together with the Fredholm operators
\[
\begin{split}
& \wih{F}^{12} := (F(\la,\mu) + \Ga) \op \pi_\nu(1) \op \pi_\tau(1) \colon \C H_1 \to \C H_2 \q \T{and} \\
& \wih{F}^{12}_\da := (F_\da(\la,\mu) + \Ga^{-1}) \op \pi_\nu(1) \op \pi_\tau(1) \colon \C H_2 \to \C H_1 \, .
\end{split}
\]
We record from Definition \ref{def:stabilisation} that we have the stabilisation isomorphisms
\[
\G S : |F^{12}| \to |\wih{F}^{12}| \q \T{and} \q \G S : |F^{12}_\da| \to |\wih{F}^{12}_\da| \, .
\]
We also introduce a couple of idempotent operators:
\[
\begin{split}
\wih{e} & := \pi_\la(1 - p) \op \pi_\mu(1 - p_0) \op 1_n \op \pi_\nu(1 - p_0) \op \pi_\tau(1 - p_0) \colon \C H_1 \to \C H_1 \\
\wih{e}_0 & := \pi_\la(1 - p_0) \op \pi_\mu(1 - p_0) \op 1_n \op \pi_\nu(1 - p_0) \op \pi_\tau(1 - p_0) \colon \C H_1 \to \C H_1 \\
\wih{f} & := \pi_\la(1 - p_0) \op \pi_\mu(1 - p) \op 1_m \op \pi_\nu(1 - p_0) \op \pi_\tau(1 - p_0) \colon \C H_2 \to \C H_2 \\
\wih{f}_0 & := \pi_\la(1 - p_0) \op \pi_\mu(1 - p_0) \op 1_m \op \pi_\nu(1 - p_0) \op \pi_\tau(1 - p_0) \colon \C H_2 \to \C H_2 \, .
\end{split}
\]

A straightforward computation shows that
\[
\wih{F}_\da^{12} ( F^{234} + \wih{f} ) \wih{F}^{12} = F_\da^{12} F^{234} F^{12} + \wih{e} \colon \C H_1 \to \C H_1
\]
and hence we obtain from Lemma \ref{l:pertcomp4} that $\wih{F}_\da^{12} ( F^{234} + \wih{f} ) \wih{F}^{12}$ is a trace class perturbation of
\[
\Om^{12}(p_0) \cd \Om^{234}(p_0) \cd \Om^{12}(p_0) + \wih{e}_0 \colon \C H_1 \to \C H_1  \, .
\]

%

The next lemma is now a consequence of Proposition \ref{p:torsta} and Proposition \ref{p:persta}.

\begin{lemma}\label{l:assoRI}
The following two trivialisations agree:
\[
\begin{split}
& \big| F^{12} \big| \otimes \big| F^{234} \big| \ot \big| F^{12}_\da \big| 
\to^{\G S \ot \G S \ot \G S} \big| \wih{F}^{12} \big| \otimes \big| F^{234} + \wih{f} \big| \ot \big| \wih{F}^{12}_\da \big| \\
& \q \to^{\G T} \big| \wih{F}_\da^{12} \cd ( F^{234} + \wih{f} ) \cd \wih{F}^{12} \big|
\to^{\G P} \big| \Om^{12}(p_0) \Om^{234}(p_0) \Om^{12}(p_0) + \wih{e}_0 \big| = (\cc,0) \q \mbox{and} \\
& \big| F^{12} \big| \otimes \big| F^{234} \big| \ot \big| F^{12}_\da \big| 
\to^{\G T} \big| F^{12}_\da F^{234} F^{12} \big| \\
& \q \to^{\G S} \big| F^{12}_\da  F^{234} F^{12} + e \big|
\to^{\G P} \big| \Om^{12}(p_0) \Om^{234}(p_0)  \Om^{12}(p_0) + e_0 \big| = (\cc,0) \, .
\end{split}
\]
\end{lemma}

To continue, we introduce the trivialisation
\begin{equation}\label{eq:dualhat}
\wih{\psi} \colon |\wih{F}^{12}| \ot |\wih{F}_\da^{12}| \to^{\G T} |\wih{F}^{12}_\da \cd \wih{F}^{12}| \to^{\G P} |\T{id}_{\C H_1}| = (\B C,0) .
\end{equation}
together with the invertible operator
\[
\De := \Om(\la,\mu)(p_0) + \Om(\la,\mu)(1 - p_0) \colon \pi_\la(1) \C H \op \pi_\mu(1) \C H \to \pi_\la(1) \C H \op \pi_\mu(1) \C H
\]
We remark that Lemma \ref{l:addcompl} and Lemma \ref{l:reduction} show that $F(\la,\mu) + \Ga = F(\la,\mu) + G(\la,\mu) + L$ is a trace class perturbation of $\De + 0_{m,n} \colon \pi_\la(1) \C H \op \pi_\mu(1) \C H \op \cc^n \to \pi_\la(1) \C H \op \pi_\mu(1) \C H \op \cc^m$ and that $F_\da(\la,\mu) + \Ga^{-1}$ is a trace class perturbation of $\De + 0_{n,m} \colon \pi_\la(1) \C H \op \pi_\mu(1) \C H \op \cc^m \to \pi_\la(1) \C H \op \pi_\mu(1) \C H \op \cc^n$. It is also convenient to define the idempotent operators
\[
f_1 := \pi_\la(1) \op \pi_\mu(1) \q \T{and} \q f_2 := \pi_\nu(1) \op \pi_\tau(1) .
\]

\begin{lemma}\label{l:assohat}
The following two trivialisations coincide:
\begin{equation}\label{eq:assohat}
\begin{split}
& \big| \wih{F}^{12}\big| \ot \big| F^{234} + \wih{f} \big| \ot \big| \wih{F}^{12}_\da \big| 
\to^{\G T} \big| \wih{F}^{12}_\da \cd (F^{234} + \wih{f}) \cd \wih{F}^{12} \big| \\
& \q \to^{\G P} \big| \Om^{12}(p_0)  \Om^{234}(p_0) \Om^{12}(p_0) + \wih{e}_0\big| = (\cc,0)  \q \mbox{and} \\
& \big| \wih{F}^{12}\big| \ot \big| F^{234} + \wih{f} \big| \ot \big| \wih{F}^{12}_\da \big|
\to^{\T{id} \ot \G P \ot \T{id}}
\big| \wih{F}^{12} \big| \otimes \big| \Om^{234}(p_0) + \wih{f}_0\big| \ot \big| \wih{F}^{12}_\da \big| \\
& \q \to^{=}
\big| \wih{F}^{12} \big| \otimes \big| \wih{F}^{12}_\da \big|
\to^{\wih{\psi}}
(\cc,0) \, .
\end{split}
\end{equation}
\end{lemma}
\begin{proof}
Using that perturbation commutes with torsion (Theorem \ref{t:percom}) together with the transitivity of perturbation (Theorem \ref{t:pervec}) the second of the two trivialisations in Equation \eqref{eq:assohat} becomes equal to
\[
\begin{split}
& \big| \wih{F}^{12}\big| \ot \big| F^{234} + \wih{f} \big| \ot \big| \wih{F}^{12}_\da \big|
\to^{\G P \ot \G P \ot \G P}
\big| \De + 0_{m,n} + f_2 \big| \otimes \big| \Om^{234}(p_0) + \wih{f}_0\big| \ot \big| \De + 0_{n,m} + f_2 \big| \\
& \q \to^{=}
\big| \De + 0_{m,n} + f_2 \big| \otimes \big| \De + 0_{n,m} + f_2 \big|
\to^{\G T}
\big| f_1 + 0_n + f_2 \big|
\to^{\G P} \big| \T{id}_{\C H_1} \big| = (\cc,0) \, .
\end{split}
\]

Next, a direct computation reveals that we have the identities of operators on $\C H_1$:
\[
(\De + 0_{n,m} + f_2)(\Om^{234}(p_0) + \wih{f}_0 )(\De + 0_{m,n} + f_2) = \Om^{12}(p_0) \Om^{234}(p_0) \Om^{12}(p_0) + e_0 + 0_n \, .
\]
Hence, using again that perturbation commutes with torsion and the transitivity of perturbation, we can replace the first of the two trivialisations in Equation \eqref{eq:assohat} with
\[
\begin{split}
& \big| \wih{F}^{12}\big| \ot \big| F^{234} + \wih{f} \big| \ot \big| \wih{F}^{12}_\da \big| 
\to^{\G P \ot \G P \ot \G P}
\big| \De + 0_{m,n} + f_2 \big| \ot \big| \Om^{234}(p_0) + \wih{f}_0 \big| \ot \big| \De + 0_{n,m} + f_2 \big| \\
& \q \to^{\G T} 
\big| \Om^{12}(p_0) \Om^{234}(p_0) \Om^{12}(p_0) + e_0 + 0_n \big| 
\to^{\G P} \big| \Om^{12}(p_0) \Om^{234}(p_0) \Om^{12}(p_0) + \wih{e}_0\big| = (\cc,0) \, .
\end{split}
\] 

We thus see that proving the identity of trivialisations claimed in the lemma amounts to verifying that the following two trivialisations agree:
\begin{equation}\label{eq:st2II}
\begin{split}
& \big| \De + 0_{m,n} + f_2 \big| \ot \big| \Om^{234}(p_0) + \wih{f}_0 \big| \ot \big| \De + 0_{n,m} + f_2   \big|
\to^{\G T} 
\big| \Om^{12}(p_0) \Om^{234}(p_0) \Om^{12}(p_0) + e_0 + 0_n \big| \\
& \q \to^{\G P} \big| \Om^{12}(p_0) \Om^{234}(p_0) \Om^{12}(p_0) + \wih{e}_0\big| = (\cc,0) \q \T{and} \\ 
& \big| \De + 0_{m,n} + f_2 \big| \otimes \big| \Om^{234}(p_0) + \wih{f}_0\big| \ot \big| \De + 0_{n,m} + f_2 \big|
\to^{=}
\big| \De + 0_{m,n} + f_2 \big| \otimes \big| \De + 0_{n,m} + f_2 \big| \\
& \q \to^{\G T}
\big| f_1 + 0_n + f_2 \big|
\to^{\G P} \big| \T{id}_{\C H_1} \big| = (\cc,0)  \, .
\end{split}
\end{equation}
That these two latter trivialisations coincide can now be checked by hand using the definition of the torsion isomorphism and the perturbation isomorphism. In fact, one may explicitly compute the kernels and cokernels of all the involved operators and the two trivialisations then both take the form 
\[
|\cc^n| \ot |\cc^m|^* \ot |\cc^m| \ot |\cc^n|^* \to^{\G T} |\cc^n| \ot |\cc^n|^* \to^{\G P} (\cc,0) \, ,
\]
Moreover, it can be seen from Definition \ref{d:torfre}, Definition \ref{def:torsion} and Example \ref{ex:indzero} that both of the trivialisations are given explicitly by the map $\la \cd (s \ot t^* \ot t \ot s^*) \mapsto \la$, for all non-trivial vectors $s \in \La^{\T{top}}(\cc^n)$ and $t \in \La^{\T{top}}(\cc^m)$ and all $\la \in \cc$. Remark however that the torsion and perturbation isomorphisms appearing are not a priori the same since they do in fact depend on different operators as reflected in Equation \eqref{eq:st2II}. This ends the proof of the lemma.
\end{proof}

We observe that the identity $\wih{F}_\da^{12} \cd \wih{F}^{12} = F_\da^{12} \cd F^{12} + \wih{e}$ of operators on $\C H_1$ holds. The following lemma is then a consequence of Proposition \ref{p:torsta} and Proposition \ref{p:persta} where we recall the definitions of the trivialisations $\wih{\psi}$ and $\psi$ from Equation \eqref{eq:dualhat} and Equation \eqref{eq:dualII}. 


\begin{lemma}\label{l:assoRIII}
The diagram here below is commutative:
\[
\xymatrix{
\big| \wih{F}^{12} \big| \ot \big| \wih{F}^{12}_\da \big| \ar[rr]^{\G S^{-1} \ot \G S^{-1}} \ar[dr]_{\wih{\psi}} 
&& \big| F^{12} \big| \otimes \big| F^{12}_\da \big| \ar[dl]^{\psi \ci (\G S^{-1} \ot \G S^{-1})}\\
& (\B C,0) &
}
\]
\end{lemma}

Combining the results obtained so far we may give a full proof of the associativity of the composition in $\G L_p$:

\begin{proof}[Proof of Theorem \ref{t:associativity}]
We need to show that the diagram in Equation \eqref{eq:fullasso} is commutative. The left hand side of this diagram commutes by Proposition \ref{p:LHSasso} so we only need to establish that the right hand side also commutes. By Proposition \ref{p:RHSassoI} this amounts to verifying that the following two trivialisations agree: 
\begin{equation}\label{eq:assoRIII}
\begin{split}
& \big| F^{12} \big| \otimes \big| F^{234} \big| \ot \big| F^{12}_\da \big| 
\to^{\G T} \big| F^{12}_\da  F^{234}  F^{12} \big| \to^{\G S} 
\big| F^{12}_\da  F^{234}  F^{12} + e \big| \\ 
& \q \to^{\G P} 
\big| \Om^{12}(p_0)  \Om^{234}(p_0)  \Om^{12}(p_0) + e_0 \big| = (\cc,0) \q \mbox{and} \\
& \big| F^{12} \big| \otimes \big| F^{234} \big| \ot \big| F^{12}_\da \big| 
\to^{\T{id} \ot \G S \ot \T{id}} 
\big| F^{12} \big| \otimes \big| F^{234} + f \big| \ot \big| F^{12}_\da \big|
\to^{\T{id} \ot \G P \ot \T{id}}
\big| F^{12} \big| \otimes \big| \Om^{234}(p_0) + e_0 \big| \ot \big| F^{12}_\da \big| \\
& \q \to^{=} \big| F^{12} \big| \ot \big| F^{12}_\da \big| 
\to^{\psi \ci \G S^{-1}} (\cc,0) \, .
\end{split}
\end{equation}
However, using Lemma \ref{l:assoRI}, Lemma \ref{l:assohat} and Lemma \ref{l:assoRIII} we see that the first of the two trivialisations in Equation \eqref{eq:assoRIII} coincides with the trivialisation
\[
\begin{split}
& \big| F^{12} \big| \otimes \big| F^{234} \big| \ot \big| F^{12}_\da \big| 
\to^{\G S} \big| F^{12} \big| \otimes \big| F^{234} + \wih{f} \big| \ot \big| F^{12}_\da \big| 
\to^{\G P} \big| F^{12} \big| \otimes \big| \Om^{234}(p_0) + \wih{e}_0 \big| \ot \big| F^{12}_\da \big| \\
& \q \to^{=} \big| F^{12} \big| \otimes \big| F^{12}_\da \big| \to^{\psi \ci \G S^{-1}} (\cc,0) \, .
\end{split}
\]
The fact that the two trivialisations in Equation \eqref{eq:assoRIII} agree is now a consequence of Proposition \ref{p:persta}. This proves the theorem.
\end{proof}

\subsection{Unitality}
We continue by proving the unitality condition in $\G L_p$. Recall that for $\la \in \La$, the unit $\T{id}_\la \colon \la \to \la$ in $\G L_p(\la,\la)$ is given by the unit in $\cc$ under the identification
\[
(\cc,0) = \big| \ma{cc}{0 & \pi_\la(p_0) \\ \pi_\la(p) & 0} \big| = | F(\la,\la)| = \G L_p(\la,\la) \, .
\]

\begin{proof}[Proof of Theorem \ref{thm:unitality}]
Since $\G L_p(\la,\mu)$ is a $\zz$-graded one-dimensional vector space over $\cc$ and since both of the degree $0$ maps $x \mapsto \G M_p(x \ot \T{id}_\mu)$ and $x \mapsto \G M_p(\T{id}_\la \ot x)$ are automorphisms of $\G L_p(\la,\mu)$, it suffices to show that they are both idempotents. However, by the associativity of the composition in $\G L_p$ (Theorem \ref{t:associativity}) this amounts to showing that $\G M_p( \T{id}_\la \ot \T{id}_\la) = \T{id}_\la$ for all $\la \in \La$. Thus, let $\la \in \La$ be given. It suffices to show that the two perturbation isomorphisms
\[
\begin{split}
& \big| F(\la,\la) \cd F_\da(\la,\la) \big| \to^{\G P} \big| \pi_\la(p_0) \op \pi_\la(p) \big| \q \T{and} \\
& \big| F_\da^{13}(\la,\la,\la) F^{23}(\la,\la,\la) F^{12}(\la,\la,\la) + \pi_\la(1 - p) \op \pi_\la(1 - p_0) \op \pi_\la(1 - p_0) \big| \\
& \q \to^{\G P}  \big| \Om^{13}(p_0) \Om^{23}(p_0) \Om^{12}(p_0) + \pi_\la(1 - p_0) \op \pi_\la(1 - p_0) \op \pi_\la(1 - p_0) \big|
\end{split}
\]
both agree with the identity automorphism when identifying all the graded lines with $(\cc,0)$. However, a straightforward computation shows that
\[
\begin{split}
& F(\la,\la) \cd F_\da(\la,\la) = \pi_\la(p_0) \op \pi_\la(p) \q \T{and} \\
& F_\da^{13}(\la,\la,\la) F^{23}(\la,\la,\la) F^{12}(\la,\la,\la) + \pi_\la(1 - p) \op \pi_\la(1 - p_0) \op \pi_\la(1 - p_0) \\
& \q = \Om^{13}(p_0) \Om^{23}(p_0) \Om^{12}(p_0) + \pi_\la(1 - p_0) \op \pi_\la(1 - p_0) \op \pi_\la(1 - p_0) = \pi_\la(1) \op \pi_\la(1) \op \pi_\la(1) \, .
\end{split}
\]
The two relevant perturbation isomorphisms are therefore indeed both equal to the identity automorphism and the result of the theorem follows.
\end{proof}

\subsection{Duality}
We end this section by giving a proof of Proposition \ref{p:dualcomp} regarding the duality relation between the compositions in $\G L_p$ and $\G L_p^\da$. We are working with respect to a fixed triple of indices $(\la,\mu,\nu)$ from the index set $\La$ and this triple will be suppressed from the notation throughout this subsection. More precisely, we apply the notation introduced in the beginning of Subsection \ref{ss:comp}.

\begin{proof}[Proof of Proposition \ref{p:dualcomp}]
Using Lemma \ref{l:inverse}, we notice that it suffices to verify that the trivialisation
\[
\big| F(\la,\mu) \big| \ot \big| F(\mu,\nu) \big| \ot \big| F_\da(\la,\nu) \big| 
\ot \big| F(\la,\nu) \big| \ot \big| F_\da(\mu,\nu) \big| \ot \big| F_\da(\la,\mu) \big|
\to^{\mu_p \ot \mu_p^\da} (\cc,0)
\]
agrees with the trivialisation
\begin{equation}\label{eq:dualityII}
\begin{split}
& \big| F(\la,\mu) \big| \ot \big| F(\mu,\nu) \big| \ot \big| F_\da(\la,\nu) \big| 
\ot \big| F(\la,\nu) \big| \ot \big| F_\da(\mu,\nu) \big| \ot \big| F_\da(\la,\mu) \big|  \\
& \q \to^{\T{id}^{\ot 2} \ot \varphi^{-1} \ot \T{id}^{\ot 2}}
\big| F(\la,\mu) \big| \ot \big| F(\mu,\nu) \big| \ot \big| F_\da(\mu,\nu) \big| \ot \big| F_\da(\la,\mu) \big| \\
& \q \to^{\T{id} \ot \psi \ot \T{id}}
\big| F(\la,\mu) \big| \ot \big| F_\da(\la,\mu) \big| \to^{\psi} (\cc,0) \, .
\end{split}
\end{equation}
Using that torsion and perturbation commute with stabilisation, see Proposition \ref{p:torsta} and Proposition \ref{p:persta}, we may alternatively describe the isomorphism of graded lines $\varphi^{-1} \colon \big| F_\da(\la,\nu)\big| \ot \big| F(\la,\nu) \big| \to (\cc,0)$ as the composition
\[
\big| F_\da(\la,\nu)\big| \ot \big| F(\la,\nu) \big|
\to^{\G S} \big| F_\da^{13}\big| \ot \big| F^{13} \big| \to^{\G T} \big| F^{13} \cd F_\da^{13} \big|
\to^{\G P} \big| \pi_\la(p_0) \op \pi_\mu(p_0) \op \pi_\nu(p) \big| = (\cc,0) \, .
\]
Since similar results apply to $\psi \colon \big| F(\mu,\nu)\big| \ot \big| F_\da(\mu,\nu) \big| \to (\cc,0)$ and $\psi \colon \big| F(\la,\mu)\big| \ot \big| F_\da(\la,\mu) \big| \to (\cc,0)$ we may rewrite the trivialisation in Equation \eqref{eq:dualityII} as the composition
\begin{equation}\label{eq:dualityIV}
\begin{split}
& \big| F(\la,\mu) \big| \ot \big| F(\mu,\nu) \big| \ot \big| F_\da(\la,\nu) \big| 
\ot \big| F(\la,\nu) \big| \ot \big| F_\da(\mu,\nu) \big| \ot \big| F_\da(\la,\mu) \big|  \\
& \q \to^{\G S} \big| F^{12} \big| \ot \big| F^{23} \big| \ot \big| F_\da^{13} \big| 
\ot \big| F^{13} \big| \ot \big| F_\da^{23} \big| \ot \big| F_\da^{12} \big|  \\
& \q \to^{\T{id}^{\ot 2} \ot (\G P \ci \G T) \ot \T{id}^{\ot 2}}
\big| F^{12} \big| \ot \big| F^{23} \big| \ot \big| F_\da^{23} \big| \ot \big| F_\da^{12} \big| \\
& \q \to^{\T{id} \ot (\G P \ci \G T) \ot \T{id}}
\big| F^{12} \big| \ot \big| F_\da^{12} \big| \to^{\G P \ci \G T} (\cc,0) \, .
\end{split}
\end{equation}
Repeated use of the transitivity of perturbation and the fact that torsion commutes with perturbation (Theorem \ref{t:pervec} and Theorem \ref{t:percom}) shows that the trivialisation in Equation \eqref{eq:dualityIV} agrees with the composition
\[
\begin{split}
& \big| F(\la,\mu) \big| \ot \big| F(\mu,\nu) \big| \ot \big| F_\da(\la,\nu) \big| 
\ot \big| F(\la,\nu) \big| \ot \big| F_\da(\mu,\nu) \big| \ot \big| F_\da(\la,\mu) \big|  \\
& \q \to^{\G S} \big| F^{12} \big| \ot \big| F^{23} \big| \ot \big| F_\da^{13} \big| 
\ot \big| F^{13} \big| \ot \big| F_\da^{23} \big| \ot \big| F_\da^{12} \big|  \\
& \q \to^{\G T}
\big| F_\da^{12} F_\da^{23} F^{13} F_\da^{13} F^{23} F^{12} \big| 
\to^{\G P} \big| \pi_\la(p) \op \pi_\mu(p_0) \op \pi_\nu(p_0) \big| = (\cc,0) \, .
\end{split}
\]
Let us define the idempotents $e := \pi_\la(1 - p) \op \pi_\mu(1 - p_0) \op \pi_\nu(1 - p_0)$ and $e_0 := \pi_\la(1 - p_0) \op \pi_\mu(1 - p_0) \op \pi_\nu(1 - p_0)$ together with the Fredholm operators $F^{123} := F_\da^{13} F^{23} F^{12}$, $F^{123}_\da := F_\da^{12} F^{23}_\da F^{13}$ and $\Om^{123}(p_0) := \Om^{13}(p_0) \Om^{23}(p_0) \Om^{12}(p_0)$. Using the associativity of torsion (Proposition \ref{p:assotors}), the result of the present proposition is then a consequence of the commutative diagram here below:
\[
\xymatrix{
\big| F^{123} \big| \ot \big| F_\da^{123} \big| \ar[r]^{\G T} \ar[d]_{\G S \ot \G S} & 
\big| F^{123}_\da F^{123} \big| \ar[d]^{\G S} \ar[dr]_{\G P} & \\  
\big| F_\da^{123} + e \big| \ot \big| F^{123} + e \big| \ar[r]^{\G T} \ar[d]_{\G P \ot \G P} & \big| F^{123} F_\da^{123} + e \big| \ar[d]_{\G P} & \big| \pi_\la(p) \op \pi_\mu(p_0) \op \pi_\nu(p_0) \big| \ar[dl]^{\G S} \\
\big| \big(\Om^{123}(p_0) \big)^{-1} + e_0 \big| \ot \big| \Om^{123}(p_0) + e_0 \big| \ar[r]^>>>>>{\G T} & \big| \pi_\la(1) \op \pi_\mu(1) \op \pi_\nu(1)\big| & 
}
\]
which in turn follows from Theorem \ref{t:percom}, Proposition \ref{p:torsta} and Proposition \ref{p:persta}.
\end{proof}

\section{Proofs of properties of the change of base point}\label{s:propchange}
In this section we give full proofs of the main properties of the change-of-base point isomorphism which we introduced in Section \ref{s:change}. Thus, we shall see that the change-of-base-point isomorphism is an isomorphism of coproduct categories. The main properties needed for establishing this result are the cofunctoriality and the functoriality of the change-of-base-point isomorphism. These properties are stated as Proposition \ref{p:cofunc} and Proposition \ref{p:func} in Section \ref{s:change}. The proof of cofunctoriality is rather straightforward and we take care of it in Subsection \ref{ss:cofunc}, but the proof of functoriality is involved and occupies Subsection \ref{ss:dual}-\ref{ss:func}. The technical core of the proof of functoriality can be found as Lemma \ref{l:crux} in Subsection \ref{ss:func}, where we show that a certain Fredholm determinant is equal to one and this allows us to relate the compositions in the coproduct categories to a sort of higher (ternary) change-of-base-point isomorphism in Proposition \ref{p:crux}. The rest of the proof (Subsection \ref{ss:dual}-\ref{ss:binter}) develops alternative descriptions of the change-of-base-point isomorphism with the aim of expressing the ternary change-of-base-point isomorphism as a tensor product of change-of-base-point isomorphisms, see Proposition \ref{p:mulcha}.

The overall setting is as follows: we are given a unital ring $R$ and an ideal $I \su R$ together with a family of representations $\{\pi_\la\}_{\la \in \La}$ satisfying Assumption \ref{a:rep}. We moreover fix two different base points, thus two idempotents $p_0$ and $p_0'$ in $R$ and we assume that they agree modulo the ideal $I$.

We will apply the results on the commutativity and associativity relations for torsion, perturbation and stabilisation from Section \ref{s:torsfred}, \ref{s:pert} and \ref{s:stab} many times, and we will do so without further notice.

\subsection{Cofunctoriality}\label{ss:cofunc}
We let $p,e$ and $q$ be idempotents in the unital ring $R$, all agreeing with the base points $p_0$ and $p_0'$ modulo the ideal $I$. Moreover, we fix two indices $\la,\mu \in \La$.

For an element $x \in R$ and indices $i,j \in \{2,3,4\}$ we let 
\[
e_{ij}(\pi_\mu(x)) \colon \pi_\la(1) \C H \op (\pi_\mu(1) \C H)^{\op 3} \to \pi_\la(1) \C H \op (\pi_\mu(1) \C H)^{\op 3}
\]
denote the bounded operator represented by the $(4 \ti 4)$-matrix with $\pi_\mu(x)$ in position $(i,j)$ and zeroes elsewhere.

We now present the proof of the cofunctoriality of the change-of-base-point isomorphism:

\begin{proof}[Proof of Proposition \ref{p:cofunc}]
We suppress the tuples of indices $(\la,\mu)$ and $(\la,\mu,\mu,\mu)$ throughout the proof.

We first notice that the isomorphism
\[
\begin{split}
\G B(p_0,p_0') \ot \G B(p_0,p_0') & \colon
\big( | F(p,p_0) | \ot | F(p_0,e)| \big) \ot \big( | F(e,p_0) | \ot | F(p_0,q)|\big) \\
& \q \to
\big( | F(p,p_0') | \ot | F(p_0',e)| \big) \ot \big( | F(e,p_0') | \ot | F(p_0',q)| \big)
\end{split}
\]
agrees with the composition of isomorphisms:
\begin{equation}\label{eq:betabeta}
\begin{split}
& \big( | F(p,p_0) | \ot | F(p_0,e)| \big) \ot \big( | F(e,p_0) | \ot | F(p_0,q)|\big) \\
& \q \to^{ (\G T \G S) \ot (\G T \G S ) }
\big| F^{13}(p_0,q,e,p) \cd F^{14}(p,q,e,p_0) \big| \ot \big| F^{12}(p_0,q,e,p) \cd F^{13}(e,q,p_0,p) \big| \\
& \q \to^{ (\G S^{-1} \G P \G S)\ot (\G S^{-1}\G P \G S)}
\big| F^{13}(p_0',q,e,p) \cd F^{14}(p,q,e,p_0') \big| \ot \big| F^{12}(p_0',q,e,p) \cd F^{13}(e,q,p_0',p) \big| \\
& \q \to^{ (\G T \G S)^{-1} \ot (\G T \G S )^{-1}}
\big( | F(p,p_0') | \ot | F(p_0',e)| \big) \ot \big( | F(e,p_0') | \ot | F(p_0',q)|\big) \, ,
\end{split}
\end{equation}
where the stabilisations in the middle are using the operators $e_{34}( \pi_\mu(1-p_0))$ and $e_{23}(\pi_\mu(1-p_0))$ as well as their primed versions (replacing $p_0$ by $p_0'$).

Likewise, we notice that the isomorphism
\[
\G B(p_0,p_0') \colon | F(p,p_0) | \ot | F(p_0,q)| \to | F(p,p_0') | \ot | F(p_0',q)|
\]
agrees with the composition of isomorphisms:
\begin{equation}\label{eq:beta}
\begin{split}
& | F(p,p_0) | \ot | F(p_0,q)|
\to^{\G T \G S}
| F^{12}(p_0,q,e,p) \cd F^{14}(p,q,e,p_0) | \\
& \q \to^{\G S^{-1} \G P \G S}
| F^{12}(p_0',q,e,p) \cd F^{14}(p,q,e,p_0') |
\to^{(\G T \G S)^{-1}}
| F(p,p_0') | \ot | F(p_0',q)| \, ,
\end{split}
\end{equation}
where the stabilisations in the middle use the operators $e_{24}( \pi_\mu(1-p_0))$ and $e_{24}(\pi_\mu(1-p_0'))$.

Recall now (from Definition \ref{d:hopf}) that the isomorphism
\[
(\De_e')^{-1} \colon
| F(p,p_0') | \ot | F(p_0',e)|  \ot | F(e,p_0') | \ot | F(p_0',q)|
\to |F(p,p_0')| \ot |F(p_0',q)|
\]
is given by the composition $(\De_e')^{-1} = (\T{id} \ot \G P \ot \T{id}) \ci (\T{id} \ot \G T \ot \T{id}) = \T{id} \ot (\varphi')^{-1} \ot \T{id}$.

Using the description of $\G B(p_0,p_0') \ot \G B(p_0,p_0')$ from Equation \eqref{eq:betabeta}, we thus obtain that
\[
\begin{split}
(\De_e')^{-1} \ci \big( \G B(p_0,p_0') \ot \G B(p_0,p_0') \big)
& \colon
\big( | F(p,p_0) | \ot | F(p_0,e)| \big) \ot \big( | F(e,p_0) | \ot | F(p_0,q)|\big) \\
& \q \to
| F(p,p_0') | \ot | F(p_0',q)|
\end{split}
\]
agrees with the composition of isomorphisms:
\[
\begin{split}
& | F(p,p_0) | \ot | F(p_0,e)|  \ot | F(e,p_0) | \ot | F(p_0,q)| \\
& \q \to^{\G T \G S }
|  F^{12}(p_0,q,e,p) \cd F^{13}(e,q,p_0,p) \cd F^{13}(p_0,q,e,p) \cd F^{14}(p,q,e,p_0) | \\
& \q \to^{\G S^{-1} \G P \G S}
| F^{12}(p_0',q,e,p) \cd F^{13}(e,q,p_0',p) \cd F^{13}(p_0',q,e,p) \cd F^{14}(p,q,e,p_0') | \\
& \q \to^{(\G T \G S)^{-1}}
| F(p,p_0') | \ot | F(p_0',e) \cd F(e,p_0') | \ot | F(p_0',q)| \\
& \q \to^{\T{id} \ot \G P \ot \T{id}}
| F(p,p_0') | \ot | F(p_0',q)| \, ,
\end{split}
\]
where the stabilisations in the middle use the invertible operators $e_{24}(\pi_\mu(1-p_0))$ and $e_{24}(\pi_\mu(1-p_0'))$. But this latter composition agrees with the composition:
\begin{equation}\label{eq:almost}
\begin{split}
& | F(p,p_0) | \ot | F(p_0,e)|  \ot | F(e,p_0) | \ot | F(p_0,q)|
\to^{ (\T{id} \ot \G P \ot \T{id}) \ci (\T{id} \ot \G T \ot \T{id}) }
| F(p,p_0) | \ot | F(p_0,q)| \\
& \q \to^{\G T \G S }
| F^{12}(p_0,q,e,p) \cd F^{14}(p,q,e,p_0) |
\to^{\G S^{-1} \G P \G S}
| F^{12}(p_0',q,e,p) \cd F^{14}(p,q,e,p_0')| \\
& \q \to^{  (\G T \G S)^{-1}}
| F(p,p_0') | \ot | F(p_0',q)| \, .
\end{split}
\end{equation}

However, using our alternative description of the change-of-base-point isomorphism
\[
\G B(p_0,p_0') \colon | F(p,p_0) | \ot | F(p_0,q)| \to | F(p,p_0') | \ot | F(p_0',q)|
\]
from Equation \eqref{eq:beta} we see that the composition of isomorphisms in Equation \eqref{eq:almost} agrees with the composition
\[
\G B(p_0,p_0') \ci \De_e^{-1} \colon
| F(p,p_0) | \ot | F(p_0,e)|  \ot | F(e,p_0) | \ot | F(p_0,q)|
\to | F(p,p_0') | \ot | F(p_0',q)| \, .
\]
We thus conclude that
\[
\begin{split}
(\De_e')^{-1} \ci \big( \G B(p_0,p_0') \ot \G B(p_0,p_0') \big)
= \G B(p_0,p_0') \ci \De_e^{-1} & \colon
| F(p,p_0) | \ot | F(p_0,e)|  \ot | F(e,p_0) | \ot | F(p_0,q)| \\
& \q \to | F(p,p_0') | \ot | F(p_0',q)| \, ,
\end{split}
\]
and this ends the proof of the proposition.
\end{proof}

\subsection{The dual change of base point}\label{ss:dual}
Throughout this subsection we fix two indices $\la,\mu \in \La$ together with two idempotents $p$ and $q$ in the unital ring $R$, both agreeing with the base points $p_0$ and $p_0'$ modulo the ideal $I \su R$. 

We develop an alternative version of the change-of-base-point isomorphism which we refer to as the \emph{dual change-of-base-point isomorphism}. It is described by the isomorphism
\[
\G B^\da(p_0,p_0') \colon \G L_q^\da(\la,\mu) \ot \G L_p(\la,\mu) \to (\G L_q')^\da(\la,\mu) \ot \G L_p'(\la,\mu) \, ,
\]
defined as the composition
\[
\begin{split}
\G B^\da(p_0,p_0') & \colon |F(p_0,q)| \ot |F(p,p_0)| \to^{\G T \G S} \big|F^{23}(q,p,p_0) F^{13}(p_0 ,p, q) \big| \\
& \q \to^{\G S^{-1} \G P \G S}
\big|F^{23}(q,p,p_0') F^{13}(p_0',p,q) \big|
\to^{(\G T \G S)^{-1}}
|F(p_0',q)| \ot |F(p,p_0')| \, ,
\end{split}
\]
where we are suppressing the tuples of indices $(\la,\mu)$ and $(\la,\la,\mu)$. The stabilisations appearing in the middle are adding the matrices $e_{21}( \pi_\la(1-p_0))$ and $e_{21}(\pi_\la(1-p_0'))$, respectively (see the beginning of Subsection \ref{ss:cofunc} for an explanation of the notation). Remark that the perturbation isomorphism in the middle exists by the argument given in the proof of Lemma \ref{l:changepert}. We will often suppress the pair of idempotents $(p_0,p_0')$ and simply denote the dual change og base point isomorphism by $\G B^\da$.

On many occasions we are going to diverge a bit from the notation introduced in Notation \ref{n:bigfred} and only refer to the idempotents which are directly involved in the Fredholm operators. So instead of writing $F^{ij}(\la_1,\la_2,\ldots,\la_n)(p_1,p_2,\ldots,p_n)$ we shall apply the notation $F^{ij}(\la_1,\la_2,\ldots,\la_n)(p_i,p_j)$ leaving it to the reader to guess the remaining idempotents (they are hopefully going to be clear from the context).

%
%

\begin{lemma}\label{l:triv}
The automorphisms $c_p$ and $c_p^\da \colon (\cc,0) \to (\cc,0)$ defined by
\[
\begin{split}
& c_p \colon (\cc,0) \to^{\psi^{-1}} \G L_p \ot \G L_p^\da \to^{\G B(p_0,p_0')}
\G L_p' \ot (\G L_p')^\da \to^{\psi'} (\cc,0) \q \mbox{and} \\
& c_p^\da \colon (\cc,0) \to^{\varphi} \G L_p^\da \ot \G L_p \to^{\G B^\da(p_0,p_0')}
(\G L_p')^\da \ot \G L_p' \to^{(\varphi')^{-1}} (\cc,0)
\end{split}
\]
are both equal to the identity.
\end{lemma}
\begin{proof}
The case of $c_p$ follows by combining Lemma \ref{l:inverse} and Proposition \ref{p:cofunc}. Indeed, it suffices to show that $c_p^2 = c_p$ and an application of Lemma \ref{l:inverse} implies that this amounts to verifying the commutativity of the diagram
\[
\xymatrix{
\G L_p \ot \G L_p^\da \ot \G L_p \ot \G L_p^\da \ar[rrr]_<<<<<<<<<<<<<<{\G B(p_0,p_0') \ot \G B(p_0,p_0')} \ar[d]^{\T{id} \ot \varphi^{-1} \ot \T{id}} &&&
\G L_p' \ot (\G L_p')^\da \ot \G L_p' \ot ( \G L_p')^\da \ar[d]_{\T{id} \ot (\varphi')^{-1} \ot \T{id}} \\
\G L_p \ot \G L_p^\da  \ar[rrr]^{\G B(p_0,p_0')} &&& \G L_p' \ot (\G L_p')^\da
}
\]
But the commutativity of this diagram follows immediately from Proposition \ref{p:cofunc}. We therefore focus on $c_p^\da \colon (\cc,0) \to (\cc,0)$. 

We show that $( c_p^\da)^2 = c_p^\da$. By Lemma \ref{l:inverse}, this amounts to verifying that the diagram
\begin{equation}\label{eq:betasqua}
\xymatrix{
\G L_p^\da \ot \G L_p \ot \G L_p^\da \ot \G L_p \ar[rr]_{\G B^\da \ot \G B^\da} \ar[d]^{\T{id} \ot \psi \ot \T{id}} &&
(\G L_p')^\da \ot \G L_p' \ot (\G L_p')^\da \ot \G L_p' \ar[d]_{\T{id} \ot \psi' \ot \T{id}} \\
\G L_p^\da \ot \G L_p  \ar[rr]^{\G B^\da} && (\G L_p')^\da \ot \G L_p'
}
\end{equation}
is commutative. To prove that this holds, we notice that the diagram
\begin{equation}\label{eq:betasquaI}
\xymatrix{
\G L_p^\da \ot \G L_p \ot \G L_p^\da \ot \G L_p
\ar[rr]_{\G B^\da \ot \G B^\da} \ar[d]^{\G T \G S} && (\G L_p')^\da \ot \G L_p' \ot (\G L_p')^\da \ot \G L_p' \ar[d]_{\G T \G S} \\
 \big|F^{34}(p,p_0) F^{24}(p_0,p) F^{24}(p,p_0) F^{14}(p_0,p) \big| \ar[rr]^{\G S^{-1}\G P \G S} && 
\big| F^{34}(p,p_0') F^{24}(p_0',p) F^{24}(p,p_0') F^{14}(p_0',p) \big| 
}
\end{equation}
commutes, where we are suppressing the tuple of indices $(\la,\la,\la,\mu)$. 
Remark that the stabilisations appearing in the lower row are adding the matrices $e_{31}(\pi_\la(1 - p_0))$ and $e_{31}(\pi_\la(1 - p_0'))$. But the commutativity of the diagram in Equation \eqref{eq:betasquaI} implies that the diagram
\[
\xymatrix{
\G L_p^\da \ot \G L_p \ot \G L_p^\da \ot \G L_p \ar[rr]_{\G B^\da \ot \G B^\da} \ar[d]^{(\G T \G S) \ci (\T{id}\ot \psi \ot \T{id})} 
& & (\G L_p')^\da \ot \G L_p' \ot (\G L_p')^\da \ot \G L_p' \ar[d]_{(\G T \G S) \ci (\T{id}\ot \psi' \ot \T{id})} \\
\big|F^{34}(p,p_0) F^{14}(p_0,p) \big| \ar[rr]^{\G S^{-1}\G P \G S} & & \big|F^{34}(p,p_0')F^{14}(p_0',p) \big|
}
\]
is commutative as well. The commutativity of the diagram in Equation \eqref{eq:betasqua} now follows.
\end{proof}


The next result does to some extent explain the relationship between the dual change-of-base-point isomorphism and the change-of-base-point isomorphism introduced in Section \ref{s:change}. As usual we suppress the pair of indices $(\la,\mu)$ and the pair of idempotents $(p_0,p_0')$ from the notation. We recall that $\epsilon$ denotes the commutativity constraint in the Picard category of $\zz$-graded complex lines, see Notation \ref{n:picard}.

\begin{prop}\label{p:dagger}
The following diagram is commutative:
\begin{equation}\label{eq:bebedag}
\xymatrix{
\G L_q^\da \ot \G L_p \ot \G L_q \ot \G L_p^\da \ar[rr]^{\G B^\da \ot \G B} \ar[d]_{\T{id} \ot \epsilon} &&
(\G L_q')^\da \ot \G L_p' \ot \G L_q' \ot (\G L_p')^\da \ar[d]^{\T{id} \ot \epsilon} \\
\G L_q^\da \ot \G L_q \ot \G L_p^\da  \ot \G L_p  \ar[rr]^{\varphi' \varphi^{-1} \ot \varphi' \varphi^{-1}} & &
(\G L_q')^\da \ot \G L_q' \ot (\G L_p')^\da  \ot \G L_p' }
\end{equation}
\end{prop}
\begin{proof}
We start by noticing that the diagram
\[
\xymatrix{
\G L_q^\da \ot \G L_p \ot \G L_q \ot \G L_p^\da \ar[rr]^{\G B^\da \ot \G B} \ar[d]^{\G T \G S} & & (\G L_q')^\da \ot \G L_p' \ot \G L_q' \ot (\G L_p')^\da \ar[d]^{\G T \G S} \\
\big|F^{13}(p_0,p) F^{14}(q,p_0) F^{23}(p,p_0) F^{13}(p_0,q)\big| \ar[rr]^{\G S^{-1}\G P \G S} & & 
\big|F^{13}(p_0',p) F^{14}(q,p_0') F^{23}(p,p_0') F^{13}(p_0',q)\big|
}
\]
is commutative, where we suppress the tuple of indices $(\la,\la,\mu,\mu)$ and where the stabilisations appearing in the lower row are adding the matrices $e_{21}(\pi_\la(1 - p_0)) + e_{34}(\pi_\mu(1 - p_0))$ and $e_{21}(\pi_\la(1 - p_0')) + e_{34}(\pi_\mu(1 - p_0'))$, respectively. Using now that $F^{14}(q,p_0)$ and $F^{23}(p,p_0)$ commute since they operate on different direct summands (and similarly with $p_0$ replaced by $p_0'$) we achieve from Proposition \ref{p:torsign} that the diagram
\begin{equation}\label{eq:fircom}
\xymatrix{
\G L_q^\da  \ot \G L_p \ot \G L_q \ot \G L_p^\da  \ar[rr]^{\G B^\da \ot \G B} \ar[d]^{\G T \G S (\T{id} \ot \epsilon \ot \T{id})} & & (\G L_q')^\da  \ot \G L_p' \ot \G L_q' \ot (\G L_p')^\da \ar[d]^{\G T \G S (\T{id} \ot \epsilon \ot \T{id})} \\
\big|F^{13}(p_0,p) F^{23}(p,p_0) F^{14}(q,p_0) F^{13}(p_0,q) \big| \ar[rr]^{\G S^{-1}\G P \G S} & & 
\big|F^{13}(p_0',p) F^{23}(p,p_0') F^{14}(q,p_0') F^{13}(p_0',q) \big|
}
\end{equation}
is commutative as well. In particular, when $p = q$ we obtain from Lemma \ref{l:triv} that the diagram
\begin{equation}\label{eq:midcom}
\xymatrix{
\G L_p^\da  \ot \G L_p \ot \G L_p \ot \G L_p^\da \ar[rr]^{\varphi' \varphi^{-1} \ot (\psi')^{-1} \psi} \ar[d]^{\G T \G S (\T{id} \ot \epsilon \ot \T{id})} & & 
(\G L_p')^\da  \ot \G L_p' \ot \G L_p' \ot (\G L_p')^\da \ar[d]^{\G T \G S (\T{id} \ot \epsilon \ot \T{id})} \\
\big|F^{13}(p_0,p) F^{23}(p,p_0) F^{14}(p,p_0) F^{13}(p_0,p) \big| \ar[rr]^{\G S^{-1}\G P \G S} & & 
\big|F^{13}(p_0',p) F^{23}(p,p_0') F^{14}(p,p_0') F^{13}(p_0',p) \big|
}
\end{equation}
is commutative. To continue, we remark that the diagram
\begin{equation}\label{eq:lastcom}
\xymatrix{
\big|F^{13}(p_0,p) F^{23}(p,p_0) F^{14}(q,p_0) F^{13}(p_0,q) \big| \ar[rr]^{\G S^{-1}\G P \G S} \ar[d]^{\G S^{-1}\G P \G S} & & 
\big|F^{13}(p_0',p) F^{23}(p,p_0') F^{14}(q,p_0') F^{13}(p_0',q) \big| \ar[d]^{\G S^{-1}\G P \G S} \\
\big|F^{13}(p_0,p) F^{23}(p,p_0) F^{14}(p,p_0) F^{13}(p_0,p) \big| \ar[rr]^{\G S^{-1}\G P \G S} & & 
\big|F^{13}(p_0',p) F^{23}(p,p_0') F^{14}(p,p_0') F^{13}(p_0',p) \big|
}
\end{equation}
is commutative, where the two stabilisations appearing in each of the two columns are adding the matrices $e_{43}(\pi_\mu(1-q))$ and $e_{43}(\pi_\mu(1-p))$, respectively. By combining the commutative diagrams in Equation \eqref{eq:fircom}-\eqref{eq:lastcom} we obtain the commutative diagram
\begin{equation}\label{eq:baselr}
\xymatrix{
\G L_q^\da \ot \G L_p \ot \G L_q \ot \G L_p^\da \ar[rrr]^{\G B^\da \ot \G B} \ar[d]_{\epsilon (\G T \G S)^{-1}(\G S^{-1} \G P \G S )(\G T \G S)\epsilon } & & & (\G L_q')^\da \ot \G L_p' \ot \G L_q' \ot (\G L_p')^\da \ar[d]^{\epsilon (\G T \G S)^{-1}(\G S^{-1} \G P \G S )(\G T \G S) \epsilon } \\
\G L_p^\da  \ot \G L_p \ot \G L_p \ot \G L_p^\da \ar[rrr]^{\varphi' \varphi^{-1} \ot (\psi')^{-1} \psi} & & &
(\G L_p')^\da  \ot \G L_p' \ot \G L_p' \ot (\G L_p')^\da
}
\end{equation}
where we clarify that the vertical commutativity constraints are really of the form $\T{id} \ot \epsilon \ot \T{id}$. We now observe that the left vertical isomorphism in the above diagram agrees with the composition of isomorphisms
\[ 
\begin{split}
& |F(p_0,q)| \ot |F(p,p_0)| \ot |F(q,p_0)| \ot |F(p_0,p)|
\to^{\T{id} \ot \epsilon \ot \T{id}} |F(p_0,q)| \ot |F(q,p_0)| \ot |F(p,p_0)| \ot |F(p_0,p)| \\
& \q \to^{\G B(q,p) \ot \T{id}^{\ot 2}}
|F(p_0,p)| \ot |F(p,p_0)| \ot |F(p,p_0)| \ot |F(p_0,p)| \\
& \q \to^{\T{id} \ot \epsilon \ot \T{id}}
|F(p_0,p)| \ot |F(p,p_0)| \ot |F(p,p_0)| \ot |F(p_0,p)|
\end{split}
\]
and hence, using Lemma \ref{l:triv}, it does in fact agree with the composition
\[ 
\begin{split}
& |F(p_0,q)| \ot |F(p,p_0)| \ot |F(q,p_0)| \ot |F(p_0,p)|
\to^{\T{id} \ot \epsilon \ot \T{id}} |F(p_0,q)| \ot |F(q,p_0)| \ot |F(p,p_0)| \ot |F(p_0,p)| \\
& \q \to^{\psi_q \ot \T{id}^{\ot 2}} |F(p,p_0)| \ot |F(p_0,p)|
\to^{\psi_p^{-1} \ot \T{id}^{\ot 2}}
|F(p_0,p)| \ot |F(p,p_0)| \ot |F(p,p_0)| \ot |F(p_0,p)| \\
& \q \to^{\T{id} \ot \epsilon \ot \T{id}}
|F(p_0,p)| \ot |F(p,p_0)| \ot |F(p,p_0)| \ot |F(p_0,p)| \, ,
\end{split}
\]
where the subscripts $p$ and $q$ on the duality isomorphism $\psi$ indicate that these idempotents function as base points. Remark that, changing the point of view and using $p_0$ as a base point, we have that $\psi_q = \varphi^{-1} \colon \G L_q^\da \ot \G L_q \to (\cc,0)$ and similarly $\psi_p = \varphi^{-1} \colon \G L_p^\da \ot \G L_p \to (\cc,0)$. Using Lemma \ref{l:inverse} we then notice that the composition of isomorphisms
\[
\begin{split}
& |F(p,p_0)| \ot |F(p_0,p)|\to^{\psi_p^{-1} \ot \T{id}^{\ot 2}}
|F(p_0,p)| \ot |F(p,p_0)| \ot |F(p,p_0)| \ot |F(p_0,p)| \\
& \q \to^{\T{id} \ot \epsilon \ot \T{id}}
|F(p_0,p)| \ot |F(p,p_0)| \ot |F(p,p_0)| \ot |F(p_0,p)|
\to^{\varphi^{-1} \ot \psi} (\cc,0)
\end{split}
\]
agrees with the composition of isomorphisms
\[
|F(p,p_0)| \ot |F(p_0,p)| \to^{\epsilon} |F(p_0,p)| \ot |F(p,p_0)| \to^{\varphi^{-1}} (\cc,0) \, .
\]
Combining these results we obtain that the composition of the left vertical isomorphism in Equation \eqref{eq:baselr} with the trivialisation
$\varphi^{-1} \ot \psi \colon \G L_p^\da \ot \G L_p \ot \G L_p \ot \G L_p^\da \to (\cc,0)$ agrees with the trivialisation
\[
\G L_q^\da \ot \G L_p \ot \G L_q \ot \G L_p^\da \to^{\T{id} \ot \epsilon}
\G L_q^\da \ot \G L_q \ot \G L_p^\da \ot \G L_p \to^{\varphi^{-1} \ot \varphi^{-1}} (\cc,0) \, .
\]
Since a similar description applies when $p_0$ is replaced by $p_0'$, the result of the present proposition follows from the commutativity of the diagram in Equation \eqref{eq:baselr}.
\end{proof}

%

\subsection{The symmetrised change of base point}
Let us fix two indices $\la$ and $\mu$ in the index set $\La$. We now derive one more alternative version of the change-of-base-point isomorphism. This alternative version is more symmetric in the indices $\la$ and $\mu$ and we refer to it as the \emph{symmetrised change-of-base-point isomorphism}. We are applying notational conventions similar to those described in the beginning of Subsection \ref{ss:cofunc} and \ref{ss:dual}.

We need a small lemma (which can be compared with Lemma \ref{l:changepert}). Remark that the invertible operator $\Om^{14}(p_0)$ appearing in the statement agrees with $\Om(\la,\mu)(p_0)$ up to stabilisation with the idempotent $\pi_\la(p) \op \pi_\mu(q)$ (and similarly with $p_0$ replaced by $p_0'$), see Notation \ref{n:bigfred}. We are also omitting the tuples of indices $(\la,\la,\mu,\mu)$ and $(\la,\mu)$ from the notation.  

\begin{lemma}\label{l:bigdif}
The difference of the two Fredholm operators
\[
\begin{split}
&  F^{13}(p_0,q) F^{24}(p,p_0) \Om^{14}(p_0) + e_{21}(\pi_\la(1 - p_0) ) + e_{34}( \pi_\mu(1 - p_0) ) 
\q \mbox{and} \\
& F^{13}(p_0',q) F^{24}(p,p_0') \Om^{14}(p_0') + e_{21}(\pi_\la(1 - p_0') ) + e_{34}( \pi_\mu(1 - p_0') ) \, ,
\end{split}
\]
with domain $\pi_\la(1)H \op \pi_\la(p) H \op \pi_\mu(q) H \op \pi_\mu(1)H$ and codomain $\pi_\la(q) H \op \pi_\la(1)H \op \pi_\mu(1)H \op \pi_\mu(p) H$, is of trace class.
\end{lemma}
\begin{proof}
This is a consequence of Lemma \ref{l:commut} and Lemma \ref{l:reduction} together with the following computation modulo $\sL^1(\C H^{\op 4})$:
\[
\begin{split}
& \Om^{13}(1) \Om^{24}(1) \Om^{14}(1) \cd ( \pi_\la(p_0 - p_0') \op 0 \op 0 \op \pi_\mu(p_0 - p_0') ) \\
& \q \sim_1
\ma{cccc}{0 & 0 & 0  & 0 \\ 0 & \pi_\la(1) & 0 & 0 \\ \pi_\mu(1) \pi_\la(1) & 0 & 0 & 0 \\ 0 & 0 & 0 & 0}
\cd \ma{cccc}{\pi_\la(1) & 0 & 0 & 0 \\ 0 & 0 & 0 & \pi_\la(1) \pi_\mu(1) \\ 0 & 0 & 0 & 0 \\ 0 & 0 & 0 & 0} \\
& \qqq \cd \ma{cccc}{0 & 0 & 0 & \pi_\la(1) \pi_\mu(1) \\ 0 & 0 & 0 & 0 \\ 0 & 0 & 0 & 0 \\ \pi_\mu(1) \pi_\la(1) & 0 & 0 & 0}
\cd ( \pi_\la(p_0 - p_0') \op 0 \op 0 \op \pi_\mu(p_0 - p_0') ) \\
& \q \sim_1 \ma{cc}{0 & 0 \\ \pi_\la(p_0 - p_0') & 0} \op \ma{cc}{0 & \pi_\mu(p_0 - p_0') \\ 0 & 0} \, . \qedhere
\end{split}
\]
\end{proof}

We recall the notation $L(\cd)$ and $R(\cd)$ from Example \ref{ex:LR} and apply Lemma \ref{l:bigdif} to define the \emph{symmetrised change-of-base-point isomorphism}
\[
\wit{\G B}(p_0,p_0') \colon \G L_p \ot \G L_q^\da \ot \G L_q \ot \G L_q^\da
\to \G L_p' \ot (\G L_q')^\da \ot \G L_q' \ot (\G L_q')^\da
\]
as the composition
\begin{equation}\label{eq:tilbet}
\begin{split}
\wit{\G B}(p_0,p_0') & \colon |F(p,p_0)| \ot |F(p_0,q)| \ot |F(q,p_0)| \ot |F(p_0,q)| \\
& \q \to^{R(\Om^{14}(p_0))(\G T \G S) \ot L(\Om^{14}(p_0))(\G T \G S)}
\big| F^{13}(p_0,q) F^{24}(p,p_0) \Om^{14}(p_0)  \big|
\ot \big| \Om^{14}(p_0) F^{24}(p_0,q) F^{13}(q,p_0)  \big| \\
& \q \to^{\G S^{-1}\G P \G S \ot \G S^{-1}\G P \G S}
\big|F^{13}(p_0',q) F^{24}(p,p_0') \Om^{14}(p_0')  \big|
\ot \big|\Om^{14}(p_0') F^{24}(p_0',q) F^{13}(q,p_0')  \big| \\
& \q \to^{(\G T  \G S)^{-1} R(\Om^{14}(p_0')) \ot (\G T \G S)^{-1} L(\Om^{14}(p_0'))}
|F(p,p_0')| \ot |F(p_0',q)| \ot |F(q,p_0')| \ot |F(p_0',q)| \, .
\end{split}
\end{equation}
The stabilisations surrounding the two perturbation isomorphisms are adding the matrices $e_{21}(\pi_\la(1 - p_0) ) + e_{34}( \pi_\mu(1 - p_0) )$ and $e_{12}(\pi_\la(1-p_0)) + e_{43}( \pi_\mu(1-p_0))$ and similarly with $p_0$ replaced by $p_0'$.

The next result explains the relationship between the symmetrised change-of-base-point isomorphism and the change-of-base-point isomorphism introduced in Section \ref{s:change}.

\begin{prop}\label{p:witbet}
The following diagram is commutative:
\[
\xymatrix{
\G L_p \ot \G L_q^\da \ot \G L_q \ot \G L_q^\da \ar[rr]^>>>>>>>>>>{\wit{\G B}(p_0,p_0')} \ar[d]_{\T{id} \ot \varphi^{-1} \ot \T{id}} && \G L_p' \ot (\G L_q')^\da \ot \G L_q' \ot (\G L_q')^\da \ar[d]^{\T{id} \ot (\varphi')^{-1} \ot \T{id}} \\
\G L_p \ot \G L_q^\da \ar[rr]^{\G B(p_0,p_0')} && \G L_p' \ot (\G L_q')^\da
}
\]
\end{prop}
\begin{proof}
We start by noticing that the symmetrised change-of-base-point isomorphism agrees with the following composition of isomorphisms:
\begin{equation}\label{eq:symmbase}
\begin{split}
& \G L_p \ot \G L_q^\da \ot \G L_q \ot \G L_q^\da
\to^{L(\Om^{15}(p_0)) R(\Om^{14}(p_0)) \G T \G S} \big| \Om^{15}(p_0) F^{25}(p_0,q) F^{13}(q,p_0) F^{13}(p_0,q) F^{24}(p,p_0) \Om^{14}(p_0) \big| \\
& \q \to^{\G S^{-1} \G P \G S}
\big| \Om^{15}(p_0') F^{25}(p_0',q) F^{13}(q,p_0') F^{13}(p_0',q) F^{24}(p,p_0') \Om^{14}(p_0') \big| \\
& \q \to^{(\G T \G S)^{-1} L(\Om^{15}(p_0')) R(\Om^{14}(p_0'))}
\G L_p' \ot (\G L_q')^\da \ot \G L_q' \ot (\G L_q')^\da \, ,
\end{split}
\end{equation}
where the stabilisations surrounding the perturbation isomorphism are adding the matrix $e_{11}(\pi_\la(1-p_0)) + e_{54}(\pi_\mu(1-p_0))$ (and similarly with $p_0$ replaced by $p_0'$) and where the tuple of indices $(\la,\la,\mu,\mu,\mu)$ has been suppressed from the notation. We then remark that the composition of isomorphisms in Equation \eqref{eq:symmbase} agrees with the following composition of isomorphisms:
\[
\begin{split}
& \G L_p \ot \G L_q^\da \ot \G L_q \ot \G L_q^\da \to^{\T{id} \ot \varphi^{-1} \T{id}}
\G L_p \ot \G L_q^\da
\to^{L(\Om^{14}(p_0)) R(\Om^{13}(p_0)) \G T \G S} \big| \Om^{14}(p_0) F^{24}(p_0,q) F^{23}(p,p_0) \Om^{13}(p_0) \big| \\
& \q \to^{\G S^{-1} \G P \G S}
\big| \Om^{14}(p_0') F^{24}(p_0',q) F^{23}(p,p_0') \Om^{13}(p_0') \big| \\
& \q \to^{(\G T \G S)^{-1} L(\Om^{14}(p_0')) R(\Om^{13}(p_0'))}
\G L_p' \ot (\G L_q')^\da \to^{\T{id} \ot \varphi' \ot \T{id}} \G L_p' \ot (\G L_q')^\da \ot \G L_q' \ot (\G L_q')^\da \, ,
\end{split}
\]
where the non-obvious stabilisations are adding the matrices $e_{11}(\pi_\la(1-p_0)) + e_{43}(\pi_\mu(1-p_0))$ and $e_{11}(\pi_\la(1-p_0')) + e_{43}(\pi_\mu(1-p_0'))$, respectively. Combining these results, we see that it suffices to show that the change-of-base-point isomorphism $\G B \colon \G L_p \ot \G L_q^\da \to \G L_p' \ot (\G L_q')^\da$ agrees with the composition of isomorphisms
\begin{equation}\label{eq:symmbaseI}
\begin{split}
& \G L_p \ot \G L_q^\da \to^{L(\Om^{14}(p_0)) R(\Om^{13}(p_0)) \G T \G S}
\big| \Om^{14}(p_0) F^{24}(p_0,q) F^{23}(p,p_0) \Om^{13}(p_0) \big| \\
& \q \to^{\G S^{-1} \G P \G S} \big| \Om^{14}(p_0') F^{24}(p_0',q) F^{23}(p,p_0') \Om^{13}(p_0') \big|
\to^{(\G T \G S)^{-1}L(\Om^{14}(p_0')) R(\Om^{13}(p_0'))} 
\G L_p' \ot (\G L_q')^\da \, .
\end{split}
\end{equation}
To prove this result we let 
\[
\begin{split}
& \Om^{14}(1-p_0) \in \sL\big( \pi_\la(1) \C H \op \pi_\la(q) \C H \op \pi_\mu(p) \C H \op \pi_\mu(1) \C H\big) \q \T{and} \\
& \Om^{13}(1 - p_0) \in \sL\big( \pi_\la(1) \C H \op \pi_\la(p) \C H \op \pi_\mu(1) \C H \op \pi_\mu(q) \C H\big)
\end{split}
\]
be stabilized by $0$ and remark that 
\[
\begin{split}
& \Om^{14}(1 - p_0)  \cd \big( F^{24}(1,p_0,p,q) F^{23}(1,p,p_0,q) + e_{43}(\pi_\mu(1-p_0)) \big) \Om^{13}(1 - p_0) \\
& \q = \Om^{14}(1 - p_0) \big( e_{11}( \pi_\la(1-p_0)) + e_{43}(\pi_\mu(1-p_0)) \big) \Om^{13}(1 - p_0)  \\
& \q = e_{11}(\pi_\la(1-p_0)) + e_{43}(\pi_\mu(1-p_0)) \, ,
\end{split}
\]
where the second identity follows from Lemma \ref{l:reduction}. We thus obtain that
\[
\begin{split}
& \Om^{14}(p_0) F^{24}(p_0,q) F^{23}(p,p_0) \Om^{13}(p_0) + e_{11}(\pi_\la(1-p_0)) + e_{43}(\pi_\mu(1-p_0)) \\
& \q = ( \Om^{14}(p_0) + \Om^{14}(1 - p_0) ) \cd \big( F^{24}(1,p_0,p,q) F^{23}(1,p,p_0,q) + e_{43}( \pi_\mu(1-p_0))\big) \\ 
& \qq \cd ( \Om^{13}(p_0) + \Om^{13}(1 - p_0) ) \, .
\end{split}
\]
A similar result applies to the case where $p_0$ has been replaced by $p_0'$. Using this computation one may verify that the composition of isomorphisms in Equation \eqref{eq:symmbaseI} agrees with the following composition of isomorphisms:
\begin{equation}\label{eq:symmbaseII}
\begin{split}
& \G L_p \ot \G L_q^\da \to^{\G S \G T \G S}
\big| F^{24}(1,p_0,p,q) F^{23}(1,p,p_0,q) + e_{43}( \pi_\mu(1-p_0)) \big|  \\
& \q \to^{L(\Om^{14}(p_0) + \Om^{14}(1 - p_0))R(\Om^{13}(p_0) + \Om^{13}(1-p_0))} 
\big| \Om^{14}(p_0) F^{24}(p_0,q) F^{23}(p,p_0) \Om^{13}(p_0)  \\
& \qqq \qqq \qqq \qqq + e_{11}(\pi_\la(1-p_0)) + e_{43}(\pi_\mu(1-p_0)) \big| \\
& \q \to^{\G P}
\big| \Om^{14}(p_0') F^{24}(p_0',q) F^{23}(p,p_0') \Om^{13}(p_0') + 
e_{11}(\pi_\la(1-p_0')) + e_{43}(\pi_\mu(1-p_0')) \big| \\
& \q \to^{L(\Om^{14}(p_0') + \Om^{14}(1 - p_0'))R(\Om^{13}(p_0') + \Om^{13}(1-p_0'))}
\big| F^{24}(1, p_0',p,q) F^{23}(1,p,p_0',q) + e_{43}( \pi_\mu(1-p_0')) \big| \\
& \q \to^{(\G S \G T \G S)^{-1}} 
\G L_p' \ot (\G L_q')^\da \, .
\end{split}
\end{equation}
This latter composition of isomorphisms can be seen to coincide with the change-of-base-point isomorphism $\G B \colon \G L_p \ot \G L_q^\da \to \G L_p' \ot (\G L_q')^\da$ multiplied with the square of the Fredholm determinant $\det\big(  \big( \Om(\la,\mu)(p_0') + \Om(\la,\mu)(1-p_0') \big) \big( \Om(\la,\mu)(p_0) + \Om(\la,\mu)(1-p_0) \big)  \big)$. The result of the proposition therefore follows since 
\[
\det\big(  \big( \Om(\la,\mu)(p_0') + \Om(\la,\mu)(1-p_0') \big) \big( \Om(\la,\mu)(p_0) + \Om(\la,\mu)(1-p_0) \big) \big)^2 = 1 \, . \qedhere
\]
\end{proof}

\subsection{Binary and ternary versions of the change of base point}\label{ss:binter}
Throughout this subsection we fix three elements in the index set $\la,\mu,\nu \in \La$ together with two idempotents $p$ and $q$ in $R$, both agreeing with the base points $p_0$ and $p_0'$ modulo the ideal $I \su R$.

As an extra technical step towards establishing the functoriality of the change-of-base-point isomorphism, we now develop binary and ternary versions of this isomorphism. Moreover, these higher analogues of the change-of-base-point isomorphism will be related to tensor products of the original change-of-base-point isomorphisms.

We begin with two lemmas asserting the existence of perturbation isomorphisms. In the first lemma we are suppressing the tuple of indices $(\la,\la,\mu,\nu)$ and in the second lemma we are suppressing the tuple of indices $(\la,\la,\la,\mu,\nu)$.

\begin{lemma}\label{l:merbetl}
The difference of the following two Fredholm operators
\[
\begin{split}
&  F^{34}(p_0,q) \cd F^{14}(q,p_0) \cd F^{24}(p_0,p) \cd F^{34}(p,p_0) \\
& \qq + e_{12}(\pi_\la(1-p_0)) + e_{44}(\pi_\nu(1-p_0)) \q \mbox{and} \\
& F^{34}(p_0',q) \cd F^{14}(q,p_0') \cd F^{24}(p_0',p) \cd F^{34}(p,p_0') \\
& \qq + e_{12}(\pi_\la(1-p_0')) + e_{44}(\pi_\nu(1-p_0')) \, ,
\end{split}
\]
both acting from the Hilbert space $\pi_\la(q) \C H \op \pi_\la(1)\C H \op \pi_\mu(p) \C H \op \pi_\nu(1)\C H$ to the Hilbert space $\pi_\la(1)\C H \op \pi_\la(p) \C H \op \pi_\mu(q) \C H \op \pi_\nu(1)\C H$, is of trace class.
\end{lemma}
\begin{proof}
We compute modulo $\sL^1(\C H^{\op 4})$, using Lemma \ref{l:commut} and Assumption \ref{a:rep} together with the fact that $p_0 - p_0' \in I$:
\begin{align*}
& F^{34}(p_0,q) \cd F^{14}(q,p_0) \cd F^{24}(p_0,p) \cd F^{34}(p,p_0) \\
& \qq - F^{34}(p_0',q) \cd F^{14}(q,p_0') \cd F^{24}(p_0',p) \cd F^{34}(p,p_0') \\
& \q \sim_1
\big( \pi_\la(p_0 - p_0') \op 0^{\op 2} \op \pi_\nu(p_0 - p_0') \big)
\cd \ma{cccc}{\pi_\la(1) & 0 & 0 & 0 \\ 0 & 0 & 0 & 0 \\ 0 & 0 & 0 & 0 \\ 0 & 0 & \pi_\nu(1) \pi_\mu(1) & 0} \\
& \qq \cd
\ma{cccc}{0 & 0 & 0 & \pi_\la(1) \pi_\nu(1) \\ 0 & 0  & 0 & 0 \\ 0 & 0 & \pi_\mu(1) & 0 \\ 0 & 0 & 0 & 0}
\cd
\ma{cccc}{0 & 0 & 0 & 0 \\ 0 & 0 & 0 & 0 \\ 0 & 0 & \pi_\mu(1) & 0 \\ 0 & \pi_\nu(1) \pi_\la(1) & 0  & 0} \\
& \qqqq \cd
\ma{cccc}{0 & 0 & 0 & 0 \\ 0 & \pi_\la(1) & 0 & 0 \\ 0 & 0 & 0 & \pi_\mu(1) \pi_\nu(1) \\ 0 & 0 & 0 & 0} \\
& \q \sim_1
e_{12}( \pi_\la(p_0 - p_0')) + e_{44}(\pi_\nu(p_0 - p_0')) \, . \qedhere
\end{align*}
\end{proof}

\begin{lemma}\label{l:multrace}
The difference of the following two Fredholm operators
\[
\begin{split}
& \Om^{14}(p_0) \cd F^{24}(p_0,q) F^{45}(p_0,q) F^{25}(q,p_0) \\
& \qqq \cd F^{35}(p_0, p) F^{45}(p,p_0) F^{34}(p,p_0) \cd \Om^{14}(p_0) \\
& \q + \pi_\la(1 - p_0)  \op 0 \op 0 \op \pi_\mu(1 - p_0) \op \pi_\nu(1 -p_0) \q \mbox{and} \\
& \Om^{14}(p_0') \cd F^{24}(p_0',q) F^{45}(p_0',q) F^{25}(q,p_0') \\
& \qqq \cd F^{35}(p_0', p) F^{45}(p,p_0') F^{34}(p,p_0') \cd \Om^{14}(p_0') \\
& \q + \pi_\la(1 - p_0')  \op 0 \op 0 \op \pi_\mu(1 - p_0') \op \pi_\nu(1 -p_0') \, ,
\end{split}
\]
both acting on the Hilbert space $\pi_\la(1)\C H \op \pi_\la(q) \C H \op \pi_\la(p) \C H \op \pi_\mu(1)\C H \op \pi_\nu(1) \C H$, is of trace class.
\end{lemma}
\begin{proof}
The argument is similar to the argument given in the proof of Lemma \ref{l:merbetl} and will therefore not be repeated here.
\end{proof}

As a consequence of Lemma \ref{l:merbetl}, we may define the following {\it binary change-of-base-point isomorphism}
\[
\begin{split}
\G N(p_0,p_0') & \colon
\G L_p(\mu,\nu) \ot \G L_p^\da(\la,\nu) \ot \G L_q(\la,\nu) \ot \G L_q^\da(\mu,\nu) \\
& \q \to
\G L_p'(\mu,\nu) \ot (\G L_p')^\da(\la,\nu) \ot \G L_q'(\la,\nu) \ot (\G L_q')^\da(\mu,\nu)
\end{split}
\]
as the composition
\begin{equation}\label{eq:ncha}
\begin{split}
& | F(\mu,\nu)(p,p_0) | \ot | F(\la,\nu)(p_0,p) |
\ot | F(\la,\nu)(q,p_0) |  \ot | F(\mu,\nu)(p_0,q) | \\
& \q \to^{\G T \G S}
\big|  F^{34}(p_0,q) \cd F^{14}(q,p_0) \cd F^{24}(p_0,p) \cd F^{34}(p,p_0) \big| \\
& \q \to^{\G S^{-1} \G P \G S}
\big|  F^{34}(p_0',q) \cd F^{14}(q,p_0') \cd F^{24}(p_0',p) \cd F^{34}(p,p_0') \big| \\
& \q \to^{(\G T \G S)^{-1}}
| F(\mu,\nu)(p,p_0') | \ot | F(\la,\nu)(p_0',p) |
\ot | F(\la,\nu)(q,p_0') |  \ot | F(\mu,\nu)(p_0',q) | \, ,
\end{split}
\end{equation}
where we are suppressing the tuple of indices $(\la,\la,\mu,\nu)$ and where the stabilisations in the middle add the invertible operators $e_{12}(\pi_\la(1-p_0)) + e_{44}(\pi_\nu(1-p_0))$ and $e_{12}(\pi_\la(1-p_0')) + e_{44}(\pi_\nu(1-p_0'))$, respectively.
\medskip

Similarly, using Lemma \ref{l:multrace}, we define the \emph{ternary change-of-base-point isomorphism}
\[
\begin{split}
\G M(p_0,p_0') & \colon \G L_p(\la,\mu) \ot \G L_p(\mu,\nu) \ot \G L_p^\da(\la,\nu)
\ot \G L_q(\la,\nu) \ot \G L_q^\da(\mu,\nu) \ot \G L_q^\da(\la,\mu) \\
& \q \to \G L_p'(\la,\mu) \ot \G L_p'(\mu,\nu) \ot (\G L_p')^\da(\la,\nu)
\ot \G L_q'(\la,\nu) \ot (\G L_q')^\da(\mu,\nu) \ot (\G L_q')^\da(\la,\mu)
\end{split}
\]
as the composition
\begin{equation}\label{eq:terbase}
\begin{split}
& \G M(p_0,p_0') \colon |F(\la,\mu)(p,p_0)| \ot |F(\mu,\nu)(p,p_0)| \ot |F(\la,\nu)(p_0,p)| \\
& \qqq \qq \ot |F(\la,\nu)(q,p_0)| \ot |F(\mu,\nu)(p_0,q)| \ot |F(\la,\mu)(p_0,q)| \\
& \q \to^{R(\Om^{14}(p_0))L(\Om^{14}(p_0))\G T \G S}
\big| \Om^{14}(p_0) \cd F^{24}(p_0,q) F^{45}(p_0,q) F^{25}(q,p_0) \\
& \qqq \qq \qq \cd F^{35}(p_0,p) F^{45}(p,p_0) F^{34}(p,p_0) \cd \Om^{14}(p_0) \big| \\
& \q \to^{\G S^{-1}\G P\G S}
\big| \Om^{14}(p_0') \cd F^{24}(p_0',q) F^{45}(p_0',q) F^{25}(q,p_0') \\
& \qqq \qq \cd F^{35}(p_0',p) F^{45}(p,p_0') F^{34}(p,p_0') \cd \Om^{14}(p_0') \big| \\
& \q \to^{(\G T \G S)^{-1} R(\Om^{14}(p_0'))L(\Om^{14}(p_0'))}
|F(\la,\mu)(p,p_0')| \ot |F(\mu,\nu)(p,p_0')| \ot |F(\la,\nu)(p_0',p)| \\
& \qqq \qq \qqq \ot |F(\la,\nu)(q,p_0')| \ot |F(\mu,\nu)(p_0',q)| \ot |F(\la,\mu)(p_0',q)| \, ,
\end{split}
\end{equation}
where we are suppressing the tuple of indices $(\la,\la,\la,\mu,\nu)$ from the notation and where the stabilisations in the middle are now adding the invertible operators $e_{11}(\pi_\la(1 - p_0)) + e_{44}( \pi_\mu(1 - p_0)) + e_{55}( \pi_\nu(1-p_0))$ and  $e_{11}(\pi_\la(1 - p_0')) + e_{44}( \pi_\mu(1 - p_0')) + e_{55}( \pi_\nu(1-p_0'))$, respectively.
\medskip

In the next proposition we express the binary change-of-base-point isomorphism in terms of the change-of-base-point isomorphism and the dual change-of-base-point isomorphism (see Definition \ref{def:change} and the beginning of Subsection \ref{ss:dual}):

\begin{prop}\label{p:merbet}
The isomorphism of $\zz$-graded complex lines
\[
\begin{split}
& \G B(p_0,p_0') \ot \G B^\da(p_0,p_0') \\ 
& \q \colon
\G L_p(\mu,\nu) \ot \G L_q^\da(\mu,\nu) \ot \G L_p^\da(\la,\nu) \ot \G L_q(\la,\nu) \to
\G L_p'(\mu,\nu) \ot (\G L_q')^\da(\mu,\nu) \ot (\G L_p')^\da(\la,\nu) \ot \G L_q'(\la,\nu)
\end{split}
\]
agrees with the composition of isomorphisms of $\zz$-graded complex lines:
\begin{equation}\label{eq:merbet}
\begin{split}
& \G L_p(\mu,\nu) \ot \G L_q^\da(\mu,\nu) \ot \G L_p^\da(\la,\nu) \ot \G L_q(\la,\nu)
\to^{\epsilon}
\G L_p(\mu,\nu) \ot \G L_p^\da(\la,\nu) \ot \G L_q(\la,\nu) \ot \G L_q^\da(\mu,\nu) \\
& \q \to^{\G N(p_0,p_0')}
\G L_p'(\mu,\nu) \ot (\G L_p')^\da(\la,\nu) \ot \G L_q'(\la,\nu) \ot (\G L_q')^\da(\mu,\nu) \\
& \q \to^{\epsilon}
\G L_p'(\mu,\nu) \ot (\G L_q')^\da(\mu,\nu) \ot (\G L_p')^\da(\la,\nu) \ot \G L_q'(\la,\nu) \, .
\end{split}
\end{equation}
\end{prop}
\begin{proof}
We suppress the tuple of indices $(\la,\la,\mu,\nu,\nu)$ from the notation.

Our first claim is that the binary change-of-base-point isomorphism $\G N(p_0,p_0')$ agrees with the composition of isomorphisms
\begin{equation}\label{eq:altenn}
\begin{split}
& \big| F(\mu,\nu)(p,p_0) \big| \ot \big| F(\la,\nu)(p_0,p) \big|
\ot \big| F(\la,\nu)(q,p_0) \big|  \ot \big| F(\mu,\nu)(p_0,q) \big| \\
& \q \to^{\G T \G S}
\big|  F^{34}(p_0,q) \cd F^{15}(q,p_0) \cd F^{25}(p_0,p) \cd F^{35}(p,p_0) \big| \\
& \q \to^{\G S^{-1} \G P \G S}
\big|  F^{34}(p_0',q) \cd F^{15}(q,p_0') \cd F^{25}(p_0',p) \cd F^{35}(p,p_0') \big| \\
& \q \to^{(\G T \G S)^{-1}}
\big| F(\mu,\nu)(p,p_0') \big| \ot \big| F(\la,\nu)(p_0',p) \big|
\ot \big| F(\la,\nu)(q,p_0') \big|  \ot \big| F(\mu,\nu)(p_0',q) \big| \, ,
\end{split}
\end{equation}
where the stabilisations in the middle use the invertible operators $e_{12}(\pi_\la(1-p_0)) + e_{45}(\pi_\nu(1-p_0))$ and $e_{12}(\pi_\la(1-p_0')) + e_{45}(\pi_\nu(1-p_0'))$, respectively.

To prove the above claim, we recall from Lemma \ref{l:triv} that the automorphism $c_q(\mu,\nu) \colon (\cc,0) \to (\cc,0)$, defined as the composition
\[
\begin{split}
& (\cc,0) \to^{\psi^{-1}} |F(\mu,\nu)(q,p_0)| \ot |F(\mu,\nu)(p_0,q)| \to^{\G T \G S}
\big| F^{34}(p_0,q) \cd F^{35}(q,p_0) \big| \\
& \q \to^{\G S^{-1}\G P \G S}  \big| F^{34}(p_0',q) \cd F^{35}(q,p_0') \big| \\
& \q \to^{(\G T \G S)^{-1}} |F(\mu,\nu)(q,p_0')| \ot |F(\mu,\nu)(p_0',q)|
\to^{\psi'} (\cc,0) \, ,
\end{split}
\]
is equal to the identity map. We are here considering $F^{34}(p_0,q) \cd F^{35}(q,p_0)$ as a Fredholm operator acting on the Hilbert space $\pi_\la(p_0) \C H \op \pi_\la(p) \C H \op \pi_\mu(q) \C H \op \pi_\nu(q)\C H \op \pi_\nu(p_0) \C H$ and we apply the invertible operator $e_{11}(\pi_\la(1-p_0)) + e_{45}(\pi_\nu(1-p_0))$ for the (non-obvious) stabilisation procedure (a similar comment applies to the primed version). This observation allows us to replace $\G N(p_0,p_0')$ with $\G N(p_0,p_0') \cd c_q(\mu,\nu)$. However, this latter isomorphism can now be described as the composition of isomorphisms
\[
\begin{split}
& | F(\mu,\nu)(p,p_0) | \ot | F(\la,\nu)(p_0,p) |
\ot | F(\la,\nu)(q,p_0) |  \ot | F(\mu,\nu)(p_0,q) | \\
& \q \to^{\T{id}^{\ot 4} \ot \psi^{-1}}
| F(\mu,\nu)(p,p_0) | \ot | F(\la,\nu)(p_0,p) |
\ot | F(\la,\nu)(q,p_0) | \\
& \qqq \qq \ot | F(\mu,\nu)(p_0,q) | \ot |F(\mu,\nu)(q,p_0)| \ot |F(\mu,\nu)(p_0,q)| \\
& \q \to^{\G T \G S}
\big| F^{34}(p_0,q) \cd F^{35}(q,p_0) \cd F^{35}(p_0,q) \\
& \qqq \qq \cd F^{15}(q,p_0) \cd F^{25}(p_0,p) \cd F^{35}(p,p_0) \big| \\
& \q \to^{\G S^{-1}\G P \G S}
\big| F^{34}(p_0',q) \cd F^{35}(q,p_0') \cd F^{35}(p_0',q) \\
& \qqq \qq \cd F^{15}(q,p_0') \cd F^{25}(p_0',p) \cd F^{35}(p,p_0') \big| \\
& \q \to^{(\G T \G S)^{-1}}
| F(\mu,\nu)(p,p_0') | \ot | F(\la,\nu)(p_0',p) |
\ot | F(\la,\nu)(q,p_0') |  \\
& \qqq \qq \ot | F(\mu,\nu)(p_0',q) | \ot |F(\mu,\nu)(q,p_0')| \ot |F(\mu,\nu)(p_0',q)| \\
& \q \to^{\T{id}^{\ot 4} \ot \psi'}
| F(\mu,\nu)(p,p_0') | \ot | F(\la,\nu)(p_0',p) |
\ot | F(\la,\nu)(q,p_0') |  \ot | F(\mu,\nu)(p_0',q) | \, ,
\end{split}
\]
where the stabilisation appearing in the middle adds the invertible operator $e_{12}(\pi_\la(1-p_0)) + e_{45}(\pi_\nu(1-p_0))$ (and similarly for the primed case). Using Lemma \ref{l:inverse} we may cancel out $F^{35}(q,p_0) \cd F^{35}(p_0,q)$ and $F^{35}(q,p_0') \cd F^{35}(p_0',q)$ from the above expression and we arrive exactly at the composition of isomorphisms given in Equation \eqref{eq:altenn}. This establishes our first claim.

The point of the alternative description of the binary change-of-base-point isomorphism given in Equation \eqref{eq:altenn} is that the Fredholm operators $F^{34}(p_0,q)$ and $F^{15}(q,p_0) \cd F^{25}(p_0, p)$ operate on different direct summands (and similarly for the primed versions). Hence, by the commutativity property of the torsion isomorphism (Proposition \ref{p:torsign}) we obtain that the isomorphism
\[
\begin{split}
& \G L_p(\mu,\nu) \ot \G L_q^\da(\mu,\nu) \ot \G L_p^\da(\la,\nu) \ot \G L_q(\la,\nu)
\to^{\epsilon}
\G L_p(\mu,\nu) \ot \G L_p^\da(\la,\nu) \ot \G L_q(\la,\nu) \ot \G L_q^\da(\mu,\nu) \\
& \q \to^{\G N(p_0,p_0')}
\G L_p'(\mu,\nu) \ot (\G L_p')^\da(\la,\nu) \ot \G L_q'(\la,\nu) \ot (\G L_q')^\da(\mu,\nu) \\
& \q \to^{\epsilon}
\G L_p'(\mu,\nu) \ot (\G L_q')^\da(\mu,\nu) \ot (\G L_p')^\da(\la,\nu) \ot \G L_q'(\la,\nu)
\end{split}
\]
agrees with the composition
\[
\begin{split}
& | F(\mu,\nu)(p,p_0)| \ot |F(\mu,\nu)(p_0,q)| \ot | F(\la,\nu)(p_0,p)| \ot | F(\la,\nu)(q,p_0)| \\
& \q \to^{\G T \G S}
\big| F^{15}(q,p_0) \cd F^{25}(p_0,p) \cd F^{34}(p_0,q) \cd F^{35}(p,p_0)\big| \\
& \q \to^{\G S^{-1} \G P \G S}
\big| F^{15}(q,p_0') \cd F^{25}(p_0',p) \cd F^{34}(p_0',q) \cd F^{35}(p,p_0') \big| \\
& \q \to^{(\G T \G S)^{-1}}
| F(\mu,\nu)(p,p_0')| \ot |F(\mu,\nu)(p_0',q)| \ot | F(\la,\nu)(p_0',p)| \ot | F(\la,\nu)(q,p_0')| \, ,
\end{split}
\]
where the stabilisations appearing in the middle add the invertible operators $e_{12}(\pi_\la(1-p_0)) + e_{45}(\pi_\nu(1-p_0))$ and $e_{12}(\pi_\la(1-p_0')) + e_{45}(\pi_\nu(1-p_0'))$, respectively. But this latter isomorphism can be seen to agree with the isomorphism $\G B(p_0,p_0') \ot \G B^\da(p_0,p_0')$ and the proposition is therefore proved.
\end{proof}

We end this subsection by providing a description of the ternary change-of-base-point isomorphism in terms of tensor products of the change-of-base-point isomorphism and the dual change-of-base-point isomorphism:

\begin{prop}\label{p:mulcha}
The isomorphism of $\zz$-graded complex lines
\[
\begin{split}
& \G B(p_0,p_0') \ot \G B(p_0,p_0') \ot \G B^\da(p_0,p_0') \\
& \q \colon \G L_p(\la,\mu) \ot \G L_q^\da(\la,\mu) \ot \G L_p(\mu,\nu) \ot \G L_q^\da(\mu,\nu) \ot \G L_p^\da(\la,\nu) \ot \G L_q(\la,\nu) \\
& \qq \to \G L_p'(\la,\mu) \ot (\G L_q')^\da(\la,\mu) \ot \G L_p'(\mu,\nu) \ot (\G L_q')^\da(\mu,\nu) \ot (\G L_p')^\da(\la,\nu) \ot \G L_q'(\la,\nu)
\end{split}
\]
agrees with the composition of isomorphisms of $\zz$-graded complex lines
\[
\begin{split}
& \G L_p(\la,\mu) \ot \G L_q^\da(\la,\mu) \ot \G L_p(\mu,\nu) \ot \G L_q^\da(\mu,\nu)
\ot \G L_p^\da(\la,\nu) \ot \G L_q(\la,\nu) \\
& \q \to^{\epsilon}
\G L_p(\la,\mu) \ot \G L_p(\mu,\nu) \ot \G L_p^\da(\la,\nu) \ot \G L_q(\la,\nu) \ot \G L_q^\da(\mu,\nu) \ot \G L_q^\da(\la,\mu)  \\
& \q \to^{\G M(p_0,p_0')}
\G L_p'(\la,\mu) \ot \G L_p'(\mu,\nu) \ot (\G L_p')^\da(\la,\nu)
\ot \G L_q'(\la,\nu) \ot (\G L_q')^\da(\mu,\nu) \ot (\G L_q')^\da(\la,\mu)  \\
& \q \to^{\epsilon}
\G L_p'(\la,\mu) \ot (\G L_q')^\da(\la,\mu) \ot \G L_p'(\mu,\nu) \ot (\G L_q')^\da(\mu,\nu)
\ot (\G L_p')^\da(\la,\nu) \ot \G L_q'(\la,\nu) \, .
\end{split}
\]
\end{prop}
\begin{proof}
We suppress the tuple of indices $(\la,\la,\la,\mu,\mu,\nu)$ from the notation.

Observe first that $\G M(p_0,p_0')$ agrees with the composition of isomorphisms
\begin{equation}\label{eq:mulbetI}
\begin{split}
& \G L_p(\la,\mu) \ot \G L_p(\mu,\nu) \ot \G L_p^\da(\la,\nu)
\ot \G L_q(\la,\nu) \ot \G L_q^\da(\mu,\nu)  \ot \G L_q^\da(\la,\mu) \\
& \q \to^{\T{id}^{\ot 5} \ot \varphi \ot \T{id}}
\G L_p(\la,\mu) \ot \G L_p(\mu,\nu) \ot \G L_p^\da(\la,\nu)
\ot \G L_q(\la,\nu) \ot \G L_q^\da(\mu,\nu) \\
& \qqq \qq \ot \G L_q^\da(\la,\mu) \ot \G L_q(\la,\mu) \ot \G L_q^\da(\la,\mu) \\
& \q \to^{L(\Om^{15}(p_0)) R(\Om^{15}(p_0))\G T \G S}
\big| \Om^{15}(p_0) \cd F^{25}(p_0,q) F^{14}(q,p_0) F^{14}(p_0,q) F^{56}(p_0,q) F^{26}(q,p_0) \\
& \qqq \qq \qq \cd F^{36}(p_0,p) F^{56}(p,p_0) F^{35}(p,p_0) \cd \Om^{15}(p_0) \big| \\
& \q \to^{\G S^{-1}\G P \G S}
\big| \Om^{15}(p_0') \cd F^{25}(p_0',q) F^{14}(q,p_0') F^{14}(p_0',q) F^{56}(p_0',q) F^{26}(q,p_0') \\
& \qqq \qq \cd F^{36}(p_0',p) F^{56}(p,p_0') F^{35}(p,p_0') \cd \Om^{15}(p_0') \big| \\
& \q \to^{(\G T \G S)^{-1} L(\Om^{15}(p_0')) R(\Om^{15}(p_0'))}
\G L_p'(\la,\mu) \ot \G L_p'(\mu,\nu) \ot (\G L_p')^\da(\la,\nu)
\ot \G L_q'(\la,\nu) \ot (\G L_q')^\da(\mu,\nu) \\
& \qqq \qq \qqq \ot (\G L_q')^\da(\la,\mu) \ot \G L_q'(\la,\mu) \ot (\G L_q')^\da(\la,\mu) \\
& \q \to^{\T{id}^{\ot 5} \ot (\varphi')^{-1} \ot \T{id}}
\G L_p'(\la,\mu) \ot \G L_p'(\mu,\nu) \ot (\G L_p')^\da(\la,\nu)
\ot \G L_q'(\la,\nu) \ot (\G L_q')^\da(\mu,\nu)  \ot (\G L_q')^\da(\la,\mu) \, ,
\end{split}
\end{equation}
where the stabilisations in the middle are adding the invertible operators $e_{11}(\pi_\la(1-p_0)) + e_{55}(\pi_\mu(1-p_0)) + e_{66}(\pi_\nu(1-p_0))$ and 
$e_{11}(\pi_\la(1-p_0')) + e_{55}(\pi_\mu(1-p_0')) + e_{66}(\pi_\nu(1-p_0'))$, respectively. Next, using that $F^{14}(p_0,q)$ and $F^{56}(p_0,q) F^{26}(q,p_0) \cd F^{36}(p_0,p) F^{56}(p,p_0)$ operate on different direct summands (and similarly for the primed version), we obtain from Proposition \ref{p:torsign} that the composition of isomorphisms in Equation \eqref{eq:mulbetI} agrees with the composition
\begin{equation}\label{eq:mulbetII}
\begin{split}
& \G L_p(\la,\mu) \ot \G L_p(\mu,\nu) \ot \G L_p^\da(\la,\nu)
\ot \G L_q(\la,\nu) \ot \G L_q^\da(\mu,\nu)  \ot \G L_q^\da(\la,\mu) \\
& \q \to^{\T{id}^{\ot 5} \ot \varphi \ot \T{id}}
\G L_p(\la,\mu) \ot \G L_p(\mu,\nu) \ot \G L_p^\da(\la,\nu)
\ot \G L_q(\la,\nu) \ot \G L_q^\da(\mu,\nu) \\
& \qqq \qq \ot \G L_q^\da(\la,\mu) \ot \G L_q(\la,\mu) \ot \G L_q^\da(\la,\mu) \\
& \q \to^{\epsilon}
\G L_p(\la,\mu) \ot \G L_q^\da(\la,\mu) \ot \G L_p(\mu,\nu) \ot \G L_p^\da(\la,\nu) \\
& \qqq \ot \G L_q(\la,\nu) \ot \G L_q^\da(\mu,\nu) \ot \G L_q(\la,\mu) \ot \G L_q^\da(\la,\mu) \\
& \q \to^{L(\Om^{15}(p_0)) R(\Om^{15}(p_0))\G T \G S}
\big| \Om^{15}(p_0) \cd F^{25}(p_0,q) F^{14}(q,p_0) F^{56}(p_0,q) F^{26}(q,p_0) \\
& \qqq \qq \qq \cd F^{36}(p_0,p) F^{56}(p,p_0) F^{14}(p_0,q) F^{35}(p,p_0) \cd \Om^{15}(p_0) \big| \\
& \q \to^{\G S^{-1}\G P \G S}
\big| \Om^{15}(p_0') \cd F^{25}(p_0',q) F^{14}(q,p_0') F^{56}(p_0',q) F^{26}(q,p_0') \\
& \qqq \q \cd F^{36}(p_0',p) F^{56}(p,p_0') F^{14}(p_0',q) F^{35}(p,p_0') \cd \Om^{15}(p_0') \big| \\
& \q \to^{(\G T \G S)^{-1}L(\Om^{15}(p_0')) R(\Om^{15}(p_0'))}
\G L_p'(\la,\mu) \ot (\G L_q')^\da(\la,\mu) \ot \G L_p'(\mu,\nu) \ot (\G L_p')^\da(\la,\nu) \\
& \qqq \qq \qqq \ot \G L_q'(\la,\nu) \ot (\G L_q')^\da(\mu,\nu) \ot \G L_q'(\la,\mu) \ot (\G L_q')^\da(\la,\mu) \\
& \q \to^{\epsilon}
\G L_p'(\la,\mu) \ot \G L_p'(\mu,\nu) \ot (\G L_p')^\da(\la,\nu)
\ot \G L_q'(\la,\nu) \\
& \qqq \ot (\G L_q')^\da(\mu,\nu) \ot (\G L_q')^\da(\la,\mu) \ot \G L_q'(\la,\mu) \ot (\G L_q')^\da(\la,\mu) \\
& \q \to^{\T{id}^{\ot 5} \ot (\varphi')^{-1} \ot \T{id}}
\G L_p'(\la,\mu) \ot \G L_p'(\mu,\nu) \ot (\G L_p')^\da(\la,\nu)
\ot \G L_q'(\la,\nu) \ot (\G L_q')^\da(\mu,\nu)  \ot (\G L_q')^\da(\la,\mu) \, .
\end{split}
\end{equation}
We then notice that the composition in Equation \eqref{eq:mulbetII} can be rewritten using the symmetrised and the binary change-of-base-point isomorphisms
\[
\begin{split}
\wit{\G B}(p_0,p_0') & \colon \G L_p(\la,\mu) \ot \G L_q^\da(\la,\mu) \ot \G L_q(\la,\mu) \ot \G L_q^\da(\la,\mu) \\
& \q \to \G L_p'(\la,\mu) \ot (\G L_q')^\da(\la,\mu) \ot \G L_q'(\la,\mu) \ot (\G L_q')^\da(\la,\mu) \q \T{and} \\
\G N(p_0,p_0') & \colon \G L_p(\mu,\nu) \ot \G L_p^\da(\la,\nu) \ot \G L_q(\la,\nu) \ot \G L_q^\da(\mu,\nu) \\
& \q \to \G L_p'(\mu,\nu) \ot (\G L_p')^\da(\la,\nu) \ot \G L_q'(\la,\nu) \ot (\G L_q')^\da(\mu,\nu)
\end{split}
\]
defined in Equation \eqref{eq:tilbet} and Equation \eqref{eq:ncha}. Indeed, using that $\wit{\G B}(p_0,p_0') = \wit{\G B}_+(p_0,p_0') \ot \wit{\G B}_-(p_0,p_0')$ already factorises as a tensor product of two isomorphisms, we obtain the following alternative description of the isomorphism in Equation \eqref{eq:mulbetII}:
\begin{equation}\label{eq:mulbetIII}
\begin{split}
& \G L_p(\la,\mu) \ot \G L_p(\mu,\nu) \ot \G L_p^\da(\la,\nu)
\ot \G L_q(\la,\nu) \ot \G L_q^\da(\mu,\nu)  \ot \G L_q^\da(\la,\mu) \\
& \q \to^{\T{id}^{\ot 5} \ot \varphi \ot \T{id}}
\G L_p(\la,\mu) \ot \G L_p(\mu,\nu) \ot \G L_p^\da(\la,\nu)
\ot \G L_q(\la,\nu) \ot \G L_q^\da(\mu,\nu) \\
& \qqq \qq \ot \G L_q^\da(\la,\mu) \ot \G L_q(\la,\mu) \ot \G L_q^\da(\la,\mu) \\
& \q \to^{\epsilon}
\G L_p(\la,\mu) \ot \G L_q^\da(\la,\mu) \ot \G L_q(\la,\mu) \ot \G L_q^\da(\la,\mu)
\ot \G L_p(\mu,\nu) \ot \G L_p^\da(\la,\nu) \ot \G L_q(\la,\nu) \ot \G L_q^\da(\mu,\nu) \\
& \q \to^{\wit{\G B}(p_0,p_0') \ot \G N(p_0,p_0')}
\G L_p'(\la,\mu) \ot (\G L_q')^\da(\la,\mu) \ot \G L_q'(\la,\mu) \ot (\G L_q')^\da(\la,\mu) \\
& \qqq \qqq \ot \G L_p'(\mu,\nu) \ot (\G L_p')^\da(\la,\nu) \ot \G L_q'(\la,\nu) \ot (\G L_q')^\da(\mu,\nu) \\
& \q \to^{\epsilon}
\G L_p'(\la,\mu) \ot \G L_p'(\mu,\nu) \ot (\G L_p')^\da(\la,\nu)
\ot \G L_q'(\la,\nu) \ot (\G L_q')^\da(\mu,\nu) \\
& \qqq \ot (\G L_q')^\da(\la,\mu) \ot \G L_q'(\la,\mu) \ot (\G L_q')^\da(\la,\mu) \\
& \q \to^{\T{id}^{\ot 5} \ot (\varphi')^{-1} \ot \T{id}}
\G L_p'(\la,\mu) \ot \G L_p'(\mu,\nu) \ot (\G L_p')^\da(\la,\nu)
\ot \G L_q'(\la,\nu) \ot (\G L_q')^\da(\mu,\nu)  \ot (\G L_q')^\da(\la,\mu) \, .
\end{split}
\end{equation}
The result of the present proposition now follows by an application of Proposition \ref{p:witbet} and Proposition \ref{p:merbet}. 
\end{proof}

\subsection{Functoriality}\label{ss:func}
Throughout this subsection, we let $\la,\mu,\nu \in \La$ be three indices and $p,q \in R$ be idempotents agreeing with the base points $p_0$ and $p_0'$ modulo the ideal $I \subseteq R$.

We are going to prove that the change-of-base-point isomorphism $\G B(p_0,p_0') \colon \G H(p,q) \to \G H'(p,q)$ respects the composition laws in the categories $\G H(p,q)$ and $\G H'(p,q)$. The main reason for the validity of this functoriality result is that a certain Fredholm determinant is equal to one and this is the statement of the next lemma. We introduce the representation
\[
\tau := \pi_\la \op \pi_\la \op \pi_\la \op \pi_\mu \op \pi_\nu \colon R \to \sL( \C H^{\op 5})
\]
and suppress the tuple of indices $(\la,\la,\la,\mu,\nu)$ from the notation.

\begin{lemma}\label{l:crux}
The invertible operator
\begin{equation}\label{eq:crux}
\begin{split}
& \Big( \Om^{14}(p_0')\Om^{24}(p_0') \Om^{45}(p_0') \Om^{25}(p_0') \Om^{35}(p_0') \Om^{45}(p_0') \Om^{34}(p_0')\Om^{14}(p_0') + \tau(1 - p_0') \Big) \\
& \q \cd
\Big( \Om^{14}(p_0) \Om^{24}(p_0) \Om^{45}(p_0) \Om^{25}(p_0) \Om^{35}(p_0) \Om^{45}(p_0) \Om^{34}(p_0) \Om^{14}(p_0) + \tau(1 - p_0) \Big)^{-1} \\
& \qq \colon \tau(1)\C H^{\op 5} \to \tau(1)\C H^{\op 5}
\end{split}
\end{equation}
is of determinant class and the determinant is equal to one.
\end{lemma}
\begin{proof}
We introduce the invertible operators acting on the Hilbert space $\tau(1)\C H^{\op 5}$:
\[
\begin{split}
& \Om^{1345}(p_0) := \Om^{14}(p_0)\Om^{45}(p_0) \Om^{35}(p_0) \Om^{45}(p_0) \Om^{34}(p_0)\Om^{14}(p_0) + \tau(1 - p_0)  \q \T{and} \\
& \Om^{1245}(p_0) := \Om^{14}(p_0)\Om^{45}(p_0) \Om^{25}(p_0) \Om^{45}(p_0) \Om^{24}(p_0)\Om^{14}(p_0) + \tau(1 - p_0) \, .
\end{split}
\]
A similar notation applies when $p_0$ is replaced by $p_0'$ or by the unit $1 \in R$. Using Lemma \ref{l:reduction} we then observe that
\[
\begin{split}
& \Big( \Om^{14}(p_0')\Om^{24}(p_0') \Om^{45}(p_0') \Om^{25}(p_0') \Om^{35}(p_0') \Om^{45}(p_0') \Om^{34}(p_0')\Om^{14}(p_0') + \tau(1 - p_0') \Big) \\
& \q \cd
\Big( \Om^{14}(p_0) \Om^{24}(p_0) \Om^{45}(p_0) \Om^{25}(p_0) \Om^{35}(p_0) \Om^{45}(p_0) \Om^{34}(p_0) \Om^{14}(p_0) + \tau(1 - p_0) \Big)^{-1} \\
& \q = \big( \Om^{1245}(p_0') \big)^{-1} \cd \Om^{1345}(p_0') \cd \big( \Om^{1345}(p_0) \big)^{-1} \cd \Om^{1245}(p_0) \, .
\end{split}
\]
To show that the invertible operator in Equation \eqref{eq:crux} is of determinant class it thus suffices to check that
\[
\Om^{1345}(p_0) - \Om^{1345}(p_0') \, \, \T{and} \, \, \, \Om^{1245}(p_0) - \Om^{1245}(p_0') \in \sL^1\big(\tau(1)\C H^{\op 5}\big) \, .
\]

Since $\Om^{1345}$ agrees with $\Om^{1245}$ up to conjugation by the permutation matrix 
\[
\Si := \pi_\la(1) \op \ma{cc}{0 & \pi_\la(1) \\ \pi_\la(1) & 0} \op \pi_\mu(1) \op \pi_\nu(1)
\]
we may focus on the first of these two differences. We compute modulo trace class operators, using Assumption \ref{a:rep}, Lemma \ref{l:commut} and Lemma \ref{l:reduction} together with the fact that $p_0 - p_0' \in I$:
\[
\begin{split}
& \Om^{1345}(p_0) - \Om^{1345}(p_0')
\sim_1 \tau(p_0' - p_0) + \tau(p_0 - p_0') \cd \Om^{1345}(1)  \\
& \q \sim_1
\tau(p_0' - p_0) \\
& \qq +
\tau(p_0 - p_0') \cd \Om^{14}(1) \left( \big( \pi_\la(1)^{\op 3} \op \ma{cc}{  0 & \pi_\mu(1) \pi_\nu(1) \\ \pi_\nu(1) \pi_\mu(1) & 0} \big) \right. \\
& \qqq \cd \big( \pi_\la(1)^{\op 2} \op \ma{ccc}{0 & 0 & \pi_\la(1) \pi_\nu(1) \\ 0 & \pi_\mu(1) & 0 \\ \pi_\nu(1) \pi_\la(1) & 0 & 0} \big) \\
& \qqqq \cd \big( \pi_\la(1)^{\op 3} \op \ma{cc}{0 & \pi_\mu(1) \pi_\nu(1) \\ \pi_\nu(1) \pi_\mu(1) & 0} \big) \\
& \qqqq \q \left. \cd \big( \pi_\la(1)^{\op 2} \op \ma{cc}{0 & \pi_\la(1) \pi_\mu(1) \\ \pi_\mu(1) \pi_\la(1) & 0} \op \pi_\nu(1) \big) \right) \Om^{14}(1)  \\
& \q \sim_1
\tau(p_0' - p_0) +
\tau(p_0 - p_0') \cd \Om^{14}(1) \cd \big(\pi_\la(1)^{\op 3} \op \pi_\mu(1) \op \pi_\nu(1) \big) \cd \Om^{14}(1)
= 0 \, .
\end{split}
\]

Now, to show that the determinant of the invertible operator in Equation \eqref{eq:crux} is equal to one, we use that the Fredholm determinant is multiplicative and invariant under conjugation, to compute that
\[
\begin{split}
& \det\Big(  \big( \Om^{1245}(p_0') \big)^{-1} \cd \Om^{1345}(p_0') \cd \big( \Om^{1345}(p_0) \big)^{-1} \cd \Om^{1245}(p_0) \Big) \\
& \q = \det\Big( \Om^{1345}(p_0') \cd \big( \Om^{1345}(p_0) \big)^{-1} \Big)
\cd \det\Big( \Om^{1245}(p_0) \cd \big( \Om^{1245}(p_0') \big)^{-1} \Big) \\
& \q =  \det\Big( \Om^{1345}(p_0') \cd \big( \Om^{1345}(p_0) \big)^{-1} \Big)
\cd \det\Big( \Si \cd \Om^{1345}(p_0) \cd \big( \Om^{1345}(p_0') \big)^{-1} \cd \Si  \Big) \\
& \q = 1 \, . \qedhere
\end{split}
\]
\end{proof}

To ease the notation, we define the $\zz$-graded complex lines
\[
\begin{split}
& \G A_p(\la,\mu,\nu) := \G L_p(\la,\mu) \ot \G L_p(\mu,\nu)  \ot \G L_p^\da(\la,\nu) \q \T{and} \\
& \G A_q^\da(\la,\mu,\nu) := \G L_q(\la,\nu) \ot \G L_q^\da(\mu,\nu)  \ot \G L_q^\da(\la,\mu) \, .
\end{split}
\]
A similar notation applies with $p_0$ replaced by $p_0'$ and hence $\G L$ replaced by $\G L'$ and $\G A$ replaced by $\G A'$.

Recall the construction of the trivialisations $\mu_p \colon \G A_p(\la,\mu,\nu) \to (\cc,0)$ and $\mu_q^\da \colon \G A_q^\da(\la,\mu,\nu) \to (\cc,0)$ from Equation \eqref{eq:comp} and Equation \eqref{eq:dualcomp}. Recall also the construction of the ternary change-of-base-point isomorphism $\G M(p_0,p_0') \colon \G A_p(\la,\mu,\nu) \ot \G A_q^\da(\la,\mu,\nu) \to \G A_p'(\la,\mu,\nu) \ot (\G A_q')^\da(\la,\mu,\nu)$ from Equation \eqref{eq:terbase}.

The next result relates the ternary change-of-base-point isomorphism to the compositions in the categories $\G H(p,q)$ and $\G H'(p,q)$.

\begin{prop}\label{p:crux}
We have the identity
\[
\G M(p_0,p_0') = \big( \mu_p' \ot (\mu_q')^\da \big)^{-1} \ci ( \mu_p \ot \mu_q^\da )
\colon \G A_p(\la,\mu,\nu) \ot \G A_q^\da(\la,\mu,\nu) \to \G A_p'(\la,\mu,\nu) \ot (\G A_q')^\da(\la,\mu,\nu) \, .
\]
of isomorphisms of $\zz$-graded complex lines.
\end{prop}
\begin{proof}
We suppress the tuple of indices $(\la,\la,\la,\mu,\nu)$.

Notice first that the trivialisation $\mu_p \ot \mu_q^\da$ can be rewritten as the composition
\[
\begin{split}
& \G A_p(\la,\mu,\nu) \ot \G A_q^\da(\la,\mu,\nu) \\
& \q \to^{L(\Om^{14}(p_0))R(\Om^{14}(p_0))\G T \G S}
\big| \Om^{14}(p_0) \cd F^{24}(p_0,q) F^{45}(p_0,q) F^{25}(q,p_0) \\
& \qqq \qqq \qq \cd F^{35}(p_0,p) F^{45}(p,p_0) F^{34}(p,p_0) \cd \Om^{14}(p_0)\big| \\
& \q \to^{\G P \G S}
\big| \Om^{14}(p_0) \cd \Om^{24}(p_0) \Om^{45}(p_0) \Om^{25}(p_0) \cd \Om^{35}(p_0) \Om^{45}(p_0) \Om^{34}(p_0) \cd \Om^{14}(p_0) + \tau(1 - p_0) \big| = (\cc,0) \, ,
\end{split}
\]
where the last stabilisation uses the invertible operator $\pi_\la(1 - p_0) \op \pi_\la(1 - q) \op \pi_\la(1 - p) \op \pi_\mu(1 - p_0) \op \pi_\nu(1 - p_0)$. Since a similar description holds for the trivialisation $\mu_p' \ot (\mu_q')^\da$, we see that the ternary change-of-base-point isomorphism
\[
\G M(p_0,p_0') \colon \G A_p(\la,\mu,\nu) \ot \G A_q^\da(\la,\mu,\nu) \to \G A_p'(\la,\mu,\nu) \ot (\G A_q')^\da(\la,\mu,\nu)
\]
agrees with the composition
\[
\big( \mu_p' \ot (\mu_q')^\da \big)^{-1} \ci ( \mu_p \ot \mu_q^\da ) \colon
\G A_p(\la,\mu,\nu) \ot \G A_q^\da(\la,\mu,\nu) \to \G A_p'(\la,\mu,\nu) \ot (\G A_q')^\da(\la,\mu,\nu)
\]
up to the Fredholm determinant of the quotient
\[
\begin{split}
& \Big( \Om^{14}(p_0')\Om^{24}(p_0') \Om^{45}(p_0') \Om^{25}(p_0') \Om^{35}(p_0') \Om^{45}(p_0') \Om^{34}(p_0')\Om^{14}(p_0') + \tau(1 - p_0') \Big) \\
& \q \cd
\Big( \Om^{14}(p_0) \Om^{24}(p_0) \Om^{45}(p_0) \Om^{25}(p_0) \Om^{35}(p_0) \Om^{45}(p_0) \Om^{34}(p_0) \Om^{14}(p_0) + \tau(1 - p_0) \Big)^{-1}
\end{split}
\]
see the description of the perturbation isomorphism in Example \ref{ex:invpert}. But this Fredholm determinant is equal to one by Lemma \ref{l:crux}.
\end{proof}

We may now combine the results achieved in Subsection \ref{ss:binter} with the results of this subsection and obtain a proof of the functoriality of the change-of-base-point isomorphism (notice that we already established the unitality condition in Section \ref{s:change}).

\begin{proof}[Proof of Proposition \ref{p:func}]
Recall from Definition \ref{d:hopf} that the multiplication operator
\[
\G M_{p,q} \colon \G L_p(\la,\mu) \ot \G L_q^\da(\la,\mu) \ot \G L_p(\mu,\nu) \ot \G L_q^\da(\mu,\nu) \to \G L_p(\la,\nu) \ot \G L_q^\da(\la,\nu)
\]
is defined as the composition
\[
\begin{split}
\G M_{p,q} & \colon \G L_p(\la,\mu) \ot \G L_q^\da(\la,\mu) \ot \G L_p(\mu,\nu) \ot \G L_q^\da(\mu,\nu) \\
& \q \to^{\T{id}^{\ot 4} \ot \varphi \ot \varphi}
\G L_p(\la,\mu) \ot \G L_q^\da(\la,\mu) \ot \G L_p(\mu,\nu) \ot \G L_q^\da(\mu,\nu)
\ot \G L_p^\da(\la,\nu) \ot \G L_p(\la,\nu) \ot \G L_q^\da(\la,\nu) \ot \G L_q(\la,\nu) \\
& \q \to^{\epsilon}
\G A_p(\la,\mu,\nu) \ot \G A_q^\da(\la,\mu,\nu) \ot \G L_p(\la,\nu) \ot \G L_q^\da(\la,\nu)
\to^{\mu_p \ot \mu_q^\da \ot \T{id}^{\ot 2}} \G L_p(\la,\nu) \ot \G L_q^\da(\la,\nu) \, .
\end{split}
\]
Using Proposition \ref{p:crux} and Proposition \ref{p:mulcha}, we may rewrite this multiplication operator as the composition
\[
\begin{split}
& \G L_p(\la,\mu) \ot \G L_q^\da(\la,\mu) \ot \G L_p(\mu,\nu) \ot \G L_q^\da(\mu,\nu) \\
& \q \to^{\T{id}^{\ot 4} \ot \varphi \ot \varphi}
\G L_p(\la,\mu) \ot \G L_q^\da(\la,\mu) \ot \G L_p(\mu,\nu) \ot \G L_q^\da(\mu,\nu)
\ot \G L_p^\da(\la,\nu) \ot \G L_p(\la,\nu) \ot \G L_q^\da(\la,\nu) \ot \G L_q(\la,\nu) \\
& \q \to^{\epsilon}
\G L_p(\la,\mu) \ot \G L_q^\da(\la,\mu) \ot \G L_p(\mu,\nu) \ot \G L_q^\da(\mu,\nu)
\ot \G L_p^\da(\la,\nu) \ot \G L_q(\la,\nu) \ot \G L_p(\la,\nu) \ot \G L_q^\da(\la,\nu) \\
& \q \to^{\G B \ot \G B \ot \G B^\da \ot \T{id}^{\ot 2}}
\G L_p'(\la,\mu) \ot (\G L_q')^\da(\la,\mu) \ot \G L_p'(\mu,\nu) \ot (\G L_q')^\da(\mu,\nu)
\ot (\G L_p')^\da(\la,\nu) \ot \G L_q'(\la,\nu) \\ 
& \qqq \qqq \ot \G L_p(\la,\nu) \ot \G L_q^\da(\la,\nu) \\
& \q \to^{\epsilon}
\G A_p'(\la,\mu,\nu) \ot (\G A_q')^\da(\la,\mu,\nu) \ot \G L_p(\la,\nu) \ot \G L_q^\da(\la,\nu)
\to^{\mu_p' \ot (\mu_q')^\da \ot \T{id}^{\ot 2}} \G L_p(\la,\nu) \ot \G L_q^\da(\la,\nu) \, .
\end{split}
\]
To continue, we combine the above description of the multiplication operator with Proposition \ref{p:dagger} to see that the composition
\[
\G B(p_0,p_0') \ci \G M_{p,q} \colon \G L_p(\la,\mu) \ot \G L_q^\da(\la,\mu) \ot \G L_p(\mu,\nu) \ot \G L_q^\da(\mu,\nu)
\to \G L_p'(\la,\nu) \ot (\G L_q')^\da(\la,\nu)
\]
can be rewritten as the composition
\[
\begin{split}
& \G L_p(\la,\mu) \ot \G L_q^\da(\la,\mu) \ot \G L_p(\mu,\nu) \ot \G L_q^\da(\mu,\nu) \\
& \q \to^{\G B(p_0,p_0') \ot \G B(p_0,p_0')}
\G L_p'(\la,\mu) \ot (\G L_q')^\da(\la,\mu) \ot \G L_p'(\mu,\nu) \ot (\G L_q')^\da(\mu,\nu) \\
& \q \to^{\T{id}^{\ot 4} \ot \varphi' \ot \varphi'} \G L_p'(\la,\mu) \ot (\G L_q')^\da(\la,\mu) \ot \G L_p'(\mu,\nu) \ot (\G L_q')^\da(\mu,\nu) \\
& \qqq \qqq \ot (\G L_p')^\da(\la,\nu) \ot \G L_p'(\la,\nu) \ot (\G L_q')^\da(\la,\nu) \ot \G L_q'(\la,\nu) \\
& \q \to^{\epsilon}
\G A_p'(\la,\mu,\nu) \ot (\G A_q^\da)'(\la,\mu,\nu) \ot \G L_p'(\la,\nu) \ot (\G L_q')^\da(\la,\nu)  \\
& \q \to^{\mu_p' \ot (\mu_q')^\da \ot \T{id}^{\ot 2}} \G L_p'(\la,\nu) \ot (\G L_q')^\da(\la,\nu) \, .
\end{split}
\]
But this is exactly the composition $\G{M}'_{p,q}\ci (\G B(p_0,p_0') \ot \G B(p_0,p_0'))$ and the theorem is therefore proved.
\end{proof}

\bibliographystyle{amsalpha-lmp}

\providecommand{\bysame}{\leavevmode\hbox to3em{\hrulefill}\thinspace}
\providecommand{\MR}{\relax\ifhmode\unskip\space\fi MR }
\providecommand{\MRhref}[2]{%
  \href{http://www.ams.org/mathscinet-getitem?mr=#1}{#2}
}
\providecommand{\href}[2]{#2}

\end{document}